\documentclass[11pt,twoside]{amsart}
\usepackage[english]{babel}
\usepackage{amsmath,amsfonts,amssymb,amsthm}
\usepackage{yfonts}
\usepackage[T1]{fontenc}
\usepackage[latin1]{inputenc}

\usepackage{enumerate}
\usepackage{verbatim}
\usepackage{graphicx}
\usepackage{pinlabel}
\usepackage{tikz}
\usepackage{fp}
\usetikzlibrary{calc}
\usetikzlibrary{shapes,arrows}
\usetikzlibrary{fit,positioning}
\usetikzlibrary{patterns,decorations.pathreplacing}

\usepackage{cancel}
\usepackage{comment}
\usepackage{mathdots}
\usepackage{mathrsfs}

\usepackage{enumitem}

\makeatletter
\def\namedlabel#1#2{\begingroup
    #2%
    \def\@currentlabel{#2}%
    \phantomsection\label{#1}\endgroup
}
\makeatother

\numberwithin{equation}{section}

\renewcommand{\>}{\rangle}
\newcommand{\ov}{\overline}
\renewcommand{\sf}{\mathsf}

\newcommand{\de}{\partial}
\newcommand{\di}{\partial_{\infty}}

\newcommand{\calL}{\mathcal L}

\newcommand{\calO}{\mathcal O}
\newcommand{\calT}{\mathcal{T}}
\newcommand{\sfG}{\mathsf{G}}
\newcommand{\sfg}{\mathsf{g}}
\newcommand{\bB}{\mathbb{B}}
\newcommand{\bD}{\mathbb{D}}

\newcommand{\SO}{{\rm SO}}
\newcommand{\PO}{{\rm PO}}
\newcommand{\OO}{{\rm O}}
\newcommand{\SL}{{\rm SL}}
\newcommand{\GL}{{\rm GL}}
\newcommand{\PGL}{{\rm PGL}}
\newcommand{\PSL}{{\rm PSL}}
\newcommand{\Hom}{{\rm Hom}}

\newcommand{\spa}{{\rm span}}

\newcommand{\hcc}{\H^{p,q}}

\newcommand{\g}{\gamma}

\renewcommand{\k}{\kappa}

\renewcommand{\L}{\Lambda}

\newcommand{\G}{\Gamma}
\newcommand{\Om}{\Omega}

\renewcommand{\r}{\rho}
\renewcommand{\g}{\gamma}

\newcommand{\N}{\mathbb N}
\newcommand{\Z}{\mathbb Z}
\newcommand{\R}{\mathbb R}
\newcommand{\C}{\mathbb C}

\renewcommand{\P}{\mathbb P}
\renewcommand{\H}{\mathbb H}
\newcommand{\bS}{\mathbb S}
\newcommand{\sfb}{\mathsf{b}}
\newcommand{\sfC}{\mathsf{C}}

\newcommand{\vcd}{\mathrm{vcd}}
\newcommand{\gph}{\mathrm{graph}}

\newcommand{\ie}{i.e.\ }
\newcommand{\eg}{e.g.\ }
\newcommand{\resp}{resp.\ }

\definecolor{dark-gray}{gray}{0.40}

\theoremstyle{plain}
\newtheorem{thm}{Theorem}[section]
\newtheorem{lem}[thm]{Lemma}
\newtheorem{prop}[thm]{Proposition}
\newtheorem{cor}[thm]{Corollary}
\newtheorem{fact}[thm]{Fact}

\newtheorem{notation}[thm]{Notation}

\theoremstyle{definition}
\newtheorem{example}[thm]{Example}
\newtheorem{examples}[thm]{Examples}
\newtheorem{defn}[thm]{Definition}
\newtheorem{remark}[thm]{Remark}
\newtheorem{remarks}[thm]{Remarks}

\title[$\H^{p,q}$-convex cocompactness and higher higher Teichm\"uller spaces]{$\H^{p,q}$-convex cocompactness and\\ higher higher Teichm\"uller spaces}

\date{\today}

\author{Jonas Beyrer}
\address{LMU, Mathematisches Institut, Theresienstrasse 39, 80333 M\"unchen, Germany}
\email{beyrer@math.lmu.de}

\author{Fanny Kassel}
\address{CNRS and Laboratoire Alexander Grothendieck, Institut des Hautes \'Etudes Scientifiques, Universit\'e Paris-Saclay, 35 route de Chartres, 91440 Bures-sur-Yvette, France}
\email{kassel@ihes.fr}

\thanks{This project received funding from the European Research Council (ERC) under the European Union's Horizon 2020 research and innovation programme (ERC starting grant DiGGeS, grant agreement No 715982). J.B was supported by the Schweizerischer Nationalfonds (SNF, Swiss Research Foundation), P2ZHP2 184022 (Early Postdoc.Mobility)}

\begin{document}

\begin{abstract}
For any integers $p\geq 2$ and $q\geq 1$, let $\H^{p,q}$ be the pseudo-Riemannian hyperbolic space of signature $(p,q)$.
We prove that if $\Gamma$ is the fundamental group of a closed aspherical $p$-manifold, then the set of representations of $\Gamma$ to $\PO(p,q+1)$ which are convex cocompact in $\H^{p,q}$ is a union of connected components of $\Hom(\Gamma,\PO(p,q+1))$.
More generally, we show that if $\Gamma$ is any finitely generated group with no infinite nilpotent normal subgroups and with virtual cohomological dimension~$p$, then the set of injective and discrete representations of $\Gamma$ to $\PO(p,q+1)$ preserving a non-degenerate non-positive $(p-1)$-sphere in the boundary of $\H^{p,q}$ is a union of connected components of $\Hom(\Gamma,\PO(p,q+1))$.
This gives new examples of higher-dimensional higher-rank Teichm\"uller spaces.
\end{abstract}

\maketitle
\tableofcontents

\section{Introduction} 

Teichm\"uller theory has a long and rich history.
The Teichm\"uller space of a closed surface $S$ of genus $\geq 2$ is a fundamental object in many areas of mathematics.
It can be viewed both as a moduli space for marked complex structures on~$S$ or, via the Uniformization Theorem, as a moduli space for marked hyperbolic structures on~$S$.
In this second point of view, the holonomy representation of the fundamental group $\pi_1(S)$ naturally realizes the Teichm\"uller space of~$S$ as a connected component of the $G$-character variety of $\pi_1(S)$ for $G=\PSL(2,\R)$, corresponding to the image, modulo conjugation by~$G$ at the target, of a connected component of $\Hom(\pi_1(S),G)$ consisting entirely of injective and discrete representations.
Here $\Hom(\pi_1(S),G)$ denotes the set of representations of $\pi_1(S)$ to~$G$, endowed with the compact-open topology.

An interesting and perhaps surprising phenomenon, which has led to a considerable amount of research in the past twenty years, is that for certain real semi-simple Lie groups $G$ of higher real rank such as $\SL(n+1,\R)$, $\mathrm{Sp}(2n,\R)$, or $\OO(2,n)$ for $n\geq 2$, there also exist connected components of $\Hom(\pi_1(S),G)$ consisting entirely of injective and discrete representations, and which are non-trivial in the sense that they are not reduced to a single representation and its conjugates by~$G$.
The images in the $G$-character variety of these components are now called \emph{higher(-rank) Teichm\"uller spaces}.
By work of Choi--Goldman \cite{cg05} (for $G=\SL(3,\R)$), Fock--Goncharov \cite{fg06}, and Labourie \cite{lab06}, these higher-rank Teichm\"uller spaces include \emph{Hitchin components} when $G$ is a real split simple Lie group, which were first investigated in the pioneering work of Hitchin \cite{hit92}.
By work of Burger--Iozzi--Wienhard \cite{biw10}, higher-rank Teichm\"uller spaces also include \emph{maximal components} when $G$ is a real simple Lie group of Hermitian type, where the \emph{Toledo invariant} (a topological invariant generalizing the Euler number) is maximized.
See \cite{biw14} for details.
More recently, new higher-rank Teichm\"uller spaces were discovered in \cite{bp,bcggo,glw,bglpw}, consisting of so-called \emph{$\Theta$-positive} representations of surface groups introduced by Guichard--Wienhard \cite{gw16,gw22}, when $G = \OO(p,q)$ for $p\neq q$ or $G$ is an exceptional simple real Lie group whose restricted root system is of type $F_4$; this conjecturally gives the full list of higher-rank Teichm\"uller spaces, see \cite{gw16}.
The richness of this very active \emph{higher-rank Teichm\"uller theory} (see \cite{biw14,wie-icm,poz-bourbaki}) comes from the many similarities between higher-rank Teichm\"uller spaces and the classical Teichm\"uller space of~$S$: see for instance \cite{hit92,cg05,fg06,lab06,lm09,biw10,gw12,gmn13,gmn14,bd14,bcls15,z15,lz17,bp17,bp,bcggo,glw,bglpw}.

It is natural to try to generalize the theory further by considering what could be called \emph{higher-dimensional higher-rank Teichm\"uller theory} (or \emph{higher higher Teichm\"uller theory} for short); this idea appears for instance in Wienhard's ICM survey \cite[\S\,14]{wie-icm}.
The goal would be, for certain closed topological manifolds $N$ of dimension $n>2$, to find connected components of representations from $\pi_1(N)$ to higher-rank semi-simple Lie groups~$G$ which consist entirely of injective and discrete representations and are non-trivial as above, and to study their images in the $G$-character variety of $\pi_1(N)$.
We are particularly interested in such \emph{higher higher Teichm\"uller spaces} for which the connected component in $\Hom(\pi_1(N),G)$ contains a representation with Zariski-dense image in~$G$.\footnote{By contrast, there exist rigid situations where all representations in the connected component factor through a semi-simple Lie group of real rank one inside~$G$ (see \eg \cite{poz15,km17}), or where the only non-trivial way to deform a representation inside~$G$ is through a compact group centralizing its image (see \eg \cite{kkp12,kli11}).}
To our knowledge, for simple~$G$ such higher higher Teichm\"uller spaces have so far been shown to exist only for $G = \PGL(n+1,\R)$ or $G = \OO(n,2)$ (see Section~\ref{subsec:intro-higher-higher-Teich} just below).
In the present paper we give new examples where $G$ can be any indefinite orthogonal group $\OO(p,q+1)$ with $p\geq 2$ and $q\geq 1$.

\subsection{Previously known higher higher Teichm\"uller spaces and new examples} \label{subsec:intro-higher-higher-Teich}

Firstly, in the context of real projective geometry, Benoist \cite{ben05} proved that for any closed topological manifold $N$ of dimension $n\geq 2$, if the fundamental group $\pi_1(N)$ has no infinite nilpotent normal subgroups, then the set of injective and discrete representations from $\pi_1(N)$ to $G := \PGL(n+\nolinebreak 1,\R)$ which are holonomies of convex projective structures on~$N$ is closed in $\Hom(\pi_1(N),G)$.
This set is also open in $\Hom(\pi_1(N),G)$ by earlier work of Koszul \cite{kos68}, and so it is a union of connected components of $\Hom(\pi_1(N),G)$.

In many examples such connected components are non-trivial, and so their images in the $G$-character variety of $\pi_1(N)$ are \emph{higher higher Teichm\"uller spaces} in the sense above.
This is the case for instance if $N$ is a closed real hyperbolic $n$-manifold with a closed totally geodesic embedded hypersurface: then bending \`a la Johnson--Millson \cite{jm87} allows to continuously deform
the natural inclusion $\pi_1(N) \hookrightarrow \mathrm{Isom}(\H^n) = \PO(n,1) \hookrightarrow \PGL(n+\nolinebreak 1,\R) = G$ (given by the holonomy representation) into representations with Zariski-dense image inside $\Hom(\pi_1(N),G)$.
Interesting examples are also obtained using Vinberg's theory \cite{vin71} of linear reflection groups: see \cite{ben08,clm18}.

We note that closedness in the case that $N$ is a closed real hyperbolic $n$-manifold was first proved by Choi--Goldman \cite{cg05} for $n=2$ and by  Kim \cite{kim01} for $n=3$.
See also \cite{ct-expository} for an expository proof when $\pi_1(N)$ is Gromov hyperbolic.
Benoist's result was extended by Marseglia \cite{mar-PhD} to holonomies of convex projective structures of finite Busemann volume on~$N$ when $N$ is not necessarily closed, and by Cooper--Tillmann \cite{ct} to holonomies of certain convex projective structures with generalized cusps (recovering also Benoist's result \cite{ben05} with a different proof).

Secondly, higher higher Teichm\"uller spaces were found in the context of anti-de Sitter geometry (\ie Lorentzian geometry of constant negative curvature).
Indeed, Mess \cite{mes90} (for $n=2$) and Barbot \cite{bar15} (for general $n\geq 2$) proved that for any closed real hyperbolic $n$-manifold $N$ with holonomy representation $\sigma_0 : \pi_1(N)\to\mathrm{Isom}(\H^n) = \PO(n,1) = \OO(n,1)/\{\pm\mathrm{I}\}$, the connected component of (a lift of) $\sigma_0$ composed with the natural inclusion $\OO(n,1)\hookrightarrow G:=\PO(n,2)$ in $\Hom(\pi_1(N),G)$ consists entirely of injective and discrete representations. 
Again, such connected components are non-trivial, and in fact contain representations with Zariski-dense image in~$G$, as soon as $N$ admits a closed totally geodesic embedded hypersurface, using bending: see \cite[\S\,6]{kas12}.
(For $n=3$, see also \cite[Th.\,1.20]{mst} for non-trivial deformations obtained by combining several bendings along intersecting hypersurfaces, in some special situation.)

One of the results of the present paper is the following generalization of Barbot's result to pseudo-Riemannian geometry of any signature $(p,q)$, where $G = \PO(p,q+1) = \OO(p,q+\nolinebreak 1)/\{\pm\mathrm{I}\}$.

\begin{thm} \label{thm:deform-Fuchsian-inj-discr}
Let $p\geq 2$ and $q\geq 1$ be integers.
Let $N$ be a closed real hyperbolic $p$-manifold with holonomy representation $\sigma_0 : \pi_1(N)\to\PO(p,1)$, and let $\rho_0 : \pi_1(N)\to\PO(p,q+1)$ be the composition of a lift of $\sigma_0$ to $\OO(p,1)$ with the natural inclusion $\OO(p,1)\hookrightarrow\PO(p,q+1)$.
Then the connected component of $\rho_0$ in $\Hom(\pi_1(N),\PO(p,q+1))$ consists entirely of injective and discrete representations.
\end{thm}

Using bending, the connected components in Theorem~\ref{thm:deform-Fuchsian-inj-discr} are non-trivial, and in fact contain representations with Zariski-dense image in $G=\PO(p,q+1)$ (or $\mathrm{PSO}(p,q+1)$), as soon as $N$ admits $q$ closed totally geodesic embedded hypersurfaces which pairwise do not intersect: see Appendix~\ref{appendix}.

Our results are actually more precise and general than Theorem~\ref{thm:deform-Fuchsian-inj-discr}, as we now explain.

\subsection{$\H^{p,q}$-convex cocompact representations}

Convex cocompactness is a classical notion in the theory of Kleinian groups: recall that a discrete subgroup $\Gamma$ of the group $\PO(p,1)$ of isometries of the real hyperbolic space~$\H^p$ is said to be \emph{convex cocompact} if it acts properly discontinuously and cocompactly on some non-empty closed convex subset $\mathscr{C}$ of~$\H^p$.
In that case $\Gamma$ is finitely generated and \emph{Gromov hyperbolic}: triangles in its Cayley graph (with respect to any given finite generating subset) are uniformly thin.
Convex cocompact subgroups of $\PO(p,1)$ admit various geometric, topological, and dynamical characterizations that make them a particularly rich and interesting class of discrete~subgroups.

Inspired by work of Mess \cite{mes90} and Barbot--M\'erigot \cite{bm12,bar15} for the anti-de Sitter space $\mathrm{AdS}^{p+1} = \H^{p,1}$, the notion of convex cocompactness was extended in \cite{dgk18} to the setting of discrete subgroups of the isometry group $\PO(p,q+1)$ of the \emph{pseudo-Riemannian hyperbolic space} $\H^{p,q}$ of signature $(p,q)$ for any integers $p\geq 1$ and $q\geq 0$.
By definition, $\H^{p,q}$ is the pseudo-Riemannian symmetric space $\PO(p,q+1)/\OO(p,q)$, which can be realized as an open set in projective space as
$$\H^{p,q} = \big\{[v]\in\P(\R^{p+q+1}) ~|~ \sfb(v,v)<0\big\}$$
where $\sfb$ is any non-degenerate symmetric bilinear form of signature $(p,q+1)$ on $\R^{p+q+1}$.
For $q=0$, this is the real hyperbolic space~$\H^p$.
For $q=1$, this is the $(p+1)$-dimensional \emph{anti-de Sitter space} $\mathrm{AdS}^{p+1}$, which is a Lorentzian counterpart of $\H^p$.
In general, $\H^{p,q}$ has a natural pseudo-Riemannian structure of signature $(p,q)$ with isometry group $\PO(p,q+1)$, induced by the symmetric bilinear form~$\sfb$ (see Section~\ref{subsec:Hpq}); it has constant negative sectional curvature, which makes it a pseudo-Riemannian analogue of~$\H^p$. 
Similarly to~$\H^p$, the space $\H^{p,q}$ has a natural \emph{boundary at infinity}
$$\di \H^{p,q} = \big\{[v]\in\P(\R^{p+q+1}) ~|~ \sfb(v,v)=0\big\}.$$

Recall that a subset of projective space is said to be \emph{convex} if it is contained and convex in some affine chart; it is said to be \emph{properly convex} if its closure is convex.
Unlike $\H^p$, for $q\geq 1$ the space $\H^{p,q}$ is not a convex subset of the projective space $\P(\R^{p+q+1})$.
However, the notion of convexity in $\H^{p,q}$ still makes sense: a subset $\mathscr{C}$ of $\H^{p,q}$ is said to be \emph{convex} if it is convex as a subset of $\P(\R^{p+q+1})$ or equivalently, from an intrinsic point of view, if any two points of~$\mathscr{C}$ are connected inside~$\mathscr{C}$ by a unique segment which is geodesic for the pseudo-Riemannian structure.
We say that $\mathscr{C}$ is \emph{properly convex} if its closure $\overline{\mathscr{C}}$ in $\P(\R^{p+q+1})$ is convex.

The notion of $\H^{p,q}$-convex cocompactness introduced in \cite{dgk18} is the following.

\begin{defn}[{\cite[Def.\,1.2]{dgk18}}] \label{def:Hpq-cc}
A discrete subgroup of $\PO(p,q+1)$ is \emph{$\hcc$-convex cocompact} if it acts properly discontinuously and cocompactly on a closed properly convex subset $\mathscr{C}$ of $\H^{p,q}$, with non-empty interior, such that the boundary at infinity $\di\mathscr{C} := \overline{\mathscr{C}} \cap \di\H^{p,q}$ of~$\mathscr{C}$ does not contain any non-trivial projective segments.

A representation of a discrete group into $\PO(p,q+1)$ is \emph{$\hcc$-convex cocompact} if it has finite kernel and $\hcc$-convex cocompact, discrete image.
\end{defn}

As in the classical case of Kleinian groups, if a discrete subgroup of $\PO(p,q+1)$ is $\H^{p,q}$-convex cocompact, then it is finitely generated and Gromov hyperbolic: see \cite[Th.\,1.7]{dgk18} and \cite[Th.\,1.24]{dgk-proj-cc}.
Moreover, the \emph{virtual cohomological dimension} $\vcd(\Gamma)$ of~$\Gamma$ (an invariant of finitely generated groups measuring a notion of their ``size'', see Section~\ref{subsec:vcd}) is bounded above by~$p$ \cite[Cor.\,11.10]{dgk-proj-cc}.
We note that Definition~\ref{def:Hpq-cc} fits into a general notion of \emph{strong convex cocompactness} in projective space developed in \cite{dgk-proj-cc}: see \cite[Th.\,1.24]{dgk-proj-cc}.

An important feature\footnote{This feature uses the assumption that $\di\mathscr{C}$ contains no segment, and is not true without it, see \cite{dgk-ex-cc}.} is that $\H^{p,q}$-convex cocompactness is an \emph{open} condition: see \cite[Cor.\,1.12]{dgk18} and \cite[Cor.\,1.25]{dgk-proj-cc}.
In the present paper, for $p\geq 2$ and $q\geq 1$, we prove that $\H^{p,q}$-convex cocompactness is also a \emph{closed} condition when $\Gamma$ is ``sufficiently large'' in the sense that $\vcd(\Gamma)=\nolinebreak p$ is maximal.

\begin{thm} \label{thm:main}
Let $p\geq 2$ and $q\geq 1$ be integers, and let $\G$ be a Gromov hyperbolic group with $\vcd(\G)=p$.
Then the set of $\H^{p,q}$-convex cocompact representations of~$\Gamma$ is a union of connected components of $\Hom(\G,\PO(p,q+1))$.
\end{thm}

Gromov hyperbolic groups $\Gamma$ with $\vcd(\G)=p$ include in particular all fundamental groups of closed real hyperbolic $p$-manifolds.
In general, if $\vcd(\Gamma)=p\geq 2$ and if $\Gamma$ admits an $\H^{p,q}$-convex cocompact representation, then the Gromov boundary of~$\Gamma$ is homeomorphic to a $(p-1)$-dimensional sphere \cite[Cor.\,11.10]{dgk-proj-cc}.
Moreover, in that case $\Gamma$ is (up to passing to a finite-index subgroup) the fundamental group of a closed aspherical differentiable manifold: this follows from \cite{abbz12} for $q=1$ and from the recent work \cite{sst} for general $q\geq 1$ and $p\geq 2$.
(For $p\geq 6$, Bartels--L\"uck--Weinberger \cite{blw10} proved that in general, any Gromov hyperbolic group with Gromov boundary a $(p-1)$-sphere is the fundamental group of a closed aspherical \emph{topological} manifold.)

Here is an immediate consequence of Theorem~\ref{thm:main}, which improves Theorem~\ref{thm:deform-Fuchsian-inj-discr}.

\begin{cor} \label{cor:deform-Fuchsian-Hpq-cc}
Let $p\geq 2$ and $q\geq 1$ be integers.
Let $N$ be a closed real hyperbolic $p$-manifold with holonomy representation $\sigma_0 : \pi_1(N)\to\PO(p,1)$, and let $\rho_0 : \pi_1(N)\to\PO(p,q+1)$ be the composition of a lift of $\sigma_0$ to $\OO(p,1)$ with the natural inclusion $\OO(p,1)\hookrightarrow\PO(p,q+1)$.
Then the connected component of $\rho_0$ in $\Hom(\pi_1(N),\PO(p,q+1))$ consists entirely of $\H^{p,q}$-convex cocompact representations.
\end{cor}

Indeed, one readily checks that in this setting $\rho_0 : \pi_1(N)\to\PO(p,q+1)$ is $\H^{p,q}$-convex cocompact (one can take for $\mathscr{C}$ a neighborhood of a copy of $\H^p$ inside $\H^{p,q}$).

The $\H^{p,q}$-convex cocompact representations in Corollary~\ref{cor:deform-Fuchsian-Hpq-cc} could be called \emph{$\H^{p,q}$-quasi-Fuchsian}, by analogy with the theory of Kleinian groups.

\begin{remark}
Theorem~\ref{thm:main} was previously proved by Barbot \cite{bar15} in the Lorentzian (anti-de Sitter) case where $q=1$.
See also the work of Mess \cite{mes90} for $p=q+1=2$.
Theorem~\ref{thm:main}  was also known for $p=2$ (and $q$ arbitrary), see \cite[\S\,7.2]{dgk18}, \cite[Cor.\,11.15]{dgk-proj-cc}, and \cite{ctt19}: in that case, $\Gamma$ is a closed surface group and the connected component of $\rho_0$ in $\Hom(\Gamma,\PO(2,q+1))$ consists of \emph{maximal representations} in the sense of \cite{biw10,biw}.
\end{remark}

Theorem~\ref{thm:main} is classically true also for $q=0$.
Indeed, assume for simplicity that $\Gamma$ is torsion-free.
For $\vcd(\Gamma)=p\geq 2$, a representation from $\Gamma$ to $\PO(p,1)$ is ($\H^{p,0}$-)convex cocompact if and only if it is injective and discrete.
For $p=2$, we have $\Gamma=\pi_1(S)$ for some closed hyperbolic surface~$S$ and the injective and discrete representations of $\Gamma$ into $\PO(p,1)$ make up one (or two, reversing orientation) connected component(s) of representations, whose image in the character variety of $\pi_1(S)$ is the Teichm\"uller space of~$S$.
For $p\geq 3$, there is only one injective and discrete representation of $\Gamma$ into $\PO(p,1)$ modulo conjugation, by Mostow rigidity.

\begin{remark} \label{rem:p=1}
Theorem~\ref{thm:main} is not true for $p=1$.
Indeed, in that case, up to replacing $\Gamma$ by a finite-index subgroup, it is a free group and $\Hom(\Gamma,\PO(p,q+1))$ is connected and contains representations that are not $\H^{p,q}$-convex cocompact (\eg the constant representation).
\end{remark}

\begin{remark}
Theorem~\ref{thm:main} is not true in general for $1\leq\vcd(\Gamma)<p$.
Indeed, if $\vcd(\Gamma)=\nolinebreak 1$, then $\Gamma$ is a free group up to finite index and we conclude as in Remark~\ref{rem:p=1}.
For $2\leq k<\nolinebreak p$, suppose that $\Gamma=\pi_1(N)$ for some closed hyperbolic $k$-manifold $N$ with a closed totally geodesic embedded hypersurface~$\mathcal{H}$ separating $N$ into two connected components $N_1$ and~$N_2$.
Then $\vcd(\Gamma)=k$.
The holonomy representation of~$N$ (with values in $\PO(k,1)$) lifts to a representation into $\OO(k,1)$ which, composed with the natural inclusion $\OO(k,1)\hookrightarrow\PO(p,q+1)$, yields an $\H^{p,q}$-convex cocompact representation $\sigma_0 : \pi_1(N)\to\PO(p,q+1)$.
The group $\sigma_0(\pi_1(\mathcal{H}))$ is contained in a copy of $\OO(k-1,1)$ in $\PO(p,q+1)$, and the centralizer of $\sigma_0(\pi_1(\mathcal{H}))$ in $\PO(p,q+1)$ is isomorphic to $\OO(p-k+1,q)$.
Bending \`a la Johnson--Millson \cite{jm87} along $\mathcal{H}$ using a compact one-parameter subgroup of this $\OO(p-k+1,q)$ yields a continuous family $(\sigma_{\theta})_{\theta\in\R/2\pi\Z}$ of representations from $\pi_1(N)$ to $\PO(p,q+1)$, containing $\sigma_0$, such that $\sigma_{\pi}$ takes values in $\OO(k,1)$ but is not injective and discrete (hence not $\H^{p,q}$-convex cocompact).
Indeed, the preimage of $\mathcal{H}$ in the universal cover $\widetilde{N} \simeq \H^k$ divides $\H^k$ into connected components, each of which projects to either $N_1$ or~$N_2$.
One easily constructs a $(\sigma_0,\sigma_{\pi})$-equivariant map $f : \H^k \to \H^k$ which is isometric in restriction to any of these connected components.
The fact that $f$ is $1$-Lipschitz and \emph{not} an isometry then implies that $\sigma_{\pi}$ cannot be injective and discrete, by \cite[Prop.\,1.13]{gk17}.
(Alternatively, for $k\geq 3$ one can argue using Mostow rigidity.)
\end{remark}

Recall that a Lorentzian manifold is called \emph{anti-de Sitter} (or \emph{AdS} for short) if its sectional curvature is constant negative.
It is called \emph{globally hyperbolic} if it is causal (\ie contains no timelike loop) and if the future of any point intersects the past of any other point in a compact set (possibly empty).
Here is another consequence of Theorem~\ref{thm:main}.

\begin{cor} \label{cor:deform-AdS-GHMC}
For $p\geq 2$ and $q\geq 1$, let $N$ be a closed $p$-manifold, let $\tau_0 : \pi_1(N)\to\PO(p,2)$ be the holonomy of a globally hyperbolic AdS spacetime homeomorphic to $N\times\nolinebreak\R$, and let $\rho_0 : \pi_1(N)\to\PO(p,q+1)$ be the composition of a lift of $\tau_0$ to $\OO(p,2)$ with the natural homomorphism $\OO(p,2)\to\PO(p,q+1)$.
Then the connected component of $\rho_0$ in $\Hom(\pi_1(N),\PO(p,q+\nolinebreak 1))$ consists entirely of $\H^{p,q}$-convex cocompact representations.
\end{cor}

Indeed, $\tau_0 : \pi_1(N)\to\PO(p,2)$ is $\H^{p,1}$-convex cocompact by \cite[Th.\,1.4]{bar15} and \cite[Th.\,1.1]{bm12}, hence $\rho_0$ is $\H^{p,q}$-convex cocompact by \cite[Th.\,1.16 \& 1.24]{dgk-proj-cc}.

The manifolds $N$ in Corollary~\ref{cor:deform-AdS-GHMC} include all closed $p$-manifolds admitting a real hyperbolic structure, but also exotic examples where $\pi_1(N)$ is not isomorphic (nor even quasi-isometric) to a lattice of $\PO(p,1)$: such examples were constructed by Lee--Marquis \cite{lm19} for $4\leq p\leq 8$ (using Coxeter groups) and by Monclair--Schlenker--Tholozan \cite{mst} for any $p\geq 4$ (using Gromov--Thurston manifolds).
Some of these examples from \cite{lm19,mst} admit non-trivial continuous deformations (even with Zariski-dense image: see Appendix~\ref{appendix}) in $G=\PO(p,q+1)$ (or $\mathrm{PSO}(p,q+1)$),  for $q>1$, and so we obtain in this way new \emph{higher higher Teichm\"uller spaces} for Gromov hyperbolic groups that are not quasi-isometric to lattices of $\PO(p,1)$.

We note that there also exist Gromov hyperbolic groups $\Gamma$ with $\vcd(\Gamma)=p$ that admit $\H^{p,q}$-convex cocompact representations for some $q\geq 2$, but no $\H^{p,q}$-convex cocompact representations for $q=0$ or~$1$: an example of a Coxeter group satisfying this property was recently constructed by Lee--Marquis \cite{lm23} for $p=4$.
This example is rigid, but we expect the existence of other non-rigid examples, which would provide additional concrete examples of higher higher Teichm\"uller spaces obtained from Theorem~\ref{thm:main}.

\subsection{Link with Anosov representations} \label{subsec:intro-Ano}

Anosov representations are representations of infinite Gromov hyperbolic groups $\Gamma$ to semi-simple Lie groups with finite kernel and discrete image, defined by strong dynamical properties.
They were introduced by Labourie \cite{lab06} and further studied by Guichard--Wienhard \cite{gw12} and many other authors.
They play a major role in recent developments on discrete subgroups of Lie groups.
All known examples of higher-rank Teichm\"uller spaces (associated to surfaces) consist entirely of Anosov representations of surface groups \cite{lab06,bilw05,glw}.

By definition, a representation $\rho : \Gamma\to\PO(p,q+1)$ is \emph{$P_1$-Anosov} if there exists a continuous, $\rho$-equivariant boundary map $\xi : \di\Gamma\to\di\H^{p,q}$ which
\begin{enumerate}[label=(\roman*),ref=\roman*]
  \item\label{item:def-Ano-transv} is injective and even \emph{transverse}: $\xi(\eta)\notin\xi(\eta')^{\perp}$ for all $\eta\neq\eta'$ in $\di\Gamma$,
  \item\label{item:flow} has an associated flow with some uniform contraction/expansion properties described in \cite{lab06,gw12}.
\end{enumerate}
We do not state condition~\eqref{item:flow} precisely, but refer instead to Definition~\ref{def:P1-Ano} below for an alternative definition using a simple condition on eigenvalues, taken from \cite{ggkw17}.
A consequence of \eqref{item:flow} is that $\xi$ is \emph{dynamics-preserving}: for any infinite-order element $\gamma\in\Gamma$, the element $\rho(\gamma)\in\PO(p,q+1)$ admits a unique attracting fixed point in $\di\H^{p,q}$, and $\xi$ sends the attracting fixed point of $\gamma$ in $\di\Gamma$ to this attracting fixed point of $\rho(\gamma)$ in $\di\H^{p,q}$.
In particular, by a density argument, the continuous map $\xi$ is unique, and the image $\xi(\di\Gamma)$ is the \emph{proximal limit set} $\Lambda_{\rho(\Gamma)}$ of $\rho(\Gamma)$ in $\di\H^{p,q}$, \ie the closure of the set of attracting fixed points of elements $\rho(\gamma)$ for $\gamma\in\Gamma$.
By \cite[Prop.\,4.10]{gw12}, if $\rho(\Gamma)$ is irreducible (in the sense that there is no non-trivial $\rho(\Gamma)$-invariant projective subspace of $\P(\R^{p,q})$), then condition~\eqref{item:flow} is automatically satisfied as soon as \eqref{item:def-Ano-transv} is.

The dynamical properties of Anosov representations are very similar to those satisfied by convex cocompact representations of Gromov hyperbolic groups into simple Lie groups of real rank one, see \eg \cite{lab06,gw12,klp-survey,ggkw17,bps19}.
The analogy is in fact also geometric, in particular for representations into $\PO(p,q+1)$, as we now explain.

For this, note that the transversality condition in \eqref{item:def-Ano-transv} means that the proximal limit set $\Lambda_{\rho(\Gamma)}$ of any $P_1$-Anosov representation $\rho : \Gamma\to\PO(p,q+1)$ lifts to a subset of the $\sfb$-isotropic vectors of~$\R^{p,q+1}$ where the symmetric bilinear form $\sfb$ is non-zero on all pairs of distinct points; we say that $\Lambda_{\rho(\Gamma)}$ is \emph{negative} if it lifts to a subset where $\sfb$ is negative on all pairs of distinct points.
With this terminology, Danciger--Gu\'eritaud--Kassel \cite{dgk18} proved that, given a Gromov hyperbolic group $\Gamma$, the $\H^{p,q}$-convex cocompact representations of $\Gamma$ to $\PO(p,q+1)$ are exactly the $P_1$-Anosov representations whose proximal limit set is negative (see Fact~\ref{fact:Hpq-cc-Ano} below); moreover, if the Gromov boundary of~$\Gamma$ is connected, then any $P_1$-Anosov representation of $\Gamma$ to $\PO(p,q+1)$ is either $\H^{p,q}$-convex cocompact or $\H^{q+1,p-1}$-convex cocompact, where we see $\H^{q+1,p-1}$ as the complement of the closure of $\H^{p,q}$ in $\P(\R^{p+q+1})$ after switching the sign of the quadratic form.
Combining this characterization of $\H^{p,q}$-convex cocompact representations with Theorem~\ref{thm:main}, we obtain the following.

\begin{cor} \label{cor:Anosov}
Let $p\geq 2$ and $q\geq 1$ be integers, and let $\Gamma$ be a Gromov hyperbolic group with $\vcd(\G)=\nolinebreak p$.
Then
\begin{itemize}
  \item the set of $P_1$-Anosov representations of $\Gamma$ whose proximal limit set is negative is a union of connected components in $\Hom(\G,\PO(p,q+1))$;
  \item if $p=q+1$, then the set of all $P_1$-Anosov representations of $\Gamma$ is a union of connected components in $\Hom(\G,\PO(p,q+1))$.
\end{itemize}
\end{cor}

In particular (see Corollaries \ref{cor:deform-Fuchsian-Hpq-cc} and~\ref{cor:deform-AdS-GHMC}):
\begin{itemize}
  \item if $\G$ is a uniform lattice in $\OO(p,1)$ and if $\rho_0 : \Gamma\subset\OO(p,1)\hookrightarrow\PO(p,q+1)$ is the natural inclusion, then the connected component of $\rho_0$ in $\Hom(\Gamma,\PO(p,q+\nolinebreak 1))$ consists entirely of $P_1$-Anosov representations;
  \item if $N$ is a closed $p$-manifold, $\tau_0 : \pi_1(N)\to\PO(p,2)$ the holonomy of a globally hyperbolic AdS spacetime homeomorphic to $N\times\R$, and $\rho_0 : \pi_1(N)\to\PO(p,q+1)$ the composition of a lift of $\tau_0$ to $\OO(p,2)$ with the natural homomorphism $\OO(p,2)\to\PO(p,q+1)$, then the connected component of $\rho_0$ in $\Hom(\pi_1(N),\PO(p,q+1))$ consists entirely of $P_1$-Anosov representations.
\end{itemize}

\subsection{Weakly spacelike $p$-graphs with proper cocompact group actions} \label{subsec:intro-suff-Hpq-cc}

In order to prove Theorem~\ref{thm:main}, we introduce the notion of a \emph{weakly spacelike $p$-graph} in~$\H^{p,q}$: this is a closed subset of~$\H^{p,q}$, homeomorphic to~$\R^p$, on which no two distinct points are in timelike position, or equivalently which meets every totally negative $q$-dimensional projective subspace of $\P(\R^{p,q+1})$ in a unique point (see Definition~\ref{def:weakly-sp-gr}, Proposition~\ref{prop:weakly-sp-gr-is-a-graph}, and Remark~\ref{r.contractible}).
The class of weakly spacelike $p$-graphs includes in particular all $p$-dimensional connected complete \emph{spacelike submanifolds} of~$\H^{p,q}$ (\ie immersed $C^1$ submanifolds $M$ such that the restriction of the pseudo-Riemannian metric of $\H^{p,q}$ to $TM$ is Riemannian), but this class is larger since we allow $M$ to not be $C^1$, and the restriction of the metric to only be positive semi-definite.
Given a weakly spacelike $p$-graph $M\subset\H^{p,q}$, we can consider its boundary at infinity $\L = \di M := \ov{M} \cap \di\H^{p,q}$ and associate to it  a convex open subset $\Om(\L)$ of~$\H^{p,q}$: for $p\geq 2$ this is the set of points of~$\H^{p,q}$ that see every point of~$\L$ in a spacelike direction (Notation~\ref{not:Omega-Lambda} and Lemma~\ref{lem:Omega-Lambda-convex}).
In the Lorentzian case where $q=1$, the set $\Om(\L)$ is called the \emph{invisible domain} of~$\L$.

Our proof of Theorem~\ref{thm:main} is based on a study of discrete subgroups of $\PO(p,q+1)$ acting properly discontinuously and cocompactly on weakly spacelike $p$-graphs in~$\H^{p,q}$.
In particular, we characterize $\H^{p,q}$-convex cocompactness as follows.

\begin{thm} \label{thm:geom-action-weakly-sp-gr-basic}
For $p,q\geq 1$, let $\Gamma$ be a discrete subgroup of $\PO(p,q+1)$ acting properly discontinuously and cocompactly on a weakly spacelike $p$-graph $M$ in $\H^{p,q}$ with boundary at infinity $\L := \di M \subset \di\H^{p,q}$ such that $M \subset \Om(\L)$.
Then the following are equivalent:
\begin{enumerate}
  \item\label{item:geom-action-1} $\Gamma$ is Gromov hyperbolic,
  \item\label{item:geom-action-2} any two points of $\L$ are transverse (\ie non-orthogonal for~$\sfb$),
  \item\label{item:geom-action-3} $\Gamma$ is $\hcc$-convex cocompact.
\end{enumerate}
\end{thm}

As proved in Proposition~\ref{prop:bdynonpossphere}.\eqref{item:bound-non-pos-sph-2} below, if the weakly spacelike $p$-graph $M$ is actually spacelike (\ie any two points of~$M$ are in spacelike position), then the assumption $M \subset \Om(\L)$ in Theorem~\ref{thm:geom-action-weakly-sp-gr-basic} is automatically satisfied.
We thus obtain the following.

\begin{cor} \label{cor:hyp-spacelike-compact->cc}
For $p,q\geq 1$, let $\Gamma$ be a discrete subgroup of $\PO(p,q+1)$ acting properly discontinuously and cocompactly on a spacelike $p$-graph $M$ in~$\H^{p,q}$ (Definition~\ref{def:weakly-sp-gr}).
Then the following are equivalent:
\begin{enumerate}
  \item\label{item:geom-action-1} $\Gamma$ is Gromov hyperbolic,
  \item\label{item:geom-action-2} any two points of $\L := \di M$ are transverse,
  \item\label{item:geom-action-3} $\Gamma$ is $\hcc$-convex cocompact.
\end{enumerate}
\end{cor}

Spacelike $p$-graphs include in particular all $p$-dimensional connected complete spacelike submanifolds of~$\H^{p,q}$ (Example~\ref{ex:spacelike-compl-submfd}).

Our proof of Theorem~\ref{thm:geom-action-weakly-sp-gr-basic} is based on convex projective geometry and on the study of geometric objects in $\di\H^{p,q}$ called \emph{crowns}, see Section~\ref{subsec:intro-strategy} below.

In Section~\ref{subsec:split-spacetimes}, we construct examples of discrete subgroups $\Gamma$ of $\PO(p,q+1)$ acting properly discontinuously and cocompactly on spacelike or weakly spacelike $p$-graphs~$M$ in $\H^{p,q}$, such that $\Gamma$ is \emph{not} Gromov hyperbolic; we prove the existence of higher higher Teichm\"uller spaces for such groups in Theorem~\ref{thm:main-general} below.
In the same Section~\ref{subsec:split-spacetimes} we also construct examples showing that if we remove the cocompactness assumption in Theorem~\ref{thm:geom-action-weakly-sp-gr-basic} or Corollary~\ref{cor:hyp-spacelike-compact->cc}, then \eqref{item:geom-action-1} does not imply \eqref{item:geom-action-2} and~\eqref{item:geom-action-3} anymore.

By Fact~\ref{fact:vcd} below, if $\vcd(\Gamma) = p$, then in Theorem~\ref{thm:geom-action-weakly-sp-gr-basic} and Corollary~\ref{cor:hyp-spacelike-compact->cc} we can replace ``$\Gamma$ acts properly discontinuously and cocompactly via~$\rho$ on'' by ``$\Gamma$ acts properly discontinuously via~$\rho$ on''.
The following proposition shows that we can then also replace it by ``$\rho$ has finite kernel and discrete image and $\rho(\Gamma)$ preserves''.

\begin{prop} \label{prop:proper-action-M}
For $p,q\geq 1$, let $\Gamma$ be a finitely generated group, let $\rho : \Gamma\to\PO(p,q+1)$ be a representation, and let $M$ be a weakly spacelike $p$-graph in~$\H^{p,q}$ preserved by $\rho(\Gamma)$, with $M \subset \Om(\L)$ where $\L := \di M \subset \di\H^{p,q}$.
Then the following are equivalent:
\begin{enumerate}
  \item\label{item:proper-action-M-1} $\Gamma$ acts properly discontinuously on~$M$ via~$\rho$,
  \item\label{item:proper-action-M-2} $\rho$ has finite kernel and discrete image.
\end{enumerate}
If this holds and if $\vcd(\Gamma) = p$, then $M$ is contained in some $\rho(\Gamma)$-invariant properly convex open subset of~$\H^{p,q}$.
\end{prop}

\subsection{A non-degeneracy result}

The second main ingredient of our proof of Theorem~\ref{thm:main} is the following.
We say that a subset $\L$ of $\di\H^{p,q}$ is \emph{non-degenerate} if the restriction of $\sfb$ to $\spa(\L)$ has trivial kernel (Definition~\ref{def:span-degen}).

\begin{prop} \label{prop:non-deg-limit}
For $p\geq 2$ and $q\geq 1$, let $\Gamma$ be a finitely generated group with no infinite nilpotent normal subgroups, such that $\vcd(\Gamma)=\nolinebreak p$.
Let $\rho : \Gamma\to\PO(p,q+1)$ be a limit of injective and discrete representations $\rho_n$ of~$\Gamma$, each preserving a weakly spacelike $p$-graph $M_n$ in $\H^{p,q}$ with $M_n \subset \Om(\L_n)$ where $\L_n := \di M_n$, such that the $\L_n$ converge to some $\L \subset \di\H^{p,q}$.
Suppose that the Zariski closure of $\rho(\Gamma)$ in $\PO(p,q+1)$ is reductive, and let $E$ be a $\rho(\Gamma)$-invariant complementary subspace of $V := \mathrm{Ker}(\sfb|_{\spa(\L)})$ in~$V^{\perp}$.
If $k<p$, suppose that $\rho(\Gamma)$ preserves a weakly spacelike $(p-k)$-graph $M_E$ in $\H^{p,q} \cap \P(E) \simeq \H^{p-k,q-k}$ with $\di M_E = \L_E := \L \cap \P(E)$ and $M_E \subset \Om(\L_E)$.
Then $\rho$ is injective and discrete, $V = \{0\}$, and $\L = \L_E$ is non-degenerate.
\end{prop}

Our proof of Proposition~\ref{prop:non-deg-limit} is based on algebraic and Lie-theoretic arguments using semi-proximal representations of algebraic groups, inspired by \cite{ben05}: see Section~\ref{subsec:intro-strategy} below.

\subsection{Existence of invariant weakly spacelike $p$-graphs}

A \emph{non-degenerate non-positive $(p-1)$-sphere} in $\di\H^{p,q}$ is a closed subset of $\di\H^{p,q}$, homeomorphic to $\bS^{p-1}$, which lifts to a subset of $\R^{p,q+1}$ on which $\sfb$ is non-degenerate and non-positive (see Definitions \ref{def:span-degen} and~\ref{def:non-pos-sphere} and Proposition~\ref{prop:lift-non-pos-sphere}).
It is not difficult to check (see Lemma~\ref{lem:Omega-Lambda-convex} and Proposition~\ref{prop:bdynonpossphere}) that if $M$ is a weakly spacelike $p$-graph in~$\H^{p,q}$ with boundary at infinity $\L := \di M$ such that $M \subset \Om(\L)$, then $\L$ is a non-degenerate non-positive $(p-1)$-sphere in $\di\H^{p,q}$; if $M$ is preserved by some subgroup of $\PO(p,q+1)$, then so is~$\L$.
The last ingredient of our proof of Theorem~\ref{thm:main} is the following converse.

\begin{fact} \label{fact:exist-p-graph}
For $p,q\geq 1$, let $\L$ be a non-degenerate non-positive $(p-1)$-sphere in $\di\H^{p,q}$, preserved by a subgroup $\Gamma$ of $\PO(p,q+1)$.
Then there exists a weakly spacelike $p$-graph $M$ in~$\H^{p,q}$, preserved by~$\Gamma$, such that $\di M = \L$ and $M \subset \Om(\L)$.
\end{fact}

This is due in full generality to Seppi--Smith--Toulisse \cite{sst}, who constructed more specifically a connected complete spacelike $p$-manifold $M$ with $\di M = \L$ which is \emph{maximal}, in the sense that its mean curvature vanishes (see Section~\ref{subsec:max-submfd}); this $p$-manifold is unique, and so it is invariant under any subgroup $\Gamma$ of $\PO(p,q+1)$ preserving~$\L$.
The existence and uniqueness of $\Gamma$-invariant maximal $p$-manifolds was first proved for $q=1$ (Lorentzian case) by Andersson--Barbot--B\'eguin--Zeghib \cite{abbz12} (for discrete $\Gamma$ with $\vcd(\Gamma) = p$) and Bonsante--Schlenker \cite{bs10}, and for $p=2$ by Collier--Tholozan--Toulisse \cite{ctt19} (when $\Gamma$ is a discrete surface group) and Labourie--Toulisse--Wolf \cite{ltw}.
See \cite{lt} for a detailed study of these $p$-manifolds when $p=2$.

\begin{remark}
For discrete~$\Gamma$, Fact~\ref{fact:exist-p-graph} expresses that the quotient $\Gamma\backslash\Om(\L)$ satisfies a weak form of \emph{global hyperbolicity} --- a classical notion in Lorentzian geometry, recently developed in signature $(p,q)$ by Troubat \cite{tro-PhD}.
\end{remark}

Using Theorem~\ref{thm:geom-action-weakly-sp-gr-basic} and Fact~\ref{fact:exist-p-graph}, we obtain the following, which was first proved for $q=1$ by Barbot \cite[Th.\,1.4]{bar15}.

\begin{thm} \label{thm:charact-Hpq-cc-sphere}
For $p,q\geq 1$, let $\Gamma$ be a Gromov hyperbolic discrete subgroup of $\PO(p,q+1)$ with $\vcd(\Gamma) = p$.
Then $\Gamma$ preserves a non-degenerate non-positive $(p-1)$-sphere $\L$ in $\di\H^{p,q}$ if and only if $\Gamma$ is $\H^{p,q}$-convex cocompact.
In this case, $\L$ is actually negative, and equal to the proximal limit set $\Lambda_{\Gamma}$ of~$\Gamma$.
\end{thm}

On the other hand, using Proposition~\ref{prop:non-deg-limit} and Fact~\ref{fact:exist-p-graph}, we obtain the following.

\begin{thm} \label{thm:non-deg-limit-sphere}
For $p\geq 2$ and $q\geq 1$, let $\Gamma$ be a finitely generated group with no infinite nilpotent normal subgroups, such that $\vcd(\Gamma)=\nolinebreak p$.
Let $\rho : \Gamma\to\PO(p,q+1)$ be a representation such that the Zariski closure of $\rho(\Gamma)$ in $\PO(p,q+1)$ is reductive, and which is a limit of representations $\rho_n : \Gamma\to\PO(p,q+1)$.
If each $\rho_n$ is injective and discrete and preserves a non-degenerate non-positive $(p-1)$-sphere in $\di\H^{p,q}$, then the same holds for~$\rho$.
\end{thm}

This implies Theorem~\ref{thm:main}.
Indeed, $\H^{p,q}$-convex cocompactness is an open condition by \cite{dgk18,dgk-proj-cc}, and it is a closed condition by Theorems \ref{thm:charact-Hpq-cc-sphere} and~\ref{thm:non-deg-limit-sphere} together with a short argument allowing to reduce to $\rho(\Gamma)$ having reductive Zariski closure (see Section~\ref{subsec:proof-main-thm}).

\subsection{Higher higher Teichm\"uller spaces for non-hyperbolic groups}

Our proof of Theorem~\ref{thm:main} only uses the existence of weakly spacelike $p$-graphs with no smoothness assumption (Fact~\ref{fact:exist-p-graph}).
However, as mentioned above, by \cite{sst} we know that if a subgroup $\Gamma$ of $\PO(p,q+1)$ preserves a non-degenerate non-positive $(p-1)$-sphere $\L$ in $\di\H^{p,q}$, then $\Gamma$ also preserves a $p$-dimensional connected complete \emph{spacelike submanifold} $M$ of~$\H^{p,q}$ with $\di M = \L$.
It follows from the Ehresmann--Thurston principle that preserving and acting properly discontinuously and cocompactly on such a submanifold is an open condition; we deduce the following.

\begin{thm} \label{thm:main-general}
Let $p\geq 2$ and $q\geq 1$, let $\Gamma$ be a finitely generated group with no infinite nilpotent normal subgroups, such that $\vcd(\Gamma) = p$.
Then the subset of $\Hom(\G,\PO(p,q+\nolinebreak 1))$ consisting of those representations with finite kernel and discrete image preserving a non-degenerate non-positive $(p-1)$-sphere in $\di\H^{p,q}$, is a union of connected components of $\Hom(\G,\PO(p,q+\nolinebreak 1))$.
\end{thm}

We refer to Section~\ref{subsec:split-spacetimes} for examples where $\Gamma$ is a direct product.

By Propositions \ref{prop:proper-action-M} and~\ref{prop:bdynonpossphere}, Fact~\ref{fact:vcd}, and \cite{sst}, we can reformulate Theorem~\ref{thm:main-general} as follows.

\begin{thm} \label{thm:main-general-spacelike-p-mfd}
Let $p\geq 2$ and $q\geq 1$, let $\Gamma$ be a finitely generated group with no infinite nilpotent normal subgroups.
Then the subset of $\Hom(\G,\PO(p,q+1))$ consisting of those representations acting properly discontinuously and cocompactly on some $p$-dimensional connected complete spacelike submanifold of~$\H^{p,q}$, is a union of connected components of $\Hom(\G,\PO(p,q+\nolinebreak 1))$.
\end{thm}

We shall call these representations \emph{spacelike cocompact} (see Definition~\ref{def:spacelike-cocompact}).

\subsection{Outline of the proofs of Theorem~\ref{thm:geom-action-weakly-sp-gr-basic} and Proposition~\ref{prop:non-deg-limit}} \label{subsec:intro-strategy}

\subsubsection{Proof of Theorem~\ref{thm:geom-action-weakly-sp-gr-basic}}

We introduce a generalization, to $\di\H^{p,q}$, of the \emph{crowns} considered by Barbot \cite{bar15} in the Einstein universe $\mathrm{Ein}^p = \di\H^{p,1}$: for $1\leq j\leq\min(p,q+1)$, we define a \emph{$j$-crown} in $\di\H^{p,q}$ to be a collection of $2j$ points of $\di\H^{p,q}$ which span a linear subspace of $\R^{p,q+1}$ which is non-degenerate of signature $(j,j)$ and such that each of the $2j$ points is orthogonal to all but one of the other points (Definition~\ref{def:j-crown}); the crowns considered by Barbot are $2$-crowns.

We prove the following strengthening of Theorem~\ref{thm:geom-action-weakly-sp-gr-basic}.

\begin{thm} \label{thm:geom-action-weakly-sp-gr}
For $p,q\geq 1$, let $\Gamma$ be a discrete subgroup of $\PO(p,q+1)$ acting properly discontinuously and cocompactly on a weakly spacelike $p$-graph $M$ in $\H^{p,q}$ with boundary at infinity $\L := \di M \subset \di\H^{p,q}$ such that $M \subset \Om(\L)$.
(For instance, $M$ could be any spacelike $p$-graph in~$\H^{p,q}$, see Proposition~\ref{prop:bdynonpossphere}.\eqref{item:bound-non-pos-sph-2}.).
Let $\mathscr{C}(\L)$ be the convex hull of $\L$ in $\Om(\L)$.
Then the following are equivalent:
\begin{enumerate}
  \item\label{item:geom-action-1} $\Gamma$ is Gromov hyperbolic,
  \item\label{item:geom-action-2} any two points of $\L$ are transverse,
  \item\label{item:geom-action-3} $\Gamma$ is $\hcc$-convex cocompact,
  \item\label{item:geom-action-4} $\L$ does not contain $2$-crowns,
  \item\label{item:geom-action-5} $\de_{\H}\mathscr{C}(\L) := \de\mathscr{C}(\Lambda) \cap \H^{p,q}$ (see Notation~\ref{not:boundaries}) does not meet $\partial\Om(\L)$.
\end{enumerate}
\end{thm}

Although the statement of Theorem~\ref{thm:geom-action-weakly-sp-gr} only involves $2$-crowns, the proof crucially makes use of $j$-crowns for $j\geq 2$: see below.
One difficulty for proving Theorem~\ref{thm:geom-action-weakly-sp-gr} is that $M$ is not assumed to be smooth.
Two implications particularly require new ideas, namely \eqref{item:geom-action-4}~$\Rightarrow$~\eqref{item:geom-action-5} and \eqref{item:geom-action-1}~$\Rightarrow$~\eqref{item:geom-action-4}.

We prove \eqref{item:geom-action-4}~$\Rightarrow$~\eqref{item:geom-action-5} by contraposition, building $2$-crowns from the existence of a non-empty intersection $\de_{\H}\mathscr{C}(\L) \cap \partial\Om(\L)$.
One difficulty is that for $q>1$ we do not have \emph{cosmological geodesics} which played a key role in the arguments of \cite{bar15} in the Lorentzian case $q=1$.
Instead, we do a more careful geometric construction in~$\H^{p,q}$.
We introduce the notion of a \emph{boundary $j$-crown} in~$\L$, which is by definition a $j$-crown in~$\L$ which is fully contained in the orthogonal of some point of~$\L$.
We first observe (Lemma~\ref{l.crown-existence}) that if $\de_{\H}\mathscr{C}(\L)$ meets $\partial\Om(\L)$, then $\L$ contains a boundary $1$-crown.
The core of the work (Proposition~\ref{p.crown-boundary}) is then to construct a $2$-crown in~$\L$ from this boundary $1$-crown.
We actually construct a $j$-crown in~$\L$ from a boundary $(j-1)$-crown in~$\L$ for any $j\geq 2$, and this is crucial in the proof of \eqref{item:geom-action-1}~$\Rightarrow$~\eqref{item:geom-action-4} below to ensure the existence of crowns in~$\L$ which are \emph{not} boundary crowns.

We prove \eqref{item:geom-action-1}~$\Rightarrow$~\eqref{item:geom-action-4} by contraposition, showing that the existence of $j$-crowns in~$\L$ with $j\geq 2$ implies that $\Gamma$ is not Gromov hyperbolic.

For $q=1$, Lorentzian geometry is heavily used in \cite{bar15} to prove this implication, through the existence of a \emph{foliation} of the invisible domain $\Om(\L)$ of~$\L$ by \emph{Cauchy hypersurfaces} which are level sets of a \emph{cosmological time function} $\tau$ of class~$C^1$; the Gauss map for these hypersurfaces yields a closed embedded submanifold $\Sigma(\tau)$ of the Riemannian symmetric space $G/K$ of $G=\PO(p,2)$, on which $\Gamma$ acts properly discontinuously and cocompactly \cite[Rem.\,4.26]{bar15}.
On the other hand, for any $2$-crown in $\L$ inside the Einstein universe $\mathrm{Ein}^p = \di\H^{p,1}$, Barbot proves \cite[\S\,5.1]{bar15} that the maximal flat of $G/K$ determined by the $2$-crown is a totally geodesic subspace contained in $\Sigma(\tau)$; therefore, $\Sigma(\tau)$ is not Gromov hyperbolic, and so the Milnor--\v{S}varc lemma implies that $\Gamma$ is not Gromov hyperbolic.

The case $q\geq 2$ is more difficult: firstly the above tools from Lorentzian geometry are not available anymore, and secondly $j$-crowns in $\di\H^{p,q}$ for $2\leq j<\min(p,q+1)$ do not determine flats in the Riemannian symmetric space of $G=\PO(p,q+1)$ (only parallel sets).
We thus do not work in this Riemannian symmetric space, but rather in $\H^{p,q}$ itself, using convex projective geometry.
One issue is that $\Om(\L)$ is not necessarily properly convex, but we show (Proposition~\ref{prop:M-in-prop-conv-Omega}) that we can find a $\Gamma$-invariant properly convex open subset $\Omega$ of $\Om(\L)$ containing~$M$; this allows us to use the Hilbert metric $d_{\Omega}$ of~$\Omega$, which is invariant under~$\Gamma$.
The aforementioned Proposition~\ref{p.crown-boundary} ensures the existence of $j$-crowns in~$\L$ which are \emph{not} boundary $j$-crowns in~$\L$, for some $j\geq 2$ (a priori $j$ could be any integer between $2$ and $\min(p,q+1)$ --- there is no control on this).
For such $j$-crowns, we introduce a natural foliation of the convex hull in~$\Omega$ of the crown by $j$-dimensional spacelike submanifolds which are orbits of a $j$-dimensional diagonal group (Section~\ref{subsec:foliate-crown}).
We then observe that the leaves of the foliation, endowed with the path metric induced by $d_{\Omega}$, are not Gromov hyperbolic (Proposition~\ref{prop:A-orbit-not-hyperb}), and that some leaf remains at bounded Hilbert~dist\-ance from~$M$ (Lemma~\ref{l.tubular-contains-A-orbit}), so that $M$ endowed with the path metric induced by $d_{\Omega}$ is not Gromov hyperbolic.
(Note that we only assume $M$ to be a weakly spacelike $p$-graph, not~a~space\-like smooth embedded manifold, hence we cannot use an induced Riemannian metric.)
Since $\Gamma$ acts properly discontinuously and cocompactly on~$M$, the non-hyperbolicity of~$M$ for the path metric of~$d_{\Omega}$, together with the Milnor--\v{S}varc lemma, implies that $\Gamma$ is not Gromov hyperbolic.

\subsubsection{Proof of Proposition~\ref{prop:non-deg-limit}}

Since $\Gamma$ has no infinite nilpotent normal subgroups, it is a standard consequence of the Kazhdan--Margulis--Zassenhaus theorem that $\rho$ has finite kernel and discrete image (see Fact~\ref{fact:fd-closed}).

We note in Section~\ref{subsec:Lambda-nondegen} that $\rho$ lifts to a representation with values in $\OO(p,q+1)$, which we still denote by~$\rho$.
We split this semi-simple representation $\rho$ into the direct sum of a representation $\rho_d$ (``degenerate part'') with values in~$\GL(k,\R)$, defining the action of $\rho(\Gamma)$ on the kernel $V \simeq \R^k$ of $\sfb|_{\spa(\L)}$, and a representation $\rho_{nd}$ (``non-degenerate part'') with values in $\OO(p-k,q+1-k)$, corresponding to the restriction of $\rho$ to a $\rho(\Gamma)$-invariant complement $E$ of $V$ inside~$V^{\perp}$.
If $k<p$, our assumption is that the non-degenerate part $\rho_{nd}$ (which we denote by $\nu$ in Section~\ref{sec:non-deg-limit}) preserves a weakly spacelike $(p-k)$-graph $M_E$ in $\H^{p,q} \cap \P(E) \simeq \H^{p-k,q-k}$, such that $\di M_E = \L_E := \L \cap \P(E)$ and $M_E \subset \Om(\L_E)$.
Up to projecting (Lemma~\ref{lem:restrict-span-L}), we may assume that $M_E \subset \P(\spa(\L))$.

Our goal is then to prove that the degenerate part $\rho_d$ (which we denote by $\kappa\oplus\kappa^*$ in Section~\ref{sec:non-deg-limit}) is trivial, \ie $V$ is trivial.
This would be automatic if we knew that $\rho(\Gamma)$ is not contained in a proper parabolic subgroup of $\PO(p,q+1)$, but we do not know this a priori.
When $q=1$, it is easy to deduce the triviality of~$V$ from the fact that for $k=q=1$, the Riemannian symmetric space of $\GL(k,\R) \times \OO(p-k,q+1-k)$ (on which $\Gamma$ acts by isometries via $\rho = (\rho_d,\rho_{nd})$) has dimension $p = \vcd(\Gamma)$: see Lemma~\ref{lem:Lambda-non-deg-q=1}.
However, when $q>1$ the Riemannian symmetric space of $\GL(k,\R) \times \OO(p-k,q+1-k)$ has dimension $>p$, and so we cannot conclude so easily.
Instead, we use a careful analysis of proximal and semi-proximal representations as in \cite{ben05} (Sections \ref{subsec:vcd-bound}--\ref{subsec:Lambda-nondegen}) to show that if $V$ is non-trivial, then we may modify the degenerate part $\rho_d$ of~$\rho$ into a sum $\tau_1\oplus\ldots\oplus\tau_{\ell}$ of irreducible representations $\tau_i : \Gamma\to\GL(V_i)$ of~$\Gamma$, where each $V_i$ is a finite-dimensional real vector space and $\dim(V_1)+\ldots+\dim(V_{\ell})\leq k$, so that $\Gamma$ acts properly discontinuously via
$$\tau_1\oplus\dots\oplus\tau_{\ell} \oplus \rho_{nd} : \Gamma \longrightarrow \GL(V_1)\times\dots\times\GL(V_{\ell}) \times \OO(p-k,q+1-k)$$
on $\Omega_1 \times \dots \times \Omega_{\ell} \times M_{nd}$, where each $\Omega_i$ is a non-empty properly convex open subset of $\P(V_i)$.
In particular, $\vcd(\Gamma) \leq \dim(\Omega_1) + \dots + \dim(\Omega_{\ell}) + \dim(M_{nd}) \leq (k-1) + (p-k) < p$, contradicting the assumption that $\vcd(\Gamma) = p$.

\subsection{Organization of the paper}

In Section~\ref{sec:reminders} we recall some well-known facts on convex projective geometry, on the pseudo-Riemannian hyperbolic space~$\H^{p,q}$, on limit cones, and on $\H^{p,q}$-convex cocompact representations.
In Section~\ref{sec:nonpos-spheres-weakly-sp-gr} we discuss some central notions for the paper, namely weakly spacelike $p$-graphs in~$\H^{p,q}$ and their boundaries, which are non-positive $(p-1)$-spheres in $\di\H^{p,q}$.
In Section~\ref{sec:two-prelim-results} we establish two preliminary results, one that allows us to pass from a convex to a properly convex open subset of $\H^{p,q}$ in order to use Hilbert metrics (Proposition~\ref{prop:M-in-prop-conv-Omega}), and that also allows us to prove Proposition~\ref{prop:proper-action-M}, and another one about non-degenerate parts of weakly spacelike $p$-graphs and their boundaries (Proposition~\ref{prop:ME-weakly-sp-gr}).
In Section~\ref{sec:suff-cond-Hpq-cc} we prove Theorem~\ref{thm:geom-action-weakly-sp-gr} (hence Theorem~\ref{thm:geom-action-weakly-sp-gr-basic}) and discuss examples of weakly spacelike $p$-graphs with proper cocompact actions of \emph{non-hyperbolic} groups.
In Section~\ref{sec:non-deg-limit} we prove Proposition~\ref{prop:non-deg-limit}.
In Section~\ref{sec:proof-main-thm} we deduce Theorems \ref{thm:charact-Hpq-cc-sphere} and~\ref{thm:non-deg-limit-sphere}, and complete the proofs of Theorems \ref{thm:main} and \ref{thm:main-general}--\ref{thm:main-general-spacelike-p-mfd}.
In Appendix~\ref{appendix}, we explain how to obtain Zariski-dense $\H^{p,q}$-convex cocompact representations for various Gromov hyperbolic groups whose boundary is a $(p-1)$-sphere, including hyperbolic lattices but also more exotic examples, using \cite{lm19,mst}.

\subsection*{Acknowledgements}

We warmly thank Thierry Barbot, whose work \cite{bar15} has been a great source of inspiration to us, for stimulating discussions with the second-named author several years ago.
We are grateful to J\'er\'emy Toulisse for enlightening discussions on maximal spacelike $p$-manifolds, and to Andrea Seppi, Graham Smith, and J\'er\'emy Toulisse for kindly sharing with us their results from \cite{sst}.
We thank Yves Benoist for useful explanations about his paper \cite{ben05}, Jean-Marc Schlenker for interesting discussions about his paper \cite{mst} with Monclair and Tholozan, and Gye-Seon Lee for kindly pointing us to the examples in Table~\ref{table:Coxeter-diag}, taken from his paper \cite{lm19} with Marquis.
We thank Alejandro Maris Natera for useful questions on Appendix~\ref{subsec:appendix-Gromov-Thurston}.
We are grateful to the referees for their valuable comments and suggestions.
Finally, we thank the IHES in Bures-sur-Yvette and the MPIM in Bonn for providing us with excellent working conditions for this project.

\section{Reminders} \label{sec:reminders}

In the whole paper, we denote the signature of a symmetric bilinear form on a finite-dimensional real vector space by $(p,m|n)$, where $p$ (\resp $m$) is the dimension of a maximal positive definite (\resp negative definitive) linear subspace, and $n$ is the dimension of the kernel.
Thus the symmetric bilinear form is non-degenerate if and only if $n=0$.

\subsection{Properly convex domains in projective space} \label{subsec:prop-conv-proj}

Let $V$ be a finite-dimensional real vector space and $\Omega$ a properly convex open subset of $\P(V)$, with boundary $\partial\Omega = \ov{\Om}\smallsetminus\Om$.
Recall the \emph{Hilbert metric} $d_{\Omega}$ on~$\Omega$:
\begin{equation} \label{eqn:d-Omega}
d_{\Omega}(y,z) := \frac{1}{2} \log \, [a,y,z,b]
\end{equation}
for all distinct $y,z\in\Omega$, where $a,b$ are the intersection points of $\partial\Omega$ with the projective line through $y$ and~$z$, with $a,y,z,b$ in this order.
Here $[\cdot,\cdot,\cdot,\cdot]$ denotes the cross-ratio on $\P^1(\R)$, normalized so that $[0,1,t,\infty]=t$ for all~$t$. 
The metric space $(\Omega,d_{\Omega})$ is complete and proper (\ie closed balls are compact), and the automorphism group
$$\mathrm{Aut}(\Omega) := \{g\in\mathrm{PGL}(V) ~|~ g\cdot\Omega=\Omega\}$$
acts on~$\Omega$ by isometries for~$d_{\Omega}$.
As a consequence, any discrete subgroup of $\mathrm{Aut}(\Omega)$ acts properly discontinuously on~$\Omega$.

Straight lines (contained in projective lines) are always geodesics for the Hilbert metric $d_{\Omega}$.
When $\Omega$ is not strictly convex, there can be other geodesics as well.

\subsection{The pseudo-Riemannian hyperbolic space $\H^{p,q}$ and its boundary at infinity} \label{subsec:Hpq}

In the whole paper, we fix integers $p,q\geq 1$ and consider a symmetric bilinear form $\sfb$ of signature $(p,q+1|0)$ on $\R^{p+q+1}$.
We denote by $\R^{p,q+1}$ the space $\R^{p+q+1}$ endowed with~$\sfb$.
As in the introduction, we consider the open subset
$$\H^{p,q} = \P(\{v\in\R^{p+q+1} ~|~ \sfb(v,v)<0\})$$
of $\P(\R^{p+q+1})$ and its double cover 
\begin{equation} \label{eqn:Hpq-double}
\hat{\H}^{p,q} = \{v\in\R^{p+q+1} ~|~ \sfb(v,v)=-1\}.
\end{equation}
The non-trivial automorphism of the covering $\hat{\H}^{p,q} \to \H^{p,q}$ is given by $v\mapsto -v$.

For any $v\in \hat{\H}^{p,q}$ the restriction of $\sfb$ to $v^{\perp}\simeq T_v\hat{\H}^{p,q}$ is a symmetric bilinear form of signature $(p,q|0)$, and this defines a pseudo-Riemannian metric $\hat{\sfg}$ of signature $(p,q)$ on $\hat{\H}^{p,q}$, which is invariant under the orthogonal group $\OO(\sfb)\simeq \OO(p,q+1)$ of~$\sfb$.
It descends to a pseudo-Riemannian metric $\sfg$ of signature $(p,q)$ on $\H^{p,q}$, invariant under the projective orthogonal group $\PO(\sfb) = \OO(\sfb)/\{\pm\mathrm{I}\}\simeq \PO(p,q+1)$.

The unparametrized complete geodesics of $(\hat{\H}^{p,q},\hat\sfg)$ are the connected components of the non-empty intersections of $\hat{\H}^{p,q}$ with $2$-planes of $\R^{p,q+1}$; the unparametrized complete geodesics of $(\H^{p,q},\sfg)$ are the non-empty intersections of $\H^{p,q}$ with projectivizations of $2$-planes of $\R^{p,q+1}$.
There are three types of geodesics: a geodesic is \emph{spacelike} (\resp \emph{lightlike}, \resp \emph{timelike}) if the defining $2$-plane of $\R^{p,q+1}$ has signature $(1,1|0)$ (\resp $(0,1|1)$, \resp $(0,2|0)$).
For a parametrized geodesic~$\gamma$, this is equivalent to $\Vert\gamma'\Vert_{\hat\sfg}$ or $\Vert\gamma'\Vert_{\sfg}$ being everywhere $>0$ (\resp $=0$, \resp $<0$).

We say that two distinct points $o,m\in \H^{p,q}$ are \emph{in spacelike} (\resp \emph{lightlike}, \resp \emph{timelike}) \emph{position} if the unique geodesic through $o$ and~$m$ is spacelike (\resp lightlike, \resp timelike).
This is equivalent to the absolute value $|\sfb(\hat{o},\hat{m})|$ being $>1$ (\resp $=1$, \resp $<1$) for some (hence any) lifts $\hat{o},\hat{m}\in\hat{\H}^{p,q} = \{v\in\R^{p+q+1} ~|~ \sfb(v,v)=-1\}$ of $o,m\in\H^{p,q}$.

The totally geodesic subspaces of $(\hat{\H}^{p,q},\hat\sfg)$ are the connected components of the non-empty intersections of $\hat{\H}^{p,q}$ with linear subspaces of $\R^{p,q+1}$; the totally geodesic subspaces of $(\H^{p,q},\sfg)$ are the non-empty intersections of $\H^{p,q}$ with projectivizations of linear subspaces of $\P(\R^{p,q+1})$.
For $k\geq 1$, we shall say that a totally geodesic $k$-dimensional subspace is \emph{spacelike} (\resp \emph{lightlike}, \emph{timelike}) if the defining $(k+1)$-plane of $\R^{p,q+1}$ has signature $(k,1|0)$ (\resp $(0,1|k)$, \resp $(0,k+1|0)$); equivalently, all geodesics contained in the $k$-dimensional subspace are spacelike (\resp lightlike, \resp timelike).
We shall also call \emph{timelike} a $(k+1)$-plane of~$\R^{p,q+1}$ of signature $(0,k+1|0)$.

We consider the projective closure 
\begin{equation} \label{eqn:closure-Hpq}
\ov\H^{p,q}=\P(\{v\in\R^{p,q+1}\smallsetminus\{0\} ~|~ \sfb(v,v)\leq 0\})
\end{equation}
of $\H^{p,q}$ and its boundary
\begin{equation} \label{eqn:boundary-Hpq}
\di\H^{p,q}=\P(\{v\in\R^{p,q+1}\smallsetminus\{0\} ~|~ \sfb(v,v)=0\}),
\end{equation}
which we call the \emph{boundary at infinity} of~$\H^{p,q}$.

\begin{notation} \label{not:boundaries}
Given a subset $A\subset \H^{p,q}$, we denote by $\ov A$ its closure in $\P(\R^{p+q+1})$ and set
\begin{itemize}
  \item $\di A := \ov A\cap \di \H^{p,q}$ (boundary at infinity of~$A$),
  \item $\de A := \ov A\smallsetminus\mathrm{Int}(A) \subset \ov\H^{p,q}$,
  \item $\de_{\H} A := \de A\cap \H^{p,q}$,
  \item $\ov A^{\H} := \ov A\cap \H^{p,q}$.
\end{itemize}
\end{notation}

\subsection{The double cover of $\H^{p,q}$} \label{subsec:Hpq-hat}

Recall the double cover $\hat{\H}^{p,q}$ of~$\H^{p,q}$ from \eqref{eqn:Hpq-double}.
It also has a boundary at infinity $\di\hat{\H}^{p,q}$ which is a double cover of the boundary at infinity $\di\H^{p,q}$ of $\H^{p,q}$ from \eqref{eqn:boundary-Hpq}.
This yields a compactification $\hat{\overline{\H}}^{p,q} = \hat{\H}^{p,q} \sqcup \di\hat{\H}^{p,q}$ of $\hat{\H}^{p,q}$ which is a double cover of the compactification $\overline{\H}^{p,q}$ of $\H^{p,q}$ from~\eqref{eqn:closure-Hpq}.
Indeed, the double cover $\hat \H^{p,q}$, its boundary at infinity, and the corresponding compactification of $\hat{\H}^{p,q}$ identify with
\begin{align} 
\hat\H^{p,q} \simeq \{v\in\R^{p,q+1}\smallsetminus\{0\} ~|~ \sfb(v,v)< 0\}/\R_{>0},\nonumber\\
\di\hat{\H}^{p,q} \simeq \{v\in\R^{p,q+1}\smallsetminus\{0\} ~|~ \sfb(v,v)= 0\}/\R_{>0},\label{eqn:di-Hpq-hat-quotient}\\
\hat{\ov\H}^{p,q} \simeq \{v\in\R^{p,q+1}\smallsetminus\{0\} ~|~ \sfb(v,v)\leq 0\}/\R_{>0}.\nonumber
\end{align}
The non-trivial automorphism of the covering $\hat{\ov\H}^{p,q} \to \ov\H^{p,q}$ is given by $[v]\mapsto [-v]$.

We shall use the following notation.

\begin{notation} \label{not:Sn}
For $n\geq 1$, we denote by $\langle\cdot,\cdot\rangle_{n+1}$ the standard Euclidean inner product on~$\R^{n+1}$, by $\bS^n := \{ y\in\R^{n+1} \,|\, \langle y,y\rangle_{n+1}=1\}$ the unit sphere of $\R^{n+1}$ equipped with its standard spherical metric $\sfg_{\bS^n}$, and by $\bB^n := \{ y=(y_0,\dots,y_n)\in\bS^n \,|\, y_0>0\}$ the upper open hemisphere of $\bS^n$ equipped with the restriction $\sfg_{\bB^n}$ of $\sfg_{\bS^n}$.
\end{notation}

We now give a useful model for $\hat{\H}^{p,q}$ (``conformal model'', see \cite[Prop.\,2.4]{bar15} for $q=1$), based on the choice of a timelike $(q+1)$-plane $T$ of~$\R^{p,q+1}$.
Fix such a $(q+1)$-plane~$T$.
Let $\mathrm{pr}_T : \R^{p,q+1}\to T$ and $\mathrm{pr}_{T^{\perp}} : \R^{p,q+1}\to T^{\perp}$ be the $\sfb$-orthogonal projections onto $T$ and~$T^{\perp}$.
We define a function $r_T : \R^{p,q+1}\to\R_{\geq 0}$ by
\begin{equation} \label{eqn:r_T}
r_T(v) := \sqrt{|\sfb(\mathrm{pr}_T(v),\mathrm{pr}_T(v))|}
\end{equation}
for all $v\in\R^{p,q+1}$.
More explicitly, choose a $\sfb$-orthogonal basis $(e_1,\dots,e_{p+q+1})$ of~$\R^{p,q+1}$ which is \emph{standard for~$T$} in the sense that $\mathrm{span}(e_1,\dots,e_p) = T^{\perp}$, that $\mathrm{span}(e_{p+1},\dots,e_{p+q+1}) = T$, and that $\sfb(e_i,e_i)$ is equal to $1$ if $1\leq i\leq p$ and to $-1$ if $p+1\leq i\leq p+q+1$.
Then $r_T$ sends any $v = (v_1,\dots,v_{p+q+1})$ in this basis to $\sqrt{v_{p+1}^2+\dots+v_{p+q+1}^2}$.
Note that $r_T(\alpha v) = \alpha\,r_T(v)$ for all $\alpha\geq 0$ and $v\in\R^{p,q+1}$.

\begin{prop} \label{p.doublecoverasproduct}
For any timelike $(q+1)$-plane $T$ of $\R^{p,q+1}$, the map
\begin{eqnarray*}
\Psi_T :\hspace{1.8cm} \hat{\H}^{p,q}\hspace{1cm} & \longrightarrow & \hspace{2cm} \bB^p \hspace{1.3cm}\times\hspace{2.3cm} \bS^q \hspace{2.3cm}\\
\hat{x} = (v_1,\dots,v_{p+q+1}) & \longmapsto & \left(\left(\frac{1}{r_T(\hat{x})} (1,v_1,\dots,v_p)\right), \left(\frac{1}{r_T(\hat{x})} (v_{p+1},\dots,v_{p+q+1})\right)\right)
\end{eqnarray*}
(where on the left-hand side $\hat{x} = (v_1,\dots,v_{p+q+1})$ is expressed in a $\sfb$-orthogonal basis of~$\R^{p,q+1}$ which is standard for~$T$) is a diffeomorphism, which defines an isometry between the pseudo-Riemannian spaces $(\hat{\H}^{p,q},\hat{\sfg})$~and
$$(\bB^p \times \bS^q, y_0^{-2}\sfg_{\bB^p} \oplus -y_0^{-2}\sfg_{\bS^q}).$$
\end{prop}

We note that $(\bB^p,y_0^{-2}\sfg_{\bB^p})$ is the upper hemisphere model of the hyperbolic space~$\H^p$.
The image of $T$ in $\hat\H^{p,q}$ corresponds via~$\Psi_T$ to the subset $\{u_0\}\times\bS^q$ of $\bB^p\times\bS^q$, where $u_0 = (1,0,\dots,0)$ is the mid-point of the upper hemisphere~$\bB^p$.

\begin{proof}
The map $\Psi_T$ is clearly a diffeomorphism.
Let us check that it defines an isometry between $\hat{\H}^{p,q}$ and $(\mathbb{B}^p\times \bS^q, y_0^{-2}\sfg_{\bB^p} \oplus -y_0^{-2}\sfg_{\bS^q})$.
Fix a point $\hat{x} = (v_1,\dots,v_{p+q+1}) \in \hat{\H}^{p,q}$ and let $r := r_T(\hat x) > 0$.
The tangent space $T_{\hat{x}}\hat{\H}^{p,q}$ is the direct orthogonal sum of $T_{\hat{x}}M$ and $T_{\hat{x}}N$ where $M$ and~$N$ are the smooth submanifolds of~$\hat{\H}^{p,q}$ defined by
\begin{align*}
M & := \big\{ (y_1,\dots,y_p,tv_{p+1},\dots,tv_{p+q+1}) \in \R^{p,q+1} ~|~ t\in\R_{>0},\ 1 + y_1^2 + \dots + y_p^2 = t^2r^2\big\},\\
N & := \big\{ (v_1,\dots,v_p,y_{p+1},\dots,y_{p+q+1}) \in \R^{p,q+1} ~|~ y_{p+1}^2 + \dots + y_{p+q+1}^2 = r^2\big\}.
\end{align*}
The pseudo-Riemannian metric $\hat{\sfg}$ is positive definite (\resp negative definite) in restriction to $M$ (\resp $N$).
If we write $\Psi_T(\hat{x}) = (u,u') \in \bB^p\times\bS^q$, then the tangent space $T_{(u,u')}(\bB^p\times\nolinebreak\bS^q)$ is the direct orthogonal sum of $T_{(u,u')}(\bB^p\times\{u'\})$ and $T_{(u,u')}(\{u\}\times\bS^q)$, and $\mathrm{d}(\Psi_T)_{\hat{x}}$ defines an isometry between $T_{\hat{x}}M$ (\resp $T_{\hat{x}}N$) and $\bB^p\times\{u'\}$ (\resp $\{u\}\times\bS^q$) endowed with the metric $y_0^{-2}\sfg_{\bB^p} \oplus -y_0^{-2}\sfg_{\bS^q}$.
\end{proof}

\begin{lem} \label{lem:not-in-timelike-position}
Consider two distinct points $\hat x_1,\hat x_2\in \hat\H^{p,q}$ with $\sfb(\hat{x}_1,\hat{x}_2)\leq 0$.
Fix a timelike $(q+1)$-plane $T$ of $\R^{p,q+1}$ and write $\Psi_T(\hat x_i) = (u_i,u'_i) \in \bB^p\times \bS^q$.
Then $\hat x_1$ and $\hat x_2$ are \emph{not} in timelike position (\resp are in spacelike position) if and only if $d_{\bB^p}(u_1,u_2) \geq d_{\bS^q}(u'_1,u'_2)$ (\resp $d_{\bB^p}(u_1,u_2) > d_{\bS^q}(u'_1,u'_2)$).
\end{lem}

\begin{proof}
Recall from Section~\ref{subsec:Hpq} that $\hat x_1$ and $\hat x_2$ are in spacelike (\resp lightlike, \resp timelike) position if and only if $|\sfb(\hat{x}_1,\hat{x}_2)| > 1$ (\resp $|\sfb(\hat{x}_1,\hat{x}_2)| = 1$, \resp $|\sfb(\hat{x}_1,\hat{x}_2)| < 1$).
From the definition of~$\Psi_T$ (Proposition~\ref{p.doublecoverasproduct}), we have
$$\sfb(\hat{x}_1,\hat{x}_2) = r_T(\hat x_1) \, r_T(\hat x_2) \big(\<u_1,u_2\>_{p+1}-\<u'_1,u'_2\>_{q+1}\big) - 1.$$
Since $\sfb(\hat{x}_1,\hat{x}_2)\leq 0$ by assumption, we deduce that $|\sfb(\hat{x}_1,\hat{x}_2)| > 1$ (\resp $|\sfb(\hat{x}_1,\hat{x}_2)| = 1$, \resp $|\sfb(\hat{x}_1,\hat{x}_2)| < 1$) if and only if $d_{\bB^p}(u_1,u_2) > d_{\bS^q}(u'_1,u'_2)$ (\resp $d_{\bB^p}(u_1,u_2) = d_{\bS^q}(u'_1,u'_2)$, \resp $d_{\bB^p}(u_1,u_2) < d_{\bS^q}(u'_1,u'_2)$).
\end{proof}

\begin{lem} \label{lem:Hpq-hat-bar-prod}
For any timelike $(q+1)$-plane $T$ of $\R^{p,q+1}$, the map $\Psi_T$ of Proposition~\ref{p.doublecoverasproduct} extends continuously to a homeomorphism
$$\ov\Psi_T : \hat{\ov\H}^{p,q} \overset{\sim}{\longrightarrow} \overline{\bB}^p \times \bS^q,$$
where $\overline{\bB}^p$ is the closure of the upper hemisphere $\bB^p$ in $\bS^p$, \ie the union of $\bB^p$ and of the equator $\bS^{p-1} \simeq \{ y=(y_0,\dots,y_p)\in\bS^p \,|\, y_0=0\}$.
\end{lem}

\begin{proof}
Let $\ov\Psi_T : \hat{\ov\H}^{p,q} \to \overline{\bB}^p \times \bS^q$ be the map obtained by identifying $\hat{\ov\H}^{p,q}$ with\linebreak $\{v\in\R^{p,q+1}\smallsetminus\nolinebreak\{0\} \,|\, \sfb(v,v)\leq 0\}/\R_{>0}$ as in \eqref{eqn:di-Hpq-hat-quotient}, then with $\{ v\in\R^{p,q+1} \,|\, \sfb(v,v)\leq\nolinebreak 0 \ \mathrm{and}\ r_T(v)=\nolinebreak 1\}$ via $[v]\mapsto v/r_T(v)$, then with $\{ v\in T^{\perp} \,|\, \sfb(v,v)\leq 1\} \times \{ v\in T \,|\, (-\sfb)(v,v)=1\}$ via $(\mathrm{pr}_{T^{\perp}}, \mathrm{pr}_T)$, then with $\ov\bD^p \times \bS^q$ (where $\ov\bD^p$ is the closed unit disk of Euclidean~$\R^p$), then with $\ov\bB^p \times \bS^q$ via $\ov\bD^p \ni (y_1,\dots,y_p) \mapsto (\sqrt{1 - (y_1^2 - \dots - y_p^2)}, y_1, \dots, y_p) \in \ov\bB^p$.
Then $\ov\Psi_T$ is a composition of diffeomorphisms, hence a diffeomorphism.
Moreover, the restriction of $\ov\Psi_T$ to $\hat\H^{p,q}$ is the map $\Psi_T$ of Proposition~\ref{p.doublecoverasproduct}.
\end{proof}

We shall sometimes denote the restriction of $\ov\Psi_T$ to $\di\hat\H^{p,q}$ by
\begin{equation} \label{eqn:di-Psi}
\di\Psi_T : \di\hat\H^{p,q} \overset{\sim}{\longrightarrow} \bS^{p-1} \times \bS^q.
\end{equation}

\begin{lem} \label{lem:b-di-Hpq-hat}
Let $\hat x_1 \in \di\hat\H^{p,q}$ and $\hat x_2 \in \hat{\ov\H}^{p,q}$.
Fix a timelike $(q+1)$-plane $T$ of $\R^{p,q+1}$ and write $\ov\Psi_T(\hat x_i) = (u_i,u'_i) \in \overline{\bB}^p \times \bS^q$ and $\hat x_i = [v_i]$ where $v_i\in\R^{p,q+1}\smallsetminus\{0\}$ satisfies $\sfb(v_i,v_i)\leq 0$ as in \eqref{eqn:di-Hpq-hat-quotient}.
Then $\sfb(v_1,v_2) \leq 0$ (\resp $\sfb(v_1,v_2) < 0$) if and only if $d_{\ov\bB^p}(u_1,u_2) \geq d_{\bS^q}(u'_1,u'_2)$ (\resp $d_{\ov\bB^p}(u_1,u_2) > d_{\bS^q}(u'_1,u'_2)$).
\end{lem}

Here $d_{\ov\bB^p}$ denotes the distance function on $\ov\bB^p$ induced by the metric $\sfg_{\bS^p}$.

\begin{proof}
After renormalizing each $v_i$ by $r_T(v_i)$, we may assume that $r_T(v_1) = r_T(v_2) = 1$.
We then identify $\{ v\in\R^{p,q+1} \,|\, \sfb(v,v)\leq 0 \ \mathrm{and}\ r_T(v)=\nolinebreak 1\}$ with $\ov\bD^p \times \bS^q$ as in the proof of Lemma~\ref{lem:Hpq-hat-bar-prod}: the image of $v_1$ is $(u_1,u'_1) \in \bS^{p-1} \times \bS^q \subset \ov\bD^p \times \bS^q$, the image of $v_2$ is $(u''_2,u'_2) \in \ov\bD^p \times \bS^q$ where $u''_2$ corresponds to $u_2$ via the identification $\ov\bD^p \simeq \ov\bB^p$ in the proof of Lemma~\ref{lem:Hpq-hat-bar-prod}, and
\begin{equation} \label{eqn:b-di-Hpq-hat}
\sfb(v_1,v_2) = \langle u_1,u''_2\rangle_p - \langle u'_1,u'_2\rangle_{q+1} = \langle u_1,u_2\rangle_{p+1} - \langle u'_1,u'_2\rangle_{q+1}.
\end{equation}
In particular, $\sfb(v_1,v_2) \leq 0$ (\resp $\sfb(v_1,v_2) < 0$) if and only if $\langle u_1,u_2\rangle_{p+1} \leq \langle u'_1,u'_2\rangle_{q+1}$ (\resp $\langle u_1,u_2\rangle_{p+1} < \langle u'_1,u'_2\rangle_{q+1}$), which is equivalent to $d_{\ov\bB^p}(u_1,u_2) \geq d_{\bS^q}(u'_1,u'_2)$ (\resp $d_{\ov\bB^p}(u_1,u_2) > d_{\bS^q}(u'_1,u'_2)$).
\end{proof}

\subsection{Non-positive subsets of $\di\H^{p,q}$}

Following \cite[Def.\,3.1]{dgk18}, we shall adopt the following terminology.

\begin{defn} \label{def:nonpos-neg}
A subset $\tilde{\L}$ of $\R^{p,q+1}$ is \emph{non-positive} (\resp \emph{negative}) if $\sfb(\tilde{x},\tilde{y}) \leq 0$ (\resp $\sfb(\tilde{x},\tilde{y}) < 0$) for all $\tilde{x},\tilde{y}\in\tilde{\L}$.

A subset of $\di \H^{p,q}$ or $\di \hat \H^{p,q}$ is \emph{non-positive} (\resp \emph{negative}) if it is the projection to $\di\H^{p,q}$ or $\di\hat\H^{p,q}$ of a non-positive (\resp negative) subset of $\R^{p,q+1}$, where we see $\di\H^{p,q}$ as a subset of $\P(\R^{p,q+1})$ and $\di \hat \H^{p,q}$ as a subset of $(\R^{p,q+1}\smallsetminus\{0\})/\R_{>0}$ as in \eqref{eqn:di-Hpq-hat-quotient}.
\end{defn}

In the Lorentzian case, where $q=1$ and $\H^{p,q}$ is the anti-de Sitter space $\mathrm{AdS}^{p+1}$, non-positive (\resp negative) subsets of $\di\H^{p,q}$ are called \emph{achronal} (\resp \emph{acausal}): see \cite[Cor.\,2.11]{bar15}.

\begin{defn} \label{def:span-degen}
Let $\L$ be a subset of $\di\H^{p,q}$ or $\di\hat\H^{p,q}$.
We say that $\L$
\begin{itemize}
  \item \emph{spans} if the linear subspace $\spa(\L)$ of $\R^{p,q+1}$ spanned by $\L$ is the whole of~$\R^{p,q+1}$;
  \item is \emph{degenerate} if the restriction of $\mathsf{b}$ to $\spa(\L)$ is degenerate (\ie has non-zero kernel), and \emph{non-degenerate} otherwise.
\end{itemize}
\end{defn}

Note that a negative subset of $\di\H^{p,q}$ or $\di\hat\H^{p,q}$ is always non-degenerate.
If a non-positive subset $\Lambda$ of $\di\H^{p,q}$ or $\di\hat\H^{p,q}$ spans, then it is non-degenerate.
Conversely, if $\Lambda$ is non-degenerate, then the restriction of $\sfb$ to $\spa(\L)$ has signature $(p',q'|0)$ for some $1\leq q'\leq q+1$ and $1\leq p'\leq p$.
Thus a non-degenerate non-positive set $\L$ always spans in $\spa(\L)\simeq \R^{p',q'}$.

\subsection{Jordan projection and limit cone} \label{subsec:limit-cone}

Let $G$ be a non-compact Lie group which is the set of real points of a connected reductive real algebraic group~$\sfG$.

Let $K$ be a maximal compact subgroup of~$G$; its Lie algebra~$\mathfrak{k}$ is the set of fixed points of some involution (\emph{Cartan involution}) of the Lie algebra $\mathfrak{g}$ of~$G$.
Let $\mathfrak{p}$ be the set of anti-fixed points of this involution, so that $\mathfrak{g} = \mathfrak{k} + \mathfrak{p}$.
We fix a maximal abelian subspace $\mathfrak{a}$ of~$\mathfrak{p}$ (\emph{Cartan subspace} of~$\mathfrak{g}$).
The \emph{Weyl group} $W_G = N_G(\mathfrak{a})/Z_G(\mathfrak{a})$ acts on~$\mathfrak{a}$ with fundamental domain a closed convex cone $\mathfrak{a}^+$ of~$\mathfrak{a}$ (\emph{closed positive Weyl chamber}).

By the Jordan decomposition, any $g\in G$ may be written in a unique way as the commuting product $g_h g_e g_u$ of an element $g_h\in G$ which is hyperbolic (\ie conjugate to some element of $\exp(\mathfrak{a}^+)$), an element $g_e$ which is elliptic (\ie conjugate to some element of~$K$), and an element $g_u\in G$ which is unipotent.
We denote by $\lambda(g)$ the unique element of~$\mathfrak{a}^+$ such that $g_h$ is conjugate to $\exp(\lambda(g))$.
This defines a map $\lambda : G\to\mathfrak{a}^+$ called the \emph{Jordan projection} or \emph{Lyapunov projection}.

\begin{example} \label{ex:Jordan-decomp-PO}
Let $G=\PO(p,q+1)$, given by the standard quadratic form $v_1^2+\dots+v_p^2-v_{p+1}^2-\dots-v_{p+q+1}^2$ on $\R^{p+q+1}$.
Let $(e_1,\dots,e_{p+q+1})$ be the standard basis of $\R^{p,q+1}$. 
Let $r:=\min(p,q+1)$.
We can take $K$ to be $\mathrm{P}(\OO(p)\times\OO(q+1))$ and $\mathfrak{a}$ to be the set of matrices which, in the basis $(e_1+e_{p+q+1}, \dots, e_r+e_{p+q+2-r}, e_{r+1}, \dots, e_{p+q+1-r}, e_1-e_{p+q+1}, \dots, e_r-e_{p+q+2-r})$ of $\R^{p+q+1}$, are diagonal of the form $\mathrm{diag}(t_1,\dots,t_r,0,\dots,0,-t_1,\dots,-t_r)$ with $t_1,\dots,t_r\in\R$.
For $p\neq q+1$ (\resp $p=q+1=r$), we can take $\mathfrak{a}^+$ to be the subset of~$\mathfrak{a}$ defined by $t_1\geq\dots\geq t_r\geq 0$ (\resp $t_1\geq\dots\geq t_{r-1}\geq |t_r|$).
For $p\neq q+1$, the Jordan projection $\lambda = (\lambda_1,\dots,\lambda_r) : \PO(p,q+1)\to\mathfrak{a}^+$ gives the logarithms of the moduli of the $r$ first (complex) eigenvalues of elements of $\PO(p,q+1)$.
\end{example}

The following notion was introduced in full generality by Benoist \cite{ben97}.

\begin{defn} \label{def:limit-cone}
The \emph{limit cone} $\calL_{\Gamma}$ of a subsemigroup $\Gamma$ of~$G$ is the closure in~$\mathfrak{a}^+$ of the cone spanned by the elements $\lambda(\gamma)$ for $\gamma\in\Gamma$.
\end{defn}

In Section~\ref{subsec:vcd-bound} we shall use the following four facts.
The second one, stated as \cite[Fait\,2.2.c]{ben05}, is an immediate consequence of \cite[Prop.\,6.2--6.3]{ben97}.
The fourth one is due to Auslander.

\begin{fact}[{\cite[Th.\,1.2]{ben97}}] \label{fact:limit-cone}
Let $\Gamma$ be a Zariski-dense subsemigroup of~$G$.
Then the limit cone $\calL_{\Gamma}$ is convex with non-empty interior.
\end{fact}

\begin{fact}[{\cite[Fait\,2.2.c]{ben05}}] \label{fact:subsemigroup-in-cone}
Let $\Gamma$ be a Zariski-dense subsemigroup of~$G$.
For any open cone $\omega$ of~$\mathfrak{a}$ meeting the limit cone $\mathcal{L}_{\Gamma}$ (Definition~\ref{def:limit-cone}), there exists an open semigroup $G'$ of~$G$ meeting~$\Gamma$ and whose limit cone is contained in~$\omega$.
\end{fact}

\begin{fact} \label{fact:proj-surj}
Let $(\tau,V)$ be a finite-dimensional linear representation of~$G$, with kernel~$G_0$.
Let $\Gamma$ be a Zariski-dense discrete subgroup of~$G$ such that the Lie algebra $\mathfrak{g}_0$ of $G_0$ meets the limit cone~$\mathcal{L}_{\Gamma}$ only in~$\{0\}$.
Then $\Gamma\cap G_0$ is finite and $\tau(\Gamma)$ is discrete in $\GL(V)$.
\end{fact}

\begin{proof}
The representation $\tau$ of~$G$ factors through the canonical projection $\pi : G\to G/G_0$.
By \cite[Lem.\,6.1]{ben05}, the group $\Gamma\cap G_0$ is finite and the image $\pi(\Gamma)$ is discrete and Zariski-dense in $G/G_0$.
The group $G/G_0$ is still the real points of a connected reductive real algebraic group.
The linear representation of $G/G_0$ induced by~$\tau$ has finite kernel, hence the image of a Zariski-dense discrete subgroup of $G/G_0$ is discrete in $\GL(V)$.
\end{proof}

\begin{fact}[{see \cite[Ch.\,VIII]{rag72} or \cite[Cor.\,5.4]{abe01}}] \label{fact:proj-ss-discrete}
Let $G^{ss}$ be the real points of the commutator subgroup of~$\sfG$, and let $\pi : G\to G^{ss}$ be the natural projection.
For any Zariski-dense discrete subgroup $\Gamma$ of~$G$, the group $\pi(\Gamma)$ is discrete and Zariski-dense in~$G^{ss}$.
\end{fact}

\subsection{Cohomological dimension} \label{subsec:vcd}

Recall that the \emph{cohomological dimension} of a group~$\Gamma$ is the largest integer $n\in\N$ for which there exists a $\Z[\Gamma]$-module $\mathcal{M}$ with $H^n(\Gamma,\mathcal{M})$ non-zero.
If $\Gamma$ admits a finite-index subgroup which is torsion-free (this is the case \eg if $\Gamma$ is a finitely generated linear group, by the Selberg lemma \cite[Lem.\,8]{sel60}), then all torsion-free finite-index subgroups of~$\Gamma$ have the same cohomological dimension, called the \emph{virtual cohomological dimension} of~$\Gamma$, denoted by $\vcd(\Gamma)$.
We shall use the following property (see \cite{ser71}).

\begin{fact} \label{fact:vcd}
Let $\Gamma$ be a group admitting a torsion-free finite-index subgroup.
Suppose $\Gamma$ acts properly discontinuously on a contractible topological manifold $M$ without boundary.
Then $\vcd(\Gamma) \leq \dim(M)$, with equality if and only if the action of $\Gamma$ on~$M$ is cocompact.
\end{fact}

\subsection{$\H^{p,q}$-convex cocompact representations} \label{subsec:Hpq-cc-def}

Recall the notion of $\H^{p,q}$-convex cocompactness from Definition~\ref{def:Hpq-cc}.
We shall use the following characterization from \cite{dgk-proj-cc}, where the \emph{full orbital limit set} $\Lambda^{\mathsf{orb}}_{\Omega}(\Gamma)$ of $\Gamma$ in~$\Omega$ is by definition the union of all accumulation points in $\partial\Om$ of all $\Gamma$-orbits in~$\Omega$ \cite[Def.\,1.10]{dgk-proj-cc}.

\begin{fact}[{\cite[Th.\,1.24]{dgk-proj-cc}}] \label{fact:Hpq-cc-proj}
Let $p,q\geq 1$.
For an infinite discrete subgroup $\G$ of $\PO(p,q+1)$, the following are equivalent:
\begin{enumerate}
  \item $\G$ is $\hcc$-convex cocompact,
  \item\label{item:Hpq-cc-proj-2} $\G$ acts \emph{convex cocompactly} on some properly convex open subset $\Om$ of~$\H^{p,q}$, \ie $\Gamma$ preserves~$\Omega$ and the convex hull of the full orbital limit set $\Lambda^{\mathsf{orb}}_{\Omega}(\Gamma)$ in~$\Omega$ is non-empty and has compact quotient by~$\Gamma$.
\end{enumerate}
\end{fact}

The notion of $\H^{p,q}$-convex cocompactness is closely related, by \cite{dgk18,dgk-proj-cc}, to the notion of a $P_1$-Anosov representation into $\PO(p,q+1)$.
Here $P_1$ denotes the stabilizer in $G:=\PO(p,q+1)$ of an isotropic line of $\R^{p,q+1}$; it is a parabolic subgroup of~$G$, and $G/P_1$ identifies with the boundary at infinity $\di\H^{p,q}$ of $\H^{p,q}$.

The following is not the original definition from \cite{lab06,gw12}, but an equivalent characterization taken from \cite[Th.\,4.2]{ggkw17}.
Recall the notions of \emph{transverse}, \emph{dynamics-preserving} and \emph{proximal limit set} from Section~\ref{subsec:intro-Ano}, and the Jordan projection $\lambda = (\lambda_1,\dots,\lambda_{\min(p,q+1)}) :\linebreak \PO(p,q+1)\to\mathfrak{a}^+\subset\R^{\min(p,q+1)}$ from Section~\ref{subsec:limit-cone}.

\begin{defn} \label{def:P1-Ano}
Let $\Gamma$ be a Gromov hyperbolic group.
A representation $\rho : \Gamma\to G=\PO(p,q+1)$ is \emph{$P_1$-Anosov} if there exists a continuous, $\rho$-equivariant boundary map $\xi : \di\Gamma\to\di\H^{p,q}$ such that
\begin{enumerate}
  \item[(i)] \label{item:ano-trans} $\xi$ is transverse,
  \item[(ii)] $\xi$ is dynamics-preserving and for any sequence $(\gamma_n)_{n\in\N}$ of elements of~$\Gamma$ in pairwise distinct conjugacy classes in~$\Gamma$,
  $$(\lambda_1-\lambda_2)(\rho(\gamma_n)) \underset{n\to +\infty}{\longrightarrow} +\infty.$$
\end{enumerate}
\end{defn}

As in Section~\ref{subsec:intro-Ano}, the image $\xi(\di\Gamma)$ is then the proximal limit set $\Lambda_{\rho(\Gamma)}$ of $\rho(\Gamma)$ in $\di\H^{p,q}$.

\begin{fact}[{\cite[Th.\,1.15--1.24 \& Cor.\,11.10]{dgk-proj-cc}}] \label{fact:Hpq-cc-Ano}
Let $p,q\geq 1$.
For an infinite discrete subgroup $\G$ of $\PO(p,q+1)$, the following are equivalent:
\begin{enumerate}
  \item $\G$ is $\hcc$-convex cocompact,
  \item $\G$ is Gromov hyperbolic, the inclusion $\G\hookrightarrow \PO(p,q+1)$ is $P_1$-Anosov, and $\Gamma$ preserves a non-empty properly convex open subset of~$\H^{p,q}$;
  \item $\G$ is Gromov hyperbolic, the inclusion $\G\hookrightarrow \PO(p,q+1)$ is $P_1$-Anosov, and the proximal limit set $\L_{\G}\subset \di\H^{p,q}$ is negative.
\end{enumerate}
If these conditions hold, then $\di \mathscr{C} = \di \Om = \L_{\G}$, where $\mathscr{C}$ is any $\Gamma$-invariant properly convex closed subset of $\H^{p,q}$ with non-empty interior as in Definition~\ref{def:Hpq-cc} of $\hcc$-convex cocompactness, and $\Om$ is any $\Gamma$-invariant properly convex open subset of $\H^{p,q}$ as in Fact~\ref{fact:Hpq-cc-proj}.\eqref{item:Hpq-cc-proj-2}.
If moreover $\vcd(\G)=p$, then $\L_{\G}$ is homeomorphic to a $(p-1)$-dimensional sphere.
\end{fact}

\begin{remark} \label{rem:Hp1-cc-Mess-BM}
The special case where $q=1$ and $\Gamma$ is a uniform lattice in~$\SO(p,1)$ follows from work of Mess \cite{mes90} for $p=2$ and is work of Barbot--M\'erigot \cite{bm12} for $p\geq 3$.
\end{remark}

The following is a consequence of Fact~\ref{fact:Hpq-cc-Ano} and of the fact \cite{lab06,gw12} that being $P_1$-Anosov is an open property.

\begin{fact}[{\cite[Cor.\,1.12]{dgk18} \& \cite[Cor.\,1.25]{dgk-proj-cc}}] \label{fact:Hpq-cc-open}
Let $p,q\geq 1$.
The set of $\hcc$-convex cocompact representations is open in $\Hom(\G,\PO(p,q+1))$.
\end{fact}

\subsection{Closedness of injective and discrete representations}

The following is a classical consequence of the Kazhdan--Margulis--Zassenhaus theorem (see \eg \cite[Fait\,2.5]{ben05}).
It will be used in Sections \ref{subsec:Lambda-nondegen}, \ref{subsec:proof-main-thm}, and~\ref{subsec:proof-main-thm-general}.

\begin{fact} \label{fact:fd-closed}
Let $\Gamma$ be a finitely generated group with no infinite nilpotent normal subgroups, and let $G$ be a real semi-simple Lie group.
Then the set of injective and discrete representations is closed in~$\Hom(\Gamma,G)$.
\end{fact}

\subsection{Maximal spacelike submanifolds} \label{subsec:max-submfd}

Let $(X,\sfg)$ be $\H^{p,q}$, $\hat{\H}^{p,q}$, or more generally a pseudo-Riemannian $C^1$ manifold of signature $(p,q)$.
We shall call \emph{spacelike submanifold of~$X$} any immersed $C^1$ submanifold $M$ of~$X$ such that the restriction of the pseudo-Riemannian metric $\sfg$ to $TM$ is Riemannian, \ie pointwise positive definite.

If $M$ is a spacelike submanifold of~$X$, then the pull-back bundle of $TX$ to $M$ splits orthogonally as $TM\oplus NM$, where $NM$ is the \emph{normal bundle} to $M$ in~$X$.
Assuming $M$ and $X$ to be $C^2$, the \emph{second fundamental form} $\mathrm{II}$ of~$M$ is the symmetric tensor on $TM$ whose value at any point $m\in M$ is the symmetric bilinear form $\mathrm{II}_m:T_mM\times T_mM\to N_mM$ defined by the equation
$$\sf g_m\big((\nabla_Y Z)_m,\xi\big) = \sf g_m\big(\mathrm{II}_m(Y_m,Z_m),\xi\big)$$
for all vector fields $Y,Z$ on~$M$ and all $\xi\in N_mM$, where $\nabla$ is the (pseudo-Riemannian) Levi--Civita connection of $(\H^{p,q},\sf g)$.
The \emph{mean curvature $H : M\to NM$} of $M$ is the trace of $\mathrm{II}$ divided by $j := \dim(M)$: for any $m\in M$ and any orthonormal basis $(e_1,\ldots,e_j)$ of $T_m M$,
$$H(m) = \frac{1}{j} \sum_{i=1}^j \mathrm{II}(e_i,e_i) \in N_mM.$$

\begin{defn} \label{def:max-submfd}
A spacelike submanifold $M$ of a pseudo-Riemannian $C^2$ manifold $(X,\sfg)$ is \emph{maximal} if $M$ is at least $C^2$ and the mean curvature of~$M$ vanishes.
\end{defn}

The terminology comes from the fact that in pseudo-Riemannian geometry of signature $(p,q)$ with $q\geq 1$, the $p$-dimensional maximal spacelike submanifolds locally maximize the $p$-dimensional volume: see \cite[Cor.\,3.24]{ltw} for $p=2$.

\section{Non-positive spheres and weakly spacelike graphs} \label{sec:nonpos-spheres-weakly-sp-gr}

In this section we introduce and discuss the notions of a \emph{non-positive $(p-1)$-sphere} in $\di\H^{p,q}$ or $\di\hat\H^{p,q}$ (Definition~\ref{def:non-pos-sphere}) and of a \emph{weakly spacelike $p$-graph} in~$\H^{p,q}$ (Definition~\ref{def:weakly-sp-gr}), which will be used throughout the rest of the paper.

\subsection{Non-positive spheres in $\di\hat{\H}^{p,q}$ and $\di\H^{p,q}$}

The following notion generalizes, to $\di\H^{p,q}$, the notion of an achronal topological sphere in the Einstein universe $\mathrm{Ein}^p = \di\H^{p,1}$.

\begin{defn} \label{def:non-pos-sphere}
Let $0\leq\ell\leq p-1$.
A \emph{non-positive $\ell$-sphere in $\di \hat \H^{p,q}$} is a subset of $\di \hat \H^{p,q}$ which is non-positive (Definition~\ref{def:nonpos-neg}) and homeomorphic to an $\ell$-sphere.
A \emph{non-positive $\ell$-sphere in $\di \H^{p,q}$} is the projection to $\H^{p,q}$ of a non-positive $\ell$-sphere in $\di \hat \H^{p,q}$.
\end{defn}

We will be particularly interested in non-positive $\ell$-spheres for $\ell=p-1$.

\begin{remark} \label{rem:non-pos-sphere-not-sphere}
Despite its name, a non-positive $(p-1$)-sphere in $\di\H^{p,q}$ is not always homeomorphic to a $(p-1)$-sphere: for $p\geq 2$, this is the case if and only if $\L$ is non-degenerate (see Proposition~\ref{prop:lift-non-pos-sphere} below).
\end{remark}

For any $n\geq 1$, we endow the sphere $\bS^n$ with its standard spherical metric, and we denote the standard Euclidean inner product on~$\R^n$ by $\langle\cdot,\cdot\rangle_n$, as in Notation~\ref{not:Sn}.

\begin{lem} \label{l.nonpossets}
Fix a splitting $\di\hat\H^{p,q} \simeq \bS^{p-1}\times\bS^q$ defined by the choice of a timelike $(q+1)$-plane of $\R^{p,q+1}$ as in \eqref{eqn:di-Psi}.
A subset $\hat \L$ of $\di\hat \H^{p,q}$ is non-positive (\resp negative) (Definition~\ref{def:nonpos-neg}) if and only if, in this splitting, it is the graph of a $1$-Lipschitz (\resp strictly $1$-Lipschitz) map $f : A\to\bS^q$ for some subset $A$ of $\bS^{p-1}$.
\end{lem}

Here \emph{strictly $1$-Lipschitz} means that $d_{\bS^q}(f(u_1),f(u_2)) < d_{\bS^{p-1}}(u_1,u_2)$ for all $u_1\neq u_2$~in~$A$.

\begin{proof}
Suppose $\hat\L \subset \di\hat \H^{p,q}$ is non-positive (\resp negative): for any $v_1,v_2\in\R^{p,q+1}$ projecting to elements $[v_1]\neq [v_2]$ of $\hat\L$ via \eqref{eqn:di-Hpq-hat-quotient}, we have $\sfb(v_1,v_2)\leq 0$ (\resp $\sfb(v_1,v_2)<0$).
By Lemma~\ref{lem:b-di-Hpq-hat}, if $[v_i] \in \di\hat\H^{p,q}$ corresponds to $(u_i,u'_i) \in \bS^{p-1}\times\bS^q$ in our splitting, then $d_{\bS^{p-1}}(u_1,u_2) \geq d_{\bS^q}(u'_1,u'_2)$ (\resp $d_{\bS^{p-1}}(u_1,u_2) > d_{\bS^q}(u'_1,u'_2)$).
Therefore the restriction to~$\hat{\L}$ of the first-factor projection $\bS^{p-1}\times\bS^q \to \bS^{p-1}$ is injective, and so $\hat{\L}$ can be written as the graph of a map $f : A\to\bS^q$ for some subset $A$ of~$\bS^{p-1}$.
Moreover, the fact that $d_{\bS^{p-1}}(u_1,u_2) \geq d_{\bS^q}(u'_1,u'_2)$ (\resp $d_{\bS^{p-1}}(u_1,u_2) > d_{\bS^q}(u'_1,u'_2)$) implies that $f$ is $1$-Lipschitz (\resp strictly $1$-Lipschitz).

Conversely, for any subset $A$ of $\bS^{p-1}$ and any $1$-Lipschitz (\resp strictly $1$-Lipschitz) map $f : A\to\bS^q$, the graph of~$f$ in $\bS^{p-1}\times \bS^q \simeq \di\hat\H^{p,q}$ is non-positive (\resp negative) (Definition~\ref{def:nonpos-neg}) by Lemma~\ref{lem:b-di-Hpq-hat}.
\end{proof}

\begin{cor} \label{cor:non-positive-sphere-as-graph}
For a subset $\hat{\L}$ of $\di\hat\H^{p,q}$, the following are equivalent:
\begin{enumerate}[label=(\arabic*),ref=(\arabic*)]
  \item\label{item:non-pos-sph-1} $\hat{\L}$ is a non-positive $(p-1)$-sphere in $\di \hat \H^{p,q}$ (Definition~\ref{def:non-pos-sphere});
\end{enumerate}

\vspace{-0.15cm}

\begin{enumerate}[label=(\arabic*),ref=(\arabic*)]\setcounter{enumi}{1}
  \item\label{item:non-pos-sph-2} for \emph{any} splitting $\di\hat\H^{p,q} \simeq \bS^{p-1}\times\bS^q$ as in \eqref{eqn:di-Psi}, defined by the choice of a timelike $(q+1)$-plane of $\R^{p,q+1}$, the set $\hat{\L}$ is in this splitting the graph of a $1$-Lipschitz map $f : \bS^{p-1}\to\bS^q$;
\end{enumerate}

\vspace{-0.15cm}

\begin{enumerate}[label=(\arabic*)',ref=(\arabic*)']\setcounter{enumi}{1}
  \item\label{item:non-pos-sph-2-bis} for \emph{some} splitting $\di\hat\H^{p,q} \simeq \bS^{p-1}\times\bS^q$ as in \eqref{eqn:di-Psi}, defined by the choice of a timelike $(q+1)$-plane of $\R^{p,q+1}$, the set $\hat{\L}$ is in this splitting the graph of a $1$-Lipschitz map $f : \bS^{p-1}\to\bS^q$.
\end{enumerate}
\end{cor}

\begin{proof}
\ref{item:non-pos-sph-1}~$\Rightarrow$~\ref{item:non-pos-sph-2}: Suppose that $\hat\L \subset \di\hat\H^{p,q} \simeq \bS^{p-1}\times \bS^q$ is a non-positive $(p-1)$-sphere.
Consider a splitting $\di\hat\H^{p,q} \simeq \bS^{p-1}\times\bS^q$ as in \eqref{eqn:di-Psi}, defined by the choice of a timelike $(q+1)$-plane of $\R^{p,q+1}$.
By Lemma~\ref{l.nonpossets}, the set $\hat\L$ is the graph of a $1$-Lipschitz map $f : A\to\bS^q$ for some subset $A$ of $\bS^{p-1}$.
Since $f$ is continuous, the first-factor projection restricts to a homeomorphism between $\hat\L$ and~$A$, hence $A$ is a compact subset of $\bS^{p-1}$ homeomorphic to a $(p-1)$-sphere.
If $p=1$, this implies $A = \bS^{p-1}$.
Suppose $p\geq 2$.
By the domain invariance theorem, any point $\hat x\in\hat\L$ admits an open neighborhood in~$\hat\L$ which is mapped by~$f$ to an open subset of~$\bS^{p-1}$ contained in~$A$, hence $A$ is open in $\bS^{p-1}$.
Since $\bS^{p-1}$ is connected we deduce $A=\bS^{p-1}$.

\ref{item:non-pos-sph-2}~$\Rightarrow$~\ref{item:non-pos-sph-2-bis} is clear.

\ref{item:non-pos-sph-2-bis}~$\Rightarrow$~\ref{item:non-pos-sph-1}: Suppose that in some splitting $\di\hat\H^{p,q} \simeq \bS^{p-1}\times\bS^q$ as in \eqref{eqn:di-Psi}, the set $\hat{\L}$ is the graph of a 1-Lipschitz map $f : \bS^{p-1}\to\bS^q$.
Then $\hat\L$ is a non-positive subset of $\di\hat\H^{p,q}$ by Lemma~\ref{l.nonpossets}.
The graph $\hat\L \subset \bS^{p-1}\times\bS^q \simeq \di\hat\H^{p,q}$ is homeomorphic to the first factor $\bS^{p-1}$, hence it is a non-positive $(p-1)$-sphere in $\di\hat\H^{p,q}$.
\end{proof}

\begin{examples} \label{ex:non-pos-sphere}
Consider a splitting $\di\hat\H^{p,q} \simeq \bS^{p-1}\times\bS^q$ as in \eqref{eqn:di-Psi}, defined by the choice of a timelike $(q+1)$-plane $T$ of $\R^{p,q+1}$.

\smallskip

(i) If $f : \bS^{p-1}\to\bS^q$ is a constant map, then the graph of~$f$ is a non-positive $(p-1)$-sphere $\hat{\L}$ in $\di\hat\H^{p,q}$ which is the intersection of $\di\hat\H^{p,q}$ with a linear subspace of $\R^{p,q+1}$ of signature $(p,1|0)$ containing~$T^{\perp}$.
  This $\hat{\L}$ is non-degenerate, and in fact negative (Definition~\ref{def:nonpos-neg}); it is the boundary of a copy of $\H^p$ in~$\hat\H^{p,q}$.

\smallskip

(ii) If $p\leq q+1$ and if $f : \bS^{p-1}\to\bS^q$ is given by $(t_1,\dots,t_p) \mapsto (|t_1|,\dots,|t_p|,0,\dots,0)$, then the graph of~$f$ is a non-degenerate non-positive $(p-1)$-sphere $\hat{\L}$ in $\di\hat\H^{p,q}$, which spans a linear subspace of $\R^{p,q+1}$ of signature $(p,p|0)$ containing~$T^{\perp}$.
The image $\L$ of $\hat{\L}$ in $\di\H^{p,q}$ is the intersection of $\di\H^{p,q}$ with the boundary of an open projective simplex $\mathcal{O}$ of $\P(\spa(\L))$ contained in $\H^{p,q}$, with vertices $x_1^{\pm},\dots,x_p^{\pm}$ where $x_i^+$ and $x_i^-$ are in spacelike position, and $\spa(x_i^+,x_i^-)$ and $\spa(x_j^+,x_j^-)$ are orthogonal for all $1\leq i<j\leq p$.
(Such a set $\{x_1^{\pm},\dots,x_p^{\pm}\}$ of vertices will be called a \emph{$p$-crown} in Section~\ref{subsec:foliate-crown}.)
More precisely, $\L$ is the union of $2^p$ closed faces of~$\overline{\mathcal{O}}$ of dimension $p-1$, each determined by $p$ vertices of the form $x_1^{\varepsilon_1},\dots,x_p^{\varepsilon_p}$ for $\varepsilon_1,\dots,\varepsilon_p\in\{\pm\}$ (see Figure~\ref{fig:A_Corbit} in Section~\ref{subsec:weakly-spacelike-graphs} below for $(p,q) = (2,1)$).

\smallskip

(iii) If $p\leq q+1$ and if $f : \bS^{p-1}\to\bS^q$ is given by $(t_1,\dots,t_p) \mapsto (t_1,\dots,t_p,0,\dots,0)$, then the graph of~$f$ is a degenerate non-positive $(p-1)$-sphere $\hat{\L}$ in $\di\hat\H^{p,q}$, which spans a totally isotropic $p$-dimensional linear subspace of $\R^{p,q+1}$.
\end{examples}

Here is a useful consequence of Corollary~\ref{cor:non-positive-sphere-as-graph}.
We consider convergence for the Hausdorff topology.

\begin{cor} \label{c.limit-non-pos-sphere}
For $X = \hat\H^{p,q}$ or $\H^{p,q}$, any sequence of non-positive $(p-1)$-spheres in $\di X$ admits a subsequence that converges to a non-positive $(p-1)$-sphere in $\di X$.
\end{cor}

\begin{proof}
We consider the case $X = \hat\H^{p,q}$, as it implies the case $X = \H^{p,q}$ (see Definition~\ref{def:non-pos-sphere}).
Let $(\hat\L_n)_{n\in \N}$ be a sequence of non-positive $(p-1)$-spheres in $\di\hat\H^{p,q}$.
Consider a splitting $\di\hat\H^{p,q}\simeq \bS^{p-1}\times \bS^q$ as in \eqref{eqn:di-Psi}, defined by the choice of a timelike $(q+1)$-plane of $\R^{p,q+1}$.
By Corollary~\ref{cor:non-positive-sphere-as-graph}, in this splitting, each $\hat\L_n$ is the graph of a $1$-Lipschitz map $f_n : \bS^{p-1} \to \bS^q$.
By the Arzel\`a--Ascoli theorem, some subsequence $(f_{\varphi(n)})_{n\in\N}$ of $(f_n)$ converges to a $1$-Lipschitz map $f : \bS^{p-1} \to \bS^q$.
Then the graph $\hat\L \subset \di\hat\H^{p,q}$ of~$f$ is the limit of $(\hat\L_{\varphi(n)})_{n\in\N}$, and $\hat\L$ is a non-positive $(p-1)$-sphere by Corollary~\ref{cor:non-positive-sphere-as-graph}.
\end{proof}

We now discuss further the non-degeneracy (Definition~\ref{def:span-degen}) of non-positive $(p-1)$-spheres.

\begin{lem} \label{lem:kernel-of-graph}
Fix a splitting $\di\hat\H^{p,q} \simeq \bS^{p-1}\times\bS^q$ defined by the choice of a timelike $(q+1)$-plane $T$ of $\R^{p,q+1}$ as in \eqref{eqn:di-Psi}.
Let $\hat{\L}$ be a non-positive $(p-1)$-sphere in $\di\hat{\H}^{p,q}$ which, in this splitting, in the graph of some $1$-Lipschitz map $f : \bS^{p-1}\to\bS^q$ (see Corollary~\ref{cor:non-positive-sphere-as-graph}).
Then $\hat\L$ is degenerate if and only if $\bS := \{u\in\nolinebreak\bS^{p-1} \,|\, f(-u) = -f(u)\}$ is non-empty.
In this case,
\begin{itemize}
  \item $\bS$ is a totally geodesic copy of $\bS^{k-1}$ in~$\bS^{p-1}$, where $k$ is the dimension of the kernel $V$ of $\sfb|_{\spa(\hat\L)}$;
  \item the restriction $f|_{\bS}$ of $f$ to $\bS$ is an isometry;
  \item the graph of $f|_{\bS}$ is the image of $V\smallsetminus\{0\}$ in $\di\hat{\H}^{p,q} \subset (\R^{p,q+1}\smallsetminus\{0\})/\R_{>0}$; in particular, the image of $V\smallsetminus\{0\}$ in $\di\hat{\H}^{p,q}$ is contained in~$\hat\L$;
  \item if $k<p$, then the set $\bS'$ of points of $\bS^{p-1}$ at distance $\pi/2$ of~$\bS$ is a totally geodesic copy of $\bS^{p-1-k}$ in $\bS^{p-1}$ and the graph of $f|_{\bS'}$ is a non-degenerate non-positive $(p-1-k)$-sphere in $\di\hat\H^{p,q}$.
\end{itemize}
\end{lem}

\begin{remark}
For $q=1$, Barbot \cite[Def.\,3.7]{bar15} uses the terminology \emph{purely lightlike} for non-positive $(p-1)$-spheres of the Einstein universe $\mathrm{Ein}^p = \di\H^{p,1}$ which are degenerate.
\end{remark}

\begin{proof}[Proof of Lemma~\ref{lem:kernel-of-graph}]
We first check that the image of $V\smallsetminus\{0\}$ in $\di\hat{\H}^{p,q} \subset (\R^{p,q+1}\smallsetminus\nolinebreak\{0\})/\R_{>0}$ is contained in the graph of $f|_{\bS}$.
Let $z\in V\smallsetminus\{0\}$.
Then $z$ is isotropic and so, up to scaling it by some element of~$\R_{>0}$, it can be written as
$$z = (u,u') \in \bS^{p-1}\times\bS^q \simeq \{ v\in\R^{p,q+1} \,|\, \sfb(v,v)=0 \ \mathrm{and}\ r_T(v) = 1\},$$
where $r_T : \R^{p,q+1}\to\R_{\geq 0}$ is given by \eqref{eqn:r_T}.
By assumption we have $\sfb(z,v)=0$ for all $v\in\gph(f)$.
Applying this to $v = (u,f(u))$ and $v = (-u,f(-u))$, Lemma~\ref{lem:b-di-Hpq-hat} gives $d_{\bS^q}(u',f(u)) = d_{\bS^{p-1}}(u,u) = 0$ and $d_{\bS^q}(u',f(-u)) = d_{\bS^{p-1}}(u,-u) = \pi$.
Therefore $u' = f(u) = -f(-u)$ and $z = (u,f(u))$ belongs to the~graph~of~$f|_{\bS}$.

Next, we check that $\bS$ is a totally geodesic subsphere of~$\bS^{p-1}$ and that the restriction $f|_{\bS}$ of $f$ to $\bS$ is an isometry.
For this, we observe that for any $u\in\bS$, since $f$ is $1$-Lipschitz and $f(u)$ and $f(-u)$ are antipodal, the restriction of~$f$ to any geodesic segment between $u$ and $-u$ is an isometric embedding.
In particular, for any $u_1,u_2\in\bS$ with $u_1\neq\pm u_2$, the restriction of~$f$ to the unique totally geodesic circle of $\bS^{p-1}$ containing $u_1$ and~$u_2$ is an isometric embedding, and such a circle is contained in~$\bS$.
This shows that $\bS$ is a totally geodesic subsphere of~$\bS^{p-1}$ and that the restriction $f|_{\bS}$ of $f$ to $\bS$ is an isometry.

For any $u\in\bS$, the fact that the restriction of~$f$ to any geodesic segment between $u$ and $-u$ is an isometric embedding also shows that $d_{\bS^{q}}(f(u),f(u_1)) = d_{\bS^{p-1}}(u,u_1)$ for all $u_1\in\bS^{p-1}$, hence $\sfb((u,f(u),(u_1,f(u_1)))=0$ for all $u_1\in\bS^{p-1}$ by Lemma~\ref{lem:b-di-Hpq-hat}, and so $z:=(u,f(u))$ belongs to~$V$.
Thus the image of $V\smallsetminus\{0\}$ in $\di\hat{\H}^{p,q} \subset (\R^{p,q+1}\smallsetminus\{0\})/\R_{>0}$ is the full graph of $f|_{\bS}$.
In particular, the dimension of the totally geodesic subsphere $\bS$ (if non-empty) is $k-1$ where $k=\dim(V)$.

Suppose $1\leq k<p$.
Then $\bS'$ is non-empty and is a totally geodesic copy of $\bS^{p-1-k}$ in $\bS^{p-1}$.
By Lemma~\ref{l.nonpossets}, the graph of $f|_{\bS'}$ is a non-positive $(p-1-k)$-sphere in $\di\hat\H^{p,q}$.
Let us check that it is non-degenerate.
Suppose by contradiction that there exists $z\neq 0$ in the kernel of $\sfb$ restricted to $\spa(f|_{\bS'})$.
Then $z$ is isotropic and, up to scaling it by some element of~$\R_{>0}$, it can be written as
$$z = (u,u') \in \bS'\times\bS^q \subset \bS^{p-1}\times\bS^q \simeq \{ v\in\R^{p,q+1} \,|\, \sfb(v,v)=0 \ \mathrm{and}\ r_T(v) = 1\}.$$
As in the first paragraph of the proof, the fact that $\sfb(z,(u,f(u))) = \sfb(z,(-u,f(-u))) = 0$ implies that $f(-u) = -f(u)$: contradiction since $u\in\bS'$.
\end{proof}

\begin{lem} \label{l.kernonpossphere}
Let $\hat{\L}$ be a non-positive $(p-1)$-sphere in $\di\hat{\H}^{p,q}$.
Let $V \subset \R^{p,q+1}$ be the kernel of $\sfb|_{\spa(\L)}$, and let $k := \dim(V)\geq 0$.
Then
\begin{itemize}
  \item $k\leq\min(p,q+1)$;
  \item if $k=p$, then $\hat{\L}$ is the image of $V\smallsetminus\{0\}$ in $\di\hat\H^{p,q} \subset (\R^{p,q+1}\smallsetminus\{0\})/\R_{>0}$; in particular, the restriction $\sfb|_{\spa(\hat\L)}$ has signature $(0,0|k)$;
  \item if $k<p$, then the restriction $\sfb|_{\spa(\hat\L)}$ has signature $(p-k,q'|k)$ for some $1\leq q'\leq q+1-k$.
\end{itemize}
In particular, if $q+1<p$, then we always have $k<q+1$.
\end{lem}

\begin{proof}
Since $V$ is a totally isotropic subspace of~$\R^{p,q+1}$, we have $k\leq\min(p,q+1)$.

Fix a splitting $\di\hat\H^{p,q} \simeq \bS^{p-1}\times\bS^q$ defined by the choice of a timelike $(q+1)$-plane of $\R^{p,q+1}$ as in \eqref{eqn:di-Psi}.
By Corollary~\ref{cor:non-positive-sphere-as-graph}, in this splitting, $\hat{\L}$ is the graph of a $1$-Lipschitz map $f : \bS^{p-1}\to\bS^q$.
Let $\bS := \{u\in\nolinebreak\bS^{p-1} \,|\, f(-u) = -f(u)\}$.
If $\bS$ is non-empty, let $\bS'$ be the set of points of $\bS^{p-1}$ at distance $\pi/2$ of~$\bS$; if $\bS$ is empty, let $\bS' := \bS^{p-1}$.

If $k=p$, then by Lemma~\ref{lem:kernel-of-graph} we have $\bS = \bS^{p-1}$ and $\hat{\L}$ is the image of $V\smallsetminus\{0\}$ in $\di\hat\H^{p,q} \subset (\R^{p,q+1}\smallsetminus\{0\})/\R_{>0}$.

Suppose $k<p$.
By Lemma~\ref{lem:kernel-of-graph}, the set $\bS'$ is a totally geodesic copy of $\bS^{p-k-1}$ in $\bS^{p-1}$ and the graph of $f|_{\bS'}$ is a non-degenerate non-positive $(p-1-k)$-sphere in $\di\hat\H^{p,q}$.
Let $(p',q'|0)$ be the signature of the restriction of $\sfb$ to $E := \spa(\gph(f|_{\bS'}))$.
We have $q'\geq 1$ since $E$ contains $\sfb$-isotropic vectors.
The graph of $f|_{\bS'}$ is a non-positive $(p-1-k)$-sphere in $\di\hat\H_E \simeq \di\hat\H^{p',q'-1}$, where $\di\hat\H_E$ is the intersection of $\di\H^{p,q} \subset (\R^{p,q+1}\smallsetminus\{0\})/\R_{>0}$ with the image of $E\smallsetminus\{0\}$.
Therefore $p-k\leq p'$ by Lemma~\ref{l.nonpossets}.
On the other hand, we have $E \subset \spa(\hat\L) \subset V^{\perp}$, where $\sfb|_{V^{\perp}}$ has signature $(p-k,q+1-k|k)$.
We deduce $p' \leq p-k$ (hence $p' = p-k$) and $q' \leq q+1-k$.
\end{proof}

Here is a consequence of Lemma~\ref{lem:kernel-of-graph}.
We denote by $\varsigma$ the non-trivial automorphism of the double covering $\di\hat\H^{p,q}\to\di\H^{p,q}$, and refer to Definition~\ref{def:span-degen} for the notion of (non-) degeneracy.

\begin{prop} \label{prop:lift-non-pos-sphere}
Let $\L$ be a non-positive $(p-1)$-sphere in $\di\H^{p,q}$, and let $\hat\L$ be a non-positive $(p-1)$-sphere in $\di\hat\H^{p,q}$ projecting onto~$\L$, so that the full preimage of $\L$ in $\di\hat\H^{p,q}$ is $\hat\L \cup \varsigma(\hat\L)$.
\begin{itemize}
  \item If $\L$ is totally degenerate (\ie spans a totally isotropic subspace of~$\R^{p,q+1}$), then $\hat\L = \varsigma(\hat\L)$ is the unique non-positive $(p-1)$-sphere of $\di\hat\H^{p,q}$ projecting onto~$\L$; the projection from $\hat\L$ to~$\L$ is a double covering.
  \item If $\L$ is degenerate but not totally degenerate, then $p\geq 2$ and $\hat\L \cup \varsigma(\hat\L)$ is connected; $\hat\L$ and $\varsigma(\hat\L)$ are exactly the two non-positive $(p-1)$-spheres of $\di\hat\H^{p,q}$ projecting onto~$\L$; the projection from $\hat\L$ to~$\L$ is not injective and not a covering.
  \item If $\L$ is non-degenerate, then $\hat\L$ and $\varsigma(\hat\L)$ are disjoint; they are exactly the two non-positive $(p-1)$-spheres of $\di\hat\H^{p,q}$ projecting onto~$\L$; the projection from $\hat\L$ or $\varsigma(\hat\L)$ to~$\L$ is a homeomorphism.
 \end{itemize}
\end{prop}

\begin{proof}
Let $\di\hat\H^{p,q} \simeq \bS^{p-1}\times\bS^q$ be the splitting defined by the choice of a timelike $(q+1)$-plane of $\R^{p,q+1}$ as in \eqref{eqn:di-Psi}.
By Corollary~\ref{cor:non-positive-sphere-as-graph}, in this splitting, $\hat\L$ is the graph of some $1$-Lipschitz map $f : \bS^{p-1}\to\bS^q$.
For any $(u,u') \in \bS^{p-1}\times\bS^q \simeq \di\hat\H^{p,q}$ we have $\varsigma((u,u')) = (-u,-u')$.
Therefore $\varsigma(\hat\L)$ is the graph of the map $f^{\varsigma} : \bS^{p-1}\to\bS^q$ given by $f^{\varsigma}(u) := -f(-u)$, which is still $1$-Lipschitz, and $\varsigma(\hat\L)$ is still a non-positive $(p-1)$-sphere of $\di\hat\H^{p,q}$ by Corollary~\ref{cor:non-positive-sphere-as-graph}.

The intersection $\hat\L \cap \varsigma(\hat\L)$ is the graph of the restriction of $f$ to
$$\{ u\in\bS^{p-1} ~|~ f(u) = f^{\varsigma}(u)\} = \{ u\in\bS^{p-1} ~|~ f(-u) = -f(u)\} =: \bS.$$
The projection from $\hat\L$ to $\L$ is two-to-one on the graph of $f$ restricted to $\bS$, and one-to-one on the graph of $f$ restricted to $\bS^{p-1}\smallsetminus\bS$.
By Lemma~\ref{lem:kernel-of-graph}, we have $\bS\neq\emptyset$ (\ie $\hat\L \cap \varsigma(\hat\L) \neq \emptyset$, which for $p\geq 2$ is also equivalent to $\hat\L \cup \varsigma(\hat\L)$ being connected) if and only if $\L$ is degenerate, and $\bS = \bS^{p-1}$ (\ie $\hat\L = \varsigma(\hat\L)$) if and only if $\L$ is totally degenerate.

If $p=1$, then $\bS \neq \emptyset$ implies $\bS = \bS^{p-1}$ since $\bS$ is a totally geodesic subsphere of $\bS^{p-1}$ by Lemma~\ref{lem:kernel-of-graph}.
Therefore $\L$ cannot be degenerate but not totally degenerate for $p=1$.

Let $\hat\L'$ be a non-positive $(p-1)$-sphere of $\di\hat\H^{p,q}$ projecting onto~$\L$.
By Corollary~\ref{cor:non-positive-sphere-as-graph}, in our splitting, $\hat\L'$ is the graph of a $1$-Lipschitz map $g : \bS^{p-1}\to\bS^q$.
For any $u\in\bS^{p-1}$, we have $g(u) = f(u)$ or $g(u) = f^{\varsigma}(u)$.
Let $A$ (\resp $B$) be the set of elements $u\in\bS^{p-1}$ such that $g(u) = f(u)$ (\resp $g(u) = f^{\varsigma}(u)$).
Then $A$ and~$B$ are closed subsets of $\bS^{p-1}$ whose union is $\bS^{p-1}$ and whose intersection is~$\bS$.
Thus $\bS^{p-1}\smallsetminus\bS$ is the disjoint union of its two closed subsets $A\smallsetminus\bS$ and $B\smallsetminus\bS$, and so these two subsets are both unions of connected components of $\bS^{p-1}\smallsetminus\bS$.
If $\bS$ has codimension at least two in~$\bS^{p-1}$ (where the codimension of the empty set is by definition~$p$) then $\bS^{p-1}\smallsetminus\bS$ is connected, while if $\bS$ has codimension one in~$\bS^{p-1}$ then $\bS^{p-1}\smallsetminus\bS$ has two connected components that are switched by $u\mapsto -u$.

We claim that $A$ and~$B$ are both stable under $u\mapsto -u$.
Indeed, if $u\in A$ and $-u\in B$, then $g(-u) = -f(u) = -g(u)$, and so Lemma~\ref{lem:kernel-of-graph} ensures that $(u,g(u)) = (u,f(u))$ belongs to the kernel $V$ of $\sfb|_{\spa(\hat\L')} = \sfb|_{\spa(\hat\L)}$, hence $u\in\bS = A\cap B$, which proves the claim.

Therefore one of the sets $A\smallsetminus\bS$ or $B\smallsetminus\bS$ is empty, which means that $\hat\L' = \hat\L$ or $\varsigma(\hat\L)$.
\end{proof}

\subsection{The convex open set $\Omega(\L)$ associated to a non-positive $(p-1)$-sphere} \label{subsec:Omega-Lambda}

Recall the notion of a non-positive subset of $\R^{p,q+1}$ or $\di\H^{p,q}$ from Definition~\ref{def:nonpos-neg}.

\begin{notation} \label{not:Omega-Lambda}
For a non-positive subset $\tilde{\L}$ of~$\R^{p,q+1}$, we denote by $\tilde{\Om}(\tilde{\L})$ (\resp $\tilde{\ov{\Om}}(\tilde{\L})$) the set of vectors $v\in\R^{p,q+1}$ such that $\sfb(v,\tilde{x}) < 0$ (\resp $\sfb(v,\tilde{x}) \leq 0$) for all $\tilde{x}\in\tilde{\L}$.

For a non-positive $(p-1)$-sphere $\L$ of $\di \H^{p,q} \subset \P(\R^{p,q+1})$, we denote by $\Om(\L)$ (\resp $\ov{\Om}(\L)$) the image in $\P(\R^{p,q+1})$ of $\tilde{\Om}(\tilde{\L})$ (\resp $\tilde{\ov{\Om}}(\tilde{\L})\smallsetminus\{0\}$) where $\tilde{\L}$ is any subset of the non-zero isotropic vectors of $\R^{p,q+1}$ whose projection to $\di\hat\H^{p,q} \subset (\R^{p,q+1}\smallsetminus\{0\})/\R_{>0}$ is a non-positive $(p-1)$-sphere projecting onto~$\L$.
\end{notation}

It follows from Proposition~\ref{prop:lift-non-pos-sphere} that $\Om(\L)$ (\resp $\ov{\Om}(\L)$) does not depend on the choice of non-positive $(p-1)$-sphere of $\di\hat\H^{p,q}$ projecting onto~$\L$.
One also readily checks that $\Om(\L)$ (\resp $\ov{\Om}(\L)$) is an open (\resp closed) convex subset of $\P(\R^{p,q+1})$, as $\tilde{\Om}(\tilde{\L})$ (\resp $\tilde{\ov{\Om}}(\tilde{\L})$) is an intersection of open (\resp closed) half-spaces of~$\R^{p,q+1}$.
If $\Om(\L)$ is non-empty, then it is the interior of $\ov{\Om}(\L)$, and $\ov{\Om}(\L) = \ov{\Om(\L)}$ is the closure of $\Om(\L)$ in $\P(\R^{p,q})$.

\begin{remark}
In the Lorentzian case, where $q=1$ and $\H^{p,q}$ is the anti-de Sitter space $\mathrm{AdS}^{p+1}$, the set $\Omega(\L)$ is often called the \emph{invisible domain} of the achronal set~$\L$; this terminology is justified by Lemma~\ref{lem:Omega-Lambda-convex}.\eqref{item:Om-Lambda-invisible-dom} just below.
Invisible domains of achronal subsets of $\di\mathrm{AdS}^{p+1}$ containing at least two points are also called \emph{AdS regular domains} \cite[Def.\,3.1]{bar15}; when $\L$ is a topological $(p-1)$-sphere the AdS regular domain $\Omega(\L)$ is said to be \emph{GH-regular} (where GH stands for \emph{globally hyperbolic}).
\end{remark}

\begin{lem} \label{lem:Omega-Lambda-convex}
Let $\L$ be a non-positive $(p-1)$-sphere in $\di\H^{p,q}$ (Definition~\ref{def:non-pos-sphere}).
Then
\begin{enumerate}
  \item\label{item:Om-Lambda-in-Hpq} $\Om(\L)$ is a convex open subset of~$\H^{p,q}$ which, if non-empty, satisfies $\di\Om(\L) = \L$; more generally, $\ov{\Om}(\L) \cap \di\H^{p,q} = \L$;
  \item\label{item:Om-Lambda-non-empty} $\Om(\L)$ is non-empty if and only if $\L$ is non-degenerate;
  \item\label{item:Om-Lambda-prop-convex} if $\L$ spans, then $\Om(\L)$ is properly convex;
  \item\label{item:Om-Lambda-invisible-dom} if $p\geq 2$, then $\Om(\L)$ is equal to the set $\H^{p,q} \smallsetminus \bigcup_{w\in\L} w^{\perp}$ of points of $\H^{p,q}$ that see every point of~$\L$ in a spacelike direction; if $p=1$ and $\Om(\L)$ is non-empty, then $\L$ consists of two points and $\Om(\L)$ is one of the two connected components of the set $\H^{p,q} \smallsetminus \bigcup_{w\in\L} w^{\perp}$, namely the component containing the geodesic line of $\H^{p,q}$ between the two points of~$\L$.
\end{enumerate}
\end{lem}

For~\eqref{item:Om-Lambda-invisible-dom}, recall that any point $x$ of $\H^{p,q}$ sees any point $w$ of $\di\H^{p,q}$ in a spacelike direction (if $x\notin w^{\perp}$) or a lightlike direction (if $x\in w^{\perp}$).

\begin{proof}
\eqref{item:Om-Lambda-in-Hpq}
Let us check that the convex open subset $\Om(\L)$ of $\P(\R^{p,q+1})$ is contained in~$\H^{p,q}$ and that $\ov{\Om}(\L) \cap \di\H^{p,q} \subset \L$.
(The reverse inclusion $\ov{\Om}(\L) \cap \di\H^{p,q} \supset \L$ is clear.)
The non-positive $(p-1)$-sphere $\L$ is the projection to $\di\H^{p,q}$ of some non-positive $(p-1)$-sphere $\hat\L$ of $\di\hat\H^{p,q}$.
Choose a timelike $(q+1)$-plane $T$ of $\R^{p,q+1}$.
As in the proof of Lemma~\ref{lem:Hpq-hat-bar-prod}, we can identify $\di\hat\H^{p,q}$ with $\{ v\in\R^{p,q+1} \,|\, \sfb(v,v)=0 \ \mathrm{and}\ r_T(v) = 1\}$ and, via $(\mathrm{pr}_{T^{\perp}},\mathrm{pr}_T)$, to
$$\{ v\in T^{\perp} ~|~ \sfb(v,v)=1\} \times \{ v\in T ~|~ (-\sfb)(v,v)=1\} \simeq \bS^{p-1}\times\bS^q.$$
Let $\tilde{\L}$ be the image of $\hat\L$ in $\{ v\in\R^{p,q+1} \,|\, \sfb(v,v)=0 \ \mathrm{and}\ r_T(v) = 1\}$; it is a non-positive subset of $\R^{p,q+1}\smallsetminus\{0\}$ which projects onto~$\L$.
It is sufficient to check that for any $v\in\tilde{\ov{\Om}}(\tilde{\L})$, if $v\in\tilde{\Om}(\tilde{\L})$ then $\sfb(v,v)<0$, and if  $\sfb(v,v)=0$ then $v$ is a non-negative multiple of some element of~$\tilde{\L}$.
By Corollary~\ref{cor:non-positive-sphere-as-graph}, the set $\tilde{\L}$ is the graph of some $1$-Lipschitz map $f : \bS^{p-1}\simeq\{ v\in T^{\perp} \,|\, \sfb(v,v)=1\} \to \{ v\in\nolinebreak T \,|\, (-\sfb)(v,v)=1\}\simeq\bS^q$.
Consider $v\in\tilde{\ov{\Om}}(\tilde{\L})$.
We can write $v = (\alpha u, \alpha' u') \in T^{\perp}\oplus T = \R^{p,q+1}$ for some $\alpha, \alpha'\geq 0$ and $u\in T^{\perp}$ and $u'\in T$ with $\sfb(u,u)=-\sfb(u',u')=1$.
Then $\sfb(v,v) = \alpha^2\,\sfb(u,u) + {\alpha'}^2\,\sfb(u',u') = \alpha^2 - {\alpha'}^2$.
Since $v$ belongs to $\tilde{\ov{\Om}}(\tilde{\L})$ and $(u,f(u))$ belongs to~$\tilde{\Lambda}$, we have
\begin{align} \label{eqn:proof-Om}
0 \geq \sfb(v,(u,f(u))) = \alpha\,\sfb(u,u) + \alpha'\,\sfb(u',f(u)) \geq \alpha - \alpha'.
\end{align}
If $v\in\tilde{\Om}(\tilde{\L})$, then the first inequality in \eqref{eqn:proof-Om} is strict, hence $\alpha<\alpha'$, and so $\sfb(v,v) = \alpha^2 -\nolinebreak {\alpha'}^2 <\nolinebreak 0$.
If $\sfb(v,v) = 0$, then $\alpha = \alpha'$, hence all inequalities in \eqref{eqn:proof-Om} must be equalities, and so $v$ is equal to $(u,f(u))\in\tilde{\L}$ rescaled by~$\alpha$.

\medskip

\eqref{item:Om-Lambda-non-empty}
Suppose that $\L$ is non-degenerate.
Let $\tilde{\mathcal{U}}$ be the interior inside $\spa(\L)$ of the $\R^+$-span of~$\tilde{\L}$; it is a non-empty open subset of~$\spa(\L)$.
For any $v\in\tilde{\mathcal{U}}$ we have $\sfb(v,v)\leq 0$, hence $\sfb(v,v)<0$ since $\tilde{\mathcal{U}}$ is open in $\spa(\L)$ and the restriction of $\sfb$ to $\spa(\L)$ is non-degenerate.
For any $v\in\tilde{\mathcal{U}}$ and $\tilde{x}\in\tilde{\L}$ we have $\sfb(v,\tilde{x})\leq 0$, hence $\sfb(v,\tilde{x})<0$ since $\tilde{\mathcal{U}}$ is open.
Thus $\tilde{\mathcal{U}}$ is contained in $\tilde{\Om}(\tilde{\L})$, and its projectivization is contained in $\Omega(\L)$.

Conversely, suppose that $\L$ is degenerate.
Let $\hat\L$ be a non-positive $(p-1)$-sphere of $\di\hat\H^{p,q}$ projecting onto~$\L$.
Consider a splitting $\hat{\ov{\H}}^{p,q} \simeq \ov{\bB}^p\times \bS^q$ as in Lemma~\ref{lem:Hpq-hat-bar-prod}, defined by the choice of a timelike $(q+1)$-plane of $\R^{p,q+1}$.
By Corollary~\ref{cor:non-positive-sphere-as-graph}, in this splitting, $\hat\L$ is the graph of some $1$-Lipschitz map $f : \bS^{p-1}\to\bS^q$.
By Lemma~\ref{lem:kernel-of-graph}, there exists $u\in\bS^{p-1}$ such that $f(-u) = -f(u)$.
Then for any $(u_1,u'_1) \in \bB^p\times\bS^q$, the sign of $\langle u_1,u\rangle_{p+1} - \langle u'_1,f(u)\rangle_{q+1}$ is opposite to the sign of $\langle u_1,-u\rangle_{p+1} - \langle u'_1,f(-u)\rangle_{q+1}$; in other words, the sign of $d_{\ov\bB^p}(u_1,u) - d_{\bS^q}(u'_1,f(u))$ is opposite to the sign of $d_{\ov\bB^p}(u_1,-u) - d_{\bS^q}(u'_1,f(-u))$.
By Lemma~\ref{lem:b-di-Hpq-hat}, this means that the sign of $\sfb(v_1,v)$ is opposite to the sign of $\sfb(v_1,w)$ for $v_1,v,w\in\R^{p,q+1}\smallsetminus\{0\}$ corresponding respectively to $(u_1,u'_1) \in \hat\H^{p,q}$, to $(u,f(u)) \in \hat\L \subset \di\hat\H^{p,q}$, and to $(-u,f(-u)) \in \hat\L \subset \di\hat\H^{p,q}$, where we see $\hat{\ov{\H}}^{p,q}$ as a subset of $(\R^{p,q+1}\smallsetminus\{0\})/\R_{>0}$ as in \eqref{eqn:di-Hpq-hat-quotient}.
Since $\Om(\L) \subset \H^{p,q}$ by \eqref{item:Om-Lambda-in-Hpq}, this shows that $\Om(\L)$ is empty.

\medskip

\eqref{item:Om-Lambda-prop-convex}
We argue by contraposition.
Suppose that the convex open set $\Omega(\L)$ is not properly convex.
Then $\ov{\Om(\L)}$ contains a complete projective line $\ell \subset \P(\R^{p,q+1})$ (see \eg \cite[Th.\,1]{gv56}).
For any $w\in\L$, for dimensional reasons, either $\ell \subset w^{\perp}$, or $\ell \cap w^{\perp}$ is a single point $\ell_w\in\P(\R^{p,q+1})$.
For $x,w \in \L$, if $\ell$ lies neither in $w^{\perp}$ nor in $x^\perp$, then we claim that the points $\ell_x$ and~$\ell_w$ are equal.
Indeed, if the points were different, then only one segment of $\ell\smallsetminus\{\ell_x,\ell_w\}$ would be contained in $\Om(\L)$, which would contradict that $\ell\subset \ov{\Om(\L)}$. 
Consequently, there is a point $o\in\ell$ such that $o\in x^{\perp}$ for all $x\in \L$, \ie $o\in\L^{\perp}$, and so $\L$ does not span.

\medskip

\eqref{item:Om-Lambda-invisible-dom}
By definition of $\Om(\L)$ (see Notation~\ref{not:Omega-Lambda}), we have $\Om(\L) \subset \P(\R^{p,q+1}) \smallsetminus \bigcup_{w\in\L} w^{\perp}$.
Therefore $\Om(\L) \subset \H^{p,q} \smallsetminus \bigcup_{w\in\L} w^{\perp}$ by~\eqref{item:Om-Lambda-in-Hpq}.

For $p\geq 2$, let us check the reverse inclusion.
Since $p\geq 2$, we can consider a \emph{connected} subset $\tilde{\L}$ of the non-zero isotropic vectors of $\R^{p,q+1}$ whose projection to $\di\hat\H^{p,q} \subset (\R^{p,q+1}\smallsetminus\{0\})/\R_{>0}$ is a non-positive $(p-1)$-sphere projecting onto~$\L$.
Let $x \in \H^{p,q} \smallsetminus \bigcup_{w\in\L} w^{\perp}$.
We can lift $x$ to $\tilde{x} \in \R^{p,q+1}$ such that $\sfb(\tilde{x},\tilde{w})<0$ for some $\tilde{w}\in\tilde{\L}$.
Since $\tilde{\L}$ is connected and $\sfb(\tilde{x},\tilde{w})\neq 0$ for all $\tilde{w}\in\tilde{\L}$, by continuity we then have $\sfb(\tilde{x},\tilde{w})<0$ for all $\tilde{w}\in\tilde{\L}$, hence $\tilde{x}\in\tilde{\Om}(\tilde{\L})$ and $x\in\Om(\L)$.

We now assume that $p=1$ and $\Om(\L)$ is non-empty.
By~\eqref{item:Om-Lambda-non-empty}, the set $\L$ is non-degenerate, and so $\#\L = \#\bS^0 = 2$ and $\L$ lifts to $\tilde\L = \{\tilde x,\tilde y\}$ where $\tilde x,\tilde y\in\R^{p,q+1}\smallsetminus\{0\}$ are isotropic and satisfy $\sfb(\tilde x,\tilde y)<0$.
The set $\H^{p,q} \smallsetminus \bigcup_{w\in\L} w^{\perp}$ has two connected components, namely $\P(\{ v\in\R^{p,q+1} \,|\, \sfb(v,\tilde x) \, \sfb(v,\tilde y) > 0\})$ and $\H^{p,q} \cap \P(\{ v\in\R^{p,q+1} \,|\, \sfb(v,\tilde x) \, \sfb(v,\tilde y) < 0\})$.
The first component contains the geodesic line of~$\H^{p,q}$ between the two points of~$\L$ and is equal to $\Om(\L)$ (Notation~\ref{not:Omega-Lambda}).
\end{proof}

\begin{remark}
Lemma~\ref{lem:Omega-Lambda-convex}.\eqref{item:Om-Lambda-in-Hpq} implies that any non-positive $(p-1)$-sphere $\L$ of $\di\H^{p,q}$ is \emph{filled} in the sense of \cite[Rem.\,3.19]{bar15}: any segment of $\di \Om(\L)$ between two points of~$\L$ is fully contained in~$\L$.
\end{remark}

\begin{notation} \label{not:C-Lambda}
For a non-degenerate non-positive $(p-1)$-sphere $\Lambda\subset \di\H^{p,q}$, we define $\mathscr{C}(\L)$ to be the convex hull of $\L$ in the convex set $\Om(\L)$.
It is a closed convex subset of the open convex set $\Om(\L)$.
\end{notation}

In other words, the closure $\overline{\mathscr{C}(\L)}$ of $\mathscr{C}(\L)$ in $\P(\R^{p,q+1})$ is the union of all closed projective simplices of $\P(\R^{p,q+1})$ with vertices in~$\L$ that are contained in $\overline{\Om(\L)}$; the set $\mathscr{C}(\L)$ is the intersection of $\overline{\mathscr{C}(\L)}$ with $\Om(\L)$.

\begin{remark} \label{rem:C-Lambda}
Let $\L \subset \di\H^{p,q}$ be a non-degenerate non-positive $(p-1)$-sphere.
Then
\begin{itemize}
  \item by Lemma~\ref{lem:Omega-Lambda-convex}.\eqref{item:Om-Lambda-non-empty} and its proof, the open convex set $\Om(\L)$ and its closed convex subset $\mathscr{C}(\L)$ are non-empty; since $\L \subset \di\mathscr{C}(\L)$, we have $\di\mathscr{C}(\L)=\di\Om(\L)=\L$ by Lemma~\ref{lem:Omega-Lambda-convex}.\eqref{item:Om-Lambda-in-Hpq};
  \item the set $\mathscr{C}(\L) \subset \H^{p,q}$ is closed in~$\H^{p,q}$ if and only if $\de_{\H}\mathscr{C}(\L)\cap \de_{\H}\Om(\L)=\emptyset$;
  \item $\mathscr{C}(\L)$ is properly convex.
  Indeed, by definition we have $\mathscr{C}(\L)\subset \Om(\L)\cap \P(\spa(\L))$, and so it is sufficient to check that $\Om(\L)\cap \P(\spa(\L))$ is properly convex.
  For this we note that $\Om(\L)\cap \P(\spa(\L))$ coincides with the set $\Om_{\spa(\L)}(\L)$ as in Notation~\ref{not:Omega-Lambda} where we see $\L$ as a subset of $\P(\spa(\L))$ (\ie $\L$ spans); we then apply Lemma~\ref{lem:Omega-Lambda-convex}.\eqref{item:Om-Lambda-prop-convex}.
\end{itemize}
\end{remark}

\begin{example}
In the setting of Examples~\ref{ex:non-pos-sphere}, let $\L$ be the image in $\di\H^{p,q}$ of the non-positive $(p-1)$-sphere~$\hat{\L}$.
In case~(i), the set $\mathscr{C}(\L)$ is the copy of~$\H^p$ bounded by~$\L$; the convex open set $\Om(\L)$ is not properly convex; it is the geometric join of $\H^p$ with a $(q-1)$-dimensional timelike totally geodesic subspace of~$\H^{p,q}$, minus this timelike totally geodesic subspace.
In case~(ii), the set $\mathscr{C}(\L)$ is the $(2p-1)$-dimensional open projective simplex spanned by~$\L$; if $p=q+1$, then $\Om(\L) = \mathscr{C}(\L)$, and if $p<q+1$, then $\Om(\L)$ is the geometric join of $\mathscr{C}(\L)$ with a $(q-p)$-dimensional timelike totally geodesic subspace of~$\H^{p,q}$, minus this timelike totally geodesic subspace.
In case~(iii), the non-positive $(p-1)$-sphere~$\L$ is degenerate and $\Om(\L)$ is empty.
\end{example}

\subsection{Spacelike and weakly spacelike $p$-graphs: definition and characterizations} \label{subsec:weakly-spacelike-graphs}

The following notion generalizes, to $\H^{p,q}$, the notion of an acausal (\resp achronal) topological hypersurface in the Lorentzian anti-de Sitter space $\mathrm{AdS}^{p+1} = \H^{p,1}$.

\begin{defn} \label{def:weakly-sp-gr}
Let $X = \H^{p,q}$ or $\hat\H^{p,q}$.
A \emph{spacelike} (\resp \emph{weakly spacelike}) \emph{$p$-graph} in~$X$ is a closed subset of~$X$ on which any two points are in spacelike position (\resp on which no two points are in timelike position) and which is, with the subset topology, homeomorphic to a connected $p$-dimensional topological manifold without boundary.
\end{defn}

Note that any spacelike $p$-graph is in particular a weakly spacelike $p$-graph.
In Section~\ref{sec:non-deg-limit} we will consider the following special class of spacelike $p$-graphs, see also \cite[Lem.\,3.5]{sst}.

\begin{example} \label{ex:spacelike-compl-submfd}
Any $p$-dimensional connected complete spacelike $C^1$ submanifold of $X = \H^{p,q}$ or $\hat\H^{p,q}$ is a spacelike $p$-graph in~$X$.
\end{example}

The goal of this section is to establish the following characterizations, which justify the terminology of Definition~\ref{def:weakly-sp-gr}.

\begin{prop} \label{prop:weakly-sp-gr-is-a-graph}
For $X = \hat\H^{p,q}$ (\resp $\H^{p,q}$) and for a subset $\hat{M}$ (\resp $M$) of~$X$, the following are equivalent:
\begin{enumerate}[label=(\arabic*),ref=(\arabic*)]
  \item\label{item:weakly-sp-gr-1} $\hat{M}$ (\resp $M$) is a weakly spacelike $p$-graph in~$X$ (Definition~\ref{def:weakly-sp-gr});
\end{enumerate}

\vspace{-0.15cm}

\begin{enumerate}[label=(\arabic*)',ref=(\arabic*)']
  \item\label{item:weakly-sp-gr-1-bis} $\hat{M}$ (\resp $M$) is closed in~$X$ and meets every $q$-dimensional timelike totally geodesic subspace of~$X$ in a unique point;
\end{enumerate}

\vspace{-0.15cm}

\begin{enumerate}[label=(\arabic*),ref=(\arabic*)]\setcounter{enumi}{1}
  \item\label{item:weakly-sp-gr-2} for \emph{any} splitting $\hat\H^{p,q} \simeq \bB^p\times \bS^q$ as in Proposition~\ref{p.doublecoverasproduct}, defined by the choice of a timelike $(q+1)$-plane of $\R^{p,q+1}$, the set $\hat{M}$ is (\resp the set $M$ lifts injectively to a subset of $\hat\H^{p,q}$ which is) in this splitting the graph of a $1$-Lipschitz map $f : \bB^p\to\bS^q$;
\end{enumerate}

\vspace{-0.15cm}

\begin{enumerate}[label=(\arabic*)',ref=(\arabic*)']\setcounter{enumi}{1}
  \item\label{item:weakly-sp-gr-2-bis} for \emph{some} splitting $\hat\H^{p,q} \simeq \bB^p\times \bS^q$ as in Proposition~\ref{p.doublecoverasproduct}, defined by the choice of a timelike $(q+1)$-plane of $\R^{p,q+1}$, the set $\hat{M}$ is (\resp the set $M$ lifts injectively to a subset of $\hat\H^{p,q}$ which is) in this splitting the graph of a $1$-Lipschitz map $f : \bB^p\to\bS^q$.
\end{enumerate}
\end{prop}

\begin{prop} \label{prop:sp-gr-is-a-graph}
For $X = \hat\H^{p,q}$ (\resp $\H^{p,q}$) and for a subset $\hat{M}$ (\resp $M$) of~$X$, the following are equivalent:
\begin{enumerate}[label=(\arabic*),ref=(\arabic*)]
  \item\label{item:sp-gr-1} $\hat{M}$ (\resp $M$) is a spacelike $p$-graph in~$X$ (Definition~\ref{def:weakly-sp-gr});
  \item\label{item:sp-gr-2} for \emph{any} splitting $\hat\H^{p,q} \simeq \bB^p\times \bS^q$ as in Proposition~\ref{p.doublecoverasproduct}, defined by the choice of a timelike $(q+1)$-plane of $\R^{p,q+1}$, the set $\hat{M}$ is (\resp the set $M$ lifts injectively to a subset of $\hat\H^{p,q}$ which is) in this splitting the graph of a strictly $1$-Lipschitz map $f : \bB^p\to\bS^q$;
\end{enumerate}

\vspace{-0.15cm}

\begin{enumerate}[label=(\arabic*)',ref=(\arabic*)']\setcounter{enumi}{1}
  \item\label{item:sp-gr-2-bis} for \emph{some} splitting $\hat\H^{p,q} \simeq \bB^p\times \bS^q$ as in Proposition~\ref{p.doublecoverasproduct}, defined by the choice of a timelike $(q+1)$-plane of $\R^{p,q+1}$, the set $\hat{M}$ is (\resp the set $M$ lifts injectively to a subset of $\hat\H^{p,q}$ which is) in this splitting the graph of a strictly $1$-Lipschitz map $f : \bB^p\to\bS^q$.
\end{enumerate}
\end{prop}

Recall that \emph{strictly $1$-Lipschitz} means that $d_{\bS^q}(f(u_1),f(u_2)) < d_{\bB^p}(u_1,u_2)$ for all $u_1\neq\nolinebreak u_2$~in~$\bB^p$.

\begin{remark} \label{r.contractible}
Propositions \ref{prop:weakly-sp-gr-is-a-graph} and~\ref{prop:sp-gr-is-a-graph} imply that any spacelike (\resp weakly spacelike) $p$-graph in $\H^{p,q}$ lifts injectively to a spacelike (\resp weakly spacelike) $p$-graph in $\hat\H^{p,q}$, and that weakly spacelike $p$-graphs in $\H^{p,q}$ or $\hat\H^{p,q}$ are always homeomorphic to~$\bB^p$, hence contractible.
\end{remark}

\begin{examples} \label{ex:spacelike-graph}
Consider $\hat\H^{p,q} \simeq \bB^p \times \bS^q$ as in Proposition~\ref{p.doublecoverasproduct}, defined by the choice of a timelike $(q+1)$-plane $T$ of $\R^{p,q+1}$.

\smallskip

(i) If $f : \bB^p\to\bS^q$ is a constant map, then the graph of~$f$ is a spacelike $p$-graph in $\hat\H^{p,q}$ which is a totally geodesic copy of~$\H^p$.
Its boundary at infinity $\hat\L = \di\hat{M}$ is as in Example~\ref{ex:non-pos-sphere}.(i).

\smallskip

(ii) If $p\leq q+1$ and if $f : \bB^p\to\bS^q$ is given by
$$f(y_0, \dots, y_p) = \Bigg(\sqrt{\frac{y_0^2}{p} + y_1^2}, \dots, \sqrt{\frac{y_0^2}{p} + y_p^2}\Bigg),$$
then the graph $\hat{M}$ of~$f$ is a spacelike $p$-graph in $\hat\H^{p,q}$.
(In fact, it is a $p$-dimensional maximal complete spacelike $C^{\infty}$ submanifold of $\hat\H^{p,q}$, see Lemma~\ref{lem:F-spacelike} below with $j=p$.)
Its boundary at infinity $\hat\L = \di\hat{M}$ is as in Example~\ref{ex:non-pos-sphere}.(ii).
See Figure~\ref{fig:A_Corbit}.
\end{examples}

\begin{figure}[h]
\centering
\labellist
\small\hair 2pt
\pinlabel {$\H^{2,1}$} [u] at 265 280
\pinlabel {$M$} [u] at 220 220
\pinlabel {$x_1^+$} [u] at 139 231
\pinlabel {$\bullet$} [u] at 139 217
\pinlabel {$x_2^+$} [u] at 117 132
\pinlabel {$\bullet$} [u] at 117 144
\pinlabel {$x_1^-$} [u] at 226 274
\pinlabel {$\bullet$} [u] at 226 260
\pinlabel {$x_2^-$} [u] at 264 100
\pinlabel {$\bullet$} [u] at 264 110
\endlabellist
\vspace{-0.5cm}
\includegraphics[scale=0.7]{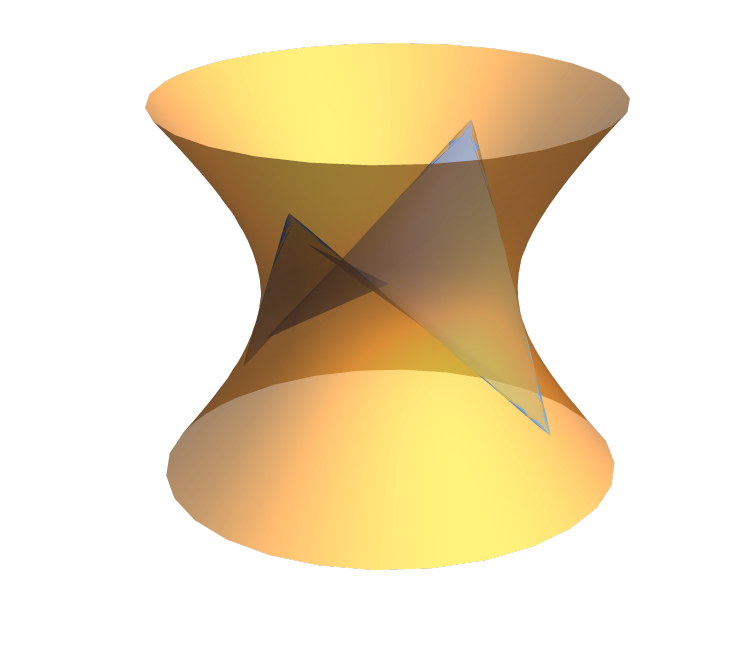}
\vspace{-0.8cm}
\caption{The anti-de Sitter space $\mathrm{AdS}^3 = \H^{2,1}$, intersected with an affine chart of $\P(\R^4)$, is the interior of the quadric of equation $v_1^2 + v_2^2 - v_3^2 = 0$.
Here we have represented, in such an affine chart, the image $M$ in $\H^{2,1}$ of a spacelike $2$-graph as in Example~\ref{ex:spacelike-graph}.(ii).
Its ideal boundary is the union of four segments of $\di\H^{2,1}$, joining points $x_1^{\pm}$ and $x_2^{\pm}$.}
\label{fig:A_Corbit}
\end{figure}

For the proof of \ref{item:weakly-sp-gr-1-bis}~$\Rightarrow$~\ref{item:weakly-sp-gr-1} in Proposition~\ref{prop:weakly-sp-gr-is-a-graph} we use the following basic observation.

\begin{remark} \label{rem:local-timelike-foliation}
Let $o\in \H^{p,q}$ and let $T$ be a timelike $(q+1)$-plane of $\R^{p,q+1}$ such that $o\in\P(T)$.
Consider a $\sfb$-orthogonal basis $(e_1,\ldots,e_{p+q+1})$ of $\R^{p,q+1}$ such that $\mathrm{span}(e_1,\dots,e_p) = T^{\perp}$ and $\mathrm{span}(e_{p+1},\dots,e_{p+q+1}) = T$, such that $o = [e_{p+q+1}]$, and such that $\sfb(e_i,e_i)$ is equal to $1$ if $1\leq i\leq p$ and to $-1$ if $p+1\leq i\leq p+q+1$.
For each $d\in T^{\perp}$ with $\sfb(d,d)<1$, consider the timelike $(q+1)$-plane
$$T_d := \mathrm{span}(e_{p+1},\dots,e_{p+q},d+e_{p+q+1})$$
of~$\R^{p,q+1}$.
We have $T_0 = T$.
In the affine chart $\{v_{p+q+1}=1\}$ of $\P(\R^{p+q+1})$, the sets $\P(T_d)$ are affine translates of $\{v_1 = \dots = v_p = 0\}$, with $\P(T_d)$ and $\P(T_{d'})$ disjoint in the affine chart if $d\neq d'$.
The map $d\mapsto\P(T_d)$ defines a foliation of the open subset
$$\mathcal{U} := \{ [v_1:\ldots:v_{p+q}:1] ~|~ v_1,\dots,v_{p+q}\in\R,\ v_1^2+\dots+v_p^2<1\}$$
of $\H^{p,q}$ (which contains $o$) by $q$-dimensional timelike totally geodesic subspaces, parametrized by the open unit disk $\bD^p$ of Euclidean~$\R^p$.
The intersection $\P(T_d) \cap \P(T_{d'}) = \P(\spa(e_{p+1},\dots,e_{p+q}))\linebreak = \P(T) \cap o^{\perp}$ is independent of the choice of $d\neq d'$ in~$\bD^p$.
\end{remark}

\begin{lem} \label{lem:weaklygraph-is-manifold}
For $X = \hat\H^{p,q}$ or $\H^{p,q}$, let $M$ be a closed subset of~$X$ meeting every $q$-dimensional timelike totally geodesic subspace of~$X$ in a unique point.
Then $M$ is homeomorphic to a connected $p$-dimensional topological manifold without boundary.
\end{lem}

\begin{proof}[Proof of Lemma~\ref{lem:weaklygraph-is-manifold}]
We give the proof for $X = \H^{p,q}$; the proof for $X = \hat\H^{p,q}$ is similar.
Let $M$ be a closed subset of $\H^{p,q}$ meeting every $q$-dimensional timelike totally geodesic subspace of $\H^{p,q}$ in a unique point.

We first fix a point $m\in M$.
Choosing a timelike $(q+1)$-plane $T$ of $\R^{p,q+1}$ such that $m\in\P(T)$, Remark~\ref{rem:local-timelike-foliation} provides an open neighborhood $\mathcal{U}_m$ of $m$ in $\H^{p,q}$ and a collection $\calT$ of timelike $(q+1)$-planes of $\R^{p,q+1}$ containing~$T$ such that $\calT$ is homeomorphic to~$\bD^p$, the sets $\P(T')\cap\mathcal{U}_m$ for $T'\in\calT$ define a foliation of~$\mathcal{U}_m$, and for any $T'\neq T''$ in $\calT$ we have $\P(T') \cap \P(T'') = \P(T) \cap m^{\perp} = \P(T') \smallsetminus \mathcal{U}_m$.
Since $M$ meets every $q$-dimensional timelike totally geodesic subspace of $\H^{p,q}$ in a unique point, we have $\P(T)\cap M = \{m\}$, hence $\P(T) \subset (X\smallsetminus M) \cup \{m\}$.
For any $T'\in\calT\smallsetminus\{T\}$, the facts that $\P(T') \smallsetminus \mathcal{U}_m \subset \P(T)$ and $m\notin\P(T')$ then imply $\P(T') \cap M \subset \mathcal{U}_m$.

Let $\varphi_m : \calT \to \mathcal{U}_m\cap M$ be the map sending any $T'$ to the unique intersection point of $\P(T')$ with $M$. 
We claim that $\varphi_m$ is a homeomorphism with respect to the subset topology on the image.
Indeed, $\varphi_m$ is bijective because $M$ meets every $q$-dimensional timelike totally geodesic subspace of $\H^{p,q}$ in a unique point and $\calT$ yields a foliation of $\mathcal{U}_m$.
The continuity of~$\varphi_m$ follows from the closedness of~$M$: if $(T_n)_{n\in\N}$ is a sequence of points of $\calT$ converging to $T'\in\calT$, then by closedness of~$M$ any accumulation point of $(\P(T_n)\cap M)_{n\in\N}$ is contained in $\P(T')\cap M$, which is a single point $\varphi(T')$, hence $\varphi(T_n)\to\varphi(T')$.
The continuity of~$\varphi_m^{-1}$ readily follows from the fact that the $\P(T')\cap\mathcal{U}_m$ for $T'\in\calT$ foliate $\mathcal{U}_m$.
Composing $\varphi_m^{-1}$ with a homeomorphism between $\calT$ and~$\bD^p$, we obtain a homeomorphism $\psi_m : \mathcal{U}_m\cap M \to \bD^p$.

We can do this for any $m\in M$.
We thus obtain an atlas of charts of~$M$ of the form $(\mathcal{U}_m\cap M,\psi_m)_{m\in M}$, with values in $\bD^p \subset \R^p$.
By construction, chart transition maps are homeomorphisms, and so $M$ is homeomorphic to a $p$-dimensional topological manifold without boundary.

We claim that $M$ is connected.
Indeed, given any two points $m_1,m_2\in M$, consider arbitrary timelike $(q+1)$-planes $T_1,T_2$ of $\R^{p,q+1}$ such that $m_i\in\P(T_i)$ for all $i\in\{1,2\}$.
Since the space of timelike $(q+1)$-planes of $\R^{p,q+1}$ is path connected (it identifies with the Riemannian symmetric space of $\PO(p,q+1)$, which is a contractible manifold), we can find a continuous path $t\mapsto T_t$ from $T_1$ to~$T_2$ in the space of timelike $(q+1)$-planes of $\R^{p,q+1}$.
Then $t\mapsto\P(T_t)\cap M$ defines a continuous path from $m_1$ to $m_2$ in~$M$.
\end{proof}

\begin{proof}[Proof of Proposition~\ref{prop:weakly-sp-gr-is-a-graph}]
\ref{item:weakly-sp-gr-1}~$\Rightarrow$~\ref{item:weakly-sp-gr-2} for~$\hat{M}$:
Suppose that $\hat M$ is a weakly spacelike $p$-graph in $\hat\H^{p,q}$.
The fact that no two points of $\hat M$ are in timelike position means that $|\sfb(\hat{x},\hat{y})|\geq 1$ for all $\hat{x},\hat{y}\in\hat{M}$ (see Section~\ref{subsec:Hpq}).
Since $\hat{M}$ is connected and $\sfb(\hat{x},\hat{x})= -1$, the continuity of~$\sfb$ implies that $\sfb(\hat{x},\hat{y})\leq -1$ for all $\hat{x},\hat{y}\in\hat{M}$.
Consider a splitting $\hat\H^{p,q} \simeq \bB^p\times \bS^q$ as in Proposition~\ref{p.doublecoverasproduct}, defined by the choice of a timelike $(q+1)$-plane of $\R^{p,q+1}$.
By Lemma~\ref{lem:not-in-timelike-position}, in our splitting we have $d_{\bB^p}(u_1,u_2) \geq d_{\bS^q}(u'_1,u'_2)$ for all $(u_1,u'_1),(u_2,u'_2)\in \hat M$.
Thus the first-factor projection restricted to $\hat M$ is injective, and so there exist a subset $A$ of $\bB^p$ and a $1$-Lipschitz map $f : A\to\bS^q$ such that $\hat M$ is the graph of~$f$.
Since $f$ is continuous, the first-factor projection restricted to $\hat M$ is a homeomorphism onto~$A$.
Since $\hat M$ is closed in $X=\hat\H^{p,q}$ by assumption, it follows that $A$ is closed in~$\bB^p$.
Moreover, since $\hat M$ is homeomorphic to a $p$-dimensional topological manifold without boundary, any point $\hat x\in\hat M$ admits an open neighborhood in~$\hat M$ whose image by the first-factor projection is an open subset of $\bB^p$ contained in~$A$, and so $A$ is open in $\bB^p$.
Since $\bB^p$ is connected we have $A=\bB^p$.
Thus $\hat M=\gph(f)$ for a $1$-Lipschitz map $f :\nolinebreak\bB^p\to\nolinebreak\bS^q$.

\ref{item:weakly-sp-gr-1}~$\Rightarrow$~\ref{item:weakly-sp-gr-2} for~$M$:
Suppose that $M$ is a weakly spacelike $p$-graph in $\H^{p,q}$.
Since $M$ is connected, the full preimage of $M$ in $\hat\H^{p,q}$ has at most two connected components.
Let $\hat{M}$ be one of these components.
As in the previous paragraph, since no two points of~$M$ are in timelike position, we have $|\sfb(\hat{x},\hat{y})|\geq 1$ for all $\hat{x},\hat{y}\in\hat{M}$ and in fact $\sfb(\hat{x},\hat{y})\leq -1$ for all $\hat{x},\hat{y}\in\hat{M}$ by connectedness of~$\hat{M}$.
In particular, if $\hat{x}\in\hat{M}$, then $-\hat{x}\notin\hat{M}$, which shows that $\hat{M}$ projects injectively onto~$M$. Hence the full preimage of $M$ in $\hat\H^{p,q}$ has exactly two connected components and $\hat M$ is a closed subset of $\hat\H^{p,q}$ which is homeomorphic to a connected $p$-dimensional topological manifold without boundary.
Moreover, the inequality $\sfb(\hat{x},\hat{y})\leq -1$ for all $\hat{x},\hat{y}\in\hat{M}$ shows that no points of $\hat{M}$ are in timelike position.
We conclude using the implication \ref{item:weakly-sp-gr-1}~$\Rightarrow$~\ref{item:weakly-sp-gr-2} for~$\hat{M}$.

\ref{item:weakly-sp-gr-2}~$\Rightarrow$~\ref{item:weakly-sp-gr-1-bis} for~$\hat{M}$:
Suppose that $\hat{M}$ satisfies~\ref{item:weakly-sp-gr-2}.
Let $T$ be a timelike $(q+1)$-plane of $\R^{p,q+1}$, and let us check that $\hat{M}$ meets the image of $T$ in $\hat\H^{p,q}$ in a unique point.
Consider the splitting $\hat\H^{p,q} \simeq \bB^p\times \bS^q$ given by~$T$ as in Proposition~\ref{p.doublecoverasproduct}.
In this splitting, the image of $T\smallsetminus\{0\}$ in $\hat{\H}^{p,q}$ is $\{u_0\}\times\bS^q$ where $u_0 = (1,0,\dots,0)$ is the mid-point of the upper hemisphere~$\bB^p$.
The fact that $\hat{M}$ is the graph of a map $f : \bB^p\to\bS^q$ then implies that $\hat{M}$ meets the image of~$T$ in a unique point (namely $(u_0,f(u_0))$).

\ref{item:weakly-sp-gr-2}~$\Rightarrow$~\ref{item:weakly-sp-gr-1-bis} for~$M$:
Suppose that $M$ lifts injectively to a subset $\hat{M}$ of $\hat{\H}^{p,q}$ satisfying~\ref{item:weakly-sp-gr-2}.
For any timelike $(q+1)$-plane $T$ of $\R^{p,q+1}$, we have just seen that $\hat{M}$ meets the image of $T$ in $\hat\H^{p,q}$ in a unique point; therefore, $M$ meets the image of $T$ in $\H^{p,q}$ in a unique point.

\ref{item:weakly-sp-gr-1-bis}~$\Rightarrow$~\ref{item:weakly-sp-gr-1}:
For $X = \hat\H^{p,q}$ (\resp $\H^{p,q}$), any two points of~$X$ which are in timelike position belong to a common $q$-dimensional timelike totally geodesic subspace of~$X$.
Therefore, the fact that $\hat M$ (\resp $M$) meets every such subspace of~$X$ in a unique point implies that no two points of $\hat M$ (\resp $M$) are in timelike position.
We conclude using Lemma~\ref{lem:weaklygraph-is-manifold}.

\ref{item:weakly-sp-gr-2}~$\Rightarrow$~\ref{item:weakly-sp-gr-2-bis} is clear.

\ref{item:weakly-sp-gr-2-bis}~$\Rightarrow$~\ref{item:weakly-sp-gr-1} for $\hat M$: Suppose that in some splitting $\hat\H^{p,q} \simeq \bB^p\times \bS^q$ as in Proposition~\ref{p.doublecoverasproduct}, the set $\hat{M}$ is the graph of a 1-Lipschitz map $f:\bB^p\to \bS^q$.
Since $f$ is continuous, the first-factor projection restricted to $\hat M$ is a homeomorphism onto~$\bB^p$.
In particular, $\hat M$ is closed in $X = \hat\H^{p,q}$ and homeomorphic to a topological $p$-dimensional manifold without boundary.
Since $f$ is $1$-Lipschitz we have $d_{\bB^p}(u_1,u_2) \geq d_{\bS^q}(u'_1,u'_2)$ for all $(u_1,u'_1),(u_2,u'_2)\in \hat M$, and so Lemma~\ref{lem:not-in-timelike-position} implies that no two points of $\hat M$ are in timelike position.

\ref{item:weakly-sp-gr-2-bis}~$\Rightarrow$~\ref{item:weakly-sp-gr-1} for $M$: Suppose that $M$ admits an injective lift $\hat{M} \subset \hat\H^{p,q}$ which, in some splitting $\hat\H^{p,q} \simeq \bB^p\times \bS^q$ as in Proposition~\ref{p.doublecoverasproduct}, is the graph of a $1$-Lipschitz map $f:\bB^p\to \bS^q$.
The previous paragraph shows that $\hat M$ is closed in $\hat\H^{p,q}$ and homeomorphic to a topological $p$-dimensional manifold without boundary.
The same holds for~$M$ since the natural projection from $\hat M$ to~$M$ is a homeomorphism.
Since no two points of $\hat M$ are in timelike position, the same holds for~$M$.
\end{proof}

\begin{proof}[Proof of Proposition~\ref{prop:sp-gr-is-a-graph}]
Observe that, by Lemma~\ref{lem:not-in-timelike-position}, for any splitting $\hat\H^{p,q} \simeq \bB^p\times \bS^q$ as in Proposition~\ref{p.doublecoverasproduct}, defined by the choice of a timelike $(q+1)$-plane of $\R^{p,q+1}$, and for any $1$-Lipschitz map $f : \bB^p\to\bS^q$, the map $f$ is strictly $1$-Lipschitz if and only if any two points of the graph of $f$ in $\hat\H^{p,q}$ are in spacelike position.

Proposition~\ref{prop:sp-gr-is-a-graph} is an immediate consequence of the equivalence \ref{item:weakly-sp-gr-1}~$\Leftrightarrow$~\ref{item:weakly-sp-gr-2}~$\Leftrightarrow$~\ref{item:weakly-sp-gr-2-bis} in Proposition~\ref{prop:weakly-sp-gr-is-a-graph} and of this observation.
\end{proof}

The following elementary observation will be useful in Section~\ref{subsec:spacelike-cocompact}.

\begin{lem} \label{lem:finite-group-preserv-M}
Let $M$ be a weakly spacelike $p$-graph in $\H^{p,q}$ (\resp $\hat\H^{p,q}$).
Then any compact subgroup of $\PO(p,q+1)$ (\resp $\OO(p,q+1)$) preserving~$M$ admits a global fixed point in~$M$.
\end{lem}

\begin{proof}
We give the proof for $\H^{p,q}$; the proof for $\hat\H^{p,q}$ is similar.
Any compact subgroup of $\PO(p,q+\nolinebreak 1)$ is contained in a maximal one, hence preserves a $q$-dimensional timelike totally geodesic subspace $\P(T)$ of $\H^{p,q}$.
If the compact subgroup preserves~$M$, then it also preserves the intersection $\P(T)\cap M$, which is a singleton by Proposition~\ref{prop:weakly-sp-gr-is-a-graph}.
\end{proof}

\subsection{Limits, boundaries, and projections of weakly spacelike $p$-graphs} \label{subsec:lim-bound-proj}

Here is an easy consequence of the equivalence \ref{item:weakly-sp-gr-1}~$\Leftrightarrow$~\ref{item:weakly-sp-gr-2-bis} in Proposition~\ref{prop:weakly-sp-gr-is-a-graph}.

\begin{cor} \label{c.limit-weakly-sp-gr}
For $X = \hat\H^{p,q}$ or $\H^{p,q}$, any sequence of weakly spacelike $p$-graphs in~$X$ admits a subsequence that converges to a weakly spacelike $p$-graph in~$X$.
\end{cor}

As in Corollary~\ref{c.limit-non-pos-sphere}, we consider convergence with respect to the Hausdorff topology for subsets of the projective space $\P(\R^{p,q+1})$, restricted to weakly spacelike $p$-graphs.
Using characterization \ref{item:weakly-sp-gr-2} of Proposition~\ref{prop:weakly-sp-gr-is-a-graph}, this coincides with the topology induced by pointwise convergence of $1$-Lipschitz maps $f:\bB^p\to\bS^q$ for any splitting $\hat\H^{p,q} \simeq \bB^p\times \bS^q$ as in Proposition~\ref{p.doublecoverasproduct}.

\begin{proof}[Proof of Corollary~\ref{c.limit-weakly-sp-gr}]
We consider the case $X = \hat\H^{p,q}$, as it implies the case $X = \H^{p,q}$ (see Proposition~\ref{prop:weakly-sp-gr-is-a-graph}).
Let $(\hat M_n)_{n\in \N}$ be a sequence of weakly spacelike $p$-graphs in $\hat\H^{p,q}$.
Consider a splitting $\hat\H^{p,q}\simeq \bB^p\times \bS^q$ as in Proposition~\ref{p.doublecoverasproduct}, defined by the choice of a timelike $(q+1)$-plane of $\R^{p,q+1}$.
By Proposition~\ref{prop:weakly-sp-gr-is-a-graph}, in this splitting, each $\hat M_n$ is the graph of a $1$-Lipschitz map $f_n : \bB^p\to \bS^q$.
By the Arzel\`a--Ascoli theorem, some subsequence $(f_{\varphi(n)})_{n\in\N}$ of $(f_n)$ converges to a $1$-Lipschitz map $f : \bB^p\to \bS^q$.
Then the graph $\hat M \subset \hat\H^{p,q}$ of~$f$ is the limit of $(\hat M_{\varphi(n)})_{n\in\N}$, and $\hat M$ is a weakly spacelike $p$-graph by Proposition~\ref{prop:weakly-sp-gr-is-a-graph}.
\end{proof}

The following will be used throughout the paper.

\begin{prop} \label{prop:bdynonpossphere}
Let $M$ be a weakly spacelike $p$-graph in~$\H^{p,q}$, with boundary at infinity $\L := \di M \subset \di\H^{p,q}$.
Then
\begin{enumerate}
  \item\label{item:bound-non-pos-sph-1} $\L$ is a non-positive $(p-1)$-sphere in $\di\H^{p,q}$ (Definition~\ref{def:non-pos-sphere}) and $M \subset \ov{\Om}(\L)$;
  \item\label{item:bound-non-pos-sph-2} if $M$ is a spacelike $p$-graph in~$\H^{p,q}$, then $M \subset \Om(\L)$; in particular, $\L$ is non-degenerate by Lemma~\ref{lem:Omega-Lambda-convex}.\eqref{item:Om-Lambda-non-empty}.
\end{enumerate}
\end{prop}

Recall from Lemma~\ref{lem:Omega-Lambda-convex}.\eqref{item:Om-Lambda-invisible-dom} that the condition $M \subset \Om(\L)$ (rather than just $M \subset \ov{\Om}(\L)$) means that any point of~$M$ sees any point of~$\L$ in a spacelike direction.

\begin{proof}[Proof of Proposition~\ref{prop:bdynonpossphere}]
\eqref{item:bound-non-pos-sph-1} Consider a splitting $\hat{\ov\H}^{p,q} \simeq \ov{\bB}^p\times \bS^q$ as in Lemma~\ref{lem:Hpq-hat-bar-prod}, defined by the choice of a timelike $(q+1)$-plane of $\R^{p,q+1}$.
By Proposition~\ref{prop:weakly-sp-gr-is-a-graph}, we can lift $M$ to a weakly spacelike $p$-graph $\hat M$ in $\hat\H^{p,q}$ which, in this splitting, is the graph of some $1$-Lipschitz map $f : \bB^p\to\bS^q$.
Then $f$ extends continuously to a $1$-Lipschitz map $f : \ov\bB^p = \bB^p\sqcup\bS^{p-1} \to \bS^q$, and $\L = \di M$ is the projection of the graph of $f|_{\bS^{p-1}}$.
Since $f|_{\bS^{p-1}}$ is $1$-Lipschitz, Corollary~\ref{cor:non-positive-sphere-as-graph} implies that $\L$ is a non-positive $(p-1)$-sphere in $\di\H^{p,q}$.

Since $f : \ov\bB^p\to\bS^q$ is $1$-Lipschitz, Lemma~\ref{lem:b-di-Hpq-hat} implies that $\sfb(v_1,v) \leq 0$ for all $v_1,v\in\R^{p,q+1}\smallsetminus\{0\}$ corresponding respectively to $(u_1,f(u_1)) \in \hat M \subset \hat\H^{p,q}$ and $(u,f(u)) \in \hat\L := \di\hat M \subset \di\hat\H^{p,q}$, where we see $\hat{\ov{\H}}^{p,q}$ as a subset of $(\R^{p,q+1}\smallsetminus\{0\})/\R_{>0}$ as in \eqref{eqn:di-Hpq-hat-quotient}.
Therefore $M \subset \ov{\Om}(\L)$ by definition of $\ov{\Om}(\L)$ (see Notation~\ref{not:Omega-Lambda}).

\eqref{item:bound-non-pos-sph-2} Suppose $M$ is in fact a spacelike $p$-graph in~$\H^{p,q}$.
By Proposition~\ref{prop:sp-gr-is-a-graph}, this means that the $1$-Lipschitz map $f : \bB^p\to\bS^q$ is in fact strictly $1$-Lipschitz, and therefore its continuous extension $f : \ov\bB^p = \bB^p\sqcup\bS^{p-1} \to \bS^q$ is also strictly $1$-Lipschitz.
Lemma~\ref{lem:b-di-Hpq-hat} then implies that $\sfb(v_1,v) < 0$ for all $v_1,v\in\R^{p,q+1}\smallsetminus\{0\}$ corresponding respectively to $(u_1,f(u_1)) \in \hat M \subset \hat\H^{p,q}$ and $(u,f(u)) \in \hat\L := \di\hat M \subset \di\hat\H^{p,q}$, where we see $\hat{\ov{\H}}^{p,q}$ as a subset of $(\R^{p,q+1}\smallsetminus\{0\})/\R_{>0}$ as in \eqref{eqn:di-Hpq-hat-quotient}.
Therefore $M \subset \Om(\L)$ by definition of $\Om(\L)$ (see Notation~\ref{not:Omega-Lambda}).
\end{proof}

We note that when $M$ is weakly spacelike but not spacelike, the non-positive $(p-1)$-sphere $\L = \di M$ may be degenerate.
Here is a most degenerate example.

\begin{example} \label{ex:most-deg-p-graph}
For $p\leq q$, take $M$ to be the projection of $\mathrm{graph}(f)\subset\hat{\H}^{p,q}$ to~$\H^{p,q}$ for the isometry $f : \bB^p\subset\bS^p\hookrightarrow\bS^q$.
Then $M$ is a weakly spacelike $p$-graph in~$\H^{p,q}$: it is a $p$-dimensional totally geodesic subspace of $\H^{p,q}$ which is \emph{lightlike} (\ie it is the projectivization of a $(p+1)$-dimensional linear subspace of $\R^{p,q+1}$ of signature $(0,1|p)$, see Section~\ref{subsec:Hpq}) and $\di M$ is the projectivization of a $p$-dimensional totally isotropic linear subspace of~$\R^{p,q+1}$.
\end{example}

The following proposition implies in particular that if a weakly spacelike $p$-graph $M$ of~$\H^{p,q}$ satisfies that $\di M$ is the projectivization of a $p$-dimensional totally isotropic linear subspace of~$\R^{p,q+1}$, then we are in the setting of Example~\ref{ex:most-deg-p-graph}.

\begin{prop} \label{prop:deg-weakly-sp-foliation}
Let $M$ be a weakly spacelike $p$-graph in~$\H^{p,q}$ whose boundary at infinity $\L := \di M \subset \di\H^{p,q}$ is degenerate, and let $V := \mathrm{Ker}(\sfb|_{\L})$.
For $1\leq\ell\leq k:=\dim(V)$, any choice of an $\ell$-dimensional linear subspace of~$V$ determines a foliation of~$M$ by $\ell$-dimensional lightlike totally geodesic subspaces of $\H^{p,q}$.
If $\ell=k=p$, then $p\leq q$ and we are in the setting of Example~\ref{ex:most-deg-p-graph}.
\end{prop}

(In particular, for $\ell=1$ this gives a foliation of~$M$ by lightlike geodesics.)

\begin{proof}
Consider a splitting $\hat{\ov\H}^{p,q} \simeq \ov{\bB}^p\times \bS^q$ as in Lemma~\ref{lem:Hpq-hat-bar-prod}, defined by the choice of a timelike $(q+1)$-plane of $\R^{p,q+1}$.
By Proposition~\ref{prop:weakly-sp-gr-is-a-graph}, we can lift $M$ to a weakly spacelike $p$-graph $\hat M$ in $\hat\H^{p,q}$ which, in this splitting, is the graph of some $1$-Lipschitz map $f : \bB^p\to\bS^q$.
Then $f$ extends continuously to a $1$-Lipschitz map $f : \ov\bB^p = \bB^p\sqcup\bS^{p-1} \to \bS^q$, and $\L = \di M$ is the projection of the graph of $f|_{\bS^{p-1}}$.
Let $\bS := \{u\in\nolinebreak\bS^{p-1} \,|\, f(-u) = -f(u)\}$.
By Lemma~\ref{lem:kernel-of-graph}, the restriction $f|_{\bS}$ of $f$ to $\bS$ is an isometry and $\P(V)$ is the projection of the graph of $f|_{\bS}$.

Let $V'$ be an $\ell$-dimensional linear subspace of~$V$, where $1\leq\ell\leq k$.
Then $\P(V')$ is the projection of the graph of $f|_{\bS'}$ for some totally geodesic copy $\bS'$ of $\bS^{\ell-1}$ in $\bS \simeq \bS^{k-1}$.
For any $u \in (\bS')^{\perp} \subset \bS^p$, the set $\bS_u := \{ tu+t'u' \,|\, u'\in\bS',\ t\in\R,\ t'>0,\ t^2+{t'}^2=1\}$ is a totally geodesic copy of $\bS^{\ell}$ in $\bS^p$.
Any point of~$\bB^p$ belongs to $\bS_u$ for some unique $u \in (\bS')^{\perp} \cap \bB^p$; this defines a foliation of $\bB^p$ by hemispheres $\bB^p \cap \bS_u$ of totally geodesic copies of $\bS^{\ell}$.
If $\ell=p$, then $(\bS')^{\perp} \cap \bB^p$ is a singleton $\{u\}$ and $\bB^p \cap \bS_u = \bB^p$.

We claim that for any $u \in (\bS')^{\perp} \cap \bB^p$, the restriction of $f$ to $\ov{\bB}^p \cap \bS_u$ is an isometric embedding.
Indeed, any $u_1\neq u_2$ in $\ov{\bB}^p \cap \bS_u$ belong to a totally geodesic circle in~$\bS_u$, which meets $\bS'$ in two antipodal points $u'$ and $-u'$, with $u', u_1, u_2, -u'$ in this cyclic order.
We have
\begin{align*}
\pi = d_{\bS^q}(f(u'),f(-u')) & \leq d_{\bS^q}(f(u'),f(u_1)) + d_{\bS^q}(f(u_1),f(u_2)) + d_{\bS^q}(f(u_2),f(-u'))\\
&  \leq d_{\bS^p}(u',u_1) + d_{\bS^p}(u_1,u_2) + d_{\bS^p}(u_2,-u') = \pi.
\end{align*}
All inequalities must be equalities, hence $d_{\bS^q}(f(u_1),f(u_2)) = d_{\bS^p}(u_1,u_2)$, proving the claim.

We deduce that for any $u \in (\bS')^{\perp} \cap \bB^p$, the image by~$f$ of the hemisphere $\bB^p \cap \bS_u$ is an $\ell$-dimensional totally geodesic subspace of $\hat\H^{p,q}$ which is the image of an $(\ell+1)$-dimensional linear subspace of $\R^{p,q+1}$ of signature $(0,1|\ell)$.
\end{proof}

We conclude this section with a projection result that will be used in the proof of Propositions \ref{prop:non-deg-limit} and~\ref{prop:proper-action-M}.

\begin{lem} \label{lem:restrict-span-L}
Let $M$ be a weakly spacelike $p$-graph in~$\H^{p,q}$, with $M \subset \Om(\L)$ where $\L := \di M \subset \di\H^{p,q}$.
Then the orthogonal projection $M'$ of $M$ to $\P(\spa(\L))$ is well defined and is still a weakly spacelike $p$-graph in~$\H^{p,q}$ with $\di M' = \L$ and $M' \subset \Om(\L)$.
If $M$ is a spacelike $p$-graph in~$\H^{p,q}$, then so is~$M'$.
\end{lem}

\begin{proof}
By Lemma~\ref{lem:Omega-Lambda-convex}.\eqref{item:Om-Lambda-non-empty} and Proposition~\ref{prop:bdynonpossphere}.\eqref{item:bound-non-pos-sph-1}, the set $\L$ is a non-degenerate non-positive $(p-1)$-sphere in $\di\hat{\H}^{p,q}$, and so we can write $\R^{p,q+1}$ as the $\sfb$-orthogonal direct sum of $\spa(\L)$ and~$\L^{\perp}$.
By Lemma~\ref{l.kernonpossphere}, the restriction $\sfb|_{\spa(\L)}$ of $\sfb$ to $\spa(\L)$ has signature $(p,q'|0)$ for some $1\leq q'\leq q+1$, and the restriction $\sfb|_{\L^{\perp}}$ of $\sfb$ to~$\L^{\perp}$ is negative definite (of signature $(0,q+1-q'|0)$).
The $\sfb$-orthogonal projection $\pi$ from $\R^{p,q+1}$ to $\spa(\L)$ induces a projection from $\P(\R^{p,q+1}) \smallsetminus \P(\L^{\perp})$ to $\P(\spa(\L))$, which we still denote by~$\pi$.

It follows from the definition of $\Om(\L)$ (Notation~\ref{not:Omega-Lambda}) that $\Om(\L)$ is contained in the domain $\P(\R^{p,q+1}) \smallsetminus \P(\L^{\perp})$ of~$\pi$ and that $\pi$ sends $\Om(\L)$ inside itself, hence (Lemma~\ref{lem:Omega-Lambda-convex}.\eqref{item:Om-Lambda-in-Hpq}) inside $\H^{p,q}$.
In particular, $\pi$ is well defined on~$M$; we set $M' := \pi(M) \subset \Om(\L) \subset \H^{p,q}$.

Since $\sfb|_{L^{\perp}}$ is negative definite, we have $\sfb(\pi(v),\pi(v)) \geq \sfb(v,v)$ for all $v\in\R^{p,q+1}$.
Therefore $\pi$ sends $\P(\R^{p,q+1}) \smallsetminus \H^{p,q}$ (\resp $\P(\R^{p,q+1}) \smallsetminus \ov{\H}^{p,q}$) inside itself.
Since $\pi$ sends any projective line of $\P(\R^{p,q+1})$ to a projective line or a point, we deduce that for any $x\neq y$ in $\Om(\L)$, if $x$ and~$y$ are not in timelike position (\resp are in spacelike position) and $\pi(x)\neq \pi(y)$, then $\pi(x)$~and~$\pi(y)$ are not in timelike position (\resp are in spacelike position).

We claim that $M' = \pi(M)$ meets every $(q'-1)$-dimensional timelike totally geodesic subspace of $\H^{p,q} \cap \P(\spa(\L)) \simeq \H^{p,q'-1}$ in a unique point.
Indeed, let $T'$ be a timelike $q'$-plane of $\spa(\L)$.
Then $T := T' \oplus \L^{\perp}$ is a timelike $(q+1)$-plane of~$\R^{p,q+1}$.
By Proposition~\ref{prop:weakly-sp-gr-is-a-graph}, the weakly spacelike $p$-graph $M \subset \H^{p,q}$ meets $\P(T)$, and so $M' = \pi(M)$ meets $\P(T')$.
The intersection $M' \cap \P(T')$ is a singleton because no two distinct points of~$M'$ are in timelike position.

Finally, we claim that $M'$ is closed in $\H^{p,q} \cap \P(\spa(\L)) \simeq \H^{p,q'-1}$.
Indeed, this follows from the fact that $M$ is closed in~$\H^{p,q}$, that $M$ meets every $q$-dimensional timelike totally geodesic subspace of $\H^{p,q}$ in a unique point (Proposition~\ref{prop:weakly-sp-gr-is-a-graph}) and that, by Remark~\ref{rem:local-timelike-foliation}, any point of $\H^{p,q} \cap \P(\spa(\L))$ admits an open neighborhood in $\H^{p,q}$ which is foliated by $q$-dimensional timelike totally geodesic subspaces of $\H^{p,q}$ of the form $\P(T' \oplus \L^{\perp})$ where $T'$ is a timelike $q'$-plane of $\spa(\L)$.

By Proposition~\ref{prop:weakly-sp-gr-is-a-graph}, the set $M'$ is a weakly spacelike $p$-graph in $\H^{p,q} \cap \P(\spa(\L)) \simeq \H^{p,q'-1}$.
It then follows from Definition~\ref{def:weakly-sp-gr} that $M'$ is also a weakly spacelike $p$-graph in $\H^{p,q}$.
Moreover, if $M$ is a spacelike $p$-graph in $\H^{p,q}$, then so is~$M'$.
\end{proof}

\section{Two useful preliminary results} \label{sec:two-prelim-results}

In this section we prepare the proofs of Proposition~\ref{prop:non-deg-limit} and Theorems \ref{thm:non-deg-limit-sphere} and~\ref{thm:geom-action-weakly-sp-gr} by establishing two preliminary results involving weakly spacelike $p$-graphs in~$\H^{p,q}$.

\subsection{Working in a properly convex open set} \label{subsec:construct-Om}

Recall from Proposition~\ref{prop:bdynonpossphere}.\eqref{item:bound-non-pos-sph-1} that the boundary at infinity $\L$ of a weakly spacelike $p$-graph of~$\H^{p,q}$ is a non-positive $(p-1)$-sphere in $\di\H^{p,q}$.
In the setting of Theorem~\ref{thm:geom-action-weakly-sp-gr}, the convex open set $\Omega(\L) \subset \H^{p,q}$ is properly convex if $\L$ spans (Lemma~\ref{lem:Omega-Lambda-convex}.\eqref{item:Om-Lambda-prop-convex}), but not in general.
In order to use Hilbert metrics in Sections \ref{subsec:max-crowns} and~\ref{subsec:crown-not-hyperb} below, we need to work in an appropriate properly convex open subset $\Om$ of $\Om(\L)$ containing~$M$.
The goal of this section is to establish the existence of such~$\Om$.

\begin{prop} \label{prop:M-in-prop-conv-Omega}
For $p,q\geq 1$, let $\Gamma$ be a discrete subgroup of $\PO(p,q+1)$ acting properly discontinuously on a weakly spacelike $p$-graph $M$ in $\H^{p,q}$ with $M \subset \Om(\L)$ where $\L := \di M \subset \di\H^{p,q}$.
Then
\begin{enumerate}
  \item\label{item:M-in-Om-1} the action of $\Gamma$ on $\Om(\L)$ is properly discontinuous.
\end{enumerate}
Assume in addition that the action of $\Gamma$ on~$M$ is cocompact.
Then
\begin{enumerate}\setcounter{enumi}{1}
  \item\label{item:M-in-Om-2} $M$ is contained in some $\Gamma$-invariant \emph{properly convex} open subset $\Om \subset \Om(\L)\subset\H^{p,q}$;
  \item\label{item:M-in-Om-3} for any $\G$-invariant closed subset $\mathcal{Z}$ of $\Om(\L)$ such that $\de_{\H}\mathcal{Z}\cap \de_{\H}\Om(\L)=\emptyset$, we can choose the set $\Om$ of~\eqref{item:M-in-Om-2} to contain~$\mathcal{Z}$, which implies that the action of $\Gamma$ on~$\mathcal{Z}$ is properly discontinuous; moreover, the action of $\Gamma$ on~$\mathcal{Z}$ is cocompact.
\end{enumerate}
\end{prop}

Recall from Lemma~\ref{lem:Omega-Lambda-convex}.\eqref{item:Om-Lambda-invisible-dom} and Proposition~\ref{prop:bdynonpossphere}.\eqref{item:bound-non-pos-sph-1} that the condition $M \subset \Om(\L)$ means that any point of~$M$ sees any point of~$\L$ in a spacelike direction.
By Lemma~\ref{lem:Omega-Lambda-convex}.\eqref{item:Om-Lambda-non-empty}, this condition implies that $\L$ is non-degenerate.

Proposition~\ref{prop:M-in-prop-conv-Omega} is an easy consequence of the following lemma.

\begin{lem} \label{lem:build-prop-conv-Omega}
For $p,q\geq 1$, let $\Gamma$ be a discrete subgroup of $\PO(p,q+1)$ acting properly discontinuously on a weakly spacelike $p$-graph $M$ in $\H^{p,q}$ with $M \subset \Om(\L)$ where $\L := \di M \subset \di\H^{p,q}$.
Then
\begin{enumerate}
  \item\label{item:build-Omega-1} any accumulation point in $\ov{\Om(\L)}$ of the $\Gamma$-orbit of a compact subset of $\Om(\L)$ is contained in~$\L$; in particular, the action of $\Gamma$ on $\Om(\L)$ is properly discontinuous;
  \item\label{item:build-Omega-2} for any compact subset $\mathcal{K}$ of $\Om(\L)$, the interior of the convex hull of $\bigcup_{\g\in\G} \g\cdot\mathcal{K}$ in $\Om(\L)$ is properly convex.
\end{enumerate}
\end{lem}

\begin{proof}[Proof of Lemma~\ref{lem:build-prop-conv-Omega}]
\eqref{item:build-Omega-1} Let $\mathcal{K}$ be a compact subset of~$\Om(\L)$, let $(o_n)_{n\in\N}$ be a sequence of points of~$\mathcal{K}$, and let $(\gamma_n)_{n\in\N}$ be a sequence of pairwise distinct elements of~$\Gamma$ such that $(\gamma_n\cdot o_n)_{n\in\N}$ converges to some $o'\in\ov{\Om(\L)}$.
Let us check that $o'\in\L$.
By Lemma~\ref{lem:Omega-Lambda-convex}.\eqref{item:Om-Lambda-in-Hpq} we have $\ov{\Om(\L)} \subset \H^{p,q}\cup\L$, hence it is enough to check that $o'\notin\H^{p,q}$.
Suppose by contradiction that $o'\in\H^{p,q}$.
Consider a point $m\in M$ in timelike position to $o'$ (such a point exists by Proposition \ref{prop:weakly-sp-gr-is-a-graph}.\ref{item:weakly-sp-gr-1-bis}).
For large enough~$n$, the point $\g_n\cdot o_n$ is in timelike position to~$m$, and so $o_n$ is in timelike position to $\g_n^{-1}\cdot m$.
Up to passing to a subsequence, we may assume that $(o_n)_{n\in\N}$ converges to some point $o\in \mathcal{K}$ (because $\mathcal{K}$ is compact) and that $(\g_n^{-1}\cdot m)_{n\in\N}$ converges to some point $w\in \L = \di M$ (because the action of $\Gamma$ on~$M$ is properly discontinuous).
Then $o$ and~$w$ are not in spacelike position: contradiction with Lemma~\ref{lem:Omega-Lambda-convex}.\eqref{item:Om-Lambda-invisible-dom}.

\eqref{item:build-Omega-2} Let $\mathcal{K}$ be a compact subset of $\Om(\L)$, and let $\Om$ be the interior of the convex hull of $\bigcup_{\g\in\G} \g\cdot\mathcal{K}$ in $\Om(\L)$.
Then $\Om$ is a $\Gamma$-invariant convex open subset of $\Om(\L)$.
We may assume that $\Omega$ is non-empty, otherwise there is nothing to prove.

We claim that $\ov\Om \cap \partial\Om(\L) \subset \ov{\mathscr{C}(\L)}$, where $\mathscr{C}(\L)$ is the convex hull of $\L$ in $\Om(\L)$ (Notation~\ref{not:C-Lambda}).
Indeed, consider a point $o\in\ov\Om$: it is a limit of points $o_n\in\Om$.
By definition of~$\Om$, each $o_n$ lies in an open projective simplex in $\Om(\L)$ whose vertices are of the form $\gamma_{i,n}\cdot k_{i,n}$ where $\gamma_{i,n}\in\G$ and $k_{i,n}\in\mathcal{K}$ (with $1\leq i\leq p+q+1$).
Up to passing to a subsequence, we may assume that the $k_{i,n}$ converge in~$\mathcal{K}$ for each~$i$, and that each sequence $(\gamma_{i,n})_{n\in\N}$ is either constant or consists of pairwise distinct elements of~$\Gamma$.
Using \eqref{item:build-Omega-1}, we deduce that $o$ lies in the relative interior of a simplex in $\ov{\Om(\L)}$ whose vertices lie in $\Om(\L) \cup \L$.
If at least one of the vertices lies in $\Om(\L)$, then $o\in\Om(\L)$.
Taking the contrapositive, we get that if $o\in\partial\Om(\L)$, then all vertices lie in $\L$, and so $o\in\ov{\mathscr{C}(\L)}$.

Suppose by contradiction that $\Om$ is not properly convex.
Arguing as in the proof of Lemma \ref{lem:Omega-Lambda-convex}.\eqref{item:Om-Lambda-prop-convex}, we find a projective line $\ell$ fully contained in $\overline{\Om} \subset \overline{\Om(\L)}$ and a point $o \in \ell \cap \partial\Om \cap \partial\Om(\L)$ such that $o\in x^{\perp}$ for all $x\in \L$.
By the claim above we have $o \in \ov{\mathscr{C}(\L)} \subset \P(\spa(\L))$, and so $\L$ is degenerate.
On the other hand, by assumption $\Om(\L)$ is non-empty, and so $\L$ is non-degenerate by Lemma~\ref{lem:Omega-Lambda-convex}.\eqref{item:Om-Lambda-non-empty}: contradiction.
\end{proof}

\begin{proof}[Proof of Proposition~\ref{prop:M-in-prop-conv-Omega}]
\eqref{item:M-in-Om-1} This is contained in Lemma~\ref{lem:build-prop-conv-Omega}.\eqref{item:build-Omega-1}.

\eqref{item:M-in-Om-2} Let $D$ be a compact fundamental domain for the action of $\Gamma$ on~$M$.
By Lemma~\ref{lem:build-prop-conv-Omega}, for any compact neighborhood $\mathcal{K}$ of $D$ in $\Om(\L)$, the interior $\Om$ of the convex hull of $\bigcup_{\g\in\G} \g\cdot\mathcal{K}$ in $\Om(\L)$ is a properly convex open subset of $\Om(\L) \subset \H^{p,q}$ which contains~$M$.

\eqref{item:M-in-Om-3} Let $\mathcal{Z}$ be a $\Gamma$-invariant closed subset of $\Om(\L)$ such that $\de_{\H}\mathcal{Z} \cap \de_{\H}\Om(\L)=\emptyset$.
Then $\mathcal{Z}$ is closed in $\H^{p,q}$, and so Lemma~\ref{lem:Omega-Lambda-convex}.\eqref{item:Om-Lambda-in-Hpq} implies that $\ov{\mathcal{Z}} = \mathcal{Z} \cup \di \mathcal{Z} \subset \mathcal{Z} \cup \L$.

The set $D'$ of points of~$\mathcal{Z}$ that are in timelike or lightlike position with a point of~$D$ is compact.
Indeed, the closure $\ov{D'}$ of $D'$ in $\ov\H^{p,q}$ still consists of points that are in timelike or lightlike position with a point of~$D$.
Moreover, $\ov{D'} \cap \di \mathcal{Z} \subset \ov{D'} \cap \L$ is empty because any point of $\L$ is in spacelike position with any point of~$D$ (Lemma~\ref{lem:Omega-Lambda-convex}.\eqref{item:Om-Lambda-invisible-dom}).

We note that $\mathcal{Z}$ is contained in the union of the $\Gamma$-translates of the compact subset~$D'$.
Indeed, any point $o\in \mathcal{Z}$ is in timelike position with some point $m\in M$ (by Definition~\ref{def:weakly-sp-gr} of a weakly spacelike $p$-graph); if $m\in\gamma\cdot D$ where $\gamma\in\Gamma$, then $o\in\gamma\cdot D'$.

Taking $\mathcal{K}$ to be a neighborhood of~$D$ containing $D'$ in the proof of~\eqref{item:M-in-Om-2} above, we obtain that the interior $\Om$ of the convex hull of $\bigcup_{\g\in\G} \g\cdot\mathcal{K}$ in $\Om(\L)$ is a properly convex open subset of $\Om(\L) \subset \H^{p,q}$ which contains $M$ and also~$\mathcal{Z}$.

Since $\Gamma$ preserves the properly convex set $\Om$, the action of $\Gamma$ on~$\Om$ is properly discontinuous (because it preserves the Hilbert metric $d_{\Omega}$, see Section~\ref{subsec:prop-conv-proj}).
Since $\mathcal{Z}\subset\Om$, the action of $\Gamma$ on~$\mathcal{Z}$ is also properly discontinuous. 
This last action is cocompact because $\mathcal{Z}$ is contained in the union of the $\Gamma$-translates of the compact subset~$D'$.
\end{proof}

\begin{proof}[Proof of Proposition~\ref{prop:proper-action-M}]
\eqref{item:proper-action-M-1}~$\Rightarrow$~\eqref{item:proper-action-M-2} is clear.

\eqref{item:proper-action-M-2}~$\Rightarrow$~\eqref{item:proper-action-M-1}:
By Lemma~\ref{lem:Omega-Lambda-convex}.\eqref{item:Om-Lambda-non-empty} and Proposition~\ref{prop:bdynonpossphere}.\eqref{item:bound-non-pos-sph-1}, the set $\L$ is a non-degenerate non-positive $(p-1)$-sphere in $\di\H^{p,q}$.
By Lemma~\ref{l.kernonpossphere}, the restriction of $\sfb$ to $\spa(\L)$ has signature $(p,q')$ for some $1\leq q'\leq q+1$.
The group $\rho(\Gamma)$ preserves $\P(\spa(\L))$, hence preserves the orthogonal projection $M'$ of $M$ to $\P(\spa(\L))$ from Lemma~\ref{lem:restrict-span-L}, which is a weakly spacelike $p$-graph in $\H^{p,q}$, and in fact in $\H^{p,q} \cap \P(\spa(\L)) \simeq \H^{p,q'-1}$.
By Lemma~\ref{lem:Omega-Lambda-convex}.\eqref{item:Om-Lambda-prop-convex}, the set $\Om(\L) \cap \P(\spa(\L))$ is a $\rho(\Gamma)$-invariant properly convex open subset of $\P(\spa(\L))$.
If $\rho$ has finite kernel and discrete image, then the action of $\Gamma$ on $\Om(\L) \cap \P(\spa(\L))$ via~$\rho$ is properly discontinuous because it preserves the Hilbert metric $d_{\Om(\L)}$, see Section~\ref{subsec:prop-conv-proj}.
In particular, the action of $\Gamma$ on~$M'$ via~$\rho$ is properly discontinuous.
By Proposition~\ref{prop:M-in-prop-conv-Omega}.\eqref{item:M-in-Om-1}, the action of $\Gamma$ on the whole of $\Om(\L)$ via~$\rho$ is properly discontinuous.
In particular, the action of $\Gamma$ on~$M$ via~$\rho$ is properly discontinuous.

Recall from Proposition~\ref{prop:weakly-sp-gr-is-a-graph} that $M$ is contractible.
If \eqref{item:proper-action-M-1} holds and if $\vcd(\Gamma) = p$, then the action on $M$ via $\rho$ is cocompact by Fact~\ref{fact:vcd}, and so $M$ is contained in some $\rho(\Gamma)$-invariant properly convex open subset of $\H^{p,q}$ by Proposition~\ref{prop:M-in-prop-conv-Omega}.\eqref{item:M-in-Om-2}.
\end{proof}

\subsection{A weakly spacelike graph in the non-degenerate part} \label{subsec:weakly-sp-gr-non-deg-part}

The following will be used in the proofs of Proposition~\ref{prop:non-deg-limit} and Theorem~\ref{thm:non-deg-limit-sphere} in Sections \ref{subsec:Lambda-nondegen} and~\ref{subsec:proof-non-deg-limit-sphere}.
We refer to Definitions \ref{def:non-pos-sphere} and~\ref{def:weakly-sp-gr} for the notions of non-positive $(p-1)$-sphere in $\di\H^{p,q}$ and of weakly spacelike $p$-graph in~$\H^{p,q}$.

\begin{prop} \label{prop:ME-weakly-sp-gr}
For $p\geq 2$ and $q\geq 1$, let $\L$ be a non-positive $(p-1)$-sphere in $\di\H^{p,q}$.
Let $V \subset \spa(\L) \subset \R^{p,q+1}$ be the kernel of $\sfb|_{\spa(\L)}$, of dimension $k:=\dim V\geq 0$.
Suppose $k<p$.
Let $E$ be a linear subspace of~$\R^{p,q+1}$ such that $V^{\perp} = V\oplus E$.
Then
\begin{enumerate}
  \item\label{item:ME-1} the restriction $\sfb|_E$ is non-degenerate, of signature $(p-k,q+1-k|0)$, so we can consider $\H_E := \H^{p,q}\cap\P(E) \simeq \H^{p-k,q-k}$;
  \item\label{item:ME-2} $\L_E := \L \cap \P(E)$ is a non-degenerate non-positive $(p-k-1)$-sphere in $\di\H_E \simeq \di\H^{p-k,q-k}$;
  \item\label{item:ME-3} for any weakly spacelike $p$-graph $M \subset \H^{p,q}$ with $\di M = \L$, the set $M_E := M\cap\nolinebreak\P(E)$ is a weakly spacelike $(p-k)$-graph in $\H_E \simeq \H^{p-k,q-k}$ with $\di M_E = \L_E$.
\end{enumerate}
\end{prop}

\begin{proof}
We may assume $k>0$, otherwise there is nothing to prove.

\eqref{item:ME-1} is clear.

\eqref{item:ME-2} It follows from \eqref{item:ME-1} that the restriction $\sfb|_{E^{\perp}}$ is non-degenerate of signature $(k,k|0)$.
Choose a timelike $k$-plane $T_{E^{\perp}}$ of $E^{\perp}$ and a timelike $(q+1-k)$-plane $T_E$ of $E$, so that $T := T_{E^{\perp}} \oplus T_E$ is a timelike $(q+1)$-plane of~$\R^{p,q+1}$.
This determines a splitting
$$\di\hat{\H}^{p,q} \simeq \{ v\in T^{\perp} ~|~ \sfb(v,v)=1\} \times \{ v\in T ~|~ (-\sfb)(v,v)=1\} \simeq \bS^{p-1}\times\bS^q$$
as in \eqref{eqn:di-Psi}.
Taking the intersection of $\{ v\in T^{\perp} \,|\, \sfb(v,v)=1\}$ with $E^{\perp}$ (\resp $E$) gives a totally geodesic copy $\bS$ (\resp $\bS'$) of $\bS^{k-1}$ (\resp $\bS^{p-k-1}$) in $\{ v\in T^{\perp} \,|\, \sfb(v,v)=1\} \simeq \bS^{p-1}$.
Similarly, taking the intersection of $\{ v\in T \,|\, (-\sfb)(v,v)=1\}$ with $E^{\perp}$ (\resp $E$) gives a totally geodesic copy $\bS''$ (\resp $\bS'''$) of $\bS^{k-1}$ (\resp $\bS^{q-k}$) in $\{ v\in T \,|\, \sfb(v,v)=1\} \simeq \bS^q$.
By construction, $\bS'$ is the set of points at distance $\pi/2$ from $\bS$ in $\bS^{p-1}$, and $\bS'''$ is the set of points at distance $\pi/2$ from $\bS''$ in $\bS^q$.

The non-positive $(p-1)$-sphere $\L$ of $\di\H^{p,q}$ is by definition the projection of a non-positive $(p-1)$-sphere $\hat\L$ of $\di\hat\H^{p,q}$, which by Corollary~\ref{cor:non-positive-sphere-as-graph} is, in the above splitting, the graph of some $1$-Lipschitz map $f : \bS^{p-1}\to\bS^q$.

Since $V\subset E^{\perp}$, the image of $V\smallsetminus\{0\}$ in $\di\hat\H^{p,q}$ is contained in $\bS\times\bS''$.
Lemma~\ref{lem:kernel-of-graph} then implies that the image of $V$ in $\di\hat\H^{p,q}$ is the graph of $f|_{\bS}$, that $\bS$ is the set of $u\in\bS^{p-1}$ such that $f(-u) = -f(u)$, and that $f(\bS) = \bS''$.

We claim that $f(\bS') \subset \bS'''$.
Indeed, for any $u\in\bS$ and any $u_1\in\bS'$ we have
\begin{align*}
\pi = d_{\bS^q}(f(u),f(-u)) & \leq d_{\bS^q}(f(u),f(u_1)) + d_{\bS^q}(f(u_1),f(-u))\\
&  \leq d_{\bS^{p-1}}(u,u_1) + d_{\bS^{p-1}}(u_1,-u) = \pi.
\end{align*}
All inequalities must be equalities, hence $d_{\bS^q}(f(u),f(u_1)) = \pi/2$ for all $u\in\bS$ and all $u_1\in\bS'$, proving the claim.

Thus $\gph(f) \cap (\bS'\times\bS''') = \gph(f|_{\bS'})$, and so $\L_E := \L\cap\P(E)$ is the image in $\di\H^{p,q}$ of $\gph(f|_{\bS'})$, where $f|_{\bS'} : \bS'\to\bS'''$ is $1$-Lipschitz.
Corollary~\ref{cor:non-positive-sphere-as-graph} then implies that $\L_E$ is a non-positive $(p-1-k)$-sphere in $\di\H_E \simeq \di\H^{p-k,q-k}$.
Since $f(-u) \neq f(u)$ for all $u\in\bS'$, this non-positive $(p-1-k)$-sphere is non-degenerate by Lemma~\ref{lem:kernel-of-graph}.

\eqref{item:ME-3} Let $M$ be a weakly spacelike $p$-graph in $\H^{p,q}$ with $\di M = \L$.
The set $M_E := M\cap\H_E$ is a closed subset of $\H_E$ with $\di M_E = \L_E$.
By Proposition~\ref{prop:weakly-sp-gr-is-a-graph}, in order to prove that $M_E$ is a weakly spacelike $(p-k)$-graph in~$\H_E$, it is enough to check that $M_E$ meets every $(q-k)$-dimensional timelike totally geodesic subspace of $\H_E$ in a unique point.
Fix a timelike $(q-k+1)$-plane $T_E$ of~$E$.
As in the proof of~\eqref{item:ME-2}, we choose a timelike $k$-plane $T_{E^{\perp}}$ of $E^{\perp}$, so that $T := T_{E^{\perp}} \oplus T_E$ is a timelike $(q+1)$-plane of~$\R^{p,q+1}$.
This determines a splitting
$\hat{\ov{\H}}^{p,q} \simeq \ov{\bB}^p\times \bS^q$ as in Lemma~\ref{lem:Hpq-hat-bar-prod}.
By Proposition~\ref{prop:weakly-sp-gr-is-a-graph}, the weakly spacelike $p$-graph $M$ admits a lift to $\hat\H^{p,q}$ which, in our splitting, is the graph of a $1$-Lipschitz map $f : \bB^p\to\bS^q$.
This map extends continuously to a $1$-Lipschitz map $f : \ov\bB^p = \bB^p\sqcup\bS^{p-1} \to \bS^q$, such that $\L = \di M$ is the projection of the graph of $f|_{\bS^{p-1}}$.
Define $\bS,\bS' \subset \bS^{p-1}$ and $\bS'',\bS''' \subset \bS^q$ as in~\eqref{item:ME-2}.
By construction (see Proposition~\ref{p.doublecoverasproduct}), the image of $T$ (\resp $T_{E^{\perp}}$, \resp $T_E$) in $\hat\H^{p,q}$ is $\{u_0\}\times\bS^q$ (\resp $\{u_0\}\times\bS''$, \resp $\{u_0\}\times\bS'''$), where $u_0 = (1,0,\dots,0)$ is the mid-point of the upper hemisphere~$\bB^p$.

We claim that $f(u_0) \in \bS'''$.
Indeed, for any $u\in\bS$ we have
\begin{align*}
\pi = d_{\bS^q}(f(u),f(-u)) & \leq d_{\bS^q}(f(u),f(u_0)) + d_{\bS^q}(f(u_0),f(-u))\\
&  \leq d_{\bS^{p-1}}(u,u_0) + d_{\bS^{p-1}}(u_0,-u) = \pi.
\end{align*}
All inequalities must be equalities, hence $d_{\bS^q}(f(u),f(u_0)) = \pi/2$ for all $u\in\bS$, proving the claim.

Thus $\gph(f)$ meets the image of $T_E$ in a unique point (namely $(u_0,f(u_0))$), and so $M_E \cap\nolinebreak \P(T_E)$ is a singleton, which completes the proof.
\end{proof}

\section{Weakly spacelike $p$-graphs with proper cocompact group actions} \label{sec:suff-cond-Hpq-cc}

The goal of this section is to prove Theorem~\ref{thm:geom-action-weakly-sp-gr}, which contains Theorem~\ref{thm:geom-action-weakly-sp-gr-basic}.

In Section~\ref{subsec:crowns}, we first introduce and discuss a generalization, in $\di\H^{p,q}$, of Barbot's \emph{crowns} \cite[\S\,4.7]{bar15} in the Einstein universe $\di\H^{p,1} = \di\mathrm{AdS}^{p+1}$.
In Section~\ref{subsec:max-crowns} we show that in the setting of Theorem~\ref{thm:geom-action-weakly-sp-gr}, crowns in~$\L$ of maximal cardinality are not boundary crowns; from this we deduce, by contraposition, the implication \eqref{item:geom-action-4}~$\Rightarrow$~\eqref{item:geom-action-5} of Theorem~\ref{thm:geom-action-weakly-sp-gr}.
In Section~\ref{subsec:foliate-crown} we introduce natural foliations of convex hulls of $j$-crowns by $j$-dimensional complete spacelike submanifolds.
In Section~\ref{subsec:crown-not-hyperb} we use such foliations and the fact that crowns in~$\L$ of maximal cardinality are not boundary crowns to show that the existence of a $j$-crown for $j\geq 2$ prevents Gromov hyperbolicity (implication \eqref{item:geom-action-1}~$\Rightarrow$~\eqref{item:geom-action-4} of Theorem~\ref{thm:geom-action-weakly-sp-gr}).
Then in Section~\ref{subsec:proof-thm-ConvexCo} we complete the proof of Theorem~\ref{thm:geom-action-weakly-sp-gr} using Proposition~\ref{prop:M-in-prop-conv-Omega} and \cite[Th.\,1.24]{dgk-proj-cc}.

On the other hand, in Section~\ref{subsec:split-spacetimes} we give examples of weakly spacelike $p$-graphs with proper cocompact actions by \emph{non-hyperbolic} groups.

\subsection{Crowns in $\di\H^{p,q}$} \label{subsec:crowns}

Recall that $x,y\in\di \H^{p,q}$ are called \emph{transverse} if $x\notin y^{\perp}$, or equivalently if $y\notin x^{\perp}$.
The following definition is a generalization of \cite[Def.\,4.30]{bar15}.

\begin{defn} \label{def:j-crown}
Let $\Lambda$ be a subset of $\di\H^{p,q}$.
For $j\geq 1$, a \emph{$j$-crown in $\Lambda$} is a collection $\sfC$ of $2j$ points $x_1^+,x_1^-,\dots,x_j^+,x_j^-$ of $\Lambda$ such that for any $1\leq i,i'\leq j$ and any $\varepsilon,\varepsilon'\in\{+,-\}$, the points $x_i^{\varepsilon}$ and $x_{i'}^{\varepsilon'}$ are transverse if and only if $i=i'$ and $\varepsilon\neq\varepsilon'$.
Such a $j$-crown $\sfC$ is a \emph{boundary $j$-crown in~$\L$} if there exists $w\in\L$ such that $\sfC\subset w^{\perp}$.
\end{defn}

\begin{remarks} \label{rem:crowns}
\begin{enumerate}
  \item\label{item:crown-1} $j$-crowns in $\di\H^{p,q}$ exist only for $j\leq\min(p,q+1)$, as the restriction of~$\sfb$ to the span of a $j$-crown has signature $(j,j|0)$.
  \item $2$-crowns in $\di\H^{p,1} = \di\mathrm{AdS}^{p+1}$ were first introduced in \cite{bar15}, where Barbot called them simply \emph{crowns}.
  See Figure~\ref{fig:A_Corbit} in Section~\ref{subsec:weakly-spacelike-graphs} for an example of a $2$-crown $\sfC = \{ x_1^{\pm},x_2^{\pm}\}$ in $\di\H^{2,1} = \di\mathrm{AdS}^3$.
  \item Let $\L$ be a subset of $\di\H^{p,q}$.
  A $1$-crown in~$\L$ is a pair of transverse points in~$\L$.
  A boundary $1$-crown in~$\L$ is a pair of transverse points of~$\L$ which are both contained in the orthogonal of some point of~$\L$.
\end{enumerate}
\end{remarks}

\begin{lem} \label{l.crown-existence}
Let $\L \subset \di\H^{p,q}$ be a non-degenerate non-positive $(p-1)$-sphere.
If\linebreak $w^{\perp} \cap\nolinebreak \de_{\H} \mathscr{C}(\L) \neq \emptyset$ for some $w\in\L$, then there is a boundary $1$-crown in $\L$.
\end{lem}

\begin{proof}
Since $\L$ is non-positive, we can lift it to $\tilde{\L} \subset \R^{p,q+1}$ such that $\sfb(\tilde{x},\tilde{y}) \leq 0$ for all $\tilde{x},\tilde{y}\in\tilde{\L}$.
Let $y \in \de_{\H} \mathscr{C}(\L)$; we can lift it to $\tilde{y}\in\R^{p,q+1}$ such that $\sfb(\tilde{x},\tilde{y})\leq 0$ for all $\tilde{x}\in\tilde\L$.
Now suppose $y \in w^{\perp}$ for some $w\in\L$: we can lift $w$ to $\tilde w\in\tilde \L$ such that $\sfb(\tilde{w},\tilde{y})=0$.
Observe that if $L$ is a subset of~$\R^d$, then any point in the convex hull of~$L$ lies in the convex hull of at most $d+1$ points of~$L$.
Applying this observation to an affine chart of $\P(\R^{p+q+1})$ containing $\Om(\L)$, and given that $y\in \de_{\H} \mathscr{C}(\L)$, we get the existence of $\tilde{x}_1,\ldots,\tilde{x}_{p+q+1}\in\tilde \Lambda$ and $a_1,\dots,a_{p+q+1}>0$ such that $\tilde{y} = \sum_{i=1}^{p+q+1} a_i \tilde{x}_i$.
By assumption on $\tilde\Lambda$ we have $\sfb(\tilde{w},\tilde{x}_i)\leq 0$ for all~$i$.
Thus $\sfb(\tilde{w},\tilde{y})=0$ implies that $\sfb(\tilde{w},\tilde{x}_i)=0$ for all $1\leq i\leq p+q+1$.
Moreover $\sfb(\tilde{y},\tilde{y})<0$ implies that, up to re-ordering, $\sfb(\tilde{x}_1,\tilde{x}_2)<0$.
Setting $x^+$ and $x^-$ to be the projections of $\tilde{x}_1$ and $\tilde{x}_2$ to $\di \H^{p,q}$, we obtain a $1$-crown $\sfC = \{ x^+,x^-\}$ in~$\L$.
This is a boundary $1$-crown in~$\L$ as $x^{\pm}\in w^{\perp}$.
\end{proof}

\subsection{Crowns of maximal cardinality are not boundary crowns} \label{subsec:max-crowns}

Recall from Remark \ref{rem:crowns}.\eqref{item:crown-1} that the cardinality of crowns in $\di\H^{p,q}$ is uniformly bounded.

\begin{prop} \label{p.crown-boundary}
Let $\Gamma$ be a discrete subgroup of $\PO(p,q+1)$ acting properly discontinuously and cocompactly on a weakly spacelike $p$-graph $M$ in $\H^{p,q}$ with $M \subset \Om(\L)$ where $\L := \di M \subset \di\H^{p,q}$.
For any $j\geq 1$, if there exists a \emph{boundary} $j$-crown in $\L$, then there exists a $(j+1)$-crown in~$\L$.
\end{prop}

Thus, if $\sfC$ is a crown of maximal cardinality in~$\L$, then it is not a boundary crown in~$\L$.

\begin{proof}
Our strategy is inspired by the proof of \cite[Prop.\,5.2]{bar15}.
Let $\sfC = \{x_1^{\pm},\ldots,x_j^{\pm}\}$ be a boundary $j$-crown in~$\L$: there exists $w\in\Lambda$ such that $\sfC\subset w^{\perp}$.
The idea is to carefully construct a suitable point $y\in M$ and a sequence $(\g_n) \in \G^{\N}$ such that, up to passing to a subsequence, $\g_n\cdot (\sfC \cup \{ w,y\})$ converges to a $(j+1)$-crown in~$\L$.

\medskip

\noindent
$\bullet$ \textbf{Preliminary set-up.}
By Proposition~\ref{prop:bdynonpossphere}.\eqref{item:bound-non-pos-sph-1}, the set $\L$ is a non-positive $(p-1)$-sphere.
We lift it to a non-positive subset $\tilde{\L}$ of $\R^{p,q+1}\smallsetminus\{0\}$ (Definition~\ref{def:nonpos-neg}).
By assumption $M \subset \Om(\L)$, hence we can lift $M$ to $\tilde M \subset \hat\H^{p,q} \subset \R^{p,q+1}$ such that $\sfb(\tilde{x},\tilde{o})<0$ for all $\tilde o\in \tilde M$ and all $\tilde{x}\in \tilde\L$.

The convex open subset $\Om(\L)$ of $\H^{p,q}$ is not necessarily properly convex.
By Proposition~\ref{prop:M-in-prop-conv-Omega}.\eqref{item:M-in-Om-2}, there exists a $\Gamma$-invariant \emph{properly} convex open subset $\Om$ of $\Om(\L) \subset \H^{p,q}$ containing~$M$, and by Lemma~\ref{lem:Omega-Lambda-convex}.\eqref{item:Om-Lambda-in-Hpq} we have $\di\Om = \di M = \di\Om(\L) = \L$.

\medskip

\noindent
$\bullet$ \textbf{Step 1: Construction of~$y$.}
Choose respective lifts $\tilde w, \tilde x_1^{\pm}, \dots, \tilde x_j^{\pm} \in \tilde \L$ of $w, x_1^{\pm}, \dots, x_j^{\pm}$, so that $\sfb(\tilde x_i^+, \tilde x_i^-) < 0$ and $\sfb(\tilde x_i^{\pm}, \tilde x_j^{\pm}) = \sfb(\tilde x_i^{\pm}, \tilde x_{i'}^{\mp}) = 0$ for all $i\neq i'$ in $\{1,\dots,j\}$.
For any $1\leq i\leq j$, we set
$$\tilde{o}_i := \frac{1}{2}\left(\tilde{x}_i^+ + \tilde{x}_i^-\right) \in \R^{p,q+1} \quad\mathrm{and}\quad o_i := [\tilde{o}_i] \in \P(\R^{p,q+1}).$$
Then $\sfb(\tilde{o}_i, \tilde{o}_i) = \sfb(\tilde x_i^+, \tilde x_i^-)/2 < 0$ and $\sfb(\tilde{o}_i, \tilde{o}_{i'}) = 0$ for all $i\neq i'$ in $\{1,\dots,j\}$.
In particular, $o_i\in\H^{p,q}$ for all~$i$ and $T_{\sfC} := \spa(\tilde{o}_1,\dots,\tilde{o}_j)$ is a timelike $j$-plane of $\R^{p,q+1}$, contained in~$w^{\perp}$.
The spacelike geodesic $(x_i^-,x_i^+)$ of $\H^{p,q}$ meets the $(j-1)$-dimensional timelike totally geodesic subspace $\P(T_{\sfC})$ of~$\H^{p,q}$ in the singleton $\{o_i\}$.

Consider a timelike $(q+1)$-plane $T_0$ of $\R^{p,q+1}$ containing~$T_{\sfC}$.
Since $M$ is a weakly spacelike $p$-graph, by Proposition~\ref{prop:weakly-sp-gr-is-a-graph} it meets $\P(T_0)$ in a unique point, which we call $y\in\H^{p,q}$.
We have $y\in M\subset\Om(\L)$ by assumption, hence $y\notin w^{\perp}$, and so $y\notin\P(T_{\sfC})$
(in particular, $j<q+1$).

\medskip

\noindent
$\bullet$ \textbf{Step 2: Construction of~$(\g_n)$.}
Lift $y\in M$ to $\tilde{y}\in\tilde{M}$.
Then $\spa(T_{\sfC},\tilde{y})$ is a timelike $(j+1)$-plane in $\R^{p,q+1}$.
By construction, the restriction of $\sfb$ to $\spa(\tilde y,\tilde w)$ has signature $(1,1|0)$, and so the restriction of $\sfb$ to $\spa(T_{\sfC},\tilde{y},\tilde{w})$ has signature $(1,j+1|0)$.
Taking the orthogonal, the restriction of $\sfb$ to $\spa(T_{\sfC},\tilde{y},\tilde{w})^{\perp}$ has signature $(p-1,q-j|0)$.
Choose a timelike $(q-j)$-plane $T'$ in $\spa(T_{\sfC},\tilde{y},\tilde{w})^{\perp}$; then $T := \spa(T_{\sfC}, T')$ is a timelike $q$-plane of~$\R^{p,q+1}$, which does \emph{not} contain~$\tilde{y}$.

For any $t\in [0,1]$ we set
$$\tilde{r}(t) := (1-t)\tilde{y} + t\tilde{w} \in \R^{p,q+1} \quad\mathrm{and}\quad r(t) := [\tilde{r}(t)] \in \P(\R^{p,q+1}).$$
Then $t\mapsto r(t)$, for $t\in [0,1)$, is a reparametrization of the geodesic ray $[y,w)$ of~$\H^{p,q}$.

We claim that $\spa(T,\tilde r(t))$ is a timelike $(q+1)$-plane of $\R^{p,q+1}$ for all $t\in [0,1)$.
Indeed, by construction we have $\tilde{r}(t)\in (T')^{\perp}$, and so it is enough to check that $\spa(T_{\sfC},\tilde r(t))$ is timelike for $t\in [0,1)$.
Let $\tilde{v}$ be the unique element of~$T_{\sfC}$ such that $\tilde{v} + \tilde{y} \in T_{\sfC}^{\perp}$.
For any $t\in [0,1)$, consider the vector
$$\tilde{x}_t := (1-t)\tilde{v} + \tilde{r}(t) \in \spa(T_{\sfC},\tilde r(t)).$$
Then $\tilde{x}_t = (1-t)(\tilde{v}+\tilde{y})+ t\tilde{w}$ belongs to $T_{\sfC}^{\perp}$.
Moreover, the (in-)equalities
\begin{align*}
\sfb(\tilde v,\tilde v)<0,\quad \sfb(\tilde y,\tilde y)<0,\quad \sfb(\tilde w,\tilde w)=0,\\
\sfb(\tilde v,\tilde w)=0,\quad \sfb(\tilde y,\tilde v)\leq 0,\quad \sfb(\tilde y,\tilde w)<0,
\end{align*}
yield $\sfb(\tilde{x}_t,\tilde{x}_t)<0$ for all $t\in [0,1)$.
Therefore $\spa(T_{\sfC},\tilde r(t)) = \spa(T_{\sfC},\tilde x_t)$ is timelike, proving that $\spa(T,\tilde r(t))$ is a timelike $(q+1)$-plane of $\R^{p,q+1}$ for $t\in [0,1)$.
It then follows from Proposition~\ref{prop:weakly-sp-gr-is-a-graph} that for any $t\in [0,1)$, the intersection  $\ov M \cap \P(\spa(T,\tilde r(t)))$ is a singleton~$\{y_t\}$.

Note that, taking $t=0$, we have $\ov M \cap \P(\spa(T,\tilde y)) = \{ y\}$.
Since $y \notin \P(T)$, we deduce $\ov M \cap \P(T) = \emptyset$.
In particular, $y_t \notin \P(T)$ for all $t \in [0,1)$.

We claim that
$$\ov M\cap \P(\spa(T,\tilde r(1)))= \{w\}.$$
Indeed, by construction we have $w \in \ov M \cap \P(\spa(T,\tilde r(1)))$.
Moreover, the restriction of $\sfb$ to $\spa(T,\tilde r(1))$ has signature $(0,q|1)$, hence $\di\H^{p,q} \cap \P(\spa(T,\tilde r(1)))$ is reduced to $\{w\}$, and so $\di M\cap \P(\spa(T,\tilde r(1))) = \{w\}$.
On the other hand, by Lemma~\ref{lem:Omega-Lambda-convex}.\eqref{item:Om-Lambda-invisible-dom}, the assumption $M \subset \Om(\L)$ implies that the geodesic ray from any point of~$M$ to~$w$ is spacelike; this geodesic ray therefore cannot be contained in $\P(\spa(T,\tilde r(1)))$, and so $M \cap \P(\spa(T,\tilde r(1))) = \emptyset$.

The $(q+1)$-planes $\spa(T,\tilde r(t))$ of $\R^{p,q+1}$ vary continuously with $t\in [0,1]$, and so does their unique intersection points with~$M$.
Therefore $y_t\to w$ as $t\to 1$.

Choose a sequence $(t_n)\in (0,1)^{\N}$ converging to~$1$.
For any~$n$ we set $y_n := y_{t_n}$.
Let $D\subset M$ be a compact fundamental domain for the action of $\G$ on~$M$.
We then define $(\g_n)_{n\in\N}$ to be a sequence of elements of~$\G$ such that $\g_n\cdot y_n\in D$ for all~$n$.

\medskip

\noindent
$\bullet$ \textbf{Passing to subsequences.}
Up to passing to a subsequence, we may assume that\linebreak $(\g_n\cdot y_n)\in D^{\N}$ converges to some $\sf z\in M$.

Let $d_{\Om}$ be the Hilbert metric on $\Om$ (Section~\ref{subsec:prop-conv-proj}).
Since $y_n\to w\in\partial\Om$, we have $d_{\Om}(y,y_n) \to +\infty$, hence $d_{\Om}(\g_n\cdot y,\g_n\cdot y_n) \to +\infty$ since $\Gamma$ acts on $(\Om,d_{\Om})$ by isometries, and so $d_{\Om}(\g_n\cdot y,z) \to +\infty$.
Therefore, up to passing further to a subsequence, we may assume that $(\g_n\cdot y)\in M^{\N}$ converges to some $\sf{y} \in \ov M\cap \partial \Om=\L$.

Since $\Lambda$ is compact, preserved by $\G$ and $x_i^{\pm}, w\in\L$, up to passing further to a subsequence, we may assume that $(\g_n\cdot x_i^{\pm})\in\Lambda^{\N}$, for $1\leq i\leq j$, and $(\g_n\cdot w)\in\Lambda^{\N}$ converge respectively to some $\sf x_i^{\pm}\in\Lambda$ and $\sf {w}\in\Lambda$.

Consider the points $o_1,\dots,o_j\in\P(T_{\sfC})\subset\H^{p,q}$ from Step~1.
Up to passing to a further subsequence, we may assume that $(\g_n\cdot o_i)_{n\in\N}$ converges in $\overline{\H}^{p,q}$ to some $\sf{o}_i$ for each $1\leq i\leq j$.
We claim that $\sf{o}_i \in \H^{p,q}$.
Indeed, for any $n\in\N$ the points $y_n$ and $o_i$ are in timelike position (as the span of $y_n$ and $T_{\sfC}$ is timelike), and the same holds for $\g_n\cdot y_n$ and $\g_n\cdot o_i$.
Thus $\sf z$ and $\sf{o}_i$ are in timelike or lightlike position.
However $\g_n\cdot o_i\in\overline{\mathscr{C}(\L)}$ for all~$n$.
Thus if $\sf{o}_i\in\di\H^{p,q}$, then $\sf o_i\in\di\mathscr{C}(\L) \subset \di\Om(\L) = \L$ (Lemma~\ref{lem:Omega-Lambda-convex}.\eqref{item:Om-Lambda-in-Hpq}).
However, by Lemma~\ref{lem:Omega-Lambda-convex}.\eqref{item:Om-Lambda-invisible-dom}, the assumption $M \subset \Om(\L)$ implies that $\sf z\in M$ sees any point of~$\L$ in a spacelike direction.
Therefore $\sf{o}_i\notin \di \H^{p,q}$.

\medskip

\noindent
$\bullet$ \textbf{Step 3: $\{\sf x_1^{\pm}, \dots, \sf x_j^{\pm}, \sf {w}, \sf{y}\}$ is a $(j+1)$-crown in~$\L$.}
We have just seen that $\sf o_i\in\H^{p,q}$ for all~$i$.
It follows that $\sf{x}_i^+$ is transverse to $\sf{x}_i^-$, as the projective line between $\sf{x}_i^+$ and $\sf{x}^-_i$ contains $\sf{o}_i\in \H^{p,q}$ and is thus not isotropic.

For any $i\neq i'$ the points $x_i^{\pm}$ and $x_{i'}^{\pm}$ are non-transverse.
The same holds for $\g_n\cdot x_i^{\pm}$ and $\g_n\cdot x_{i'}^{\pm}$ for all~$n$.
Since being non-transverse is a closed condition, by passing to the limit we find that $\sf{x}_i^{\pm}$ and $\sf{x}_{i'}^{\pm}$ are non-transverse for $i\neq i'$.

Similarly, since $\spa(x_1^{\pm},\dots,x_j^{\pm}) \subset w^{\perp}$, we have $\spa(\g_n\cdot x_1^{\pm},\dots,\g_n\cdot x_j^{\pm}) \subset \g_n\cdot w^{\perp}$ for all~$n$, hence $\spa(\sf{x}_1^{\pm},\ldots,\sf{x}_j^{\pm}) \subset \sf{w}^{\perp}$ by passing to the limit.

We claim that
$$\spa(\sf{x}_1^{\pm},\ldots,\sf{x}_j^{\pm}) \subset \sf{y}^{\perp}.$$
Indeed, given any $1\leq i\leq j$, the point $o_i$ sees $y$ in a timelike direction, hence the same holds for $\g_n\cdot o_i$ and $\g_n\cdot y$ for all~$n$, and by passing to the limit we get that $\sf{o}_i$ sees $\sf{y}$ in a timelike or lightlike direction.
Since $\sf{y}\in \L\subset \di\H^{p,q}$, it has to be a lightlike direction, \ie $\sf{o}_i \in \sf{y}^{\perp}$.
Consider respective lifts $\tilde{\sf y},\tilde{\sf x}_i^+,\tilde{\sf x}_i^-\in\tilde\L\subset \R^{p,q+1}$ of $\sf{y},\sf{x}_i^+,\sf{x}_i^-\in\L$.
Since $\sf{o}_i \in [\sf{x}_i^+,\sf{x}_i^-]$, there exists $t_i\in (0,1)$ such that $\sf{o}_i$ lifts to $\tilde{\sf{o}}_i := t_i\tilde{\sf x}_i^+ + (1-t_i)\tilde{\sf x}_i^- \in \R^{p,q+1}$.
Since $\tilde{\L}$ is non-positive we have $\sfb(\tilde{\sf y},\tilde{\sf x}_i^+)\leq 0$ and $\sfb(\tilde{\sf y},\tilde{\sf x}_i^-)\leq 0$, and since $\sf{o}_i \in \sf{y}^{\perp}$ we have $\sfb(\tilde{\sf y},\tilde{\sf{o}}_i) = t_i\,\sfb(\tilde{\sf y},\tilde{\sf x}_i^+) + (1-t_i)\,\sfb(\tilde{\sf y},\tilde{\sf x}_i^-) = 0$.
Therefore $\sfb(\tilde{\sf y},\tilde{\sf x}_i^+) = \sfb(\tilde{\sf y},\tilde{\sf x}_i^-) = 0$, which proves the claim. 

It remains to show that $\sf{w}$ is transverse to $\sf{y}$.
For this, lift each $y_n\in M$ to $\tilde{y}_n\in\tilde{M}$.
We have seen in Step~2 that $y_n\in\P(\spa(\tilde{r}(t_n),T))$ does not belong to $\P(T)$.
Therefore $\P(\spa(T,\tilde{y}_n)) = \P(\spa(T,\tilde{r}(t_n)))$ is a projective hyperplane of $\P(\spa(T,\tilde{w},\tilde{y}))$ which separates the properly convex subset $\ov{\mathscr{C}(\L)}\cap\P(\spa(T,\tilde{w},\tilde{y}))$ of $\P(\spa(T,\tilde{w},\tilde{y}))$ into two connected components, one containing $w$ and the other one containing~$y$.
(Recall that $\ov{\mathscr{C}(\L)}$ is properly convex by Remark~\ref{rem:C-Lambda}.)

Up to passing again to a subsequence, we may assume that $(\g_n\cdot\spa(T,\tilde y_n))_{n\in\N}$ (\resp $(\g_n\cdot\nolinebreak\spa(T,\tilde w,\tilde y))_{n\in\N}$) converges to some element $\sf{T}_{\sf z}$ (\resp $\sf{T}_{\sf w,\sf y}$) in the (compact) Grassmannian of $(q+1)$-planes (\resp $(q+2)$-planes) of~$\R^{p,q+1}$, with $\sf{T}_{\sf z} \subset \sf{T}_{\sf w,\sf y}$ and $\sf{z}\in\P(\sf{T}_{\sf{z}})$ and $\sf{w},\sf{y}\in\P(\sf{T}_{\sf w,\sf y})$.
We note that $\sf{w},\sf{y}\notin  \P(\sf{T}_{\sf z})$, as both are in spacelike position to $\sf{z}\in M$, while $\sf{T}_{\sf z}$ is a limit of timelike $(q+1)$-planes, hence negative semi-definite with respect to $\sfb$ and therefore cannot contain spacelike geodesics.
Therefore, translating by~$\gamma_n$ and passing to the limit in the previous paragraph, we obtain that $\P(\sf{T}_{\sf{z}})$ is a projective hyperplane of $\P(\sf{T}_{\sf w,\sf y})$ which separates the properly convex subset $\ov{\mathscr{C}(\L)}\cap\P(\sf{T}_{\sf w,\sf y})$ of $\P(\sf{T}_{\sf w,\sf y})$ into two connected components, one containing $\sf w$ and the other one containing~$\sf y$.

Suppose by contradiction that the two points $\sf w,\sf y \in \L \subset \di\H^{p,q}$ are not transverse.
Then the segment $[\sf w,\sf y]$ in $\ov{\Om(\L)}$ is contained in $\di\Om(\L) = \ov{\Om(\L)} \cap \di\H^{p,q}$, hence in~$\L$ (Lemma~\ref{lem:Omega-Lambda-convex}.\eqref{item:Om-Lambda-in-Hpq}).
Let $\sf u\in \L$ be the intersection point $[\sf w,\sf y]\cap \P(\sf{T}_{\sf z})$, which exists and is unique because $\sf w$ and $\sf y$ lie in different components of $(\ov{\mathscr{C}(\L)}\cap\P(\sf{T}_{\sf w,\sf y})) \smallsetminus \P(\sf{T}_{\sf z})$.
Then $\sf u\in \L\cap \P(\sf{T}_{\sf z})$ and $\sf z\in M$ are in spacelike position (Lemma~\ref{lem:Omega-Lambda-convex}.\eqref{item:Om-Lambda-invisible-dom}).
However $\sf{T}_{\sf z}$ is the limit of timelike planes and therefore does not contain a spacelike geodesic: contradiction.
Therefore $\sf{w}$ is transverse to $\sf{y}$.

This shows that $\{\sf x_1^{\pm}, \dots, \sf x_j^{\pm}, \sf {w}, \sf{y}\}$ is a $(j+1)$-crown in~$\L$, as desired.
\end{proof}

\subsection{Convex hulls of $j$-crowns} \label{subsec:foliate-crown}

We shall use the following terminology.

\begin{defn} \label{def:crown-conv-hull}
Let $\sfC = \{ x_1^{\pm},\dots,x_j^{\pm}\}$ be a $j$-crown in $\di\H^{p,q}$.
\begin{itemize}
  \item An \emph{open convex hull} of $\sfC$ in $\H^{p,q}$ is an open projective simplex of $\P(\spa(\sfC))$ contained in $\H^{p,q}$, with set of vertices~$\sfC$; in other words, it is a connected component of $\P(\spa(\sfC)) \smallsetminus \bigcup_{i=1}^j ((x_i^+)^{\perp} \cup (x_i^-)^{\perp})$ contained in~$\H^{p,q}$.
   \item A \emph{basis adapted to~$\sfC$} is a basis $(e'_1,\dots,e'_{2j})$ of $\spa(\sfC)$ such that $x_i^+ = [e'_i]$ and $x_i^- = [e'_{j+i}]$ for all $1\leq i\leq j$ and such that $\sfb(e'_i,e'_{i'}) = -\delta_{j+i,i'}$ for all $1\leq i\leq i'\leq 2j$.
\end{itemize}
\end{defn}

A basis adapted to~$\sfC$ always exists since the restriction of $\sfb$ to $\spa(\sfC)$ has signature $(j,j|0)$.
It is unique up to the action of $(\R^*)^j$ given by
$$(t_1,\dots,t_j) \cdot (e'_1,\dots,e'_{2j}) = (t_1e'_1,\dots,t_je'_j,t_1^{-1}e'_{j+1},\dots,t_j^{-1}e'_{2j}).$$

\begin{remark} \label{rem:basis-adapted-C}
Let $(e'_1,\dots,e'_{2j})$ be a basis adapted to~$\sfC$.
Then
\begin{itemize}
  \item $\sfb(e'_i,e'_{i'})\leq 0$ for all $1\leq i,i'\leq 2j$; this shows that $\sfC$ is a non-positive subset of $\di\H^{p,q}$ (Definition~\ref{def:nonpos-neg});
  \item $(e'_1,\dots,e'_{2j})$ defines a unique open convex hull $\mathcal{O}$ of $\sfC$ in $\H^{p,q}$, namely the projectivization of the $\R_{>0}$-span of $(e'_1,\dots,e'_{2j})$; the set $\di\mathcal{O} = \ov{\mathcal{O}}\cap\di\H^{p,q}$ is a non-degenerate non-positive $(j-1)$-sphere in $\di\H^{p,q} \cap \P(\spa(\sfC)) \simeq \di\H^{j,j-1}$ as in Example~\ref{ex:non-pos-sphere}.(ii) with $T = \spa(e'_1+e'_{j+1},\dots,e'_j+e'_{2j})$; in particular, $\di\mathcal{O} = \ov{\mathcal{O}}\cap\di\H^{p,q}$ is a union of $2^j$ closed faces of~$\ov{\mathcal{O}}$ of dimension $j-1$, each determined by $j$ vertices of the form $x_1^{\varepsilon_1},\dots,x_j^{\varepsilon_j}$ for $\varepsilon_1,\dots,\varepsilon_j\in\{\pm\}$;
  \item if $t_1,\dots,t_j\in\R^*$ all have the same sign, then $(t_1,\dots,t_j) \cdot (e'_1,\dots,e'_{2j})$ defines the same open convex hull $\mathcal{O}$ as $(e'_1,\dots,e'_{2j})$;
  \item by replacing $(e'_1,\dots,e'_{2j})$ with $(t_1,\dots,t_j) \cdot (e'_1,\dots,e'_{2j})$ for $t_1,\dots,t_j\in\R^*$ of varying signs, we see that there are $2^{j-1}$ possible open convex hulls of~$\sfC$ in $\H^{p,q}$, and $2^{j-1}$ possible corresponding non-degenerate non-positive $(j-1)$-spheres of $\di\H^{p,q}$.
\end{itemize}
\end{remark}

\begin{lem} \label{lem:crown-conv-hull-Om-Lambda}
Let $\L$ be a non-degenerate non-positive $(p-1)$-sphere in $\di\H^{p,q}$ and $\sfC = \{ x_1^{\pm},\dots,x_j^{\pm}\}$ a $j$-crown in~$\L$.
Then there is a unique open convex hull of $\sfC$ in $\H^{p,q}$ that is contained in $\ov{\Om(\L)}$; we shall denote it by $\mathcal{O}_{\L}(\sfC)$.
We have $\Om(\L) \cap \P(\spa(\sfC)) \subset \mathcal{O}_{\L}(\sfC)$, with equality if and only if $\sfC$ is \emph{not} a boundary $j$-crown in~$\L$.
\end{lem}

\begin{proof}
As in Notation~\ref{not:Omega-Lambda}, let $\tilde{\L}$ be a subset of the non-zero isotropic vectors of $\R^{p,q+1}$ whose projection to $\di\hat\H^{p,q} \subset (\R^{p,q+1}\smallsetminus\{0\})/\R_{>0}$ is a non-positive $(p-1)$-sphere projecting onto~$\L$.
We can lift each point $x_i^{\pm}\in\sfC$ to a point $\tilde{x}_i^{\pm}\in\tilde{\L}$.
We then have $\sfb(\tilde{x}_i^-,\tilde{x}_i^+) < 0$ for all $1\leq i\leq j$ and $\sfb(\tilde{x}_i^{\pm},\tilde{x}_{i'}^{\pm}) = 0$ for all $1\leq i<i'\leq j$, as well as $\sfb(\tilde{x}_i^{\pm},\tilde{w}) \leq 0$ for all $1\leq i\leq j$ and $\tilde{w}\in\tilde{\L}$.

Consider a vector $v \in \spa(\sfC)$.
We can write $v = \sum_{i=1}^j (s_i \tilde{x}_i^- + t_i \tilde{x}_i^+)$ for some $s_i,t_i\in\R$.
For any $\tilde{w}\in\tilde{\L}$ we have $\sfb(v,\tilde{w}) = \sum_{i=1}^j (s_i\,\sfb(\tilde{x}_i^-,\tilde{w}) + t_i\,\sfb(\tilde{x}_i^+,\tilde{w}))$.
In particular, $\sfb(v,\tilde{x}_i^+) = s_i\,\sfb(\tilde{x}_i^-,\tilde{x}_i^+)$ and $\sfb(v,\tilde{x}_i^-) = t_i\,\sfb(\tilde{x}_i^-,\tilde{x}_i^+)$ for all~$i$.
Therefore we have $v \in \tilde{\ov{\Om}}(\tilde{\L})$ if and only if $s_i,t_i\geq 0$ for all~$i$, and if $v \in \tilde{\Om}(\tilde{\L})$ then $s_i,t_i>0$ for all~$i$.
This shows that the projectivization $\mathcal{O}_{\L}(\sfC)$ of the $\R_{>0}$-span of $\tilde{x}_1^{\pm},\dots,\tilde{x}_j^{\pm}$ is the unique open convex hull of $\sfC$ contained in $\ov{\Om(\L)}$, and that $\Om(\L) \cap \P(\spa(\sfC)) \subset \mathcal{O}_{\L}(\sfC)$.

By Lemma~\ref{lem:Omega-Lambda-convex}.\eqref{item:Om-Lambda-invisible-dom}, we have $\mathcal{O}_{\L}(\sfC) \cap \partial\Om(\L) \neq \emptyset$ if and only if there exists $w\in\L$ such that $\mathcal{O}_{\L}(\sfC) \cap w^{\perp} \neq \emptyset$.
Consider $\tilde{w}\in\tilde{\L}$ lifting~$w$.
Then $\mathcal{O}_{\L}(\sfC) \cap w^{\perp} \neq \emptyset$ if and only if there exist $s_i,t_i>0$ such that $v := \sum_{i=1}^j (s_i \tilde{x}_i^- + t_i \tilde{x}_i^+)$ satisfies $\sfb(v,\tilde{w}) = 0$.
Since $\sfb(v,\tilde{w}) = \sum_{i=1}^j (s_i\,\sfb(\tilde{x}_i^-,\tilde{w}) + t_i\,\sfb(\tilde{x}_i^+,\tilde{w}))$ where $\sfb(\tilde{x}_i^{\pm},\tilde{w})\leq 0$ for all~$i$, we deduce that $\mathcal{O}_{\L}(\sfC) \cap w^{\perp} \neq\nolinebreak \emptyset$ if and only if $\sfb(\tilde{x}_i^{\pm},\tilde{w})=0$ for all~$i$, or in other words if and only if $\sfC \subset w^{\perp}$.
Thus $\mathcal{O}_{\L}(\sfC) \cap\nolinebreak \partial\Om(\L) \neq \emptyset$ if and only if $\sfC$ is a boundary $j$-crown in~$\L$.
\end{proof}

Let $\rm{stab}_{\PO(p,q+1)}(\spa(\sfC))$ be the stabilizer of $\spa(\sfC)$ in $\PO(p,q+1)$.
Let $G_{\sfC} \simeq \PO(j,j)$ be the subgroup of $\rm{stab}_{\PO(p,q+1)}(\spa(\sfC))$ centralizing $\spa(\sfC)^{\perp}$, and let $A_{\sfC}$ be the identity component of the stabilizer of $\sfC$ in $G_{\sfC}$.
Then $A_{\sfC}$ is a closed subgroup of $\PO(p,q+1)$ isomorphic to~$\R^j$.
Indeed, in a basis adapted to~$\sfC$ (Definition~\ref{def:crown-conv-hull}), the elements of $A_{\sfC}$ restricted to $\spa(\sfC)$ are diagonal of the form
\begin{equation} \label{eqn:A-C}
{\boldsymbol\exp}(a_1,\ldots,a_j):=\mathrm{diag}(\exp(a_1),\ldots,\exp(a_j),\exp(-a_1),\ldots,\exp(-a_j))
\end{equation}
for $a_1,\ldots,a_j\in\R$; the map ${\boldsymbol\exp}$ defines an isomorphism between $\R^j$ and~$A_{\sfC}$.
Note that this isomorphism depends only on~$\sfC$ (not on the chosen basis adapted to~$\sfC$).

We consider the foliation of $\P(\spa(\sfC)) \smallsetminus \bigcup_{i=1}^j ((x_i^+)^{\perp} \cup (x_i^-)^{\perp}$) by $A_{\sfC}$-orbits.
It is not difficult to see that the stabilizer in~$A_{\sfC}$ of any $x \in \P(\spa(\sfC)) \smallsetminus \bigcup_{i=1}^j ((x_i^+)^{\perp} \cup (x_i^-)^{\perp}$ is trivial, hence we can identify any leaf $F = A_{\sfC}\cdot x$ of the foliation with~$A_{\sfC}$.
The following lemma is designed to clarify the context; it will not be used anywhere in the sequel.

\begin{lem} \label{lem:F-spacelike}
Let $\sfC = \{ x_1^{\pm},\dots,x_j^{\pm}\}$ be a $j$-crown in $\di\H^{p,q}$ and let $(e'_1,\ldots,e'_{2j})$ be a basis adapted to~$\sfC$, defining an open convex hull $\mathcal{O}$ of $\sfC$ in $\H^{p,q}$ (Definition~\ref{def:crown-conv-hull} and Remark~\ref{rem:basis-adapted-C}).
Then
\begin{enumerate}
  \item\label{item:F-spacelike-1} any $A_{\sfC}$-orbit $F$ in~$\mathcal{O}$ is a $j$-dimensional complete spacelike $C^{\infty}$ submanifold of $\H^{p,q} \cap \P(\spa(\sfC)) \simeq \H^{j,j-1}$, with boundary at infinity $\di F = \di\mathcal{O} = \ov{\mathcal{O}}\cap\di\H^{p,q}$; in particular, the convex hull of $F$ in~$\mathcal{O}$ is all of~$\mathcal{O}$;
  \item\label{item:F-spacelike-2} there is exactly one such $A_{\sfC}$-orbit which is maximal in the sense of Definition~\ref{def:max-submfd}, namely the one passing through $[e'_1 + \dots + e'_{2j}]$;
  \item\label{item:F-spacelike-3} in the splitting $\hat\H^{p,q} \cap \spa(\sfC) \simeq \hat\H^{j,j-1} \simeq \bB^j \times \bS^{j-1}$ from Proposition~\ref{p.doublecoverasproduct} associated to the $\sfb$-orthogonal basis $(e_1,\dots,e_{2j}) := (e'_1 - e'_{j+1}, \dots, e'_j - e'_{2j}, e'_1 + e'_{j+1}, \dots, e'_j + e'_{2j})/\sqrt{2}$ of $\spa(\sfC)$ which is standard for $T := \spa(e_{j+1},\dots,e_{2j})$, the $A_{\sfC}$-orbit of $\sum_{i=1}^j \tau_i e_{j+i}$ (where $\tau_1,\dots,\tau_j\in\R$ satisfy $\sum_{i=1}^j \tau_i^2 = 1$) is the graph of the $1$-Lipschitz map $f : \bB^j\to\bS^{j-1}$ given by
$$f(y_0, \dots, y_j) = \Big(\sqrt{\tau_1^2\,y_0^2 + y_1^2}, \dots, \sqrt{\tau_j^2\,y_0^2 + y_j^2}\Big).$$
\end{enumerate}
\end{lem}

In the case $j=p=2$, the maximal $A_{\sfC}$-orbit in~\eqref{item:F-spacelike-2} is called in \cite{lt} a \emph{Barbot surface}.
See Figure~\ref{fig:A_Corbit} in Section~\ref{subsec:weakly-spacelike-graphs} for an illustration.

\begin{proof}
\eqref{item:F-spacelike-1} Let $F$ be an $A_{\sfC}$-orbit in~$\mathcal{O}$.
One easily checks that $F$ is a $j$-dimensional $C^{\infty}$ embedded submanifold of $\H^{p,q}$.
Fix $x\in F$ and let us check that the restriction of the metric $\sfg$ to the tangent space $T_xF$ is positive definite.
Since by construction $\mathcal{O}$ is the projectivization of the $\R_{>0}$-span of $(e'_1,\dots,e'_{2j})$, we can lift $x$ to a point of the double cover $\hat\H^{p,q} = \{ v\in\nolinebreak\R^{p,q+1}\,|\, \sfb(v,v)=-1\}$ of the form $\hat{x} = \sum_{i=1}^{2j} v_i e'_i$ with $v_i>0$ for all~$i$.
It is sufficient to check that the restriction of the metric $\hat{\sfg}$ to the tangent space $T_{\hat{x}}\hat\H^{p,q}$ is positive definite.
Recall that $T_{\hat{x}}\hat\H^{p,q}$ identifies with the orthogonal ${\hat{x}}^{\perp}$ of $\hat{x}$ in~$\R^{p,q+1}$, and the restriction of $\hat{\sfg}$ to $T_{\hat{x}}\hat\H^{p,q}$ with the restriction of $\sfb$ to ${\hat{x}}^{\perp}$ (see Section~\ref{subsec:Hpq}).
For any $a = (a_1,\ldots,a_j)\in \R^j\smallsetminus\{0\}$, the tangent vector
\begin{equation} \label{eqn:v-a}
\sf{v}_a := \frac{d}{dt}\Big|_{t=0} \, {\boldsymbol\exp}(ta_1,\ldots,ta_j)\cdot \hat{x} \in T_{\hat{x}}\hat\H^{p,q}
\end{equation}
(see \eqref{eqn:A-C}) is given by $\sf{v}_a = \sum_{i=1}^j a_i (v_i\,e'_i - v_{j+i}\,e'_{j+i})$, hence $\sfb(\sf{v}_a,\sf{v}_a)= 2\sum_{i=1}^j a_i^2\,v_i\,v_{j+i} >\nolinebreak 0$.
This shows that the restriction of $\sfg$ to $T_xF$ is positive definite.
This holds for any $x\in F$, hence $F$ is spacelike.

Since $F$ is a closed subset of~$\mathcal{O}$, we have $\di F \subset \di\mathcal{O}$.
Let us check the reverse inclusion.
As in Remark~\ref{rem:basis-adapted-C}, the boundary at infinity $\di\mathcal{O} = \ov{\mathcal{O}}\cap\di\H^{p,q}$ is a union of $2^j$ closed faces of~$\ov{\mathcal{O}}$ of dimension $j-1$, each determined by $j$ vertices of the form $x_1^{\varepsilon_1},\dots,x_j^{\varepsilon_j}$ for $\varepsilon_1,\dots,\varepsilon_j\in\{\pm\}$.
Thus any point $z$ of $\di\mathcal{O}$ lifts to a nonzero vector of $\R^{p,q+1}$ of the form $\sum_{i\in I^+} z_i e'_i + \sum_{i\in I^-} z_i e'_{j+i}$ where $I^+,I^-$ are disjoint subsets of $\{1,\dots,j\}$ and where $z_i>0$ for all $i\in I^+\cup I^-$.
Let $x$ be any point of~$F$, lifting to $\hat{x} = \sum_{i=1}^{2j} v_i e'_i$ with $v_i>0$ for all~$i$.
Setting $a_{i,t} := \log(tz_i/ v_i)$ (\resp $-\log(tz_i/v_i)$) for all $i\in I^+$ (\resp $I^-$) and $a_{i,t} := 0$ for all $i\in\{ 1,\dots,j\}\smallsetminus (I^+\cup I^-)$, we then have ${\boldsymbol\exp}(a_{1,t},\dots,a_{j,t}) \cdot x \to z$ as $t\to +\infty$.
This shows that $\di F = \di\mathcal{O}$.

\eqref{item:F-spacelike-2} Let us compute the mean curvature of an $A_{\sfC}$-orbit $F$ in~$\mathcal{O}$.
Fix $x\in F$ and lift it to a point of $\hat\H^{p,q}$ of the form $\hat{x} = \sum_{i=1}^{2j} v_i e'_i$ with $v_i>0$ for all $1\leq i\leq 2j$.
For each $1\leq i\leq j$, we set $u_i := v_i e'_i + v_{j+i} e'_{j+i} \in \R^{p,q+1}$ (so that $\hat{x} = \sum_{i=1}^j u_i$) and $a^{(i)} := (0,\dots,0,1,0,\dots,0) \in \R^j$, where $1$ is at the $i$-th position.
Using notation \eqref{eqn:v-a}, we then have $\sf{v}_{a^{(i)}}= v_i e'_i - v_{j+i} e'_{j+i}$, the $\sf{v}_{a^{(i)}}$ are $\sfb$-orthogonal to each other, and $\sfb(\sf{v}_{a^{(i)}},\sf{v}_{a^{(i)}})=|\sfb(u_i,u_i)|$ for all~$i$.
Thus the $j$-tuple $(\sf{v}_{a^{(1)}}/\sqrt{|\sfb(u_1,u_1)|},\dots,\sf{v}_{a^{(j)}}/\sqrt{\sfb(u_j,u_j)})$ is an orthonormal basis of $T_{\hat{x}}(A_{\sfC}\cdot\hat{x})$ in $T_{\hat{x}}\hat\H^{p,q} \simeq {\hat{x}}^{\perp}$, and so the mean curvature of $A_{\sfC}\cdot\hat{x}$ at~$\hat{x}$ (which is the same as the mean curvature of $F = A_{\sfC}\cdot x$ at~$x$) is given by
$$H(\hat{x}) = \sum_{i=1}^j \frac{1}{j\,|\sfb(u_i,u_i)|} \, \mathrm{II}(\sf{v}_{a^{(i)}},\sf{v}_{a^{(i)}})$$
(see Section~\ref{subsec:max-submfd}).
Since the Levi-Civita connection of $\hat{\H}^{p,q}$ is given by the orthogonal projection with respect to $\sfb$ of the standard connection on $\R^{p+q+1}$, for each $1\leq i\leq j$ the vector $\mathrm{II}(\sf{v}_{a^{(i)}},\sf{v}_{a^{(i)}}) \in T_{\hat{x}}\hat\H^{p,q} \simeq {\hat{x}}^{\perp}$ is the orthogonal projection to $(T_{\hat{x}}(A_{\sfC}\cdot\hat{x}))^{\perp} \cap {\hat{x}}^{\perp}$ of
$$\frac{d^2}{dt^2}\Big|_{t=0} \, {\boldsymbol\exp}(ta^{(i)})\cdot \hat{x} = u_i.$$
Observe that $u_i\in\R^{p,q+1}$ belongs to the $\sfb$-orthogonal complement of $T_{\hat{x}}(A_{\sfC}\cdot\hat{x})$ in $\R^{p,q+1}$; therefore, its orthogonal projection to $(T_{\hat{x}}(A_{\sfC}\cdot\hat{x}))^{\perp} \cap {\hat{x}}^{\perp}$ is just its orthogonal projection to ${\hat{x}}^{\perp}$, and $H(\hat{x})$ is the orthogonal projection to ${\hat{x}}^{\perp}$ of $\sum_{i=1}^j u_i/(j\,|\sfb(u_i,u_i)|)$.
In particular, $H(\hat{x}) = 0$ if and only if $\sum_{i=1}^j u_i/(j\,|\sfb(u_i,u_i)|)$ is a multiple of $\hat{x} = \sum_{i=1}^j u_i$, if and only if $\sfb(u_1,u_1) = \dots = \sfb(u_j,u_j)$.
This is equivalent to $e'_1 + \dots + e'_{2j}$ belonging to $A_{\sfC}\cdot\hat{x}$, \ie to $[e'_1 + \dots + e'_{2j}]$ belonging to $F = A_{\sfC}\cdot x$.

\eqref{item:F-spacelike-3} For $\tau_1,\dots,\tau_j\in\R$ with $\sum_{i=1}^j \tau_i^2 = 1$, the $A_{\sfC}$-orbit of $\sum_{i=1}^j \tau_i e_{j+i}$ is the set of points of $\hat\H^{p,q} \cap \spa(\sfC) \simeq \hat\H^{j,j-1}$ of the form $\sum_{i=1}^j \tau_i \, (\sinh(a_i)\,e_i + \cosh(a_i)\,e_{j+i})$ where $a_1,\dots,a_j \in \R$, or in other words of the form $\sum_{i=1}^j \tau_i \, \big(s_i\,e_i + \sqrt{1+s_i^2}\,e_{j+i})$ where $s_1,\dots,s_j \in \R$.
In the splitting $\hat\H^{p,q} \cap \spa(\sfC) \simeq \hat\H^{j,j-1} \simeq \bB^j \times \bS^{j-1}$ given explicitly by Proposition~\ref{p.doublecoverasproduct}, this corresponds to the graph of $f : (y_0,\dots,y_j) \mapsto (\sqrt{\tau_1^2\,y_0^2 + y_1^2}, \dots, \sqrt{\tau_j^2\,y_0^2 + y_j^2})$.
\end{proof}

If $\sfC$ is a $1$-crown, then it has a unique open convex hull in $\H^{p,q}$, namely the open projective segment $\H^{p,q}\cap\spa(\sfC)$, which is a single $A_{\sfC}$-orbit.
On the other hand, when $j\geq 2$ the following holds.

\begin{prop} \label{prop:A-orbit-not-hyperb}
For $j\geq 2$, let $\sfC$ be a $j$-crown in $\di\H^{p,q}$, let $\mathcal{O}$ be an open convex hull of $\sfC$ in $\H^{p,q}$ (Definition~\ref{def:crown-conv-hull}), and let $F$ be an $A_{\sfC}$-orbit in $\mathcal{O}$.
Then for any $\delta>0$, there exist geodesic quadrilaterals for the Hilbert metric $d_{\mathcal{O}}$, consisting of four projective line segments contained in~$F$, such that a side of the quadrilateral is not contained in the uniform $\delta$-neighborhood for $d_{\mathcal{O}}$ of the other three sides.
\end{prop}

\begin{proof}
Let $(e'_1,\ldots,e'_{2j})$ be a basis adapted to~$\sfC$ (Definition~\ref{def:crown-conv-hull}), defining~$\mathcal{O}$ as in Remark~\ref{rem:basis-adapted-C}: namely, $\mathcal{O}$ is the projectivization of the $\R_{>0}$-span of $(e'_1,\dots,e'_{2j})$.
Let ${\boldsymbol\exp} : \R^j \overset{\sim}{\longrightarrow} A_{\sfC}$ be the isomorphism given by \eqref{eqn:A-C}.

We first observe that for any $x\in\mathcal{O}$ and any $a_1,\dots,a_j\in\R$,
\begin{equation} \label{eqn:Hilbert-dist-simplex}
d_{\mathcal{O}}\big(x,{\boldsymbol\exp}(a_1,\dots,a_j)\cdot x\big) = \max_{1\leq i\leq j} |a_i|.
\end{equation}
Indeed, consider a lift $\tilde{x} = \sum_{i=1}^{2j} v_i e'_i$ of $x$ to the $\R_{>0}$-span of $(e'_1,\dots,e'_{2j})$, where $v_i>0$ for all~$i$.
Then $y := {\boldsymbol\exp}(a_1,\dots,a_j)\cdot x$ lifts to $\tilde{y} = \sum_{i=1}^j (e^{a_i}\,v_i\,e'_i + e^{-a_i}\,v_{j+i}\,e'_{j+i})$, and the projective line $\P(\spa(x,y))$ meets $\partial\mathcal{O}$ in two points, namely $[\tilde{x} - e^{-\max_i |a_i|}\,\tilde{y}]$ and $[\tilde{y} - e^{-\max_i |a_i|}\,\tilde{x}]$.
We conclude using the formula \eqref{eqn:d-Omega} for the Hilbert metric $d_{\mathcal{O}}$.

In particular, for any $x\in\mathcal{O}$ and any $(a_1,\dots,a_j) \in \R^j\smallsetminus\{ (0,\dots,0)\}$, the path $t \mapsto {\boldsymbol\exp}(ta_1,\dots,ta_j)\cdot x$ is a geodesic line for the Hilbert metric $d_{\mathcal{O}}$, which is parametrized at speed $\max_i |a_i|$, and whose image is contained in a projective line.

Fix $x\in F$ and, for any $R>0$, consider the four points
$$\begin{array}{ll}
\mathtt{a}_R := {\boldsymbol\exp}(R,-R,\dots,-R)\cdot x, & \mathtt{b}_R := {\boldsymbol\exp}(-R,R,\dots,R)\cdot x,\\
\mathtt{c}_R := {\boldsymbol\exp}(-R,-3R,\dots,-3R)\cdot x, & \mathtt{d}_R := {\boldsymbol\exp}(-3R,-R,\dots,-R)\cdot x
\end{array}$$
of $F \subset \mathcal{O}$.
By what we have just seen, the geodesic segment in $(\mathcal{O},d_{\mathcal{O}})$ between $\mathtt{a}_R$ and $\mathtt{b}_R$ (\resp $\mathtt{a}_R$ and $\mathtt{c}_R$, \resp $\mathtt{b}_R$ and $\mathtt{d}_R$, \resp $\mathtt{c}_R$ and $\mathtt{d}_R$) which is contained in a projective line is given by $t \mapsto {\boldsymbol\exp}(-t,t,\dots,t)\cdot\nolinebreak\mathtt{a}_R$ (\resp $t \mapsto {\boldsymbol\exp}(-t,-t,\dots,-t)\cdot\mathtt{a}_R$, \resp $t \mapsto {\boldsymbol\exp}(-t,-t,\dots,-t)\cdot\nolinebreak\mathtt{b}_R$, \resp $t \mapsto {\boldsymbol\exp}(-t,t,\dots,t)\cdot\mathtt{c}_R$) for $t\in [0,2R]$, hence its image is contained in~$F$.
In particular, $x$ belongs to the projective line segment $[\mathtt{a}_R,\mathtt{b}_R]$.
By \eqref{eqn:Hilbert-dist-simplex} we have $d_{\mathcal{O}}(x,y)\geq R$ for all $y \in [\mathtt{a}_R,\mathtt{c}_R] \cup [\mathtt{c}_R,\mathtt{d}_R] \cup [\mathtt{d}_R,\mathtt{b}_R]$.
Therefore $[\mathtt{a}_R,\mathtt{b}_R]$ is \emph{not} contained in the uniform $\delta$-neighborhood for $d_{\mathcal{O}}$ of $[\mathtt{a}_R,\mathtt{c}_R] \cup [\mathtt{c}_R,\mathtt{d}_R] \cup [\mathtt{d}_R,\mathtt{b}_R]$ for $\delta<R$.
\end{proof}

\subsection{Crowns prevent Gromov hyperbolicity} \label{subsec:crown-not-hyperb}

In this section we use Propositions \ref{prop:M-in-prop-conv-Omega}, \ref{p.crown-boundary}, and~\ref{prop:A-orbit-not-hyperb}, together with Lemma~\ref{lem:crown-conv-hull-Om-Lambda}, to prove the implication \eqref{item:geom-action-1}~$\Rightarrow$~\eqref{item:geom-action-4} of Theorem~\ref{thm:geom-action-weakly-sp-gr}.

\begin{prop} \label{p.boundaries dont intersect for hyperbolic}
Let $\Gamma$ be a discrete subgroup of $\PO(p,q+1)$ acting properly discontinuously and cocompactly on a weakly spacelike $p$-graph $M$ in $\H^{p,q}$ with $M \subset \Om(\L)$ where $\L := \di M \subset \di\H^{p,q}$.
If there is a $j$-crown in~$\L$ for some $j\geq 2$, then $\Gamma$ is not Gromov hyperbolic.
\end{prop}

We start with the following preliminary result.

\begin{lem} \label{l.tubular-contains-A-orbit}
In the setting of Proposition~\ref{p.boundaries dont intersect for hyperbolic}, let $\sfC$ be a $j$-crown in~$\L$, and suppose that $j$ is maximal in the sense that there does not exist any $(j+1)$-crown in~$\L$.
Let $F$ be an $A_{\sfC}$-orbit in the open convex hull $\calO_{\L}(\sfC)\subset\ov{\Om(\L)}$ of Lemma~\ref{lem:crown-conv-hull-Om-Lambda}.
Then there exist $R\geq 0$ and a $\Gamma$-invariant properly convex open subset $\Om$ of $\Om(\L)$ containing~$M$ such that $F$ is contained in the closed uniform $R$-neighborhood $\mathcal{U}_R(M)$ of $M$ in $(\Om,d_{\Om})$.
\end{lem}

Recall that in this setting the convex open set $\Omega(\L) \subset \H^{p,q}$ is properly convex if $\L$ spans (Lemma~\ref{lem:Omega-Lambda-convex}.\eqref{item:Om-Lambda-prop-convex}), but not in general.

\begin{proof}[Proof of Lemma~\ref{l.tubular-contains-A-orbit}]
By Proposition~\ref{p.crown-boundary}, the maximality of~$j$ implies that $\sfC$ is \emph{not} a boundary $j$-crown in~$\L$.
Therefore, Lemma~\ref{lem:crown-conv-hull-Om-Lambda} implies that $F \subset \calO_{\L}(\sfC) \subset \Om(\L)$.

Consider the $\Gamma$-invariant closed subset
$$\mathcal{Z} := \Om(\L) \cap \overline{\bigcup_{y\in F} \Gamma\cdot y}$$
of $\Om(\L)$, which contains~$F$.
It is sufficient to prove that $\de_{\H}\mathcal{Z} \cap \de_{\H}\Om(\L) = \emptyset$.
Indeed, if this is the case, then by Proposition~\ref{prop:M-in-prop-conv-Omega}.\eqref{item:M-in-Om-3} there exists a $\Gamma$-invariant properly convex open subset $\Omega$ of $\Omega(\L)$ containing $M$ and~$\mathcal{Z}$ (hence~$F$), and the action of $\Gamma$ on~$\mathcal{Z}$ is properly discontinuous and cocompact.
Consider a compact fundamental domain for the action of $\Gamma$ on~$\mathcal{Z}$: it is contained in some closed uniform neighborhood $\mathcal{U}_R(M)$ of $M$ in $(\Om,d_{\Om})$.
Since $M$ and~$\mathcal{Z}$, and the Hilbert metric $d_{\Om}$, are all $\Gamma$-invariant, we have $\mathcal{Z} \subset \mathcal{U}_R(M)$, hence $F \subset  \mathcal{U}_R(M)$, as desired.

Our goal for the rest of the proof is to show that $\de_{\H}\mathcal{Z} \cap \de_{\H}\Om(\L) = \emptyset$.
For this, consider sequences $(\gamma_n)\in\Gamma^{\N}$ and $(y_n)\in F^{\N}$, and let us show that the sequence $(\gamma_n\cdot y_n)_{n\in\N}$ in~$\H^{p,q}$ cannot converge to a point of $\de_{\H}\Om(\L)$.

Write $\sfC = \{ x_1^{\pm},\ldots,x_j^{\pm}\}$ where $x_i^-$ and~$x_i^+$ are transverse and $x_i^{\pm}$ and $x_{i'}^{\pm}$ are \emph{not} transverse for $1\leq i\neq i'\leq j$.
By definition of $\calO_{\L}(\sfC)$ (see Lemma~\ref{lem:crown-conv-hull-Om-Lambda}), we can lift $\L$ to a non-positive subset $\tilde\L$ of $\R^{p,q+1}\smallsetminus\{0\}$, and lift each $x_i^{\pm}$ to $\tilde{x}_i^{\pm}\in\tilde\L$ and $\calO_{\L}(\sfC)$ to $\tilde{\calO}_{\L}(\sfC) \subset \R^{p,q+1}\smallsetminus\{0\}$ so that $\tilde{\calO}_{\L}(\sfC)$ is the $\R_{>0}$-span of $\tilde{\sfC} := \{ \tilde{x}_1^{\pm}, \dots, \tilde{x}_j^{\pm}\}$.

For each $n$, we claim that there exists a timelike $(q+1)$-plane $T_n$ of $\R^{p,q+1}$ whose projectivization $\P(T_n)$ contains $y_n$ and meets all $j$ geodesics $(x_i^-,x_i^+)$ of $\calO_{\L}(\sfC)$ for $1\leq i\leq j$.
Indeed, consider a lift $\tilde{y}_n \in \tilde{\calO}_{\L}(\sfC)$ of $y_n \in \calO_{\L}(\sfC)$: there exist $s_i,t_i>0$ such that $\tilde{y}_n = \sum_{i=1}^j (s_i\tilde x^-_i + t_i\tilde x^-_i)$.
Then the span of the $s_i\tilde x^-_i + t_i\tilde x^-_i$ for $1\leq i\leq j$ is a timelike $j$-plane of~$\R^{p,q+1}$ whose projectivization contains~$y_n$, and we can extend it to a timelike $(q+1)$-plane $T_n$ of $\R^{p,q+1}$ with the desired properties.

Since $M$ is a weakly spacelike $p$-graph, by Proposition~\ref{prop:weakly-sp-gr-is-a-graph}, for each~$n$ there is a unique intersection point $\{ z_n\} := \P(T_n) \cap M$.

Let $D$ be a compact fundamental domain for the action of $\G$ on~$M$.
For any~$n$ there exists $\gamma'_n\in\Gamma$ such that $\gamma'_n\gamma_n\cdot z_n\in D$.
Up to passing to a subsequence, we may and shall assume that $(\gamma'_n\gamma_n\cdot z_n)_{n\in\N}$ converges to some $\sf{z} \in D$, that $(\gamma'_n\gamma_n\cdot x_i^{\pm})_{n\in\N}$ converges to some $\sf{x}_i^{\pm}\in\L$ for each~$i$, and that $(\gamma'_n\gamma_n\cdot y_n)_{n\in\N}$ converges to some $\sf{y} \in \ov{\H}^{p,q}$. 

We claim that $\sfC_{\infty}:= \{\sf{x}_1^{\pm},\ldots,\sf{x}_j^{\pm}\}$ is a $j$-crown in~$\L$.
Indeed, consider $1\leq i\neq i'\leq\nolinebreak j$.
For any~$n$ the points $\gamma'_n\gamma_n\cdot x_i^{\pm}$ and $\gamma'_n\gamma_n\cdot x_{i'}^{\pm}$ are non-transverse (\ie orthogonal), and being non-transverse is a closed condition; therefore $\sf{x}_i^{\pm}$ and $\sf{x}_{i'}^{\pm}$ are non-transverse. 
Assume by contradiction that $\sf{x}_i^-$ and $\sf{x}_i^+$ are non-transverse.
Then the segment $[\sf{x}_i^-,\sf{x}_i^+]$ is contained in $\di\Om$, which is equal to $\L$ by Lemma~\ref{lem:Omega-Lambda-convex}.\eqref{item:Om-Lambda-in-Hpq} and Proposition~\ref{prop:bdynonpossphere}.\eqref{item:bound-non-pos-sph-1}.
But each $z_n \in \P(T_n)$ is in timelike position with $\P(T_n) \cap [x_i^-,x_i^+]$, hence $\gamma'_n\gamma_n\cdot z_n$ is in timelike position with $\P(\gamma'_n\gamma_n\cdot T_n) \cap [\gamma'_n\gamma_n\cdot x_i^-,\gamma'_n\gamma_n\cdot x_i^+]$.
By passing to a subsequence and taking a limit, we find a point of $[\sf{x}_i^-,\sf{x}_i^+]$ that is in timelike or lightlike position with~$\sf z$.
This is a contradiction, as $\sf z\in M$ is in spacelike position with all of~$\L$ by Lemma~\ref{lem:Omega-Lambda-convex}.\eqref{item:Om-Lambda-invisible-dom}.
Thus $\sf{x}_i^-$ and $\sf{x}_i^+$ are transverse.
This shows that $\sfC_{\infty}$ is a $j$-crown in~$\L$.

Note that $\sf{y} \in \ov{\calO_{\L}(\sfC_{\infty})}$.
Indeed, $y_n \in F \subset \calO_{\L}(\sfC)$ for all~$n$, hence $\gamma'_n\gamma_n\cdot y_n \in \gamma'_n\gamma_n\cdot\nolinebreak\calO_{\L}(\sfC) = \calO_{\L}(\gamma'_n\gamma_n\cdot\sfC)$ for all~$n$, hence $\sf{y} \in \ov{\calO_{\L}(\sfC_{\infty})}$ by passing to the limit.

We claim that $\sf{y} \in \H^{p,q}$.
Indeed, we have $\sf{y} \in \ov{\calO_{\L}(\sfC_{\infty})} \subset \ov{\Om(\L)}$.
By Lemma~\ref{lem:Omega-Lambda-convex}.\eqref{item:Om-Lambda-in-Hpq} we have $\ov{\Om(\L)} \cap \di\H^{p,q} = \di\Om(\L) = \L$, hence it is enough to check that $\sf{y} \notin \L$.
Since $y_n$ and $z_n$ are in timelike position for all~$n$, the same holds for $\gamma'_n\gamma_n\cdot y_n$ and $\gamma'_n\gamma_n\cdot z_n$, and by passing to the limit $\sf{y}$ is in timelike or lightlike position with $\sf{z}\in M$.
Therefore $\sf y\notin \L$ by Lemma~\ref{lem:Omega-Lambda-convex}.\eqref{item:Om-Lambda-invisible-dom}, proving the claim.

We now check that $\sf{y}$ belongs to $\calO_{\L}(\sfC_{\infty})$ (not only $\ov{\calO_{\L}(\sfC_{\infty})}$).
Observe that the group $\PO(p,q+1)$ acts transitively on the set of $2j$-tuples of vectors of $\R^{p,q+1}$ forming, for some $j$-crown $\sfC'$ of $\di\H^{p,q}$, a basis adapted to~$\sfC'$ (Definition~\ref{def:crown-conv-hull}).
Therefore we can find a converging sequence $(g_n)\in\PO(p,q+1)^{\N}$ with limit $g_{\infty}\in \PO(p,q+1)$ such that for any $n\in\N$ we have $g_n\gamma'_n\gamma_n\cdot x_i^{\pm} = x_i^{\pm}$ for all $1\leq i\leq j$ and $g_n \cdot \calO_{\L}(\g_n\cdot\sfC) = \calO_{\L}(\sfC)$, and $g_{\infty}\cdot\sf x_i^{\pm} = x_i^{\pm}$ for all $1\leq i\leq j$ and $g_{\infty} \cdot \calO_{\L}(\sfC_{\infty}) = \calO_{\L}(\sfC)$.
Note that $\calO_{\L}(\gamma'_n\gamma_n\cdot\sfC) = \gamma'_n\gamma_n\cdot\calO_{\L}(\sfC)$ for all~$n$, by $\Gamma$-invariance of~$\L$.
Since the subgroup of the stabilizer of $\calO_{\L}(\sfC)$ in $G_{\sfC}$ fixing $\sfC$ pointwise is equal to~$A_{\sfC}$, for any~$n$ we can write $g_n\gamma'_n\gamma_n = a_n h_n$ where $a_n\in A_{\sfC}$ and where $h_n\in\PO(p,q+1)$ fixes $\spa(\sfC)$ pointwise.
Then for any~$n$ we have $g_n\gamma'_n\gamma_n\cdot F = a_nh_n\cdot F = a_n\cdot F = F$. 
In particular, $g_n\gamma'_n\gamma_n\cdot y_n\in F$ for all $n\in\N$. 
Moreover, $(g_n\gamma'_n\gamma_n\cdot y_n)_{n\in\N}$ converges to $g_{\infty}\cdot\sf y$, which belongs to $\H^{p,q}$ since $\sf y\in \H^{p,q}$.
Therefore $g_{\infty}\cdot\sf y \in F\subset\calO_{\L}(\sfC)$, and so $\sf y\in g^{-1}_{\infty}\cdot\calO_{\L}(\sfC)=\calO_{\L}(\sfC_{\infty})$, as desired.

By Proposition~\ref{p.crown-boundary}, the maximality of~$j$ implies that $\sfC_{\infty}$ is \emph{not} a boundary $j$-crown in~$\L$.
Therefore, $\calO_{\L}(\sfC_{\infty}) \subset \Om(\L)$ by Lemma~\ref{lem:crown-conv-hull-Om-Lambda}.
In particular, we have $\sf y = \lim_n \gamma'_n\gamma_n\cdot\nolinebreak y_n \in \Om(\L)$.

Now suppose that $(\gamma_n\cdot y_n)_{n\in\N}$ converges to some $y_{\infty} \in \H^{p,q}$, and let us check that $y_{\infty}\in\nolinebreak\Om(\L)$.
Let $\mathcal{K}$ be a small compact neighborhood of $\sf{y}$ contained in $\Om(\L)$.
For all large enough~$n$, we have $\gamma'_n\gamma_n\cdot y_n\in\mathcal{K}$, \ie $\gamma_n\cdot y_n\in{\gamma'_n}^{-1}\cdot\mathcal{K}$.
By Proposition~\ref{prop:M-in-prop-conv-Omega}.\eqref{item:M-in-Om-1}, the action of $\Gamma$ on $\Om(\L)$ is properly discontinuous.
This implies that the sequence $(\gamma'_n)_{n\in\N}$ has to be bounded: otherwise some subsequence of ${\gamma'_n}^{-1}\cdot\mathcal{K}$ would converge to a point of $\di\Om(\L)$, contradicting the assumption that $y_{\infty} = \lim_n\,\gamma_n\cdot y_n \in \H^{p,q}$.
Therefore, up to subsequence all $\gamma'_n$ are equal to some given $\gamma'\in\Gamma$, and $y_{\infty} = {\gamma'}^{-1}\cdot\sf y$ belongs to $\Om(\L)$.

This shows that $\de_{\H}\mathcal{Z} \cap \de_{\H}\Om(\L) = \emptyset$, hence concludes the proof.
\end{proof}

\begin{lem} \label{lem:path-metric-exists}
In the setting of Proposition~\ref{p.boundaries dont intersect for hyperbolic}, let $\Om$ be a $\Gamma$-invariant properly convex open subset of $\Om(\L)$ containing~$M$ (Proposition~\ref{prop:M-in-prop-conv-Omega}.\eqref{item:M-in-Om-2}).
Then there exists $R_0\geq 0$ such that for any $R\geq R_0$, the path metric with respect to $d_{\Om}$ exists on the closed uniform $R$-neighborhood $\mathcal{U}_R(M)$ of $M$ in $(\Om,d_{\Om})$.
\end{lem}

\begin{proof}[Proof of Lemma~\ref{lem:path-metric-exists}]
Let $D\subset M$ be a compact fundamental domain for the action of $\Gamma$ on~$M$.
The convex hull $\mathscr{C}(D)$ of $D$ in~$\Om$ is still compact, hence is contained in $\mathcal{U}_{R_0}(M)$ for some $R_0\geq 0$.
Then $\mathscr{C}(D) \subset \mathcal{U}_R(M)$ for all $R\geq R_0$.

Fix $R\geq R_0$, and let us check that the path metric with respect to $d_{\Om}$ exists on $\mathcal{U}_R(M)$.
It is enough to show that any two points $o,o'$ of $\mathcal{U}_R(M)$ can be joined by a finite sequence of closed projective line segments completely contained in $\mathcal{U}_R(M)\subset\Om$, as these are paths of finite length for~$d_{\Om}$.
For this, consider points $m_o,m_{o'}\in M$ such that $d_{\Om}(o,m_o) = \min_{m\in M} d_{\Om}(o,m)$ and $d_{\Om}(o',m_{o'}) = \min_{m'\in M} d_{\Om}(o',m')$.
The projective line segments $[o,m_o]$ and $[o',m_{o'}]$ are contained in $\mathcal{U}_R(M)$.
Because $\mathcal{U}_R(M)$ is preserved by $\G$, we can assume without loss of generality that $m_o\in D$.
By cocompactness we find finitely many translates $\g_i\cdot D$ for $i=1,\ldots,\ell$ such that there is path from $m_o$ to $m_{o'}$ contained in these translates; we order them such that there exists $m_i\in \g_i\cdot D \cap \g_{i+1}\cdot D$ for all~$i$ and $m_{o'}\in \g_{\ell}\cdot D$.
Then each segment $[m_i,m_{i+1}]$ lies inside a $\G$-translate of $\mathscr{C}(D)$, and since $\G$ preserves $\mathcal{U}_R(M)$, each such segment lies in~$\mathcal{U}_R(M)$.~Thus the segments $[o,m_o],[o',m_{o'}]$ and $[m_i,m_{i+1}]$ for $i=1,\ldots,\ell$ constitute a finite sequence of closed projective line segments completely contained in $ \mathcal{U}_R(M) \subset \Om$ connecting $o$ and $o'$ as~desired.
\end{proof}

\begin{proof}[Proof of Proposition~\ref{p.boundaries dont intersect for hyperbolic}]
Suppose there is a $j$-crown in~$\L$ for some $j\geq 2$.
Take $j$ to be maximal in the sense that there does not exist any $(j+1)$-crown in~$\L$.
Let $\sfC$ be a $j$-crown in~$\L$.
By Proposition~\ref{p.crown-boundary} we know that $\sfC$ is \emph{not} a boundary $j$-crown in~$\L$.
Let $\mathcal{O}_{\L}(\sfC)$ be as in Lemma~\ref{lem:crown-conv-hull-Om-Lambda}, and choose an $A_{\sfC}$-orbit $F$ in $\mathcal{O}_{\L}(\sfC)$.
By Lemma~\ref{l.tubular-contains-A-orbit}, there exist $R\geq 0$ and a $\Gamma$-invariant properly convex open subset $\Om$ of $\Om(\L)$ containing~$M$ such that $F$ is contained in the closed uniform $R$-neighborhood $\mathcal{U}_R(M)$ of $M$ in $(\Om,d_{\Om})$.
By Lemma~\ref{lem:path-metric-exists}, up to increasing $R$ we can equip $\mathcal{U}_R(M)$ with the path metric of~$d_{\Om}$.

By Proposition~\ref{prop:A-orbit-not-hyperb}, for any $\delta>0$, there is a geodesic quadrilateral for the Hilbert metric $d_{\calO_{\L}(\sfC)}$, consisting of four projective line segments contained in~$F$, such that a side of the quadrilateral is not contained in the uniform $\delta$-neighborhood of the other three sides in $(\calO_{\L}(\sfC),d_{\calO_{\L}(\sfC)})$.
We claim that a side of the quadrilateral is also not contained in the uniform $\delta$-neighborhood of the other three sides in $(\Om,d_{\Om})$.
Indeed, for this it is sufficient to check that $\calO_{\L}(\sfC) = \Om \cap \P(\spa(\sfC))$; then the restriction of $d_{\Om}$ to $\calO_{\L}(\sfC)$ coincides with $d_{\calO_{\L}(\sfC)}$, as $\calO_{\L}(\sfC)$ is geodesically convex with respect to $d_{\Om}$.
The inclusion $\calO_{\L}(\sfC) \subset \Om \cap\nolinebreak\P(\spa(\sfC))$ follows from the fact that $\calO_{\L}$ is contained in the interior of $\mathscr{C}(\L)\cap\nolinebreak\P(\spa(\sfC))$; the reverse inclusion follows from the equality $\calO_{\L}(\sfC) = \Om(\L) \cap \P(\spa(\sfC))$ in Lemma~\ref{lem:crown-conv-hull-Om-Lambda}.

Note that the same property (a side of the quadrilateral is not contained in the uniform $\delta$-neighborhood of the other three sides) also holds for the path metric, as passing to the path metric is distance non-decreasing (and geodesics for $d_{\Om}$ are also geodesics for the path metric).
Therefore $\mathcal{U}_R(M)$ with the path metric is not Gromov hyperbolic.

The group $\Gamma$ preserves $\mathcal{U}_R(M)$ (which is a geodesic metric space with the path metric of $d_{\Om}$) and acts properly discontinuously and cocompactly on it by isometries.
Since $\mathcal{U}_R(M)$ is not Gromov hyperbolic, the Milnor--\v{S}varc lemma implies that $\Gamma$ is also not Gromov hyperbolic. 
\end{proof}

\subsection{Proof of Theorem~\ref{thm:geom-action-weakly-sp-gr}.} \label{subsec:proof-thm-ConvexCo}

The implication \eqref{item:geom-action-1}~$\Rightarrow$~\eqref{item:geom-action-4} is Proposition~\ref{p.boundaries dont intersect for hyperbolic}.
The implication \eqref{item:geom-action-2}~$\Rightarrow$~\eqref{item:geom-action-4} is immediate from Definition~\ref{def:j-crown} of a $j$-crown.
The implication \eqref{item:geom-action-3}~$\Rightarrow$~\eqref{item:geom-action-2} follows from Definition~\ref{def:Hpq-cc} of $\H^{p,q}$-convex cocompactness and from Lemma~\ref{lem:Omega-Lambda-convex}.\eqref{item:Om-Lambda-in-Hpq}.
The implication \eqref{item:geom-action-3}~$\Rightarrow$~\eqref{item:geom-action-1} is contained in \cite[Th.\,1.24]{dgk-proj-cc} (see Fact~\ref{fact:Hpq-cc-Ano}).

The implication \eqref{item:geom-action-4}~$\Rightarrow$~\eqref{item:geom-action-5} follows from Lemma~\ref{l.crown-existence} and Proposition~\ref{p.crown-boundary}: indeed, if $\de_{\H}\mathscr{C}(\L)$ meets $\partial\Om(\L)$, then it meets $w^{\perp}$ for some $w\in\L$ (see Notation~\ref{not:Omega-Lambda}).

We now check the implication \eqref{item:geom-action-5}~$\Rightarrow$~\eqref{item:geom-action-3}.
Suppose that $\de_{\H} \mathscr{C}(\L)$ does not meet $\partial\Om(\L)$.
By Proposition~\ref{prop:M-in-prop-conv-Omega}.\eqref{item:M-in-Om-3}, there exists a $\G$-invariant properly convex open subset $\Om$ of $\Om(\L)\subset \H^{p,q}$ containing $\mathscr{C}(\L)$, and $\mathscr{C}(\L)$ has compact quotient by~$\Gamma$.
By Lemma~\ref{lem:Omega-Lambda-convex}.\eqref{item:Om-Lambda-in-Hpq} we have $\di\Om = \di\mathscr{C}(\L) = \L$, and so $\di\mathscr{C}(\L) = \L$ contains the full orbital limit set $\Lambda^{\mathsf{orb}}_{\Omega}(\Gamma)$ (see Section~\ref{subsec:Hpq-cc-def}).
In fact, $\L = \di M$ is equal to $\L^{\mathsf{orb}}_{\Omega}(\Gamma)$ because $\Gamma$ acts cocompactly on~$M$.
Therefore $\Gamma$ acts convex cocompactly on~$\Om$, and so \cite[Th.\,1.24]{dgk-proj-cc} implies that $\Gamma$ is $\H^{p,q}$-convex cocompact (see Fact~\ref{fact:Hpq-cc-proj}).

\subsection{Split spacetimes} \label{subsec:split-spacetimes}

The following gives examples of properly discontinuous and cocompact actions of \emph{non-hyperbolic} groups on (weakly) spacelike $p$-graphs in $\hat\H^{p,q}$ and~$\H^{p,q}$.

\begin{prop} \label{prop:split-spacetime}
For $p,q\geq 2$ and $2\leq r\leq\min(p,q)$, let $V_1,\dots,V_r$ be linear subspaces of $\R^{p,q+1}$ such that the restriction of $\sfb$ to any $V_i$ is non-degenerate of signature $(k_i,\ell_i)$ for some $k_i,\ell_i\geq 1$ and $\R^{p,q+1}$ is the $\sfb$-orthogonal direct sum of $V_1,\dots,V_r$.
For any $1\leq i\leq r$, let $\Gamma_i$ be a discrete subgroup of $\OO(\sfb|_{V_i}) \simeq \OO(k_i,\ell_i)$ acting properly discontinuously on a weakly spacelike (\resp a spacelike) $k_i$-graph $\hat M_i$ in $\hat\H_i := \hat\H^{p,q}\cap V_i \simeq \hat\H^{k_i,\ell_i-1}$ with boundary at infinity $\hat\L_i := \di\hat M_i \subset \di\hat\H_i$.
Then
\begin{enumerate}
  \item\label{item:split-1} $\hat M := \{ \sum_{i=1}^r m_i/\sqrt{r} \,|\, m_i\in\hat M_i\}$ is a weakly spacelike (\resp a spacelike) $p$-graph in $\hat\H^{p,q}$;
  \item\label{item:split-2} the discrete subgroup $\Gamma := \Gamma_1 \times \dots \times \Gamma_r$ of $\OO(p,q+1)$ acts properly discontinuously on~$\hat M$; this action is cocompact if and only if the action of $\Gamma_i$ on~$\hat M_i$ is cocompact for all $1\leq i\leq r$;
  \item\label{item:split-3} with the identification $\di\hat\H^{p,q} \simeq \{ v\in\R^{p,q+1}\smallsetminus\{0\} \,|\, \sfb(v,v)=0\}/\R_{>0}$ of \eqref{eqn:di-Hpq-hat-quotient}, the boundary at infinity $\hat\L := \di\hat M$ of $\hat M$ is the image of $\tilde\L := \{ \sum_{i=1}^r v_i \,|\, v_i\in\tilde\L_i\cup\{ 0\}\} \smallsetminus \{0\}$ where $\tilde\L_i$ is the preimage of~$\hat\L_i$ in $\{ v\in\R^{p,q+1}\smallsetminus\{0\} \,|\, \sfb(v,v)=0\}$;
  \item\label{item:split-4} $\hat M \subset \tilde\Om(\tilde\L)$ if and only if $\hat M_i \subset \tilde\Om(\tilde\L_i)$ for all $1\leq i\leq r$.
\end{enumerate}
\end{prop}

\begin{proof}
\eqref{item:split-1} Since each $\hat M_i$ is a non-positive (\resp negative) subset of $\hat\H^{p,q}$ (Definition~\ref{def:nonpos-neg}), so is~$\hat M$.
For any $1\leq i\leq r$, choose a timelike $\ell_i$-plane $T_i$ of $V_i \simeq \R^{k_i,\ell_i}$, so that $T := T_1\oplus\dots\oplus T_r$ is a timelike $(q+1)$-plane of $\R^{p,q+1}$, with orthogonal complement $T^{\perp} = (T_1^{\perp}\cap V_1) \oplus \dots \oplus (T_r^{\perp}\cap V_r)$.
By Proposition \ref{prop:weakly-sp-gr-is-a-graph} (\resp \ref{prop:sp-gr-is-a-graph}), the set $\hat M_i$ is the graph of a $1$-Lipschitz (\resp strictly $1$-Lipschitz) map in the splitting $\hat\H_i \simeq \bB^{k_i}\times\bS^{\ell_i-1}$ associated to~$T_i$ (see Proposition~\ref{p.doublecoverasproduct}).
The fact that $\hat M_i$ is a graph means (see the proof of Lemma~\ref{lem:Hpq-hat-bar-prod}) that for any $v\in T^{\perp}\cap V_i = T_i^{\perp}\cap V_i$, there is a unique $v'\in T_i$ such that some positive multiple of $v+v'$ belongs to~$\hat M_i$.
We deduce that for any $v\in T^{\perp}$, there is a unique $v'\in T$ such that some positive multiple of $v+v'$ belongs to~$\hat M$.
Therefore $\hat M$ is a graph in the splitting $\hat\H^{p,q} \simeq \bB^p \times \bS^q$ associated to~$T$ as in Proposition~\ref{p.doublecoverasproduct}.
Since $\hat M$ is non-positive (\resp negative), it is the graph of a $1$-Lipschitz (\resp strictly $1$-Lipschitz) map by Lemma~\ref{l.nonpossets}.
Therefore $\hat M$ is a weakly spacelike (\resp a spacelike) $p$-graph in $\hat\H^{p,q}$ by Proposition \ref{prop:weakly-sp-gr-is-a-graph} (\resp \ref{prop:sp-gr-is-a-graph}).

\eqref{item:split-2} and \eqref{item:split-3} are clear.

\eqref{item:split-4} The fact that $\hat M \subset \tilde\Om(\tilde\L)$ means (see Notation~\ref{not:Omega-Lambda}) that $\sfb(m,\tilde x)<0$ for all $m\in\hat M$ and all $\tilde x\in\tilde\L$.
By definition of~$\hat M$, by \eqref{item:split-3}, and by Proposition~\ref{prop:bdynonpossphere}.\eqref{item:bound-non-pos-sph-1}, this holds if and only if $\sfb(m_i,\tilde x_i)<0$ for all $m_i\in\hat M_i$, all $\tilde x_i\in\tilde\L_i$, and all~$i$; equivalently, $\hat M_i \subset \tilde\Om(\tilde\L_i)$ for all~$i$.
\end{proof}

\begin{remark} \label{rem:split-spacetime-lift}
Let $V_1,\dots,V_r$ be as in Proposition~\ref{prop:split-spacetime}.
For any $1\leq i\leq r$, let $\Gamma_i$ be a discrete subgroup of $\OO(\sfb|_{V_i}) \simeq \OO(k_i,\ell_i)$ acting properly discontinuously on a weakly spacelike (\resp a spacelike) $k_i$-graph $M_i$ in $\H_i := \H^{p,q}\cap\P(V_i) \simeq \H^{k_i,\ell_i-1}$ with boundary at infinity $\L_i := \di M_i \subset \di\H_i$.
Then each $\Gamma_i$ has a subgroup $\Gamma'_i$ of index $\leq 2$ which preserves a weakly spacelike (\resp a spacelike) $k_i$-graph $\hat M_i$ in $\hat\H_i := \hat\H^{p,q}\cap V_i \simeq \hat\H^{k_i,\ell_i-1}$ lifting~$M_i$.
Proposition~\ref{prop:split-spacetime} applies to the $\Gamma'_i$ and~$\hat M_i$; it gives that the image $M$ of $\{ \sum_{i=1}^r m_i/\sqrt{r} \,|\, m_i\in\hat M_i\}$ in~$\H^{p,q}$ is a weakly spacelike (\resp a spacelike) $p$-graph, on which $\Gamma' := \Gamma'_1 \times \dots \times \Gamma'_r$ acts properly discontinuously; this action is cocompact if and only if the action of $\Gamma_i$ on~$M_i$ is cocompact for all $1\leq i\leq r$.
Moreover, $M \subset \Om(\L)$ (where $\L := \di M$) if and only if $M_i \subset \Om(\L_i)$ for all~$i$.
\end{remark}

Note that by Fact~\ref{fact:vcd}, the action of $\Gamma_i$ on~$M_i$ is cocompact if and only if $\vcd(\Gamma_i)=k_i$.
This is the case for instance if $\Gamma_i$ is a uniform lattice in a copy of $\OO(k_i,1)$ inside $\OO(\sfb|_{V_i}) \simeq \OO(k_i,\ell_i)$, in which case $\Gamma_i$ acts properly discontinuously and cocompactly on a copy $M_i$ of $\H^{k_i}$ in $\H^{p,q}\cap\P(V_i) \simeq \H^{k_i,\ell_i-1}$ (a totally geodesic spacelike $k_i$-manifold).
For $q=1$ and $r=2$, such $\Gamma = \Gamma_1 \times \Gamma_2$ with $\Gamma_i$ a uniform lattice in $\SO(k_i,1) = \SO(k_i,\ell_i)$ were considered in \cite[\S\,4.6]{bar15}, where the corresponding quotients $\Gamma\backslash\Om(\L)$ were called \emph{split AdS spacetimes}.

\begin{remark} \label{rem:split-spacetime}
In the setting of Proposition~\ref{prop:split-spacetime} and Remark~\ref{rem:split-spacetime-lift},
\begin{enumerate}
  \item\label{item:split-spacetime-1} the set $\L = \di M$ is a non-positive $(p-1)$-sphere in $\di\H^{p,q}$ (Proposition~\ref{prop:bdynonpossphere}.\eqref{item:bound-non-pos-sph-1}) which contains \emph{non-transverse} points (\eg a point of $\L_i$ is never transverse to a point of~$\L_j$ for $1\leq i<j\leq r$);
  \item\label{item:split-spacetime-2} if $M \subset \Om(\L)$, then $\L$ is non-degenerate (Lemma~\ref{lem:Omega-Lambda-convex}.\eqref{item:Om-Lambda-non-empty}) and contains $r$-crowns of the form $\sfC = \{x_1^{\pm},\ldots,x_r^{\pm}\}$ where $x_i^+$ and~$x_i^-$ are any two transverse points of~$\L_i$;
  \item\label{item:split-spacetime-3} if at least two groups $\Gamma_i$ are infinite (\eg they act cocompactly on their corresponding $M_i$), then $\Gamma = \Gamma_1\times\dots\times\Gamma_r$ and $\Gamma' = \Gamma'_1\times\dots\times\Gamma'_r$ are \emph{not} Gromov hyperbolic;
  \item\label{item:split-spacetime-4} if one group $\Gamma_i$ is $\H^{k_i,\ell_i-1}$-convex cocompact and the other groups $\Gamma_i$ are finite, then $\Gamma = \Gamma_1\times\dots\times\Gamma_r$ and $\Gamma' = \Gamma'_1\times\dots\times\Gamma'_r$ are Gromov hyperbolic, but the actions of $\Gamma$ on~$\hat M$ and of $\Gamma'$ on~$M$ are \emph{not} cocompact.
\end{enumerate}
Points \eqref{item:split-spacetime-1}--\eqref{item:split-spacetime-2}--\eqref{item:split-spacetime-4} show that if we remove the cocompactness assumption in Theorem~\ref{thm:geom-action-weakly-sp-gr-basic} or Corollary~\ref{cor:hyp-spacelike-compact->cc}, then the Gromov hyperbolicity of~$\Gamma$ does not imply the transversality of~$\L$ anymore, hence does not imply that $\Gamma$ is $\H^{p,q}$-convex cocompact.
\end{remark}

\begin{example}
In the setting of Proposition~\ref{prop:split-spacetime} and Remark~\ref{rem:split-spacetime-lift}, suppose there exists $1\leq j\leq r$ such that $k_i=\ell_i=1$ for all $1\leq i\leq j$.
For any $1\leq i\leq j$ we can take for $M_i$ a spacelike geodesic of $\H_i$ with endpoints $x_i^+$ and~$x_i^-$, so that $\L_i := \di M_i = \{x_i^+,x_i^-\}$, and for $\Gamma_i \simeq \Z$ a discrete subgroup of $\OO(\sfb|_{V_i}) \simeq \OO(1,1)$ generated by a proximal element with attracting fixed point $x_i^+$ and repelling fixed point~$x_i^-$.
Then the set $\sfC := \L_1 \cup \dots \cup \L_j$ is a $j$-crown (Definition~\ref{def:j-crown}) in the non-positive $(p-1)$-sphere $\L = \di M$, and the subgroup $\Gamma_1\times\dots\times\Gamma_j \simeq \Z^j$ of~$\Gamma$ is a lattice in the group $A_{\sfC} \simeq \R^j$ of Lemma~\ref{lem:F-spacelike}.
If $j=r=p$, then $M$ is an $A_{\sfC}$-orbit of an open convex hull $\mathcal{O}$ of $\sfC$ as in Lemma~\ref{lem:F-spacelike}; in particular, $M$ is a complete spacelike embedded $p$-submanifold of~$\H^{p,q}$ and $\L = \di\mathcal{O}$.
For general $2\leq j\leq r$, if each $M_i$ for $j+1\leq i\leq r$ is a weakly spacelike (\resp a spacelike) smooth embedded submanifold of $\H_i$, then $M$ is a weakly spacelike (\resp a spacelike) smooth embedded $p$-submanifold of~$\H^{p,q}$.
Note that the sectional curvature on~$M$ has vanishing directions.
\end{example}

\section{Non-degeneracy for reductive limits} \label{sec:non-deg-limit}

The goal of this section is to prove Proposition~\ref{prop:non-deg-limit}.

In the case $q=1$, one can give the following short argument, which implies Proposition~\ref{prop:non-deg-limit} using Fact~\ref{fact:fd-closed}.

\begin{lem} \label{lem:Lambda-non-deg-q=1}
For $p\geq 2$, let $\Gamma$ be a finitely generated group with no infinite nilpotent normal subgroups, such that $\vcd(\Gamma)=p$.
Let $\rho : \Gamma\to\PO(p,2)$ be a representation with finite kernel and discrete image, such that the Zariski closure of $\rho(\Gamma)$ in $\PO(p,2)$ is reductive, and let $M$ be a $\rho(\Gamma)$-invariant weakly spacelike $p$-graph in~$\H^{p,1}$.
Then $\L := \di M$ is non-degenerate.
\end{lem}

\begin{proof}
By Proposition~\ref{prop:bdynonpossphere}.\eqref{item:bound-non-pos-sph-1}, the set $\L = \di M$ is a non-positive $(p-1)$-sphere in $\di\H^{p,1}$.
Let $0\leq k\leq 2$ be the dimension of the kernel $V$ of $\sfb|_{\spa(\L)}$.
We note that $k\leq 1$: indeed, this follows from Lemma~\ref{l.kernonpossphere} if $p>2=q+1$, and from Proposition~\ref{prop:deg-weakly-sp-foliation} if $p=2=q+1$.

Suppose by contradiction that $k=1$.
Then $\rho(\Gamma)$ is contained in the stabilizer of an isotropic line of~$\R^{p,2}$, which is a proper parabolic subgroup of $\PO(p,2)$.
Since the Zariski closure $G$ of $\rho(\Gamma)$ in $\PO(p,2)$ is reductive, it is actually contained in a Levi subgroup of this parabolic subgroup, namely a Lie subgroup $L$ isomorphic to $(\GL(1,\R) \times \OO(p-1,1))/\{\pm\mathrm{Id}\}$.
The group $G$ is the set of real points of a connected reductive real algebraic group~$\sfG$.
Let $G^{ss}$ be the set of real points of the commutator subgroup of~$\sfG$, and let $\pi : G\to G^{ss}$ be the natural projection.
By Fact~\ref{fact:proj-ss-discrete}, the group $\pi(\Gamma)$ is discrete and Zariski-dense in~$G^{ss}$.
By assumption $\Gamma$ has no infinite nilpotent normal subgroups, hence the restriction of $\pi$ to~$\Gamma$ has finite kernel.
Therefore the action of $\Gamma$ via~$\pi$ on the Riemannian symmetric space of~$G^{ss}$ is properly discontinuous.
But this Riemannian symmetric space has dimension $<p$.
Indeed, $G^{ss}$ is a semi-simple subgroup of $L \simeq (\GL(1,\R) \times \OO(p-1,1))/\{\pm\mathrm{Id}\}$, hence its intersection with the central subgroup $\GL(1,\R)$ of~$L$ is trivial (or of cardinality~$2$), and so $G^{ss}$ is isomorphic to (or a double covering of) its projection to $\PO(p-1,1)$; therefore the Riemannian symmetric space of~$G^{ss}$ embeds into that of $\PO(p-1,1)$, namely $\H^{p-1}$.
Since $\vcd(\Gamma)=p$, we obtain a contradiction with Fact~\ref{fact:vcd}.

This shows that $k=0$, \ie $\L$ is non-degenerate.
\end{proof}

When $q>1$, proving the non-degeneracy of~$\L$ (Proposition~\ref{prop:non-deg-limit}) is more difficult, as the Riemannian symmetric space of $\GL(k,\R) \times \OO(p-k,q+1-k)$ has dimension $k(k+1)/2 + (p-\nolinebreak k)(q+1-k)$ which is then strictly larger than $p = \vcd(\Gamma)$ (the case $k=q+1<p$ is excluded by Lemma~\ref{l.kernonpossphere}).
Our proof of Proposition~\ref{prop:non-deg-limit} then goes by establishing a more complicated general bound on cohomological dimension (Proposition~\ref{prop:semiprox-vcd}), involving positively semi-proximal representations in the sense of \cite{ben05}.
We next show that we can control this bound in our situation to prove that $V=\{0\}$: see Section~\ref{subsec:Lambda-nondegen}.

\subsection{Reminders on semi-proximality} \label{subsec:semi-proximality}

Let $V$ be a finite-dimensional real vector space.
An element $g\in\GL(V)$ is called \emph{proximal} if it has a unique complex eigenvalue of maximal modulus (this eigenvalue is then necessarily real), and this eigenvalue has multiplicity~$1$; equivalently, $g$ has a unique attracting fixed point in the projective space $\P(V)$.
The element $g\in\GL(V)$ is called \emph{positively proximal} if it is proximal and its eigenvalue of maximal modulus is positive.
The element $g\in\GL(V)$ is called \emph{semi-proximal} (\resp \emph{positively semi-proximal}) if it admits a real (\resp positive) eigenvalue which has maximal modulus among all complex eigenvalues of~$g$ (without assuming anything on multiplicity nor uniqueness of the eigenvalues of maximal modulus).

\begin{remark} \label{rem:prox-semi-prox}
Proximality implies semi-proximality, and positive proximality is equivalent to the conjunction of proximality and positive semi-proximality.
\end{remark}

We now fix a connected reductive real algebraic group $\sfG$, and denote by $G$ its set of real points.
By a \emph{linear representation} $(\tau,V)$ of~$G$ we mean the restriction to the real points of a homomorphism of real algebraic groups from $\sfG$ to $\mathsf{GL(V)}$.
Fix such a linear representation $(\tau,V)$ of~$G$.

As in Section~\ref{subsec:limit-cone}, let $\mathfrak{a}$ be a Cartan subspace of the Lie algebra of~$G$, let $\mathfrak{a}^+$ be a closed Weyl chamber in~$\mathfrak{a}$, and let $\lambda : G\to\mathfrak{a}^+$ be the Jordan projection.
An element $g\in G$ is called \emph{loxodromic} (or \emph{$\R$-regular}) if $\lambda(g)$ belongs to the interior $\mathrm{Int}(\mathfrak{a}^+)$ of~$\mathfrak{a}^+$.

\begin{defn}
Let $\Gamma$ be a subsemigroup of~$G$.
The linear representation $(\tau,V)$ of~$G$ is \emph{$\Gamma$-proximal} (\resp \emph{positively $\Gamma$-proximal}, \resp \emph{$\Gamma$-semi-proximal}, \resp \emph{positively $\Gamma$-semi-proximal}) if for any element $\gamma\in\Gamma$ which is loxodromic in~$G$, the image $\tau(\gamma)\in\GL(V)$ is proximal (\resp positively proximal, \resp semi-proximal, \resp positively semi-proximal).
\end{defn}

Let $T$ be a maximal compact torus in the centralizer of $A := \exp(\mathfrak{a})$ in~$G$.

\begin{remark} \label{rem:char-pos-semiprox}
Let $g$ be a loxodromic element of~$G$.
Then $g$ is conjugate in~$G$ to an element of the form $ta$ where $t\in T$ and $a\in\exp(\mathrm{Int}(\mathfrak{a}^+))$.
The element $\tau(g)$ is positively semi-proximal if and only if $\tau(t)$ fixes some non-zero vector in the highest eigenspace $V^+$ of $\tau(a)$ in~$V$, or in other words $\det(\tau(t)|_{V^+} - \mathrm{Id}_{V^+})=0$.
\end{remark}

\begin{remark} \label{rem:pos-semiprox-remove-doubles}
Let $\Gamma$ be a subsemigroup of~$G$ and let $\tau_1,\dots,\tau_{\ell}$ be linear representations of~$G$.
\begin{enumerate}
  \item\label{item:semiprox-doubles-1} Suppose $\tau_1$ and $\tau_2$ are isomorphic representations.
  Then $\tau_1\oplus\dots\oplus\tau_{\ell}$ is positively $\Gamma$-semi-proximal if and only if $\tau_2\oplus\dots\oplus\tau_{\ell}$ is.
  \item\label{item:semiprox-doubles-2} If $\tau_1,\dots,\tau_{\ell}$ are all positively $\Gamma$-semi-proximal, then so is $\tau_1\oplus\dots\oplus\tau_{\ell}$.
\end{enumerate}
\end{remark}

We shall also use the following fact; recall that the linear representation $(\tau,V)$ of~$G$ is called \emph{irreducible} if $V$ does not admit any non-trivial $\tau(G)$-invariant linear subspace.

\begin{fact}[{\cite[Cor.\,5.1]{ben05}}] \label{fact:semiprox-Gamma-G}
Let $\Gamma$ be a Zariski-dense subgroup of~$G$.
If the linear representation $(\tau,V)$ of~$G$ is irreducible and $\Gamma$-semi-proximal, then it is in fact $G$-semi-proximal.
\end{fact}

The statement is not true in general if one changes ``semi-proximal'' into ``positively semi-proximal''.

\begin{lem} \label{lem:semiprox-Gamma'}
Let $\Gamma$ be a Zariski-dense subgroup of~$G$.
Let $G'$ be an open subsemigroup of~$G$ meeting~$\Gamma$, and $\Gamma' := \Gamma\cap G'$.
If the linear representation $(\tau,V)$ of~$G$ is positively $\Gamma'$-semi-proximal, then it is also positively $\Gamma$-semi-proximal.
\end{lem}

\begin{proof}
By contraposition, suppose $(\tau,V)$ is not $\Gamma$-semi-proximal: there exists an element $\gamma\in\Gamma$ which is loxodromic in~$G$ and such that $\tau(\gamma)$ is not positively semi-proximal.
By Remark~\ref{rem:char-pos-semiprox}, the element $\gamma$ is conjugate in~$G$ to an element of the form $ta$ where $t\in T$ and $a\in\exp(\mathrm{Int}(\mathfrak{a}^+))$, such that $\det(\tau(t)|_{V^+} - \mathrm{Id}_{V^+})\neq 0$ where $V^+$ is the highest eigenspace of $\tau(a)$ in~$V$.
By \cite[Cor.\,8.5]{ben05}, we can find elements $\gamma'\in\Gamma'$ that are proximal in~$G$ and conjugate to elements of the form $t'a'$ with $a'\in\exp(\mathrm{Int}(\mathfrak{a}^+))$ and $t'\in T$ arbitrarily close to~$t$.
In particular, we still have $\det(\tau(t')|_{V^+} - \mathrm{Id}_{V^+})\neq 0$, and so $\tau(\gamma')$ is not positively semi-proximal, which shows that $\tau$ is not positively $\Gamma'$-semi-proximal.
\end{proof}

\begin{cor} \label{cor:pos-semiprox-domin}
Let $\Gamma$ be a Zariski-dense subgroup of~$G$.
For $1\leq i\leq\ell$, let $(\tau_i,V_i)$ be an irreducible linear representation of~$G$ with highest weight $\chi_i\in\mathfrak{a}^*$.
Suppose that the linear representation $\tau_1\oplus\dots\oplus\tau_{\ell}$ of~$G$ is positively $\Gamma$-semi-proximal.
\begin{enumerate}
  \item\label{item:pos-semiprox-domin-1} (\cite[Lem.\,6.3]{ben05}) If there exists an element $X$ of the limit cone $\mathcal{L}_{\Gamma}$ such that $\langle\chi_1,X\rangle > \langle\chi_i,X\rangle$ for all $2\leq i\leq\ell$, then $\tau_1$ is positively $\Gamma$-semi-proximal.
  \item\label{item:pos-semiprox-domin-2} If there exists an element $X$ of the limit cone $\mathcal{L}_{\Gamma}$ such that $\langle\chi_1,X\rangle > \langle\chi_2,X\rangle$, then $\tau_1\oplus\tau_3\oplus\dots\oplus\tau_{\ell}$ is positively $\Gamma$-semi-proximal.
\end{enumerate}
\end{cor}

We refer \eg to \cite[\S\,2.1--2.2]{ben05} for reminders on highest (restricted) weights of irreducible linear representations of~$G$.

\begin{proof}
\eqref{item:pos-semiprox-domin-1} Suppose there exists $X \in \mathcal{L}_{\Gamma}$ such that $\langle\chi_1,X\rangle > \langle\chi_i,X\rangle$ for all $2\leq i\leq\ell$.~Then
$$\omega := \{ X\in\mathfrak{a} ~|~ \langle\chi_1,X\rangle > \max_{2\leq i\leq\ell} \langle\chi_i,X\rangle\}$$
is an open convex cone of~$\mathfrak{a}$ that meets $\mathcal{L}_{\Gamma}$.
By Fact~\ref{fact:subsemigroup-in-cone}, there is an open semigroup $G'$ of~$G$ meeting~$\Gamma$ and whose limit cone is contained in~$\omega$.
The representation $\tau_1$ is positively $\Gamma'$-semi-proximal, where $\Gamma' := \Gamma\cap G'$, because $\tau_1\oplus\dots\oplus\tau_{\ell}$ is positively $\Gamma$-semi-proximal.
Therefore $\tau_1$ positively $\Gamma$-semi-proximal by Lemma~\ref{lem:semiprox-Gamma'}.

\eqref{item:pos-semiprox-domin-2} Suppose there exists $X \in \mathcal{L}_{\Gamma}$ such that $\langle\chi_1,X\rangle > \langle\chi_2,X\rangle$.
Then
$$\omega := \{ X\in\mathfrak{a} ~|~ \langle\chi_1,X\rangle > \langle\chi_2,X\rangle\}$$
is an open convex cone of~$\mathfrak{a}$ that meets $\mathcal{L}_{\Gamma}$.
By Fact~\ref{fact:subsemigroup-in-cone}, there is an open subsemigroup $G'$ of~$G$ meeting~$\Gamma$ whose limit cone is contained in~$\omega$.
The representation $\tau_1\oplus\tau_3\oplus\dots\oplus\tau_{\ell}$ is positively $\Gamma'$-semi-proximal, where $\Gamma' := \Gamma\cap G'$, because $\tau_1\oplus\dots\oplus\tau_{\ell}$ is positively $\Gamma$-semi-proximal.
Therefore $\tau_1\oplus\tau_3\oplus\dots\oplus\tau_{\ell}$ is positively $\Gamma$-semi-proximal by Lemma~\ref{lem:semiprox-Gamma'}.
\end{proof}

The following is well known, and follows \eg from \cite[Lem.\,5.5]{ben05} by a straightforward induction on the number of factors.

\begin{fact}[{see \eg \cite[Lem.\,5.5]{ben05}}] \label{fact:irred-rep-prod}
Suppose $\sfG = \sfG_1 \times \dots \times \sfG_r$ is a direct product of finitely many connected reductive real algebraic groups $\sfG_s$, with real points~$G_s$.
Suppose the linear representation $(\tau,V)$ of~$G$ is irreducible.
Then there exist $\mathbb{K}=\R$ or~$\C$ and, for each $1\leq s\leq r$, an irreducible real linear representation $(\sigma_s,W_s)$ of~$G_s$ such that
$$(\tau, V) = (\sigma_1 \otimes_{\mathbb{K}} \dots \otimes_{\mathbb{K}} \sigma_r, W_1 \otimes_{\mathbb{K}} \dots \otimes_{\mathbb{K}} W_r),$$
where $\mathbb{K}=\R$ if some non-trivial $\sigma_s$ has an irreducible complexification, and $\mathbb{K}=\C$ if each $W_s$ has an invariant complex structure for $1\leq s\leq r$.
Moreover, $\tau$ is $G$-proximal (\resp $G$-semi-proximal) if and only if each non-trivial $\sigma_s$ is $G_s$-proximal (\resp $G_s$-semi-proximal) for $1\leq s\leq r$.
\end{fact}

The following is readily obtained from \cite[Lem.\,5.4 \& 5.7]{ben05}, as we explain just below.

\begin{fact}[{\cite[Lem.\,5.4 \& 5.7]{ben05}}] \label{fact:semiprox-nonprox-nonmin}
Suppose the algebraic group $\sfG$ is connected, simply connected, semi-simple.
Let $(\sigma,W)$ be a non-trivial irreducible linear representation of~$G$.
Suppose that $\sigma$ is $G$-semi-proximal but not $G$-proximal.
Then there exists an irreducible linear representation $(\varsigma,\mathcal{W})$ of~$G$ which is $G$-proximal, with $\dim(\mathcal{W})<\dim(W)$, and satisfies the following property: for any connected reductive real algebraic group with real points~$G'$, any irreducible linear representation $(\sigma',W')$ of $G'$, and any Zariski-dense subgroup $\Gamma$ of $G\times G'$, if $\sigma\otimes\sigma' : G\times G' \to W\otimes W'$ is positively $\Gamma$-semi-proximal, then so is $\varsigma\otimes\sigma' : G\times G' \to \mathcal{W}\otimes W'$.
\end{fact}

\begin{proof}
By \cite[Lem.\,5.7]{ben05}, there exists an irreducible linear representation $(\varsigma,\mathcal{W})$ of~$G$ which is $G$-proximal, with $\dim(\mathcal{W})<\dim(W)$, and such that $(\varsigma,\mathcal{W})$ has the same \emph{sign} (in the sense of \cite[Def.\,5.3]{ben05}) as $(\sigma,W)$.
By \cite[Lem.\,5.4.(a)]{ben05}, this means that the highest weights $\chi_{\sigma}$ of~$\sigma$ and $\chi_{\varsigma}$ of~$\varsigma$ differ by an element of $2P$, where $P \subset \mathfrak{a}^*$ is the (restricted) weight lattice of~$G$.

Consider any connected reductive real algebraic group with real points~$G'$.
Choose a Cartan subspace $\mathfrak{a}'$ of the Lie algebra of~$G'$, and a closed Weyl chamber ${\mathfrak{a}'}^+$ of~$\mathfrak{a'}$.
For any irreducible linear representation $(\sigma',W')$ of $G'$, with highest weight $\chi_{\sigma'} \in {\mathfrak{a}'}^*$, the representation $\sigma\otimes\sigma'$ of $G\times G'$ is irreducible with highest weight $\chi_{\sigma} + \chi_{\sigma'} \in (\mathfrak{a} + \mathfrak{a}')^*$, where we see $\chi_{\sigma}$ (\resp $\chi_{\sigma'}$) as an element of $(\mathfrak{a} + \mathfrak{a}')^*$ vanishing on $\mathfrak{a'}$ (\resp $\mathfrak{a}$).
Similarly, the representation $\varsigma\otimes\sigma'$ of $G\times G'$ is irreducible with highest weight $\chi_{\varsigma} + \chi_{\sigma'} \in (\mathfrak{a} + \mathfrak{a}')^*$.
The fact that $\chi_{\sigma}$ and $\chi_{\varsigma}$ differ by an element of $2P$ implies that $\chi_{\sigma} + \chi_{\sigma'}$ and $\chi_{\varsigma} + \chi_{\sigma'}$ differ by an element of $2P''$, where $P'' \subset (\mathfrak{a} + \mathfrak{a}')^*$ is the (restricted) weight lattice of $G\times G'$.
By \cite[Lem.\,5.4.(c)]{ben05}, this implies that for any Zariski-dense subgroup $\Gamma$ of $G\times G'$, the representation $\sigma\otimes\sigma'$ is positively $\Gamma$-semi-proximal if and only if $\varsigma\otimes\sigma'$ is.
\end{proof}

We note that $G'$ and $(\sigma',W')$ are allowed to be trivial in this property, yielding (using Remark~\ref{rem:prox-semi-prox}) that if $\sigma$ is positively $\Gamma$-semi-proximal, then $\varsigma$ is positively $\Gamma$-proximal.

\begin{fact}[{\cite[Fait\,2.9 \& Lem.\,3.2]{ben05}}] \label{fact:pos-semiprox-preserve-conv}
Let $\Gamma$ be a subgroup of~$G$.
\begin{enumerate}
  \item\label{item:semiprox-pres-conv-1} Suppose $\tau$ is irreducible and $\Gamma$ is Zariski-dense in~$G$.
  Then $\tau$ is positively $\Gamma$-proximal if and only if $\tau(\Gamma)$ preserves a non-empty properly convex open cone in~$V$.
  \item\label{item:semiprox-pres-conv-2} In general, if $\tau(\Gamma)$ preserves a non-empty properly convex open cone in~$V$, then any element of $\tau(\Gamma)$ is positively semi-proximal.
\end{enumerate}
\end{fact}

\subsection{A preliminary bound on cohomological dimension} \label{subsec:vcd-bound}

For any non-trivial linear representation $(\tau,V)$ of a reductive real Lie group~$G$, we set
\begin{equation} \label{eqn:delta-tau}
\delta(\tau) := \dim(V) - \ell,
\end{equation}
where $\tau$ splits into irreducible factors as $\tau = \tau_1 \oplus \ldots \oplus \tau_{\ell}$.

\begin{prop} \label{prop:semiprox-vcd}
Let $\sfG = \sfG_1\times\dots\times\sfG_r$ be a real algebraic group which is the direct product of finitely many factors $\sfG_s$, each of which is either the multiplicative group or a connected, simply connected, simple algebraic group.
Let $\mathsf{Z(G)}$ (\resp $\sfG^{ss}$) be the center (\resp the commutator subgroup) of~$\sfG$, \ie the product of those factors $\sfG_s$ which are the multiplicative group (\resp a simple group), so that $\sfG = \mathsf{Z(G)} \times \sfG^{ss}$.
Let $G$, $Z(G)$, and~$G^{ss}$ be the real points of $\sfG$, $\mathsf{Z(G)}$, and~$\sfG^{ss}$, respectively, and $\varpi : G\to Z(G)$ and $\pi : G\to G^{ss}$ the natural projections.
Let $\Gamma$ be a Zariski-dense discrete subgroup of~$G$, such that $\varpi(\Gamma)$ lies in the identity component of $Z(G)$ for the real topology, and the restriction of $\pi$ to~$\Gamma$ has finite kernel.
Let $(\upsilon,F)$ be a linear representation of~$G$ which either is the trivial representation, or satisfies that $\upsilon(\Gamma)$ preserves a properly convex open cone $\tilde{\Om}_F$ of~$F$ and, inside its projectivization $\Om_F \subset \P(F)$, a non-empty subset $M_F$ which is homeomorphic to a contractible topological manifold without boundary.
Let $(\tau,V)$ be a (possibly trivial) linear representation of~$G$ with the following two properties:
\begin{enumerate}[label=(\roman*),ref=\roman*]
  \item\label{item:semiprox-hyp-i} $\tau\oplus\upsilon$ is positively $\Gamma$-semi-proximal and non-trivial;
  \item\label{item:semiprox-hyp-ii} the Lie algebra of the kernel of $\tau\oplus\upsilon$ (or equivalently, of the kernel of $(\tau\oplus\upsilon)|_{G^{ss}}$) meets the limit cone $\mathcal{L}_{\pi(\Gamma)}$ of $\pi(\Gamma)$ (Definition~\ref{def:limit-cone}) only in~$\{0\}$.
\end{enumerate}
Then $\vcd(\Gamma) \leq \delta(\tau) + \dim(M_F)$, where we set $\delta(\tau) := 0$ if $\tau$ is trivial, and $\dim(M_F) := 0$ if $\upsilon$ is trivial.
\end{prop}

(In condition~\eqref{item:semiprox-hyp-ii} we implicitly choose a positive Weyl chamber $(\mathfrak{a}^{ss})^+$ in a Cartan subspace $\mathfrak{a}^{ss}$ of the Lie algebra of~$G^{ss}$; the condition does not depend on this choice.)

\begin{proof}
It is sufficient to prove the inequality $\vcd(\Gamma) \leq \delta(\tau) + \dim(M_F)$ for a linear representation $(\tau,V)$ of~$G$ which satisfies \eqref{item:semiprox-hyp-i} and \eqref{item:semiprox-hyp-ii} and for which $\delta(\tau)$ is minimal among all linear representations satisfying \eqref{item:semiprox-hyp-i} and \eqref{item:semiprox-hyp-ii}.
Let us fix such a representation $(\tau,V)$.
If $\tau$ is non-trivial, then we write $\tau = \bigoplus_{i\in I} \tau_i$ as a sum of non-trivial irreducible representations of~$G$, where $I$ is some non-empty finite set; otherwise, we set $I := \emptyset$.
Similarly, if $\upsilon$ is non-trivial, then we write $\upsilon = \bigoplus_{j\in J} \upsilon_j$ as a sum of non-trivial irreducible representations of~$G$, where $J$ is some non-empty finite set; otherwise, we set $J := \emptyset$.
Let $\mathfrak{a}$ be a Cartan subspace of the Lie algebra of~$G$, with $\mathfrak{a} = (\mathfrak{a}\cap\mathfrak{z}(\mathfrak{g})) \oplus \mathfrak{a}^{ss}$ where $\mathfrak{z}(\mathfrak{g})$ is the Lie algebra of $Z(G)$.
For $a$ in $I$ (\resp $J$), let $\chi_a\in\mathfrak{a}^*$ be the highest weight of $\tau_a$ (\resp $\upsilon_a$).

We start with the following three observations:
\begin{enumerate}[label=(\alph*),ref=\alph*]
  \item\label{item:compose-with-pi-1} a linear representation of~$G$ is irreducible if and only if its precomposition with~$\pi$~is;
  \item\label{item:compose-with-pi-2} an element $g\in G$ is loxodromic in~$G$ if and only if $\pi(g)$ is loxodromic in~$G^{ss}$;
  \item\label{item:compose-with-pi-3} since $\varpi(\Gamma)$ lies in the identity component of $Z(G)$ for the real topology, a linear repre\-sentation of~$G$ is positively $\Gamma$-semi-proximal if and only if its precomposition with~$\pi$~is.
\end{enumerate}
In particular, the restriction $(\tau\oplus\upsilon)|_{G^{ss}}$ of $\tau\oplus\upsilon$ to~$G^{ss}$ is positively $\pi(\Gamma)$-semi-proximal.

\medskip

\noindent
$\bullet$ \textbf{Step 1: For any $i_1\neq i_2$ in~$I$, the linear representation $\tau_{i_1}|_{G^{ss}}$ of~$G^{ss}$ is not isomorphic to $\tau_{i_2}|_{G^{ss}}$ nor to its dual $\tau_{i_2}^*|_{G^{ss}}$.}
Suppose by contradiction that there exist $i_1\neq\nolinebreak i_2$ in~$I$ with $\tau_{i_1}|_{G^{ss}} \simeq \tau_{i_2}|_{G^{ss}}$ or $\tau_{i_1}|_{G^{ss}} \simeq \tau_{i_2}^*|_{G^{ss}}$.
Consider the non-trivial subrepresentation $\tau' := \bigoplus_{i\in I\smallsetminus\{i_2\}} \tau_i$ of~$\tau$.
We have $\delta(\tau') < \delta(\tau)$.
The kernel of $(\tau'\oplus\upsilon)|_{G^{ss}}$ is equal to the kernel of $(\tau\oplus\upsilon)|_{G^{ss}}$, hence $\tau'\oplus\upsilon$ still satisfies~\eqref{item:semiprox-hyp-ii}.

If $\tau_{i_1}|_{G^{ss}} \simeq \tau_{i_2}|_{G^{ss}}$, then $(\tau'\oplus\upsilon)|_{G^{ss}}$ is still positively $\pi(\Gamma)$-semi-proximal by Remark~\ref{rem:pos-semiprox-remove-doubles}.\eqref{item:semiprox-doubles-1}, hence $\tau'\oplus\upsilon$ is still positively $\Gamma$-semi-proximal by Observations \eqref{item:compose-with-pi-2} and~\eqref{item:compose-with-pi-3} above: contradiction with the minimality of $\delta(\tau)$.

So we now assume $\tau_{i_1}|_{G^{ss}} \simeq \tau_{i_2}^*|_{G^{ss}}$ and $\tau_{i_1}|_{G^{ss}} \not\simeq \tau_{i_2}|_{G^{ss}}$.
In particular, the highest weight $\chi_{i_1}|_{\mathfrak{a}^{ss}}$ of $\tau_{i_1}|_{G^{ss}}$ is different from the highest weight $\chi_{i_2}|_{\mathfrak{a}^{ss}}$ of $\tau_{i_2}|_{G^{ss}}$, and these two weights are images of each other by the opposition involution.
Since $\pi(\Gamma)$ is Zariski-dense in~$G^{ss}$ (see Fact~\ref{fact:proj-ss-discrete}), its limit cone $\mathcal{L}_{\pi(\Gamma)} \subset \mathfrak{a}^{ss}$ has non-empty interior in~$\mathfrak{a}^{ss}$ (Fact~\ref{fact:limit-cone}), and so it is \emph{not} contained in the hyperplane $\{ \langle\chi_{i_1},\cdot\rangle = \langle\chi_{i_2},\cdot\rangle\}$ of~$\mathfrak{a}^{ss}$.
Since $\mathcal{L}_{\pi(\Gamma)}$ is invariant under the opposition involution, it meets the open cone $\omega := \{ X\in\mathfrak{a}^{ss} ~|~ \langle\chi_{i_1},X\rangle > \langle\chi_{i_2},X\rangle\}$ of~$\mathfrak{a}^{ss}$.
By Corollary~\ref{cor:pos-semiprox-domin}.\eqref{item:pos-semiprox-domin-2}, the representation $(\tau'\oplus\upsilon)|_{G^{ss}}$ is still positively $\pi(\Gamma)$-semi-proximal, hence $\tau'\oplus\upsilon$ is positively $\Gamma$-semi-proximal by Observations \eqref{item:compose-with-pi-2} and~\eqref{item:compose-with-pi-3} above: contradiction again with the minimality of $\delta(\tau)$.
\medskip

\noindent
$\bullet$ \textbf{Step 2: For any $i\in I$, the representation $\tau_i$ of~$G$ is positively $\Gamma$-semi-proximal.}
Suppose $I\neq\emptyset$.
Let $I'$ be the set of elements $i\in I$ for which there exists $X\in\mathcal{L}_{\pi(\Gamma)}$ with $\langle\chi_i,X\rangle > \max_{a\in (I\cup J)\smallsetminus\{i\}} \langle\chi_a,X\rangle$.
For any $i\in I'$, the representation $\tau_i|_{G^{ss}}$ is positively $\pi(\Gamma)$-semi-proximal by Corollary~\ref{cor:pos-semiprox-domin}.\eqref{item:pos-semiprox-domin-1}, hence $\tau_i$ is positively $\Gamma$-semi-proximal by Observations \eqref{item:compose-with-pi-2} and~\eqref{item:compose-with-pi-3} above.

Consider the subrepresentation $\tau' := \oplus_{i\in I'}\,\tau_i$ of~$\tau$.
We claim that the representation $\tau'\oplus\upsilon$ of~$G$ is still non-trivial (even though $\tau'$ may be trivial).
Indeed, this is clear if $\upsilon$ is non-trivial, so let us assume that $\upsilon$ is trivial (\ie $J=\emptyset$) and check that $I'\neq\emptyset$.
By Step~1, the highest weights $\chi_i$, $i\in I$, are all distinct on~$\mathfrak{a}^{ss}$.
Since $\pi(\Gamma)$ is Zariski-dense in~$G^{ss}$, its limit cone $\mathcal{L}_{\pi(\Gamma)}$ has non-empty interior (Fact~\ref{fact:limit-cone}), and so $\mathcal{L}_{\pi(\Gamma)}$ is \emph{not} contained in the union of the hyperplanes $\{ \langle\chi_{i_1},\cdot\rangle = \langle\chi_{i_2},\cdot\rangle\}$ for $i_1\neq i_2$ in~$I$.
This implies that $I'\neq\emptyset$ if $\upsilon$ is trivial.

We claim that $I' = I$ (independently of whether $\upsilon$ is trivial or not).
Indeed, we have just seen that the representation $\tau'\oplus\upsilon$ of~$G$ is non-trivial.
Recall that $(\tau\oplus\upsilon)|_{G^{ss}}$ is positively $\pi(\Gamma)$-semi-proximal.
By Corollary~\ref{cor:pos-semiprox-domin}.\eqref{item:pos-semiprox-domin-2}, the representation $(\tau'\oplus\upsilon)|_{G^{ss}}$ is still positively $\pi(\Gamma)$-semi-proximal, hence $\tau'\oplus\upsilon$ is positively $\Gamma$-semi-proximal by Observations \eqref{item:compose-with-pi-2} and~\eqref{item:compose-with-pi-3} above.
By construction, for any $X\in\mathcal{L}_{\pi(\Gamma)}$ we have $\max_{a\in I\cup J} \langle\chi_a,X\rangle = \max_{a\in I'\cup J} \langle\chi_i,X\rangle$.
Therefore, if $X\in\mathcal{L}_{\pi(\Gamma)}$ satisfies $\langle\chi_a,X\rangle = 0$ for all $a\in I'\cup J$, then it satisfies $\langle\chi_a,X\rangle = 0$ for all $a\in I\cup J$.
This implies that $\tau'\oplus\upsilon$ still satisfies~\eqref{item:semiprox-hyp-ii}.
By minimality of $\delta(\tau)$, we must have $\tau = \tau'$, \ie $I' = I$.

\medskip

\noindent
$\bullet$ \textbf{Step 3: For any $i\in I$, the representation $\tau_i$ is positively $\Gamma$-proximal.}
Suppose $I\neq\emptyset$ and consider $i\in I$.
Let $V_i$ be the representation space of~$\tau_i$.
By Step~2, we know that the non-trivial irreducible representation $(\tau_i,V_i)$ of~$G$ is $\Gamma$-semi-proximal; let us check that it is actually $\Gamma$-proximal.
By Fact~\ref{fact:irred-rep-prod}, there exist $\mathbb{K}=\R$ or~$\C$ and irreducible real linear representations $(\sigma_s,W_s)$ of~$G_s$, for $1\leq s\leq r$, such that $V_i = W_1 \otimes_{\mathbb{K}} \dots \otimes_{\mathbb{K}} W_r$ and $\tau_i = \sigma_1 \otimes_{\mathbb{K}} \dots \otimes_{\mathbb{K}} \sigma_r$.
Suppose by contradiction that $\tau_i$ is not $\Gamma$-proximal.
Then $\tau_i$ is also not $G$-proximal, and so by Fact~\ref{fact:irred-rep-prod} there exists $1\leq s\leq r$ such that $\sigma_s : G_s\to\GL(W_s)$ is non-trivial and not $G_s$-proximal.
Note that $\sfG_s$ cannot be the multiplicative group, hence it is a connected, simply connected, simple factor of~$\sfG$.
Let $\sfG'_s$ be the product of the $\sfG_t$ for $t\neq s$, and $G'_s$ its real points, so that $\sfG = \sfG_s \times \sfG'_s$ and $G=G_s\times G_s'$.
We can write $\tau_i = \sigma_s \otimes \sigma'_s$ where $\sigma'_s$ is the
 irreducible representation of~$G'_s$ which is the tensor product of the $\sigma_t$ for $t\neq s$, with representation space $W'_s$ which is the tensor product of the $W_t$ for $t\neq s$.
Since $\tau_i$ is $\Gamma$-semi-proximal and irreducible, it is $G$-semi-proximal by Fact~\ref{fact:semiprox-Gamma-G}, and so $\sigma_s$ is $G_s$-semi-proximal by Fact~\ref{fact:irred-rep-prod}.
By Fact~\ref{fact:semiprox-nonprox-nonmin}, there is an irreducible linear representation $(\varsigma_s, \mathcal{W}_s)$ of~$G_s$ which is $G_s$-proximal, with $\dim(\mathcal{W}_s) < \dim(W_s)$, such that $\varsigma_s\otimes\sigma'_s$ is still positively $\Gamma$-semi-proximal.
If we set
$$\tau'' := (\varsigma_s\otimes\sigma'_s) \oplus \bigoplus_{j\neq i} \tau_j,$$
then $\tau''\oplus\upsilon$ is still positively $\Gamma$-semi-proximal by Remark~\ref{rem:pos-semiprox-remove-doubles}.\eqref{item:semiprox-doubles-2}.
Note that the kernel $\mathrm{Ker}(\tau''\oplus\nolinebreak\upsilon)$ of $\tau''\oplus\upsilon$ is the direct product of $\mathrm{Ker}(\tau''\oplus\upsilon) \cap G_s$ which is finite (because $\varsigma_s$ is a non-trivial irreducible representation of~$G_s$) and of $\mathrm{Ker}(\tau''\oplus\upsilon) \cap G'_s = \mathrm{Ker}(\tau\oplus\upsilon)$.
Therefore the Lie algebra of the kernel of $\tau''\oplus\upsilon$ is contained in (in fact, equal to) the Lie algebra of the kernel of $\tau\oplus\upsilon$, and so $\tau''\oplus\upsilon$ still satisfies~\eqref{item:semiprox-hyp-ii}.
Since $\delta(\tau'') < \delta(\tau)$, we get a contradiction with the minimality of $\delta(\tau)$.
This shows that $\tau_i$ must be $\Gamma$-proximal, hence positively $\Gamma$-proximal by Remark~\ref{rem:prox-semi-prox}.

\medskip

\noindent
$\bullet$ \textbf{Step 4: Proof of the inequality $\vcd(\Gamma) \leq \delta(\tau) + \dim(M_F)$.}
By Fact~\ref{fact:proj-ss-discrete}, the group $\pi(\Gamma)$ is discrete and Zariski-dense in~$G^{ss}$.
By Fact~\ref{fact:proj-surj} applied to $(\pi(\Gamma),\sfG^{ss},(\tau\oplus\upsilon)|_{G^{ss}})$ instead of $(\Gamma,\sfG,\tau)$, the group $(\tau\oplus\upsilon)(\pi(\Gamma))$ is discrete in $\GL(V\oplus E)$.
We deduce that the group $(\tau\oplus\upsilon)(\Gamma)$ is discrete in $\GL(V\oplus E)$. 
Indeed, the projection to the semi-simple part of $(\tau\oplus\upsilon)(G)$ restricted to $(\tau\oplus\upsilon)(\Gamma)$ has finite kernel and the image of this projection equals $(\tau\oplus\upsilon)(\pi(\Gamma))$.

Moreover, by Fact~\ref{fact:proj-surj} the intersection of $\pi(\Gamma)$ with the kernel of $\tau\oplus\upsilon$ is finite.
Since the restriction of $\pi$ to~$\Gamma$
 has finite kernel by assumption, this implies that the intersection of $\Gamma$ with the kernel of $\tau\oplus\upsilon$ is finite.
In particular, $\vcd(\Gamma) = \vcd((\tau\oplus\upsilon)(\Gamma))$.

By Step~2 and Fact~\ref{fact:pos-semiprox-preserve-conv}.\eqref{item:semiprox-pres-conv-1}, if $I\neq\emptyset$, then for any $i\in I$ the group $\tau_i(\G)$ preserves a non-empty properly convex open set $\Omega_i$ in $\P(V_i)$.
Thus the discrete subgroup $(\tau\oplus\upsilon)(\Gamma)$ of $\GL(V\oplus E) \simeq (\prod_{i\in I} \GL(V_i)) \times \GL(E)$ (where $I$ is now allowed to be empty) preserves the contractible topological manifold without boundary $(\prod_{i\in I} \Om_i) \times M_F$, which has dimension $\delta(\tau) + \dim(M_F)$.
Moreover, this group acts by isometries on the sum of the Hilbert metrics on each factor, including $M_F$, where the metric is taken with respect to the properly convex set $\Om_F$.
Since the Hilbert metric is proper, $(\tau\oplus\upsilon)(\G)$ acts properly discontinuously on $(\prod_{i\in I} \Om_i) \times M_F$.
By Fact~\ref{fact:vcd}, this implies $\vcd((\tau\oplus\upsilon)(\Gamma)) \leq \delta(\tau) + \dim(M_F)$.
\end{proof}

\subsection{Proof of Proposition~\ref{prop:non-deg-limit}} \label{subsec:Lambda-nondegen}

Since $\Gamma$ has no infinite nilpotent normal subgroups, the Kazhdan--Margulis--Zassenhaus theorem implies that the set of injective and discrete representations of~$\Gamma$ is closed in $\Hom(\Gamma,\PO(p,q+1))$ (Fact~\ref{fact:fd-closed}).
Therefore $\rho$, which is the limit of the injective and discrete representations~$\rho_n$, is also injective and discrete.

By Corollary~\ref{c.limit-weakly-sp-gr}, up to passing to a subsequence, we may assume that the $M_n$ converge to some weakly spacelike $p$-graph $M$ in~$\H^{p,q}$.
Then $M$ is $\rho(\Gamma)$-invariant.

By Propositions \ref{prop:weakly-sp-gr-is-a-graph}--\ref{prop:sp-gr-is-a-graph}, we can lift injectively each $M_n$ to a spacelike $p$-graph $\hat{M}_n$ in $\hat\H^{p,q}$, and $M$ to a weakly spacelike $p$-graph $\hat{M}$ in $\hat\H^{p,q}$, in such a way that the $\hat{M}_n$ converge to~$\hat{M}$.
Each element $\rho_n(\gamma)$ for $\gamma\in \Gamma$ has a unique lift to $\OO(p,q+1)$ that preserves $\hat M_n$ (this follows \eg from Proposition~\ref{prop:lift-non-pos-sphere} as such elements also preserve $\hat\L_n := \di\hat{M}_n \subset \di\hat\H^{p,q}$, which is a non-degenerate non-positive $(p-1)$-sphere by Proposition~\ref{prop:bdynonpossphere}.\eqref{item:bound-non-pos-sph-1} and Lemma~\ref{lem:Omega-Lambda-convex}.\eqref{item:Om-Lambda-non-empty}).
Choosing the lift preserving $\hat M_n$ for each $\gamma\in \Gamma$ gives a representation from $\Gamma$ to $\OO(p,q+1)$ lifting $\rho_n$; we still denote it by $\rho_n : \Gamma\to\OO(p,q+1)$. 
Similarly, we can lift $\rho$ to a representation from $\Gamma$ to $\OO(p,q+1)$, still denoted by~$\rho$, such that $\rho(\Gamma)$ preserves~$\hat{M}$ and $\lim_n \rho_n=\rho$.

By Proposition~\ref{prop:proper-action-M}, for any~$n$, the group $\rho_n(\Gamma)$ preserves a properly convex open subset of $\H^{p,q}$ containing~$M_n$.
Therefore any element of $\rho_n(\Gamma)$ is positively semi-proximal by Fact~\ref{fact:pos-semiprox-preserve-conv}.\eqref{item:semiprox-pres-conv-2}.
Passing to the limit, the same property holds for $\rho(\Gamma)$ (see Remark~\ref{rem:char-pos-semiprox}).

Let $\sfG$ be the Zariski closure of $\rho(\Gamma)$ in $\OO(p,q+1)$; it is reductive by assumption.
Let $G$ be the real points of~$\sfG$.
We can view $\Gamma$ as a discrete subgroup of~$G$ and $\rho$ as a linear representation of~$G$ (in the sense of Section~\ref{subsec:semi-proximality}), whose kernel in~$G$ is by construction trivial, and which is positively $\Gamma$-semi-proximal by the argument above.

Up to replacing the algebraic group $\sfG$ by its identity component and $\Gamma$ by its intersection with the real points of the identity component of~$\sfG$ (which is a finite-index subgroup of~$\Gamma$), we may and shall assume that $\sfG$ is connected.
Furthermore, up to replacing $\sfG$ by a finite cover, $\Gamma$ by its preimage in this finite cover (which does not change $\vcd(\Gamma)$), and $\rho$ by the composition of $\rho$ with the covering map, we may and shall assume that $\sfG = \sfG_1 \times \cdots \times \sfG_r$ is a direct product of real algebraic groups $\sfG_s$, each of which is either the multiplicative group or a connected, simply connected, simple algebraic group (see \eg \cite[Prop.\,14.2 \& Th.\,22.10]{bor91}).
The kernel of $\rho : G\to\OO(p,q+1)$ is then finite.

As in Proposition~\ref{prop:semiprox-vcd}, let $\mathsf{Z(G)}$ (\resp $\sfG^{ss}$) be the center (\resp the commutator subgroup) of~$\sfG$, so that $\sfG = \mathsf{Z(G)} \times \sfG^{ss}$.
Let $Z(G)$ and~$G^{ss}$ be the real points of $\mathsf{Z(G)}$ and~$\sfG^{ss}$, respectively, and $\varpi : G\to Z(G)$ and $\pi : G\to G^{ss}$ the natural projections.
Up to replacing $\Gamma$ by a finite-index subgroup (which again does not change $\vcd(\Gamma)$), we may assume that $\varpi(\Gamma)$ lies in the identity component of $Z(G)$ for the real topology.
Since $\Gamma$ has no infinite nilpotent normal subgroups, the restriction of $\pi$ to~$\Gamma$ has finite kernel.

By Proposition~\ref{prop:bdynonpossphere}.\eqref{item:bound-non-pos-sph-1}, the set $\L = \di M \subset \di\H^{p,q}$ is a non-positive $(p-1)$-sphere.
Let $V \subset \spa(\L) \subset \R^{p,q+1}$ be the kernel of $\sfb|_{\spa(\L)}$, of dimension $k:=\dim V\geq 0$.
It is a $\rho(\Gamma)$-invariant, totally isotropic subspace of~$\R^{p,q+1}$.
Its orthogonal $V^{\perp}$ is also $\rho(\Gamma)$-invariant.
Since the Zariski closure of $\rho(\Gamma)$ in $\OO(p,q+1)$ is reductive, we can find
\begin{itemize}
  \item a $\rho(\Gamma)$-invariant complementary subspace $E$ of $V$ in~$V^{\perp}$;
  \item a $\rho(\Gamma)$-invariant complementary subspace $V_{tr}$ of $V^{\perp}$ in~$\R^{p,q+1}$.
\end{itemize}
The signature of $\sfb|_{V\oplus V_{tr}}$ is $(k,k|0)$, and that of $\sfb|_E$ is $(p-k,q+1-k|0)$.
Thus we have a $\sfb$-orthogonal splitting of $\R^{p,q+1}$ into $\rho(\Gamma)$-invariant, non-degenerate subspaces
$$\R^{p,q+1} = (V+V_{tr}) \oplus E.$$
The representation $\r$ splits correspondingly as
\begin{equation} \label{eqn:split-rho}
\r = \k \oplus \k_{tr} \oplus \nu : \G \longrightarrow \GL(V) \times \GL(V_{tr}) \times \OO(\sfb|_E) \subset \OO(p,q+1).
\end{equation}
Since $V$ and~$V_{tr}$ are $\rho(\Gamma)$-invariant totally isotropic subspaces of~$\R^{p,q+1}$ that are transverse, the form $\sfb$ naturally identifies $V_{tr}$ with the dual $V^*$ of~$V$, and $\k_{tr}$ with the dual representation~$\k^*$.

Suppose $k=p$.
Then $\nu$ takes values in the compact group $\OO(\sfb|_E) \simeq \OO(q+1-k)$, and Proposition~\ref{prop:semiprox-vcd} applies with $\tau = \kappa\oplus\kappa^*$ and the trivial representation~$\upsilon$.

Suppose $k<p$.
Then $\H_E := \H^{p,q} \cap \P(E) \simeq \H^{p-k,q-k}$.
By Proposition~\ref{prop:ME-weakly-sp-gr}, the set $\L_E := \L \cap \P(E)$ is a non-degenerate non-positive $(p-k-1)$-sphere in $\di\H_E \simeq \di\H^{p-k,q-k}$.
By assumption, the group $\nu(\Gamma)$ preserves a weakly spacelike $(p-k)$-graph $M_E$ in $\H_E$ with $\di M_E = \L_E$ and $M_E \subset \Om(\L_E)$.
By Lemma~\ref{l.kernonpossphere}, the restriction $\sfb|_{\spa\L_E}$ has signature $(p-k,q'|0)$ for some $1\leq q'\leq q+1-k$, and $\nu$ splits as a direct sum $\nu = \nu_E \oplus \nu_{\perp}$ where $\nu_E$ takes values in the indefinite orthogonal group $\OO(\sfb|_{\spa(\L_E)}) \simeq \OO(p-k,q')$ and $\nu_{\perp}$ in the compact group $\OO(\sfb|_{\L_E^{\perp}}) \simeq \OO(q+1-k-q')$.
By Lemma~\ref{lem:restrict-span-L}, the orthogonal projection $M'_E$ of $M_E$ to $\P(\spa(\L_E))$ is well defined, and is still a weakly spacelike $(p-k)$-graph in~$\H_E$ with $\di M'_E = \L_E$ and $M'_E \subset \Om_E(\L_E)$.
Since $M'_E$ is a weakly spacelike $(p-k)$-graph, it is homeomorphic to a contractible topological manifold without boundary (Definition~\ref{def:weakly-sp-gr} and Remark~\ref{r.contractible}).
Since $M'_E$ is contained in $\Om_E(\L_E)$, it is contained in $\Om_E(\L_E) \cap \P(\spa(\L_E))$, which is a $\nu_E(\Gamma)$-invariant properly convex open subset of $\P(\spa(\L_E))$ by Lemma~\ref{lem:Omega-Lambda-convex}.\eqref{item:Om-Lambda-prop-convex}.
Moreover, $\nu_E(\Gamma)$ preserves a properly convex open cone of $\spa(\L_E)$ projecting to $\Om_E(\L_E) \cap \P(\spa(\L_E))$: indeed, we saw that $\rho(\Gamma)$ preserves a weakly spacelike $p$-graph $\hat{M}$ in $\hat\H^{p,q}$ lifting~$M$; the group $\nu(\Gamma)$ preserves the intersection $\hat\L_E$ of $\di\hat{M}$ with $(E\smallsetminus\{0\})/\R_{>0}$, hence $\nu_E(\Gamma)$ preserves the properly convex open cone $\tilde{\Om}_E(\tilde{\L}_E) \cap \spa(\L_E)$ of $\spa(\L_E)$ where $\tilde{\L}_E$ is any subset of the $\sfb$-isotropic vectors of $E\smallsetminus\{0\}$ whose projection to $(E\smallsetminus\{0\})/\R_{>0}$ is~$\hat\L_E$, as in Notation~\ref{not:Omega-Lambda}.
Since $\nu_{\perp}$ takes values in a compact group, we see that Proposition~\ref{prop:semiprox-vcd} applies with $\tau = \kappa\oplus\kappa^*$ and $\upsilon = \nu_E$ and $M_F = M'_E$.

Suppose by contradiction that the non-positive $(p-1)$-sphere $\L := \di M \subset \di\H^{p,q}$ is degenerate, which means that $\kappa$ is non-trivial, with representation space $V$ of dimension $k\geq 1$.
Write $\kappa = \bigoplus_{i\in I}\, \kappa_i$ as a sum of non-trivial irreducible representations of~$G$, where $I$ is some non-empty finite set.
Let $\tau$ be a subrepresentation of $\k\oplus\k^*$ satisfying properties \eqref{item:semiprox-hyp-i} and \eqref{item:semiprox-hyp-ii} of Proposition~\ref{prop:semiprox-vcd}, and such that $\delta(\tau)$ is minimal among all such subrepresentation of $\k\oplus\k^*$.
We can write $\tau = \bigoplus_{i\in I_{\kappa}} \kappa_i \oplus \bigoplus_{i\in I_{\kappa^*}} \kappa^*_i$ for some subsets $I_{\kappa},I_{\kappa^*}$ of~$I$.
Arguing exactly as in Step~1 of the proof of Proposition~\ref{prop:semiprox-vcd}, we see that $I_{\kappa}\cap I_{\kappa^*}=\emptyset$.
Therefore the dimension of the representation space of~$\tau$ is at most that of~$\kappa$, which is~$k$, and so $\delta(\tau)<k$.
By Proposition~\ref{prop:semiprox-vcd}, we have $p \leq \delta(\tau) + p-k < p$: contradiction.

\section{Closedness of $\H^{p,q}$-convex cocompact and spacelike cocompact representations} \label{sec:proof-main-thm}

In this final section we complete the proofs of the main results of the paper.
We first deduce Theorems \ref{thm:charact-Hpq-cc-sphere} and~\ref{thm:non-deg-limit-sphere} from Theorem~\ref{thm:geom-action-weakly-sp-gr-basic}, Proposition~\ref{prop:non-deg-limit}, Fact~\ref{fact:exist-p-graph}, and Proposition~\ref{prop:ME-weakly-sp-gr}.
We then explain how these results imply Theorem~\ref{thm:main}.
Next we discuss the representations appearing in Theorems \ref{thm:main-general} and~\ref{thm:main-general-spacelike-p-mfd}, which we call \emph{spacelike cocompact}.
Finally, we prove Theorems \ref{thm:main-general} and~\ref{thm:main-general-spacelike-p-mfd}.

\subsection{Proof of Theorem~\ref{thm:charact-Hpq-cc-sphere}}

Suppose that $\Gamma$ preserves a non-degenerate non-positive $(p-1)$-sphere $\L$ in $\di\H^{p,q}$.
By Fact~\ref{fact:exist-p-graph}, it also preserves a weakly spacelike $p$-graph $M$ with $\di M = \L$ and $M \subset \Om(\L)$.
By Proposition~\ref{prop:proper-action-M} and Fact~\ref{fact:vcd}, the action of $\Gamma$ on~$M$ is properly discontinuous and cocompact.
Therefore $\Gamma$ is $\H^{p,q}$-convex cocompact by Theorem~\ref{thm:geom-action-weakly-sp-gr-basic}.

Conversely, suppose $\Gamma$ is $\H^{p,q}$-convex cocompact.
By Fact~\ref{fact:Hpq-cc-Ano}, its proximal limit set $\L_{\Gamma}$ is a non-degenerate non-positive $(p-1)$-sphere in $\di\H^{p,q}$, which is actually negative.
This proximal limit set is invariant under~$\Gamma$.

\subsection{Proof of Theorem~\ref{thm:non-deg-limit-sphere}} \label{subsec:proof-non-deg-limit-sphere}

By assumption, the representation $\rho$ is a limit of injective and discrete representations $\rho_n$, each preserving a non-degenerate non-positive $(p-1)$-sphere $\L_n$ in $\di\H^{p,q}$.
By Fact~\ref{fact:exist-p-graph}, for each~$n$ there is a $\rho_n(\Gamma)$-invariant weakly spacelike $p$-graph $M_n$ in $\H^{p,q}$ with $\di M_n = \L_n$ and $M_n \subset \Om(\L_n)$.
By Corollary~\ref{c.limit-non-pos-sphere}, up to passing to a subsequence, we may assume that the $\L_n$ converge to some $\rho(\Gamma)$-invariant non-positive $(p-1)$-sphere $\L$ in $\di\H^{p,q}$.
Since $\rho(\Gamma)$ has reductive Zariski closure, there is a $\rho(\Gamma)$-invariant complementary subspace $E$ of $V := \mathrm{Ker}(\sfb|_{\spa(\L)})$ in~$V^{\perp}$.
By Proposition~\ref{prop:ME-weakly-sp-gr}, if $k<p$, then $\L_E := \L \cap \P(E)$ is a non-degenerate $(p-k-1)$-sphere in $\di\H^{p,q} \cap \P(E) \simeq \di\H^{p-k,q-k}$, and so Fact~\ref{fact:exist-p-graph} gives the existence of a $\rho(\Gamma)$-invariant weakly spacelike $(p-k)$-graph $M_E$ in $\H^{p,q} \cap \P(E) \simeq \H^{p-k,q-k}$ with $\di M_E = \L_E$ and $M_E \subset \Om(\L_E)$.
We can then apply Proposition~\ref{prop:non-deg-limit}, and obtain that $\rho$ is injective and discrete, $V = \{0\}$, and $\L = \L_E$ is non-degenerate.

\subsection{Proof of Theorem~\ref{thm:main}} \label{subsec:proof-main-thm}

Theorems \ref{thm:charact-Hpq-cc-sphere} and~\ref{thm:non-deg-limit-sphere} imply the following.

\begin{thm} \label{thm:main-closed-finite-kernel}
Let $p\geq 2$ and $q\geq 1$, and let $\Gamma$ be a Gromov hyperbolic group with $\vcd(\Gamma) = p$.
Then for any finite normal subgroup $\Gamma'$ of~$\Gamma$, the set $\Hom_{\text{cc}}^{\Gamma'}(\G,\PO(p,q+1))$ of $\H^{p,q}$-convex cocompact representations of~$\Gamma$ with kernel exactly~$\Gamma'$ is closed in $\Hom(\Gamma,\PO(p,q+1))$.
\end{thm}

\begin{proof}
Since $\vcd(\Gamma)=p\geq 2$, the Gromov hyperbolic group $\Gamma$ has no infinite nilpotent normal subgroups (see \eg \cite[Prop.\,III.$\Gamma$.3.20]{bh99}).

Consider a sequence $(\rho_n)_{n\in\N}$ of elements of $\Hom_{\text{cc}}^{\Gamma'}(\G,\PO(p,q+1))$ converging to a representation $\rho \in \Hom(\Gamma,\PO(p,q+1))$.
Then $\Gamma'$ is contained in the kernel of~$\rho$.
Each representation $\rho_n$ factors through an injective $\H^{p,q}$-convex cocompact representation $\overline{\rho}_n : \Gamma/\Gamma'\to\PO(p,q+1)$, the representation $\rho$ factors through a representation $\overline{\rho} : \Gamma/\Gamma'\to\PO(p,q+1)$, and $\overline{\rho}_n\to\overline{\rho}$.
By Fact~\ref{fact:fd-closed}, the representation $\overline{\rho}$ is injective and discrete, hence $\rho$ has kernel exactly~$\Gamma'$.
Let us check that $\rho$ is still $\H^{p,q}$-convex cocompact.

If the Zariski closure of $\overline{\rho}(\Gamma)=\rho(\Gamma)$ in $\PO(p,q+1)$ is reductive, then $\overline{\rho}$ is $\H^{p,q}$-convex cocompact by Theorems \ref{thm:charact-Hpq-cc-sphere} and~\ref{thm:non-deg-limit-sphere}, and so $\rho$ is also $\H^{p,q}$-convex cocompact.

We now treat the general case where the Zariski closure $G$ of $\rho(\Gamma)$ might not necessarily be reductive.
We can write $G$ as a semi-direct product $L\ltimes U$ where $L$ is reductive and $U$ is unipotent (Levi decomposition).
Let $\pi_L : G\to L$ be the natural projection.
Following \cite{ggkw17}, we call \emph{semi-simplification} of~$\rho$ the representation $\rho^{ss} :=\pi_L\circ\rho : \Gamma\to L$.
(It is unique up to conjugation by an element of~$U$.)
By construction, the Zariski closure of $\rho^{ss}(\Gamma)$ is reductive, and $\rho^{ss}$ is a limit of conjugates $g_m\rho(\cdot)g_m^{-1}$ of $\rho$ where $(g_m)\in\PO(p,q+1)^{\N}$.
For each~$m$, the sequence $(g_m\rho_n(\cdot)g_m^{-1})_{n\in\N}$ converges to $g_m\rho(\cdot)g_m^{-1}$, so by a diagonal extraction argument, there exists $m_n\to +\infty$ such that $(\rho'_n)_{n\in\N} := (g_{m_n}\rho_n(\cdot)g_{m_n}^{-1})_{n\in\N}$ converges to~$\rho^{ss}$.
Each $\rho'_n$ is still $\H^{p,q}$-convex cocompact because this property is invariant under conjugation.
By the reductive case treated above, $\rho^{ss}$ is $\H^{p,q}$-convex cocompact.
But $\H^{p,q}$-convex cocompactness is an open condition (Fact~\ref{fact:Hpq-cc-open}), so for large enough~$m$ the representation $g_m\rho(\cdot)g_m^{-1}$ is $\H^{p,q}$-convex cocompact.
Therefore $\rho$ is $\H^{p,q}$-convex cocompact.
\end{proof}

\begin{cor} \label{cor:main-finite-kernel}
Let $p\geq 2$ and $q\geq 1$, and let $\Gamma$ be a Gromov hyperbolic group with $\vcd(\Gamma) = p$.
Then the set $\Hom_{\text{cc}}(\G,\PO(p,q+1))$ of $\H^{p,q}$-convex cocompact representations from $\Gamma$ to $\PO(p,q+1)$ is a union of connected components of $\Hom(\Gamma,\PO(p,q+1))$.
More precisely, for any finite normal subgroup $\Gamma'$ of~$\Gamma$, the set $\Hom_{\text{cc}}^{\Gamma'}(\G,\PO(p,q+1))$ of $\H^{p,q}$-convex cocompact~rep\-resentations with kernel exactly~$\Gamma'$ is a union of connected components of $\Hom(\Gamma,\PO(p,q+1))$.
\end{cor}

\begin{proof}
Recall (Definition~\ref{def:Hpq-cc}) that the kernel of any $\H^{p,q}$-convex cocompact representation of~$\Gamma$ is a finite normal subgroup $\Gamma'$ of~$\Gamma$.
Since $\Gamma$ is Gromov hyperbolic, there are only finitely many such subgroups~$\Gamma'$ (see \eg \cite[Th.\,III.$\Gamma$.3.2]{bh99}).
Thus the set $\Hom_{\text{cc}}(\G,\PO(p,q+1))$ is a \emph{finite} disjoint union of closed subsets $\Hom_{\text{cc}}^{\Gamma'}(\G,\PO(p,q+1))$ of $\Hom(\G,\PO(p,q+1))$, for $\Gamma'$ ranging through the finite normal subgroups of~$\Gamma$.
In particular, $\Hom_{\text{cc}}(\G,\PO(p,q+1))$ is closed in $\Hom(\G,\PO(p,q+1))$, and each subset $\Hom_{\text{cc}}^{\Gamma'}(\G,\PO(p,q+1))$ is open in $\Hom_{\text{cc}}(\G,\PO(p,q+1))$.
By Fact~\ref{fact:Hpq-cc-open}, the set $\Hom_{\text{cc}}(\G,\PO(p,q+1))$ is open in $\Hom(\G,\PO(p,q+1))$, and so each subset $\Hom_{\text{cc}}^{\Gamma'}(\G,\PO(p,q+1))$ is open in $\Hom(\G,\PO(p,q+1))$.
\end{proof}

Theorems \ref{thm:deform-Fuchsian-inj-discr} and~\ref{thm:main} are contained in Corollary~\ref{cor:main-finite-kernel}.

\subsection{Spacelike cocompact representations} \label{subsec:spacelike-cocompact}

The following is a consequence of Propositions \ref{prop:proper-action-M} and~\ref{prop:bdynonpossphere}, Fact~\ref{fact:vcd}, and \cite{sst}.

\begin{lem} \label{lem:spacelike-cocompact}
Let $p,q\geq 1$, let $\Gamma$ be a finitely generated group, and let $\rho : \Gamma\to\PO(p,q+1)$ be a representation.
Then the following are equivalent:
\begin{enumerate}
  \item\label{item:sc-1} $\vcd(\Gamma) = p$ and $\rho$ has finite kernel and discrete image which preserves a non-degenerate non-positive $(p-1)$-sphere in $\di\H^{p,q}$,
  \item\label{item:sc-2} $\Gamma$ acts properly discontinuously and cocompactly via~$\rho$ on some $p$-dimensional connected complete spacelike submanifold of~$\H^{p,q}$;
  \item\label{item:sc-3} $\Gamma$ acts properly discontinuously and cocompactly via~$\rho$ on some $p$-dimensional connected maximal complete spacelike submanifold of~$\H^{p,q}$.
\end{enumerate}
\end{lem}

\begin{proof}
\eqref{item:sc-1}~$\Rightarrow$~\eqref{item:sc-3}: Let $\L$ be a non-degenerate non-positive $(p-1)$-sphere in $\di\H^{p,q}$, preserved by $\rho(\Gamma)$.
By \cite{sst}, there is a $p$-dimensional connected maximal complete spacelike submanifold $M$ of~$\H^{p,q}$, preserved by $\rho(\Gamma)$, such that $\di M = \L$.
Since $\rho$ has finite kernel and discrete image, $\Gamma$ acts properly discontinuously on~$M$ via~$\rho$ by Proposition~\ref{prop:proper-action-M}.
Since $\vcd(\Gamma) = p$ and $M$ is contractible (see Proposition~\ref{prop:weakly-sp-gr-is-a-graph}), this action is cocompact by Fact~\ref{fact:vcd}.

\eqref{item:sc-3}~$\Rightarrow$~\eqref{item:sc-2}: clear.

\eqref{item:sc-2}~$\Rightarrow$~\eqref{item:sc-1}: Let $M$ be a $p$-dimensional connected complete spacelike submanifold $M$ of~$\H^{p,q}$ on which $\Gamma$ acts properly discontinuously and cocompactly via~$\rho$.
By Proposition~\ref{prop:bdynonpossphere}, the ideal boundary $\L := \di M$ is a non-degenerate non-positive $(p-1)$-sphere in $\di\H^{p,q}$, preserved by $\rho(\Gamma)$.
Since the action of $\rho$ on~$M$ is properly discontinuous, $\rho$ has finite kernel and discrete image by Proposition~\ref{prop:proper-action-M}.
Since this action is cocompact and $M$ is contractible, we have $\vcd(\Gamma) = p$ by Fact~\ref{fact:vcd}.
\end{proof}

We shall use the following terminology.

\begin{defn} \label{def:spacelike-cocompact}
Let $p,q\geq 1$ and let $\Gamma$ be a finitely generated group with $\vcd(\Gamma) = p$.
A representation $\rho : \Gamma\to\PO(p,q+1)$ is \emph{spacelike cocompact} if it satisfies the equivalent conditions \eqref{item:sc-1}, \eqref{item:sc-2}, \eqref{item:sc-3} of Lemma~\ref{lem:spacelike-cocompact}.
\end{defn}

Observe that if $\rho$ is spacelike cocompact, then so is the image of~$\rho$ by conjugation at the target by any element of $\PO(p,q+1)$.
By Theorem~\ref{thm:charact-Hpq-cc-sphere}, if $\Gamma$ is Gromov hyperbolic with $\vcd(\Gamma) = p$, then \emph{spacelike cocompact} is equivalent to \emph{$\H^{p,q}$-convex cocompact}.

\begin{lem} \label{lem:finitely-many-finite-subgroups}
Let $p,q\geq 1$ and let $\Gamma$ be a finitely generated group with $\vcd(\Gamma) = p$.
If $\Gamma$ admits a spacelike cocompact representation into $\PO(p,q+1)$, then $\Gamma$ has only finitely many conjugacy classes of finite subgroups.
In particular, $\Gamma$ has only finitely many finite normal subgroups.
\end{lem}

\begin{proof}
Suppose there is a spacelike cocompact representation $\rho : \Gamma\to\PO(p,q+1)$: the group $\Gamma$ acts properly discontinuously and cocompactly via~$\rho$ on some $p$-dimensional connected complete spacelike submanifold $M$ of~$\H^{p,q}$.
Let $D\subset M$ be a compact fundamental domain for this action.
By proper discontinuity, the set $\mathcal{F}$ of elements $\gamma\in\Gamma$ such that $D\cap\rho(\gamma)\cdot D\neq\emptyset$ is finite.
Let $\Gamma'$ be a finite subgroup of~$\Gamma$.
By Lemma~\ref{lem:finite-group-preserv-M}, the group $\rho(\Gamma')$ has a global fixed point in~$M$.
Up to conjugation in~$\Gamma$ we may assume that this fixed point is contained in~$D$, hence $\Gamma'\subset\mathcal{F}$.
\end{proof}

\subsection{Proof of Theorems \ref{thm:main-general} and~\ref{thm:main-general-spacelike-p-mfd}} \label{subsec:proof-main-thm-general}

We check openness and closedness.

\begin{prop} \label{prop:sc-open}
Let $p,q\geq 1$ and let $\Gamma$ be a finitely generated group with $\vcd(\Gamma) = p$.
Then the set $\Hom_{\text{sc}}(\G,\PO(p,q+1))$ of spacelike cocompact representations from $\Gamma$ to $\PO(p,q+1)$ is open in $\Hom(\G,\PO(p,q+1))$.
\end{prop}

\begin{proof}
We may assume that $\Hom_{\text{sc}}(\G,\PO(p,q+1))$ is non-empty, otherwise there is nothing to prove.
By the Selberg lemma \cite[Lem.\,8]{sel60}, the finitely generated group $\Gamma$ admits a finite-index subgroup $\Gamma_0$ which is torsion-free.

Let $\rho \in \Hom_{\text{sc}}(\G,\PO(p,q+1))$, and let $M$ be a $p$-dimensional connected complete spacelike submanifold of~$\H^{p,q}$ on which $\Gamma$ acts properly discontinuously and cocompactly via~$\rho$.
Let $\L := \di M$.
Then $M \subset \Om(\L)$ by Proposition~\ref{prop:bdynonpossphere}; moreover, $M$ is contained in some $\rho(\Gamma)$-invariant \emph{properly convex} open subset $\Om \subset \Om(\L)$ by Proposition~\ref{prop:M-in-prop-conv-Omega}.
For $\varepsilon>0$, consider the closed $\rho(\Gamma)$-invariant neighborhood
$$\mathcal{C} := \{\exp_x(v) ~|~ x\in M,\ v\in N_xM\subset T_x\H^{p,q},\ |\sfg(v,v)|\leq\varepsilon\}$$
of $M$ in~$\H^{p,q}$, where $NM$ is the normal bundle of $M$ in $\H^{p,q}$ and $\exp$ the pseudo-Riemannian exponential map.
For $\varepsilon$ small enough, $\mathcal{C}$ is diffeomorphic to $M\times\ov\bD^q$, where $\ov\bD^q$ is the closed unit disk of Euclidean~$\R^q$; moreover, since the action of $\Gamma$ on $M$ via~$\rho$ is cocompact, for $\varepsilon$ small enough $\mathcal{C}$ is contained in~$\Omega$.
The action of $\Gamma_0$ on~$\mathcal{C}$ via~$\rho$ is then properly discontinuous, cocompact, and free, and $\underline{\mathcal{C}} := \rho(\Gamma_0)\backslash\mathcal{C}$ is a compact $(\PO(p,q+1),\H^{p,q})$-manifold with boundary, containing the closed submanifold $\underline{M} := \rho(\Gamma_0)\backslash M$.
By the Ehresmann--Thurston principle (see \eg \cite[Ch.\,1]{ceg87} or \cite[\S\,7.2]{gol22}), there is a neighborhood $\mathcal{U}$ of $\rho$ in $\Hom(\G,\PO(p,q+1))$ consisting entirely of representations whose restriction to~$\Gamma_0$ is the holonomy of a $(\PO(p,q+1),\H^{p,q})$-structure on~$\underline{\mathcal{C}}$.
Since the developing map depends continuously in the $C^{\infty}$ topology on the representation \cite[Ch.\,1]{ceg87} and $\underline{M}$ is compact, we can assume, up to making the neighborhood $\mathcal{U}$ of~$\rho$ smaller, that for any $\rho' \in \mathcal{U}$ the pseudo-Riemannian metric coming from the associated $(\PO(p,q+1),\H^{p,q})$-structure restricted to the tangent bundle of $\underline{M}$ is positive definite.
Then $\underline{M}$ lifts in the universal cover of~$\underline{\mathcal{C}}$ to a subset whose image $M_{\rho'}$ under the developing map is a $p$-dimensional connected complete spacelike submanifold of~$\H^{p,q}$ on which $\Gamma_0$ acts properly discontinuously, cocompactly, and freely via~$\rho'$.
In particular, the restriction of $\rho'$ to~$\Gamma_0$ is injective and discrete, and so $\rho'$ itself has finite kernel and discrete image.
In order to conclude that $\rho'$ is spacelike cocompact, we can either check that the non-degenerate non-positive $(p-1)$-sphere $\di M_{\rho'}$ is invariant under the full group $\rho'(\Gamma)$ by showing that the limit set of $\rho'(\Gamma)$ in $\di\H^{p,q}$ in the sense of \cite[Def.\,5.1]{ggkw17} is a closed $\rho'(\Gamma)$-invariant subset of $\di M_{\rho'}$ containing all extremal points of $\di M_{\rho'}$, or we can check that $M_{\rho'}$ is invariant under the full group $\rho'(\Gamma)$ by using \cite[Lem.\,9.3]{sst} which shows that $M_{\rho'}$ varies analytically and $\Gamma$-equivariantly with~$\rho'$.
\end{proof}

Theorems \ref{thm:main-general} and~\ref{thm:main-general-spacelike-p-mfd} are equivalent by Lemma~\ref{lem:spacelike-cocompact}, and contained in the following more precise statement.

\begin{thm} \label{thm:main-general-finite-kernel}
Let $p\geq 2$ and $q\geq 1$, and let $\Gamma$ be a finitely generated group with no infinite nilpotent normal subgroups, such that $\vcd(\Gamma) = p$.
Then the set $\Hom_{\text{sc}}(\G,\PO(p,q+1))$ of spacelike cocompact representations from $\Gamma$ to $\PO(p,q+1)$ is a union of connected components of $\Hom(\Gamma,\PO(p,q+1))$.
More precisely, for any finite normal subgroup $\Gamma'$ of~$\Gamma$, the set $\Hom_{\text{sc}}^{\Gamma'}(\G,\PO(p,q+1))$ of spacelike cocompact representations with kernel exactly~$\Gamma'$ is a union of connected components of $\Hom(\Gamma,\PO(p,q+1))$.
\end{thm}

\begin{proof}
Let us show that for any finite normal subgroup $\Gamma'$ of~$\Gamma$, the set $\Hom_{\text{sc}}^{\Gamma'}(\G,\PO(p,q+1))$ is closed in $\Hom(\Gamma,\PO(p,q+1))$.
Let $(\rho_n)_{n\in\N}$ be a sequence of elements of $\Hom_{\text{sc}}^{\Gamma'}(\G,\PO(p,q+\nolinebreak 1))$ converging to a representation $\rho \in \Hom(\Gamma,\PO(p,q+1))$.
Arguing exactly as in the proof of Theorem~\ref{thm:main-closed-finite-kernel} in Section~\ref{subsec:proof-main-thm}, we see that $\rho$ has kernel exactly~$\Gamma'$.
If the Zariski closure of $\overline{\rho}(\Gamma)=\rho(\Gamma)$ in $\PO(p,q+1)$ is reductive, then $\rho$ is spacelike cocompact by Theorem~\ref{thm:non-deg-limit-sphere}.
For a general Zariski closure, we argue again exactly as in the proof of Theorem~\ref{thm:main}, using the fact that the set $\Hom_{\text{sc}}(\G,\PO(p,q+1))$ is open in $\Hom(\G,\PO(p,q+1))$ (Proposition~\ref{prop:sc-open}) and invariant under conjugation.

Recall (Definition~\ref{def:spacelike-cocompact}) that the kernel of any spacelike cocompact representation of~$\Gamma$ is a finite normal subgroup $\Gamma'$ of~$\Gamma$.
By Lemma~\ref{lem:finitely-many-finite-subgroups}, there are only finitely many such subgroups~$\Gamma'$.
Thus the set $\Hom_{\text{sc}}(\G,\PO(p,q+1))$ is a \emph{finite} disjoint union of closed subsets $\Hom_{\text{sc}}^{\Gamma'}(\G,\PO(p,q+1))$ of $\Hom(\G,\PO(p,q+1))$, for $\Gamma'$ ranging through the finite normal subgroups of~$\Gamma$.
In particular, $\Hom_{\text{sc}}(\G,\PO(p,q+1))$ is closed in $\Hom(\G,\PO(p,q+1))$, and each subset $\Hom_{\text{sc}}^{\Gamma'}(\G,\PO(p,q+1))$ is open in $\Hom_{\text{sc}}(\G,\PO(p,q+1))$.
By Proposition~\ref{prop:sc-open}, the set $\Hom_{\text{sc}}(\G,\PO(p,q+1))$ is open in $\Hom(\G,\PO(p,q+1))$, and so each subset $\Hom_{\text{sc}}^{\Gamma'}(\G,\PO(p,q+1))$ is open in $\Hom(\G,\PO(p,q+1))$.
\end{proof}

\appendix

\section{$\H^{p,q}$-convex cocompact representations with Zariski-dense image} \label{appendix}

In this appendix we explain how to obtain Zariski-dense $\H^{p,q}$-convex cocompact representations for various Gromov hyperbolic groups $\Gamma$ with $\vcd(\Gamma)=p$.
The groups $\Gamma$ we consider include hyperbolic lattices (Section~\ref{subsec:appendix-hyperbolic}), but also more exotic groups which are not commensurable to lattices of $\PO(p,1)$.
Examples of $\H^{p,q}$-convex cocompact representations of such groups were constructed in \cite{lm19,mst} for $q=1$; in Sections \ref{subsec:appendix-Coxeter} and~\ref{subsec:appendix-Gromov-Thurston} below, we explain how to deform them to get Zariski-dense $\H^{p,q}$-convex cocompact representations for $q>1$.

\subsection{Hyperbolic lattices} \label{subsec:appendix-hyperbolic}

Bending \`a la Johnson--Millson \cite{jm87} allows to prove the following.

\begin{prop} \label{prop:Z-dense-deform}
Let $p\geq 2$ and $q\geq 1$.
Let $N$ be a closed orientable real hyperbolic $p$-manifold, and let $\rho_0 : \pi_1(N)\to G:=\SO(p,q+1)$ be the composition of the holonomy representation of~$N$ with the natural inclusion $\PO(p,1)_0\simeq\OO(p,1)_0\hookrightarrow G$.
If $N$ contains $q$ two-sided, connected, totally geodesic embedded hypersurfaces which are pairwise disjoint, then any neighborhood of $\rho_0$ in $\Hom(\pi_1(N),G)$ contains representations whose image is Zariski-dense in~$G$.
\end{prop}

The bending construction was originally introduced in \cite{jm87} for deformations into $\SO(p+\nolinebreak 1,1)$ or $\PGL(p+1,\R)$.
For the reader's convenience, we shall describe it for deformations into $\SO(p,q+1)$, and check Zariski density in this case.
We refer to \cite[\S\,6]{kas12} for the case of deformations into $\SO(p,2)$.

Hyperbolic manifolds $N$ as in Proposition~\ref{prop:Z-dense-deform} exist, which yields the following corollary.
This shows that Theorem~\ref{thm:deform-Fuchsian-inj-discr} gives new examples of higher higher Teichm\"uller spaces corresponding to connected components of $\Hom(\pi_1(N),G)$ containing Zariski-dense representations.

\begin{cor} \label{cor:Z-dense-deform}
Let $p\geq 2$ and $q\geq 1$.
Then there exists a closed orientable real hyperbolic $p$-manifold $N$ with the following property: let $\rho_0 : \pi_1(N)\to G:=\SO(p,q+1)$ be the composition of the holonomy representation of~$N$ with the natural inclusion $\PO(p,1)_0\simeq\OO(p,1)_0\hookrightarrow\nolinebreak G$; then any neighborhood of $\rho_0$ in $\Hom(\pi_1(N),G)$ contains representations whose image is Zariski-dense in~$G$.
\end{cor}

\subsubsection{Preliminary observations}

The following lemma is useful for proving Proposition~\ref{prop:Z-dense-deform}.

\begin{lem} \label{lem:Lie-alg}
Let $p\geq 2$ and $q\geq q'\geq 1$.
\begin{enumerate}
  \item\label{item:Lie-alg-1} The only Lie subalgebra of $\mathfrak{o}(p,q+1)$ that strictly contains $\mathfrak{o}(p,q)$ is $\mathfrak{o}(p,q+1)$.
  \item\label{item:Lie-alg-2} There is an element $X \in \mathfrak{o}(p,q+1)$, centralizing $\mathfrak{o}(p-1,1)$, such that the Lie subalgebra of $\mathfrak{o}(p,q+1)$ generated by $\mathfrak{o}(p,q')$ and $\mathrm{ad}(X)(\mathfrak{o}(p,q'))$ is $\mathfrak{o}(p,q'+1)$.
\end{enumerate}
\end{lem}

\begin{proof}
\eqref{item:Lie-alg-1}
The $\OO(p,q)$-module $\mathfrak{o}(p,q+1)$ decomposes as the direct sum of the $\OO(p,q)$-module $\mathfrak{o}(p,q)$ (where $\OO(p,q)$ acts by the adjoint action) and of the irreducible $\OO(p,q)$-module $\R^{p,q}$ (where $\OO(p,q)$ acts by the natural action).

\eqref{item:Lie-alg-2}
Let $J = E_{1,1} + \dots + E_{p,p} - E_{p+1,p+1} - \dots - E_{p+q+1,p+q+1} \in \mathrm{M}_{p+q+1}(\R)$ be the diagonal matrix with $p$ times the entry~$1$ and $q+1$ times the entry $-1$.
Viewing $\mathfrak{o}(p,q+1)$ as the set of elements $Z\in\mathrm{M}_{p+q+1}(\R)$ such that ${}^{\scriptscriptstyle{T}}\!ZJ + JZ= 0$, and $\mathfrak{o}(p-1,1)$ as its intersection with $\spa(\{E_{i,j} \,|\, 2\leq i,j\leq p+1\})$, we can take $X = E_{1,p+q'+1} + E_{p+q'+1,1}$.
Indeed, $X$ centralizes $\mathfrak{o}(p-1,1)$.
Moreover, $\mathrm{ad}(X)(\mathfrak{o}(p,q'))$ is contained in $\mathfrak{o}(p,q'+1)$ but not in $\mathfrak{o}(p,q')$; therefore the Lie subalgebra of $\mathfrak{o}(p,q+1)$ generated by $\mathfrak{o}(p,q')$ and $\mathrm{ad}(X)(\mathfrak{o}(p,q'))$ is $\mathfrak{o}(p,q'+1)$ by~\eqref{item:Lie-alg-1}.
\end{proof}

\subsubsection{Proof of Proposition~\ref{prop:Z-dense-deform}}

Let $\mathcal{H}_1,\dots,\mathcal{H}_q$ be two-sided, connected, totally geodesic embedded hypersurfaces of~$N$ which are pairwise disjoint.
For each~$i$, the group $\rho_0(\pi_1(\mathcal{H}_i))$ is a uniform lattice in some conjugate of $\OO(p-1,1)$ in $\SO(p,1)$.

Following \cite[\S\,5]{jm87}, we define a finite oriented graph $Y$ in the following way: the vertices $v_1,\dots,v_r$ of~$Y$ correspond to the connected components $N_1,\dots,N_r$ of $N \smallsetminus \bigcup_{i=1}^q \mathcal{H}_i$, with one oriented edge $e_i$ (for some arbitrary choice of orientation) between vertices $v_j$ and~$v_{j'}$ for each hypersurface $\mathcal{H}_i$ between components $N_j$ and~$N_{j'}$.
Then $\Gamma := \pi_1(N)$ is the fundamental group of the graph of groups associated to~$Y$ with edge groups $G_{e_i} = \pi_1(\mathcal{H}_i)$, vertex groups $G_{v_j} = \pi_1(N_j)$, and natural injections $\varphi_{e,0} : G_e\hookrightarrow G_{v_j}$ and $\varphi_{e,1} : G_e\hookrightarrow G_{v_{j'}}$ for each oriented edge $e$ from $v_j$ to $v_{j'}$ (see \cite[Cor.\ to Lem.\,5.3]{jm87}).

Let $T$ be a maximal tree in~$Y$, and let $\Gamma_T$ be the group generated by the vertex groups $G_{v_1},\dots,G_{v_r}$ with the relations $\varphi_{e,0}(g) = \varphi_{e,1}(g)$ for all $g\in G_e$ for all edges $e$ lying in~$T$.
Bass--Serre theory says that $\Gamma$ is generated by $\Gamma_T$ and by one element $\gamma_e$ for each edge $e$ \emph{not} in~$T$, with the relations $\gamma_e\,\varphi_{e,0}(g)\,\gamma_e^{-1} = \varphi_{e,1}(g)$ for all $g\in G_e$ for all edges $e$ \emph{not} in~$T$ (see \cite[\S\,I.5.4]{ser77}).

Up to renumbering, we may assume that the edges in~$T$ are $e_1,\dots,e_r$, and that for any $1\leq i\leq r$, the oriented edge $e_i$ goes between a vertex $v_{j_i}$ in the subgraph of~$T$ spanned by $e_1,\dots,e_{i-1}$, and another vertex $v_{j'_i}$ outside of this subgraph.
We set $\Gamma_0 := \pi_1(N_{j_1})$ and for $1\leq i\leq r$, by induction, $\Gamma_i := \Gamma_{i-1} *_{\pi_1(\mathcal{H}_i)} \pi_1(N_{j'_i})$ (amalgamated product).
Then $\Gamma_r = \Gamma_T$.

For each $1\leq i\leq r$, let $X_i \in \mathfrak{o}(p,q+1)$ be an element of the Lie algebra of the centralizer of $\rho_0(\pi_1(\mathcal{H}_i))$ in $\SO(p,q+1)$.
As in \cite[Lem.\,5.6]{jm87}, we can define by induction a representation $\rho : \Gamma_T\to G$ by
\begin{itemize}
  \item $\rho|_{\Gamma_0}: = \rho_0|_{\Gamma_0}$;
  \item for $1\leq i\leq r$, if the geodesic segment from $v_{j_1}$ to $v_{j'_i}$ in~$T$ consists of the edges $e_{k_1},\dots,$ $e_{k_{m_i}},e_i$ in this order (disregarding orientation), then $\rho|_{\Gamma_i}$ is defined by $\rho|_{\Gamma_{i-1}}$~and~by
  $$\rho|_{\pi_1(N_{j'_i})} := (e^{X_i} e^{X_{k_{m_i}}}\dots e^{X_{k_1}}) \circ \rho_0|_{\pi_1(N_{j'_i})} \circ (e^{X_i} e^{X_{k_{m_i}}}\dots e^{X_{k_1}})^{-1}.$$
\end{itemize}
(Indeed, the fact that $X_i$ belongs to the Lie algebra of the centralizer of $\rho_0(\pi_1(\mathcal{H}_i))$ ensures that $\rho|_{\Gamma_{i-1}}$ and $\rho|_{\pi_1(N_{j'_i})}$ agree on $\pi_1(\mathcal{H}_i)$.)

We claim that for some appropriate choice of the~$X_i$, the Zariski closure of $\rho(\Gamma_T)$ in $\SO(p,q+\nolinebreak 1)$ is $\SO(p,r+1)$.
Indeed, by \cite[Lem.\,5.9]{jm87}, for any $1\leq j\leq r$ the group $\rho_0(\pi_1(N_j))$ is Zariski-dense in $\SO(p,1)$.
Since the group $\rho(\Gamma_1)$ is generated by $\rho_0(\pi_1(N_{j_1}))$ and $e^{X_1}\,\rho_0(\pi_1(N_{j'_2}))\,e^{-X_1}$, it has the same Zariski closure in $\SO(p,q+1)$ as the group generated by $\SO(p,1)$ and\linebreak $e^{X_1}\,\SO(p,1)\,e^{-X_1}$.
Up to conjugation in $\SO(p,1)$, the group $\rho(\pi_1(\mathcal{H}_1)) = \rho_0(\pi_1(\mathcal{H}_1))$ is a uniform lattice in $\OO(p-1,1)$, hence its Zariski closure in $\SO(p,q+1)$ has Lie algebra $\mathfrak{o}(p-1,1)$.
By Lemma~\ref{lem:Lie-alg}.\eqref{item:Lie-alg-2} with $q'=1$, we may choose $X_1$ so that the Lie subalgebra of $\mathfrak{o}(p,q+1)$ generated by $\mathfrak{o}(p,1)$ and $\mathrm{ad}(X_1)(\mathfrak{o}(p,1))$ is $\mathfrak{o}(p,2)$.
Then the Zariski closure of the group generated by $\SO(p,1)$ and $e^{X_1}\,\SO(p,1)\,e^{-X_1}$ is $\SO(p,2)$, and so the Zariski closure of $\rho(\Gamma_1)$ in $\SO(p,q+1)$ is also $\SO(p,2)$.
Similarly, by induction, using Lemma~\ref{lem:Lie-alg}.\eqref{item:Lie-alg-2} with $q'=i$, for any $1\leq i\leq r$ we may choose $X_i$ so that the Zariski closure of $\rho(\Gamma_i)$ in $\SO(p,q+1)$ is $\SO(p,i+1)$.
In particular, the Zariski closure of $\rho(\Gamma_T) = \rho(\Gamma_r)$ in $\SO(p,q+1)$ is $\SO(p,r+1)$.

For $r+1\leq i\leq q$, we define by induction $\Gamma_i := \Gamma_{i-1} *_{\pi_1(\mathcal{H}_i)}$ (HNN extension): namely, $\Gamma_i$ is generated by $\Gamma_{i-1}$ and $\gamma_{e_i}$ with the relations $\gamma_{e_i}\,\varphi_{e_i,0}(g)\,\gamma_{e_i}^{-1} = \varphi_{e_i,1}(g)$ for all $g\in G_{e_i} = \pi_1(\mathcal{H}_i)$.
Then $\Gamma_q = \Gamma$.
Again, choose an element $X_i \in \mathfrak{o}(p,q+1)$ in the Lie algebra of the centralizer of $\rho_0(\pi_1(\mathcal{H}_i))$ in $\SO(p,q+1)$.
As in \cite[Lem.\,5.7]{jm87}, we can define by induction a representation $\rho : \Gamma\to G$ by
\begin{itemize}
  \item $\rho|_{\Gamma_r} = \rho|_{\Gamma_T}$ as above;
  \item for $r+1\leq i\leq q$, if the edge $e_i$ goes from $v_{j_i}$ to $v_{j'_i}$ and if the geodesic segment from $v_{j_1}$ to $v_{j_i}$ (\resp $v_{j'_i}$) in~$T$ consists of the edges $e_{k_1},\dots,e_{k_{m_i}}$ (\resp $e_{\ell_1},\dots,e_{\ell_{n_i}}$) in this order (disregarding orientation), then $\rho|_{\Gamma_i}$ is defined by $\rho|_{\Gamma_{i-1}}$~and~by
  $$\rho(\gamma_{e_i}) := (e^{X_{\ell_{n_i}}}\dots e^{X_{\ell_1}}) \, \rho_0(\gamma_{e_i}) \, e^{X_i} \,(e^{X_{k_{m_i}}}\dots e^{X_{k_1}})^{-1}.$$
\end{itemize}
(Indeed, by construction we have $\rho|_{\pi_1(N_{j_i})} = (e^{X_{k_{m_i}}}\dots e^{X_{k_1}}) \circ \rho_0|_{\pi_1(N_{j_i})} \circ (e^{X_{k_{m_i}}}\dots e^{X_{k_1}})^{-1}$ and $\rho|_{\pi_1(N_{j'_i})} = (e^{X_{\ell_{n_i}}}\dots e^{X_{\ell_1}}) \circ \rho_0|_{\pi_1(N_{j'_i})} \circ (e^{X_{\ell_{n_i}}}\dots e^{X_{\ell_1}})^{-1}$; the fact that $X_i$ belongs to the Lie algebra of the centralizer of $\rho_0(\pi_1(\mathcal{H}_i))$ ensures that the relations $\rho(\gamma_{e_i})\,\rho(\varphi_{e_i,0}(g))\,\rho(\gamma_{e_i})^{-1} = \rho(\varphi_{e_i,1}(g))$ are satisfied for all $g\in G_{e_i} = \pi_1(\mathcal{H}_i)$.)

We claim that for some appropriate choice of the~$X_i$, the group $\rho(\Gamma)$ is Zariski-dense in $\SO(p,q+1)$.
Indeed, the Zariski closure of $\rho(\Gamma_r)$ in $\SO(p,q+1)$ is $\SO(p,r+1)$.
Since the group $\rho(\Gamma_{r+1})$ is generated by $\rho(\Gamma_r)$ and $\rho(\gamma_{e_{r+1}})$, it has the same Zariski closure in $\SO(p,q+1)$ as the group generated by $\SO(p,r+1)$ and $\rho(\gamma_{e_{r+1}})$.
Since $\rho(\gamma_{e_{r+1}})$ is equal to $e^{X_{r+1}}$ multiplied on the left and on the right by elements of $\SO(p,r+1)$, the group generated by $\SO(p,r+1)$ and $\rho(\gamma_{e_{r+1}})$ contains the group generated by $\SO(p,r+1)$ and $e^{X_{r+1}}\,\SO(p,r+1)\,e^{-X_{r+1}}$.
By Lemma~\ref{lem:Lie-alg}.\eqref{item:Lie-alg-2} with $q'=r+1$, we may choose $X_{r+1}$ so that the Lie subalgebra of $\mathfrak{o}(p,q+1)$ generated by $\mathfrak{o}(p,r+1)$ and $\mathrm{ad}(X_{r+1})(\mathfrak{o}(p,r+1))$ is $\mathfrak{o}(p,r+2)$.
Then the Zariski closure of the group generated by $\SO(p,r+1)$ and $e^{X_{r+1}}\,\SO(p,r+1)\,e^{-X_{r+1}}$ is $\SO(p,r+2)$, and so the Zariski closure of $\rho(\Gamma_{r+1})$ in $\SO(p,q+1)$ is also $\SO(p,r+2)$.
Similarly, by induction, using Lemma~\ref{lem:Lie-alg}.\eqref{item:Lie-alg-2} with $q'=i$, for any $r+1\leq i\leq q$ we may choose $X_i$ so that the Zariski closure in $\SO(p,q+1)$ of the group $\rho(\Gamma_i)$ generated by $\rho(\Gamma_{i-1})$ and $\rho(\gamma_{e_i})$ is $\SO(p,i+1)$.
In particular, $\rho(\Gamma) = \rho(\Gamma_q)$ is Zariski-dense in $\SO(p,q+1)$.

Finally, we observe that for any neighborhood $\mathcal{U}$ of $\rho_0$ in $\Hom(\Gamma,\SO(p,q+1))$, up to replacing each $X_i$ by $\varepsilon X_i$ for some small $\varepsilon>0$, we may assume that the representation $\rho$ we have constructed belongs to~$\mathcal{U}$.
This completes the proof of Proposition~\ref{prop:Z-dense-deform}.

\begin{remark}
In the proof we could also replace each $X_i$ by $t_i X_i$ for some arbitrary $t_i>0$.
By Theorem~\ref{thm:main}, the representation $\rho : \Gamma\to\SO(p,q+1)$ obtained in this way is still $\H^{p,q}$-convex cocompact, hence in particular injective and discrete (even when the $t_i$ are arbitrarily large, \ie $\rho$ is arbitrary far away from $\rho_0$ in the connected component of $\rho_0$ inside $\Hom(\Gamma,\SO(p,q+1))$).
\end{remark}

\subsubsection{Proof of Corollary~\ref{cor:Z-dense-deform}}

For a standard uniform arithmetic lattice of $\SO(p,1)$, the corresponding closed real hyperbolic manifold $N_0$ admits a closed totally geodesic hypersurface~$\mathcal{H}_0$: see \cite[\S\,7]{jm87} or \cite[\S\,2]{bhw11}.
By \cite[Cor.\,1.12]{bhw11}, the group $\pi_1(N_0)$ is separable over geometrically finite subgroups.
Therefore we can find a finite covering $N$ of~$N_0$ such that $N$ is orientable and contains $q$ lifts of~$\mathcal{H}_0$ which are \emph{pairwise disjoint} two-sided, connected, totally geodesic embedded hypersurfaces in~$N$.
We conclude by applying Proposition~\ref{prop:Z-dense-deform}.

\subsection{Exotic examples via Coxeter groups} \label{subsec:appendix-Coxeter}

Lee--Marquis \cite[Th.\,E]{lm19} found examples of Coxeter groups $\Gamma$ in $p+3$ generators (for $p=4$ or~$6$, see Table~\ref{table:Coxeter-diag}) with $\vcd(\Gamma) = p$, which are not commensurable to lattices in $\PO(p,1)$, and which admit pairs $(\rho_1,\rho_2)$ of representations as reflection groups in~$\R^{p+2}$ in the sense of Vinberg \cite{vin71} which are $\H^{p,1}$-convex cocompact and such that $\rho_1$ cannot be continuously deformed to~$\rho_2$ inside the space of $\H^{p,1}$-convex cocompact representations of~$\Gamma$.
We now observe that these representations can be deformed to representations in $\Hom(\Gamma,\GL(p+3,\R))$ which are $\H^{p,2}$-convex cocompact and whose image is Zariski-dense in $\OO(p,3)$.

\begin{prop} \label{prop:Z-dense-deform-LM}
Let $\Gamma$ be a Coxeter group in $p+3$ generators as in Table~\ref{table:Coxeter-diag} below, where $p \in \{ 4,6\}$.
Then $\Gamma$ is a Gromov hyperbolic group with $\vcd(\Gamma) = p$ and there exist a one-parameter family $(\rho_t)_{t>0} \subset \Hom(\Gamma,\GL(p+3,\R))$ of representations of $\Gamma$ as a reflection group in~$\R^{p+3}$ in the sense of Vinberg~\cite{vin71} and two positive numbers $t_1<t_2$ such that
\begin{enumerate}[label=(\roman*),ref=\roman*]
  \item for any $t\in (t_1,t_2)$, the group $\rho_t(\Gamma)$ preserves a quadratic form $Q_t$ of signature $(p,3|0)$ on $\R^{p+3}$, is Zariski-dense in $\mathrm{Aut}(Q_t) \simeq \OO(p,3)$ and $\H^{p,2}$-convex cocompact;
  \item for any $t\in (0,t_1)\cup (t_2,+\infty)$, the group $\rho_t(\Gamma)$ preserves a quadratic form $Q_t$ of signature $(p+1,2|0)$ on $\R^{p+3}$, is Zariski-dense in $\mathrm{Aut}(Q_t) \simeq \OO(p+1,2)$ and $\H^{p+1,1}$-convex cocompact;
  \item for any $t\in\{ t_1,t_2\}$, the group $\rho_t(\Gamma)$ preserves a hyperplane $V_t$ of $\R^{p+3}$ and a quadratic form $Q_t$ of signature $(p,2|0)$ on~$V_t$, and the restriction of $\rho_t(\Gamma)$ to~$V_t$ is Zariski-dense in $\mathrm{Aut}(Q_t) \simeq \OO(p,2)$ and $\H^{p,1}$-convex cocompact.
\end{enumerate}
\end{prop}

Case~(iii) is \cite[Th.\,E]{lm19}.
Below we explain how to obtain cases (i) and~(ii) from computations made in \cite[\S\,8]{lm19} for the various examples in Table~\ref{table:Coxeter-diag}.

\medskip

\begin{table}[ht!]
\begin{tabular}{ccc}
\begin{tikzpicture}[thick,scale=0.7, every node/.style={transform shape}]
\node[draw,circle] (1) at (36-18:1.0514) {};
\node[draw,circle] (2) at (108-18:1.0514){};
\node[draw,circle] (3) at (180-18:1.0514){};
\node[draw,circle,fill=dark-gray] (4) at (252-18:1.0514){};
\node[draw,circle,fill=dark-gray] (5) at (324-18:1.0514){};
\draw (1)--(2)--(3)--(4)--(5)--(1);
\node[draw,circle,right=1cm of 1] (6) {};
\node[draw,circle,left=1cm of 3] (7) {};
\draw (1)--(6) node[above,midway] {$\ell \geqslant 9$};
\draw (3)--(7) node[above,midway] {$k \geqslant 9$};
\draw (4)--(5) node[below,midway] {$\infty$};
\end{tikzpicture}
&
\begin{tikzpicture}[thick,scale=0.7, every node/.style={transform shape}]
\node[draw,circle] (1) at (36-18:1.0514) {};
\node[draw,circle] (2) at (108-18:1.0514){};
\node[draw,circle] (3) at (180-18:1.0514){};
\node[draw,circle,fill=dark-gray] (4) at (252-18:1.0514){};
\node[draw,circle,fill=dark-gray] (5) at (324-18:1.0514){};
\draw (1)--(2)--(3)--(4)--(5)--(1);
\node[draw,circle,right=1cm of 1] (6) {};
\node[draw,circle] (7) at ($(0,0)$) {};
\draw (1)--(6) node[above,midway] {$k \geqslant 11$};
\draw (4)--(5) node[below,midway] {$\infty$};
\draw (4)--(3);
\draw (5)--(1);
\draw (7)--(2);
\draw (7)--(3) node[midway] {$5$};
\draw (7)--(4) ;
\end{tikzpicture}
&
\begin{tikzpicture}[thick,scale=0.7, every node/.style={transform shape}]
\node[draw,circle,fill=dark-gray] (1) at (0,0) {};
\node[draw,circle] (2) at (0,2) {};
\node[draw,circle] (4) at (2,2) {};
\node[draw,circle] (3) at (1,2) {};
\node[draw,circle] (5) at (2,1) {};
\node[draw,circle,fill=dark-gray] (6) at (2,0) {};
\draw (1)--(2)--(3)--(4)--(5)--(6)--(1);
\node[draw,circle] (7) at (3,1) {};
\draw (5)--(7) node[above,midway] {$\ell \geqslant 8$};
\draw (2)--(3) node[above,midway] {$k \geqslant 7$};
\draw (4)--(3) node[above,midway] {$4$};
\draw (1)--(2) node[left,midway] {$4$};
\draw (1)--(6) node[below,midway] {$\infty$};
\end{tikzpicture}
\\
\begin{tikzpicture}[thick,scale=0.7, every node/.style={transform shape}]
\node[draw,circle] (1) at (36-18:1.0514) {};
\node[draw,circle] (2) at (108-18:1.0514){};
\node[draw,circle] (3) at (180-18:1.0514){};
\node[draw,circle,fill=dark-gray] (4) at (252-18:1.0514){};
\node[draw,circle,fill=dark-gray] (5) at (324-18:1.0514){};
\draw (1)--(2)--(3)--(4)--(5)--(1);
\node[draw,circle,right=1cm of 1] (6) {};
\node[draw,circle,left=1cm of 3] (7) {};
\draw (1)--(6) node[above,midway] {$\ell \geqslant 8$};
\draw (3)--(7) node[above,midway] {$k \geqslant 8$};
\draw (4)--(5) node[below,midway] {$\infty$};
\draw (4)--(3) node[left,midway] {$4$};
\draw (5)--(1) node[right,midway] {$4$};
\end{tikzpicture}
&
\begin{tikzpicture}[thick,scale=0.7, every node/.style={transform shape}]
\node[draw,circle] (1) at (36-18:1.0514) {};
\node[draw,circle] (2) at (108-18:1.0514){};
\node[draw,circle] (3) at (180-18:1.0514){};
\node[draw,circle,fill=dark-gray] (4) at (252-18:1.0514){};
\node[draw,circle,fill=dark-gray] (5) at (324-18:1.0514){};
\draw (1)--(2)--(3)--(4)--(5)--(1);
\node[draw,circle,right=1cm of 1] (6) {};
\node[draw,circle] (7) at ($(0,0)$) {};
\draw (1)--(6) node[above,midway] {$k \geqslant 9$};
\draw (4)--(5) node[below,midway] {$\infty$};
\draw (4)--(3) node[left,midway] {$4$};
\draw (5)--(1) node[right,midway] {$4$};
\draw (7)--(2);
\draw (7)--(3) node[midway] {$4$};
\draw (7)--(4) node[right,midway] {$4$};
\end{tikzpicture}
&
\begin{tikzpicture}[thick,scale=0.7, every node/.style={transform shape}]
\node[draw,circle,fill=dark-gray] (1) at (0,0) {};
\node[draw,circle] (2) at (0,2) {};
\node[draw,circle] (3) at (1,2) {};
\node[draw,circle] (4) at (2,2) {};
\node[draw,circle] (5) at (2,1) {};
\node[draw,circle,fill=dark-gray] (6) at (2,0) {};
\draw (1)--(2)--(3)--(4)--(5)--(6)--(1);
\node[draw,circle] (7) at (3,1) {};
\draw (5)--(7) node[midway] {$4$};
\draw (2)--(3) node[above,midway] {$k \geqslant 9$};
\draw (4)--(3) node[above,midway] {$4$};
\draw (1)--(2) node[left,midway] {$4$};
\draw (1)--(6) node[below,midway] {$\infty$};
\draw (4)--(7);
\draw (6)--(7);
\end{tikzpicture}
\\
&
\begin{tikzpicture}[thick,scale=0.7, every node/.style={transform shape}]
\node[draw,circle] (1) at (36-18:1.0514) {};
\node[draw,circle] (2) at (108-18:1.0514){};
\node[draw,circle] (3) at (180-18:1.0514){};
\node[draw,circle,fill=dark-gray] (4) at (252-18:1.0514){};
\node[draw,circle,fill=dark-gray] (5) at (324-18:1.0514){};
\draw (1)--(2)--(3)--(4)--(5)--(1);
\node[draw,circle,right=1cm of 1] (6) {};
\node[draw,circle,right=1cm of 6] (7) {};
\node[draw,circle,right=1cm of 7] (8) {};
\draw (1)--(6)--(7)--(8);
\draw (7)--(8) node[above,midway] {$5$};
\node[draw,circle,left=1cm of 3] (9) {};
\draw (3)--(9) node[below,midway] {$k \geqslant 11$};
\draw (4)--(5) node[below,midway] {$\infty$};
\end{tikzpicture}
&
\end{tabular}
\caption{Coxeter diagrams of some Gromov hyperbolic Coxeter groups to which Proposition~\ref{prop:Z-dense-deform-LM} applies, taken from \cite[Tables 12--13]{lm19}. In the top-left (\resp top-right, \resp bottom-left) example we ask $(k,\ell) \neq (9,9),(9,10),\dots$, $(9,18),(10,10)$ (\resp $(k,\ell) \neq (7,8),(7,9),(8,8)$, \resp $(k,\ell) \neq (8,8)$).}
\label{table:Coxeter-diag}
\end{table}

\begin{proof}[Proof of Proposition~\ref{prop:Z-dense-deform-LM}]
For the fact that $\Gamma$ is a Gromov hyperbolic group with $\vcd(\Gamma) = p$, see \cite[\S\,8]{lm19}.
To simplify notation, we set $N := p+3$.
Consider a presentation of~$\Gamma$ by generators and relations given by the Coxeter diagram of~$\Gamma$ in Table~\ref{table:Coxeter-diag}:
$$\Gamma = \big\langle s_1,\dots,s_N ~|~ (s_i s_j)^{m_{i,j}}=1\ \forall \ 1\leq i,j\leq N\big\rangle,$$
where $m_{i,i}=1$ and $m_{i,j} = m_{j,i} \in\{2,3,\dots,\infty\}$ for $i\neq j$.
(By convention, $(s_i s_j)^{\infty}=1$ means that $s_is_j$ has infinite order in the group~$\Gamma$.)
For each $t\geq 0$, let $\mathcal{A}_t$ be the $(N\times N)$ real matrix whose $(i,j)$-entry is given by $-2\cos(\pi/m_{i,j})$ if $m_{i,j}\neq\infty$, and $-2-t$ if $m_{i,j}=\infty$.
Let $(e_1,\dots,e_N)$ be the canonical basis of~$\R^N$, and $(e_1^*,\dots,e_N^*)$ its dual basis.
For each $1\leq i\leq N$, we set $v_{i,t} := \mathcal{A}_t\cdot e_i$ and let
$$\rho_t(s_i) := \big(v \longmapsto v - e_i^*(v)\,v_{i,t}\big)$$
be the reflection in the hyperplane $\spa\{ e_j\,|\,j\neq i\}$ satisfying $\rho_t(s_i)(v_{i,t}) = -v_{i,t}$.
This defines a representation $\rho_t : \Gamma\to\GL(N,\R)$ of $\Gamma$ as a reflection group in~$\R^N$ in the sense of Vinberg \cite{vin71}: see \eg \cite[\S\,3.3]{dgklm}.
The representation $\rho_t$ preserves the subspace $V_t := \spa(v_{1,t},\dots,v_{N,t}) = \mathrm{Im}(\mathcal{A}_t)$, and acts irreducibly on~$V_t$ (see \eg \cite[Prop.\,3.23.(2)]{dgklm}).
Since $\mathcal{A}_t$ is symmetric, there is a non-degenerate quadratic form $Q_t$ on~$V_t$ such that $Q_t(v_i,v_j)$ is the $(i,j)$-entry of~$\mathcal{A}_t$ for all $1\leq i,j\leq N$ (see \cite[Th.\,6]{vin71}).
Let $(p_t,q_t|0)$ be the signature of~$Q_t$.
Then the restriction of $\rho_t(\Gamma)$ to~$V_t$ is Zariski-dense in $\mathrm{Aut}(Q_t) \simeq \OO(p_t,q_t)$ by an easy generalization of \cite{bh04}: see \cite{adlm}.
Moreover, if $t>0$, then this restriction is $\H^{p_t,q_t-1}$-convex cocompact: see \cite{dgk18,lm19}.
We now compute the signature $(p_t,q_t|0)$ of~$Q_t$ by computing the signature of the symmetric matrix~$\mathcal{A}_t$.

First, observe that if $I$ is the subset of $\{1,\dots,N\}$ corresponding to the $N-2$ white vertices in the Coxeter diagram of Table~\ref{table:Coxeter-diag}, then the symmetric matrix $\mathcal{A}_t^I$ obtained from $\mathcal{A}_t$ by restricting to coefficients $(i,j)$ in~$I^2$ has signature $(N-3,1|0)$.
Indeed, the Coxeter diagram associated to~$I$ contains a Coxeter sub-diagram with $N-3$ vertices which is spherical (of type $I_2(k) \times I_2(m)$ for some $m\geq 3$, or type $I_2(k)\times H_4$ in the last example), and also a Coxeter sub-diagram with $3$ vertices which is Lann\'er (corresponding to $i_1,i_2,i_3 \in \{1,\dots,N\}$ with $m_{i_1,i_2} = k$, $m_{i_1,i_3} = 2$, and $m_{i_2,i_3} = 3$ or~$4$).
This implies (see \cite{lm19}) that $\mathcal{A}_t^I$ contains a submatrix of the form $\mathcal{A}_t^{I'}$ with $\# I' = N-3$ which is positive definite, and also a submatrix of the form $\mathcal{A}_t^{I''}$ with $\# I'' = 3$ with signature $(2,1)$.
Therefore $\mathcal{A}_t^I$ has signature $(N-3,1|0)$.

Note that $t \mapsto \det(\mathcal{A}_t)$ is a second-degree polynomial, because there is exactly one pair $(i,j)$ in $\{ 1,\dots,N\}$ such that $m_{i,j} = \infty$.
By \cite[Th.\,E \& \S\,8]{lm19}, this polynomial has two positive roots $t_1 < t_2$, and the symmetric matrices $\mathcal{A}_{t_1}$ and $\mathcal{A}_{t_2}$ have signature $(N-3,2|1)$, which implies that $V_{t_1}$ and~$V_{t_2}$ are hyperplanes and that $(p_{t_1},q_{t_1}) = (p_{t_2},q_{t_2}) = (N-3,2)$.
For $t\in (0,+\infty) \smallsetminus \{t_1,t_2\}$, the symmetric matrix $\mathcal{A}_t$ is non-degenerate, hence $V_t = \R^N$ and the signature $(p_t,q_t|0)$ of $Q_t$ is equal to the signature of~$\mathcal{A}_t$, which satisfies $p_t\geq N-3$ and $q_t\geq 1$ by the above observation.
By developing along the last row and the last column, we see that the leading coefficient of the polynomial $t \mapsto \det(\mathcal{A}_t)$ is $-\det(\mathcal{A}_t^I)$, which is positive by the above observation.
We deduce that $(p_t,q_t) = (N-2,2)$ if $t\in (0,t_1)\cup (t_2,+\infty)$, and $(p_t,q_t) = (N-3,3)$ if $t\in (t_1,t_2)$.
\end{proof}

\subsection{Exotic examples via Gromov--Thurston manifolds} \label{subsec:appendix-Gromov-Thurston}

Gromov--Thurston \cite{gt87} constructed closed orientable manifolds admitting a negatively-curved Riemannian metric, but no hyperbolic metric.
They are defined as follows.
Fix any $p\geq 4$.
By \cite{gt87}, there exist closed oriented hyperbolic $p$-dimensional manifolds $N$ admitting two connected embedded totally geodesic hypersurfaces $\mathcal{H}_1$ and $\mathcal{H}_2$ such that
\begin{itemize}
  \item each $\mathcal{H}_i$ is the set of fixed points of some isometric involution $\sigma_i$ of~$N$,
  \item $\mathcal{H}_i$ is homologically trivial,
  \item $\mathcal{S} := \mathcal{H}_1 \cap \mathcal{H}_2$ is connected,
  \item $\mathcal{H}_1$ and~$\mathcal{H}_2$ intersect along~$\mathcal{S}$ with an angle $\pi/n$, where $n\geq 2$.
\end{itemize}
We call such an~$N$ an \emph{$n$-dihedral} hyperbolic manifold.
Indeed, the group generated by $\sigma_1$ and~$\sigma_2$ is isomorphic to the dihedral group $D_{2n}$.
The product $\sigma_1\sigma_2$ generates an index-two subgroup which is cyclic of order~$n$.
Let $\underline{N}$ be the quotient of~$N$ by this cyclic subgroup.
The quotient map $N \to \underline{N}$ is a ramified covering of degree~$n$.
A cyclically ramified cover $N^{k/n}$ of~$\underline{N}$ of degree $k\neq n$ is called a $k$-ramified \emph{Gromov--Thurston manifold} over~$N$.

Monclair--Schlenker--Tholozan \cite[Th.\,1.1--1.2]{mst} constructed $\H^{p,1}$-convex cocompact representations of $\pi_1(N^{k/n})$ into $\SO(p,2)$ for such Gromov--Thurston manifolds $N^{k/n}$ with $k>n$.
These representations have Zariski-dense image in $\SO(p,2)$: see \cite[Prop.\,1.4]{gm24}.
We now observe that these representations can be deformed, inside $\Hom(\pi_1(N^{k/n}),\SO(p,q+\nolinebreak 1))$, into $\H^{p,q}$-convex cocompact representations whose image is Zariski-dense in $\SO(p,q+1)$, for any $1\leq q\leq 2k-2$.

\begin{prop} \label{prop:Z-dense-deform-MST}
For $p\geq 4$ and $k>n\geq 2$, let $N^{k/n}$ be a $k$-ramified Gromov--Thurston manifold over an $n$-dihedral hyperbolic manifold of dimension~$p$.
Then for any $1\leq q\leq 2k-2$, the $\H^{p,1}$-convex cocompact representations from $\pi_1(N^{k/n})$ to $\SO(p,2)$ constructed in \cite{mst} can be deformed continuously inside $\Hom(\pi_1(N^{k/n}),\SO(p,q+\nolinebreak 1))$ to representations whose image is Zariski-dense in $\SO(p,q+1)$.
\end{prop}

These continuous deformations are $\H^{p,q}$-convex cocompact by Theorem~\ref{thm:main}.

\subsubsection{Preliminary observations} \label{subsubsec:prelim-Gromov-Thurston}

As in \cite{mst}, for $k>n\geq 2$ we define an \emph{equilateral spacelike $2k$-polygon of side length $\pi/n$ in $\R^{2,1}$} to be a cyclically ordered family $\mathcal{P} = (v_j)_{j\in\Z/2k\Z}$ of points of $\R^{2,1}$ such that for any $j\in\Z/2k\Z$, we have $\sfb(v_j,v_j) = 1$ and $\sfb(v_j,v_{j+1}) = \cos(\pi/n)$, the vectors $v_{j+1} - \cos(\pi/n) v_j$ and $v_{j-1} - \cos(\pi/n) v_j$ belong to two different connected components of spacelike vectors in $v_j^{\perp} \simeq \R^{1,1}$, and the interior of the segment $[v_i,v_{i+1}]$ does not intersect the segment $[v_j,v_{j+1}]$ for $i\neq j$.
We denote by $\mathcal{E}_{k,n}$ the space of such polygons.

For instance, if $(e_1,e_2,e_3)$ is an orthogonal basis of~$\R^{2,1}$ in which the quadratic form $\sfb$ has matrix $\mathrm{diag}(1,1,-1)$, then we can take
\begin{equation} \label{eqn:v_j}
v_j := \sqrt{\alpha_{k,n}} \, \Big(\cos\Big(\frac{j\pi}{k}\Big)\,e_1 + \sin\Big(\frac{j\pi}{k}\Big)\,e_2\Big) + \sqrt{\alpha_{k,n} - 1} \ e_3
\end{equation}
where $\alpha_{k,n} := (1 - \cos(\pi/n))/(1 - \cos(\pi/k))$.

\begin{remark} \label{rem:polygon}
For $k>n\geq 2$, if $\mathcal{P} = (v_j)_{j\in\Z/2k\Z}$ is an equilateral spacelike $2k$-polygon of side length $\pi/n$ in $\R^{2,1}$, then for each $j\in\Z/2k\Z$ the subspace $\spa(v_j,v_{j+1})$ is a $2$-plane on which $\sfb$ is positive definite, and there exists $j_0\in\Z/2k\Z$ such that $\spa(v_{j_0-1},v_{j_0},v_{j_0+1}) = \R^{2,1}$.
\end{remark}

\begin{lem} \label{lem:deform-triple-R2q}
For $q\geq 1$, let $\sfb$ be the standard quadratic form of signature $(2,q)$ on $\R^{2,q}$.
Let $v_0,v_2\in\R^{2,q}$ be two distinct vectors such that $\sfb(v_0,v_0) = \sfb(v_2,v_2) = 1$.
For $0\leq c<1$, let $\mathcal{V}_{v_0,v_2,c}$ be the set of vectors $v_1\in\R^{2,q}$ such that $\sfb(v_1,v_1) = 1$ and $\sfb(v_0,v_1) = \sfb(v_1,v_2) = c$.
\begin{enumerate}
  \item\label{item:triple-R2q-1} The restriction of~$\sfb$ to $\spa(v_0,v_2)$ has signature $(2,0|0)$ (\resp $(1,1|0)$) if and only if $|\sfb(v_0,v_2)|<1$ (\resp $|\sfb(v_0,v_2)|>1$).
  \item\label{item:triple-R2q-2} Suppose that $c\neq 0$ and $-1 < \sfb(v_0,v_2) < 2c^2-1$, or that $c=0$ and $|\sfb(v_0,v_2)|>1$.
  Then $\mathcal{V}_{v_0,v_2,c}$ is non-empty and does not intersect $\spa(v_0,v_2)$.
  If $q=1$, then $\mathcal{V}_{v_0,v_2,c}$ consists of two vectors, which vary continuously with $v_0$ and~$v_2$ (for fixed~$c$).
  \item\label{item:triple-R2q-3} For any $v_1 \in \mathcal{V}_{v_0,v_2,c}$ such that $\spa(v_0,v_1,v_2) = \R^{2,1} \subset \R^{2,q}$, there is a continuous path $(g_t)_{t\in [0,1]} \subset \SO(2,q)$ with $g_0 = \mathrm{Id}$ such that $g_t\cdot v_0 = v_0$ and $g_t\cdot v_2 = v_2$ and $g_t\cdot v_1 \in \mathcal{V}_{v_0,v_2,c} \smallsetminus \R^{2,q-1}$ for all $t>0$.
\end{enumerate}
\end{lem}

\begin{proof}
\eqref{item:triple-R2q-1} The polynomial $t \mapsto \sfb(v_0+tv_2,v_0+tv_2) = t^2 + 2\sfb(v_0,v_2)\,t + 1$ admits no real roots (\resp two distinct real roots) if and only if $|\sfb(v_0,v_2)|<1$ (\resp $|\sfb(v_0,v_2)|>1$).

\eqref{item:triple-R2q-2} Suppose that $c\neq 0$ and $-1 < \sfb(v_0,v_2) < 2c^2-1$.
By~\eqref{item:triple-R2q-1}, the restriction of~$\sfb$ to $\spa(v_0,v_2)$ has signature $(2,0|0)$.
Let $v$ be the unique vector of $\spa(v_0,v_2)$ such that $\sfb(v,v) =\nolinebreak 1$ and $c' := \sfb(v,v_0) = \sfb(v,v_2) > 0$.
We have $\sfb(v_0,v_2) = \cos(2\arccos(c')) = 2 {c'}^2 - 1$, hence the inequality $\sfb(v_0,v_2) < 2c^2-1$ implies $c'<c$.
Let $v'_1 := (c/c')\,v$.
Then $1 - \sfb(v'_1,v'_1) = 1 - (c/c')^2 < 0$, and
$$\mathcal{V}_{v_0,v_2,c} = \big\{ v'_1+w ~|~ w\in\spa(v_0,v_2)^{\perp},\ \sfb(w,w) = 1 - \sfb(v'_1,v'_1)\big\}.$$
The restriction of~$\sfb$ to $\spa(v_0,v_2)^{\perp} \simeq \R^{0,q}$ is negative definite, hence $\mathcal{V}_{v_0,v_2,c}$ is non-empty and does not intersect $\spa(v_0,v_2)$.
If $q=1$, then $\mathcal{V}_{v_0,v_2,c}$ consists of two vectors, which vary continuously with $v_0$ and~$v_2$ (for fixed~$c$).

Suppose that $c=0$ and $|\sfb(v_0,v_2)|>1$.
Then $\mathcal{V}_{v_0,v_2,c} = \{ v_1\in\spa(v_0,v_2)^{\perp} ~|~ \sfb(v_1,v_1) = 1\}$.
By~\eqref{item:triple-R2q-1}, the restriction of~$\sfb$ to $\spa(v_0,v_2)^{\perp}$ has signature $(1,q-1|0)$, hence $\mathcal{V}_{v_0,v_2,c}$ is non-empty; moreover, $\mathcal{V}_{v_0,v_2,c}$ does not intersect $\spa(v_0,v_2)$.
If $q=1$, then $\mathcal{V}_{v_0,v_2,c}$ consists of two vectors, which vary continuously with $v_0$ and~$v_2$ (for fixed~$c$).

\eqref{item:triple-R2q-3} Fix $v_1 \in \mathcal{V}_{v_0,v_2,c}$ such that $\spa(v_0,v_1,v_2) = \R^{2,1} \subset \R^{2,q}$.
Choose a nonzero vector $w \in \R^{2,q}$ in the orthogonal of $\R^{2,q-1}$; then $\sfb(w,w) < 0$.

Suppose that the restriction of $\sfb$ to $\spa(v_0,v_2)$ is non-degenerate.
Then the restriction of $\sfb$ to the $2$-plane $E := (\spa(v_0,v_2)^{\perp}\cap\R^{2,1}) \oplus \R w$ is non-degenerate.
The orthogonal projection of $v_1$ to~$E$ is non-zero and non-isotropic for~$\sfb$.
The identity component of the fixator of $E^{\perp}$ in $\SO(2,q)$ is a one-parameter subgroup $(g_t)_{t\in\R}$ such that $g_t\cdot v_1 \notin \R^{2,q-1}$ for all $0<t\leq 1$.
Since it fixes $v_0$ and~$v_2$, it preserves $\mathcal{V}_{v_0,v_2,c}$.

Suppose that the restriction of $\sfb$ to $\spa(v_0,v_2)$ is degenerate.
Consider a non-zero vector $v \in \spa(v_0,v_2)\cap\spa(v_0,v_2)^{\perp}$; then $\sfb(v,v) = 0$ and $\spa(v_0,v_2) = \spa(v_0,v)$.
The restriction of $\sfb$ to $E := (v_0^{\perp} \cap \R^{2,1}) \oplus \R w$ has signature $(1,2)$.
The orthogonal projection of $v_1$ to~$E$ is non-zero and not a multiple of~$v$.
The identity component of the fixator of $E^{\perp}\oplus\R v$ in $\SO(2,q)$ is a one-parameter unipotent subgroup $(g_t)_{t\in\R}$ such that $g_t\cdot v_1 \notin \R^{2,q-1}$ for all $t>0$.
Since it fixes $v_0$ and~$v_2$, it preserves $\mathcal{V}_{v_0,v_2,c}$.
\end{proof}

\begin{lem} \label{lem:deform-quadruple-R21}
Fix $0\leq c<1$.
Let $v_{-1},v_0,v_1,v_2\in\R^{2,1}$ satisfy $\sfb(v_j,v_j) = 1$ for all $-1\leq j\leq 2$ and $\sfb(v_{-1},v_0) = \sfb(v_0,v_1) = \sfb(v_1,v_2) = c$.
Suppose that $\spa(v_{-1},v_0,v_1) = \R^{2,1}$ and $\spa(v_0,v_2) = \spa(v_0,v_1)$.
Then there exist $t_0>0$ and continuous deformations $(v_{0,t})_{t\in [0,t_0]}$ of $v_0 = v_{0,0}$ and $(v_{1,t})_{t\in [0,t_0]}$ of $v_1 = v_{1,0}$ such that $\sfb(v_{0,t},v_{0,t}) = \sfb(v_{1,t},v_{1,t}) = 1$ and $\sfb(v_{-1},v_{0,t}) = \sfb(v_{0,t},v_{1,t}) = \sfb(v_{1,t},v_2) = c$ for all~$t$, and such that $\spa(v_{0,t},v_{1,t},v_2) = \R^{2,1}$ for all $t>0$.
\end{lem}

\begin{proof}
By Lemma~\ref{lem:deform-triple-R2q}.\eqref{item:triple-R2q-1}, the restriction of $\sfb$ to $\spa(v_{-1},v_0)$ has signature $(2,0|0)$, and similarly for $\spa(v_0,v_1)$.

We can write $v_0 = cv_{-1} + v''_0$ where $v''_0 \in (v_{-1})^{\perp} \cap \spa(v_{-1},v_0)$ satisfies $\sfb(v''_0,v''_0) = 1-c^2 > 0$.
Consider a non-zero vector $w \in \spa(v_{-1},v_0)^{\perp}$; then $\sfb(w,w)<0$.
Since $v_{-1} \notin \spa(v_0,v_1) = \spa(v_0,v_2) \simeq \R^{2,0}$, we have $v_2 \notin \spa(v_{-1},v_0)$, hence $\sfb(w,v_2) \neq 0$.
Up to replacing $w$ by $-w$, we may assume that $\sfb(w,v_2) < 0$.
For any $t\geq 0$ we have $\sfb(v_0+tv''_0,v_0+tv''_0) - 1 = (2t+t^2)(1-c^2) \geq 0$, and so we can define a continuous deformation $(v_{0,t})_{t\geq 0}$ of $v_0 = v_{0,0}$ by
$$v_{0,t} := v_0 + tv''_0 + \sqrt{\frac{\sfb(v_0+tv''_0,v_0+tv''_0)-1}{|\sfb(w,w)|}}\ w.$$
It satisfies $\sfb(v_{-1},v_{0,t}) = c$ and $\sfb(v_{0,t},v_{0,t}) = 1$ for all $t\geq 0$.

Since the restriction of $\sfb$ to $\spa(v_0,v_2) = \spa(v_0,v_1)$ has signature $(2,0|0)$ and since $\sfb(v_0,v_1) = \sfb(v_1,v_2) = c$, we have $\sfb(v_0,v_2) = \cos(2\arccos(c)) = 2c^2 - 1$.
On the other hand,
$$\sfb(v_{0,t},v_2) = \sfb(v_0,v_2) + \sqrt{\frac{2(1-c^2)t}{|\sfb(w,w)|}}\ \sfb(w,v_2) + O(t)$$
as $t\to 0$, and $\sfb(w,v_2) < 0$ by construction, hence there exists $t_0>0$ such that $-1 < \sfb(v_{0,t},v_2) < 2c^2-1$ for all $0<t\leq t_0$.
By Lemma~\ref{lem:deform-triple-R2q}.\eqref{item:triple-R2q-2}, we can find a continuous deformation $(v_{1,t})_{t\in [0,t_0]}$ of $v_1 = v_{1,0}$ such that $\sfb(v_{1,t},v_{1,t}) = 1$ and $\sfb(v_{0,t},v_{1,t}) = \sfb(v_{1,t},v_2) = c$ for all~$t$, and such that $\spa(v_{0,t},v_{1,t},v_2) = \R^{2,1}$ for all $t>0$.
\end{proof}

\begin{lem} \label{lem:deform-polygon-R21}
Let $k>n\geq 2$.
Any element $\mathcal{P} = (v_j)_{j\in\Z/2k\Z}$ of~$\mathcal{E}_{k,n}$ can be continuously deformed inside~$\mathcal{E}_{k,n}$ into an element $\mathcal{P}' = (v'_j)_{j\in\Z/2k\Z}$ such that $\spa(v'_{j-1},v'_j,v'_{j+1}) = \R^{2,1}$ for all $j\in\Z/2k\Z$.
\end{lem}

\begin{proof}
By Remark~\ref{rem:polygon}, there exists $j_0\in\Z/2k\Z$ such that $\spa(v_{j_0-1},v_{j_0},v_{j_0+1}) = \R^{2,1}$.
It follows from Lemma~\ref{lem:deform-quadruple-R21} that for any $j_0 \in \Z/2k\Z$, if $\spa(v_{j_0-1}, v_{j_0},v_{j_0+1}) = \R^{2,1}$ but $\spa(v_{j_0},v_{j_0+1}, v_{j_0+2}) \neq \R^{2,1}$, then we can continuously deform $\mathcal{P}$ inside~$\mathcal{E}_{k,n}$ (by varying $v_{j_0}$ and $v_{j_0+1}$ and keeping all the other $v_j$ fixed) into an element $\mathcal{P}'' = (v''_j)_{j\in\Z/2k\Z}$ of~$\mathcal{E}_{k,n}$ such that $\spa(v''_{j_0},v''_{j_0+1},v''_{j_0+2}) = \R^{2,1}$; moreover, by taking the deformation small enough we can ensure that $\spa(v''_{j_0-1},v''_{j_0},v''_{j_0+1}) = \R^{2,1}$, and that $\spa(v''_{j_0-2},v''_{j_0-1},v''_{j_0}) = \R^{2,1}$ as soon as $\spa(v_{j_0-2},v_{j_0-1},v_{j_0}) = \R^{2,1}$.
By applying this process iteratively, we can continuously deform $\mathcal{P}$ inside~$\mathcal{E}_{k,n}$ into an element $\mathcal{P}' = (v'_j)_{j\in\Z/2k\Z}$ such that $\spa(v'_{j-1},v'_j,v'_{j+1}) = \R^{2,1}$ for all~$j$.
\end{proof}

\subsubsection{Proof of Proposition~\ref{prop:Z-dense-deform-MST}}

Let $\sigma_1$ and~$\sigma_2$ be the involutions of~$N$ defining its $n$-dihedral structure and, as above, let $\mathcal{H}_1$ (\resp $\mathcal{H}_2$) be the set of fixed points of $\sigma_1$ (\resp $\sigma_2$) and $\mathcal{S} := \mathcal{H}_1 \cap \mathcal{H}_2$.
Let $H_1^N$ (\resp $H_2^N$) be the closure of a connected component of $\mathcal{H}_1 \smallsetminus \mathcal{S}$ (\resp $\mathcal{H}_2 \smallsetminus \mathcal{S}$), chosen so that the oriented angle at~$\mathcal{S} $ from $H_1^N$ to~$H_2^N$ is $\pi/n$.
Let $\underline{H}_1$ and~$\underline{H}_2$ be the projections of $H_1^N$ and~$H_2^N$ to $\underline{N} = N/\langle\sigma_1\sigma_2\rangle$.
As in \cite[\S\,2.2]{mst}, we denote by $H_1,\dots,H_{2k}$ the lifts of $\underline{H}_1$ and~$\underline{H}_2$ to the ramified cover $N^{k/n}$, in cyclic order around $S := H_1\cap H_2$, and by $V_j$ the connected component of $N^{k/n} \smallsetminus \bigcup_{i=1}^{2k} H_i$ bounded by $H_j$ and~$H_{j+1}$, for each $j\in\Z/2k\Z$.

Recall the double cover $\hat{\H}^{p,q}$ of $\H^{p,q}$ from \eqref{eqn:Hpq-double}.
We choose a point $z \in \hat{\H}^{p,1} \subset \hat{\H}^{p,q}$ and a totally geodesic copy $Z$ of $\H^{p-2}$ containing~$z$ in $\hat{\H}^{p,1} \subset \hat{\H}^{p,q}$.
We can write the tangent space $T_z\hat{\H}^{p,1} \simeq \R^{p,1}$ as the direct sum of $T_zZ \simeq \R^{p-2,0}$ and of its orthogonal $(T_zZ)^{\perp} \simeq \R^{2,1}$.
We now consider equilateral spacelike $2k$-polygons of side length $\pi/n$ in $\R^{2,1}$ as in Section~\ref{subsubsec:prelim-Gromov-Thurston}, which we see as subsets of $(T_zZ)^{\perp} \subset T_z\hat{\H}^{p,1}$.
As in \cite[Lem.\,5.4]{mst}, such a polygon $\mathcal{P} = (v_j)_{j\in\Z/2k\Z}$ defines a spacelike $p$-graph (Definition~\ref{def:weakly-sp-gr}) in~$\hat{\H}^{p,1}$ which is a finite union $\bigcup_{j=1}^{2k} X_j$ where
\begin{itemize}
  \item $X_j := \{ \exp_z(u+tw) \,|\, u\in T_zZ,\ w\in [v_j,v_{j+1}],\ t\geq 0\}$ is a convex set with nonempty interior inside a totally geodesic copy of $\H^p$ inside~$\hat{\H}^{p,1}$;
  \item the relative boundary of~$X_i$ inside~$\hat{\H}^{p,1}$ is $Y_j \cup Y_{j+1}$, where $Y_j := \{ \exp_z(u+tv_j) \,|\linebreak u\in T_zZ,\ t\geq 0\}$ is a half-space inside a totally geodesic copy of~$\H^{p-1}$ inside $\hat{\H}^{p,1}$;
  \item $X_j \cap X_{j+1} = Y_{j+1}$ and $Y_j \cap Y_{j+1} = Z$ for all $j\in\Z/2k\Z$.
\end{itemize}
As in \cite[Lem.\,5.6]{mst}, we can then construct an atlas of charts on $N^{k/n}$, with values in $\hat{\H}^{p,1}$, in the following way.
For each $x\in N^{k/n}$, choose a small connected open neighborhood $U_x$ of $x$ in $N^{k/n}$ such that for any $j\in\Z/2k\Z$, if $x \notin \ov{H_j}$, then $U_x \cap \ov{H_j} = \emptyset$; choose a continuous chart $\varphi_x : U_x \to \hat{\H}^{p,1}$ mapping isometrically $U_x\cap V_j$ into~$X_j$ and $U_x\cap H_j$ into~$Y_j$ for all $j\in\Z/2k\Z$.
We then consider the atlas of charts $(U_x,g\circ\varphi_x)_{x\in N^{k/n},\,g\in\SO(p,2)_0}$.
Note that each transition map is given by a unique element of $\SO(p,2)_0$, because if $\mathcal{U}_1$ and~$\mathcal{U}_2$ are two relatively open subsets inside two copies of $\H^p$ inside $\hat{\H}^{p,1}$, then any orientation-preserving and time-orientation-preserving isometry between $\mathcal{U}_1$ and~$\mathcal{U}_2$ is given by a unique element of $\SO(p,2)_0$.
Therefore, as explained in \cite[Cor.\,3.32]{mst}, one can argue similarly to the case of classical $(G,X)$-structures to get that this atlas of charts defines a developing map from the universal cover $\widetilde{N}^{k/n}$ of $N^{k/n}$ to~$\hat{\H}^{p,1}$ which is equivariant with respect to a holonomy representation $\rho : \pi_1(N^{k/n}) \to \SO(p,2)$.
By \cite[Th.\,3.33]{mst}, the image of this developing map is a spacelike $p$-graph in $\hat{\H}^{p,1}$, and the representation $\rho : \pi_1(N^{k/n}) \to \SO(p,2)$ is $\H^{p,1}$-convex cocompact.
We note that for any $j\in\Z/2k\Z$, there are natural embeddings $\pi_1(H_j) \hookrightarrow \pi_1(V_j) \hookrightarrow \pi_1(N^{k/n})$, and the Zariski closure of $\rho(\pi_1(H_j))$ (\resp $\rho(\pi_1(V_j))$) in $\SO(p,2)$ is a copy of $\SO(p-1,1)$ (\resp $\SO(p,1)$),
preserving the copy of $\H^{p-1}$ (\resp $\H^p$) in~$\hat{\H}^{p,1}$ with tangent space $T_zZ \oplus \R v_j$ (\resp $T_zZ \oplus \spa(v_j,v_{j+1})$) at~$z$.
Since the vectors $v_j$ are not all collinear, the groups $\pi_1(V_j)$ do not all preserve the same copy of $\H^p$ in $\hat{\H}^{p,1}$;
therefore the group $\rho(\pi_1(N^{k/n}))$ does not preserve a copy of $\H^p$ in $\hat{\H}^{p,1}$, and so it is Zariski-dense in $\SO(p,2)$ by Lemma~\ref{lem:Lie-alg}.\eqref{item:Lie-alg-1}.

By Lemma~\ref{lem:deform-polygon-R21}, up to a continuous deformation inside the space of equilateral spacelike $2k$-polygons of side length $\pi/n$ in $\R^{2,1}$ (which induces a continuous deformation inside the space of $\H^{p,1}$-convex cocompact representations of $\pi_1(N^{k/n})$), we may assume that our polygon $\mathcal{P} = (v_j)_{j\in\Z/2k\Z}$ satisfies that $\spa(v_{j-1},v_j,v_{j+1}) = \R^{2,1}$ for all~$j$.
By the above argument (based on Lemma~\ref{lem:Lie-alg}.\eqref{item:Lie-alg-1}), the image under~$\rho$ of the subgroup generated by $\pi_1(V_{j-1})$ and $\pi_1(V_j)$ is then Zariski-dense in $\SO(p,2)$ for all~$j$.

Let us fix $q\leq 2k-2$, and view $\rho$ as a representation from $\pi_1(N^{k/n})$ to $\SO(p,q+1)$ by composing with the natural inclusion $\SO(p,2) \hookrightarrow \SO(p,q+1)$; then $\rho$ is $\H^{p,q}$-convex cocompact (see \cite{dgk18,dgk-proj-cc}).
If $q=1$, then we already know that $\rho(\pi_1(N^{k/n}))$ is Zariski-dense in $\SO(p,q+1)$, so we now assume $q\geq 2$.
Our goal is to deform $\rho$ continuously into a representation whose image is Zariski-dense in $\SO(p,q+1)$; then $\H^{p,q}$-convex cocompactness will still hold by Theorem~\ref{thm:main}.

The deformation will be done by $q-1$ successive bendings \`a la Johnson--Millson (see Section~\ref{subsec:appendix-hyperbolic}).
For this, we observe that for every $j\in\Z/2k\Z$, the union of the codimension-$1$ submanifolds $H_{j-1}\cup H_{j+1}$ is path-connected and separates $N^{k/n}$ into two connected components, namely $V_{j-1} \cup V_j \cup H_j \smallsetminus S$ and $N^{k/n} \smallsetminus (\overline{V_{j-1}} \cup \overline{V_j})$.
Therefore, by van Kampen's theorem, we can write
\begin{equation} \label{eqn:Gromov-Thurston-amalgam}
\pi_1(N^{k/n}) = \pi_1\big(V_{j-1} \cup V_j \cup H_j \smallsetminus S\big) *_{\pi_1(H_{j-1}\cup H_{j+1})} \pi_1\big(N^{k/n} \smallsetminus (\overline{V_{j-1}} \cup \overline{V_j})\big).
\end{equation}

Consider our point $z \in \hat{\H}^{p,1} \subset \hat{\H}^{p,q}$ and our totally geodesic copy $Z$ of $\H^{p-2}$ containing~$z$ in $\hat{\H}^{p,1} \subset \hat{\H}^{p,q}$ as above.
We write the tangent space $T_z\hat{\H}^{p,q} \simeq \R^{p,q}$ as the direct sum of $T_zZ \simeq \R^{p-2,0}$ and of its orthogonal $(T_zZ)^{\perp} \simeq \R^{2,q}$.
Let $(e_1,\dots,e_{q+2})$ be an orthogonal basis of $(T_zZ)^{\perp} \simeq \R^{2,q}$ in which the quadratic form $\sfb$ has matrix $\mathrm{diag}(1,1,-1,\dots,-1)$.
By construction, we have $\spa(v_{j-1},v_j,v_{j+1}) = \spa(e_1,e_2,e_3)$ for all $j\in\Z/2k\Z$.
The subgroup of $\SO(p,q+1)$ fixing $Z$ pointwise is isomorphic to $\SO(2,q)$, acting on $(T_zZ)^{\perp} = \spa(e_1,\dots,e_{q+2})$.

We proceed in $q-1$ steps: for $1\leq j\leq q-1$, the $j$-th step produces a continuous family $(\rho_t^{(j)})_{t\in [0,1]} \subset \Hom(\pi_1(N^{k/n}),\SO(p,q+1))$ and a corresponding continuous deformation $(v_{j,t})_{t\in [0,1]} \subset \spa(e_1,\dots,e_{j+3}) \smallsetminus \spa(e_1,\dots,e_{j+2})$ of~$v_j$  such that
\begin{itemize}
  \item for any $t\neq 0$, the group $\rho_t^{(j)}(\pi_1(N^{k/n}))$ is contained and Zariski-dense in $\SO(p,j+2)$,
  \item for any $t\neq 0$, the image under $\rho_t^{(j)}$ of the subgroup of $\pi_1(N^{k/n})$ generated by $\pi_1(V_1), \dots$, $\pi_1(V_j), \pi_1(V_{2k-2}), \pi_1(V_{2k-1})$ is already Zariski-dense in $\SO(p,j+2)$,
  \item for any~$t$, we have $\sfb(v_{j,t},v_{j,t}) = 1$ and $\sfb(v_{j,t},v_{j+1}) = \cos(\pi/n)$, and the restriction of $\sfb$ to $\spa(v_{j,t},v_{j+1},v_{j+2})$ has signature $(2,1|0)$.
\end{itemize}
The family $(\rho_t^{(j)})_{t\in [0,1]}$ of the $j$-th step is related to the family $(\rho_t^{(j-1)})_{t\in [0,1]}$ of the $(j-1)$-th step by $\rho_0^{(j)} = \rho_1^{(j-1)}$, so that we can then concatenate the deformation paths and obtain a continuous deformation of~$\rho$ with Zariski-dense image in $\SO(p,q+1)$.
We now explain how to perform the first step, and how to perform the $j$-th step having performed the $(j-1)$-th step.

We start with the first step.
By construction, we have $\sfb(v_1,v_1) = 1$ and $\sfb(v_0,v_1) = \sfb(v_1,v_2) = \cos(\pi/n)$, and $v_1 \in \spa(e_1,e_2,e_3) \smallsetminus \spa(v_0,v_2)$.
Therefore Lemma~\ref{lem:deform-triple-R2q}.\eqref{item:triple-R2q-3} gives a continuous family $(g_t^{(1)})_{t\in [0,1]}$ of elements of $\SO(p,3) \subset \SO(p,q+1)$ fixing $Z$ pointwise and fixing $v_0$ and $v_2$ inside $(T_zZ)^{\perp} = \spa(e_1,\dots,e_{q+2})$, such that $g_0^{(1)} = \mathrm{Id}$, and such that for any $t \in (0,1]$, the vector $v_{1,t} := g_t^{(1)} \cdot v_1$ belongs to $\spa(e_1,e_2,e_3,e_4) \smallsetminus \spa(e_1,e_2,e_3)$ and still satisfies $\sfb(v_{1,t},v_{1,t}) = 1$ and $\sfb(v_0,v_{1,t}) = \sfb(v_{1,t},v_2) = \cos(\pi/n)$; moreover, we may assume that the restriction of $\sfb$ to $\spa(v_{1,t},v_2,v_3)$ still has signature $(2,1|0)$ for all $t\in [0,1]$.
By \eqref{eqn:Gromov-Thurston-amalgam}, for each~$t$ we can define a representation $\rho_t^{(1)} : \pi_1(N^{k/n}) \to \SO(p,q+1)$ to be equal to $\rho$ on $\pi_1(N^{k/n} \smallsetminus (\overline{V_0} \cup \overline{V_1}))$, and to $g_t^{(1)}\rho(\cdot) (g_t^{(1)})^{-1}$ on $\pi_1(V_0 \cup V_1 \cup H_1 \smallsetminus S)$.
Indeed, since $g_t^{(1)}$ fixes $Z$ pointwise and fixes $v_0$, it acts trivially on the copy of $\H^{p-1}$ in $\hat{\H}^{p,q}$ with tangent space $T_zZ \oplus \R v_0$ at~$z$, hence it centralizes $\rho(\pi_1(H_0))$; similarly, since $g_t^{(1)}$ fixes $Z$ pointwise and fixes $v_2$, it centralizes $\rho(\pi_1(H_2))$; therefore $g_t^{(1)}$ centralizes $\rho(\pi_1(H_0\cup H_2))$.
We thus obtain a continuous family $(\rho_t^{(1)})_{t\in [0,1]} \subset \Hom(\pi_1(N^{k/n}),\SO(p,q+1))$ with values in $\SO(p,3)$, such that $\rho_0^{(1)} = \rho$.
We claim that for any $t\neq 0$, the image under $\rho_t^{(1)}$ of the subgroup $\Gamma_1$ of $\pi_1(N^{k/n})$ generated by $\pi_1(V_1), \pi_1(V_{2k-2}), \pi_1(V_{2k-1})$ is Zariski-dense in $\SO(p,3)$.
Indeed, $\rho_t^{(1)}$ coincides with~$\rho$ on the subgroup generated by $\pi_1(V_{2k-2})$ and~$\pi_1(V_{2k-1})$, hence the Zariski closure of $\rho_t^{(1)}(\Gamma_1)$ contains $\SO(p,2)$ by the above.
Moreover, $\rho_t^{(1)}(\pi_1(V_1))$ is Zariski-dense in the group $g_t^{(1)}\SO(p,1)(g_t^{(1)})^{-1}$, which preserves the copy of $\H^p$ with tangent space $T_zZ \oplus \spa(v_0,v_{1,t})$ at~$z$.
Since $v_{1,t} \notin \spa(e_1,e_2,e_3)$, the group $\rho_t^{(1)}(\pi_1(V_1))$ does not preserve $\hat{\H}^{p,1}$; therefore the group generated by $\rho_t^{(1)}(\pi_1(V_1)), \rho_t^{(1)}(\pi_1(V_{2k-2})), \rho_t^{(1)}(\pi_1(V_{2k-1}))$ does not preserve $\hat{\H}^{p,1}$, and so it is Zariski-dense in $\SO(p,3)$ by Lemma~\ref{lem:Lie-alg}.\eqref{item:Lie-alg-1}.

Suppose we have performed the $(j-1)$-th step of our process: we have obtained a continuous family $(\rho_t^{(j-1)})_{t\in [0,1]} \subset \Hom(\pi_1(N^{k/n}),\SO(p,q+1))$ and a corresponding continuous deformation $(v_{j-1,t})_{t\in [0,1]} \subset \spa(e_1,\dots,e_{j+2}) \smallsetminus \spa(e_1,\dots,e_{j+1})$ of~$v_{j-1}$ such that
\begin{itemize}
  \item for any $t\neq 0$, the group $\rho_t^{(j-1)}(\pi_1(N^{k/n}))$ is contained and Zariski-dense in $\SO(p,j+1)$,
  \item for any $t\neq 0$, the image under $\rho_t^{(j-1)}$ of the subgroup of $\pi_1(N^{k/n})$ generated by $\pi_1(V_1), \dots, \pi_1(V_{j-1}), \pi_1(V_{2k-2}), \pi_1(V_{2k-1})$ is already Zariski-dense in $\SO(p,j+1)$,
  \item for any~$t$, we have $\sfb(v_{j-1,t},v_{j-1,t}) = 1$ and $\sfb(v_{j-1,t},v_j) = \cos(\pi/n)$, and the restriction of $\sfb$ to $\spa(v_{j-1,t},v_j,v_{j+1})$ has signature $(2,1|0)$.
\end{itemize}
Let us construct $(\rho_t^{(j)})_{t\in [0,1]}$ and $(v_{j,t})_{t\in [0,1]}$, with $\rho_0^{(j)} = \rho_1^{(j-1)}$.
By Lemma~\ref{lem:deform-triple-R2q}.\eqref{item:triple-R2q-3}, there is a continuous family $(g_t^{(j)})_{t\in [0,1]}$ of elements of $\SO(p,j+2) \subset \SO(p,q+1)$ fixing $Z$ pointwise and fixing $v_{j-1,1}$ and $v_{j+1}$ inside $(T_zZ)^{\perp} = \spa(e_1,\dots,e_{q+2})$, such that $g_0^{(j)} = \mathrm{Id}$, and such that for any $t \in (0,1]$, the vector $v_{j,t} := g_t^{(j)} \cdot v_j$ belongs to $\spa(e_1,\dots,e_{j+3}) \smallsetminus \spa(e_1,\dots,e_{j+2})$ and still satisfies $\sfb(v_{j,t},v_{j,t}) = 1$ and $\sfb(v_{j-1,1},v_{j,t}) = \sfb(v_{j,t},v_{j+1}) = \cos(\pi/n)$; moreover, we may assume that the restriction of $\sfb$ to $\spa(v_{j,t},v_{j+1},v_{j+2})$ still has signature $(2,1|0)$ for all $t\in [0,1]$.
As above, by \eqref{eqn:Gromov-Thurston-amalgam}, for each~$t$ we can define a representation $\rho_t^{(j)} : \pi_1(N^{k/n}) \to \SO(p,q+1)$ to be equal to $\rho_1^{(j-1)}$ on $\pi_1(N^{k/n} \smallsetminus (\overline{V_{j-1}} \cup \overline{V_j}))$, and to $g_t^{(j)}\rho_1^{(j-1)}(\cdot) (g_t^{(j)} )^{-1}$ on $\pi_1(V_{j-1} \cup V_j \cup H_j \smallsetminus S)$.
We thus obtain a continuous family $(\rho_t^{(j)})_{t\in [0,1]} \subset \Hom(\pi_1(N^{k/n}),\SO(p,q+1))$, with values in $\SO(p,j+2)$, such that $\rho_0^{(j)} = \rho_1^{(j-1)}$.
We claim that for any $t\neq 0$, the image under $\rho_t^{(j)}$ of the subgroup $\Gamma_j$ of $\pi_1(N^{k/n})$ generated by $\pi_1(V_1), \dots, \pi_1(V_j), \pi_1(V_{2k-2}), \pi_1(V_{2k-1})$ is Zariski-dense in $\SO(p,j+2)$.
Indeed, $\rho_t^{(j)}$ coincides with~$\rho_1^{(j-1)}$ on the subgroup generated by $\pi_1(V_1), \dots, \pi_1(V_{j-1}), \pi_1(V_{2k-2}), \pi_1(V_{2k-1})$, hence the Zariski closure of $\rho_t^{(j)}(\Gamma_j)$ contains $\SO(p,j+1)$ by the $(j-1)$-th step.
Moreover, $\rho_t^{(j)}(\pi_1(V_j))$ is Zariski-dense in a conjugate of $\SO(p,1)$ which preserves the copy of $\H^p$ with tangent space $T_zZ \oplus \spa(v_{j,t},v_{j+1})$ at~$z$.
Since $v_{j,t} \notin \spa(e_1,\dots,e_{j+2})$, the group $\rho_t^{(j)}(\pi_1(V_j))$ does not preserve $\hat{\H}^{p,j}$; therefore the group $\rho_t^{(j)}(\Gamma_j)$ also does not preserve $\hat{\H}^{p,j}$, and so it is Zariski-dense in $\SO(p,j+2)$ by Lemma~\ref{lem:Lie-alg}.\eqref{item:Lie-alg-1}.

In the end, by concatenating the continuous families $(\rho_t^{(j)})_{t\in [0,1]}$, for $j$ ranging from $1$ to $q-1$, we obtain a continuous deformation of~$\rho$ inside $\Hom(\pi_1(N^{k/n}),\SO(p,q+1))$ whose image is eventually Zariski-dense in $\SO(p,q+1)$.
These representations are still $\H^{p,q}$-convex cocompact by Theorem~\ref{thm:main}.
This completes the proof of Proposition~\ref{prop:Z-dense-deform-MST}.

\vspace{0.5cm}

\end{document}